\documentclass[11pt]{amsart}
\usepackage{amssymb}
\usepackage{amsmath} 
\usepackage{amsthm} 
\usepackage{epsfig}
\usepackage{graphicx}
\usepackage{color}
\usepackage{transparent}
\usepackage{soul}
\usepackage{subfig}
\usepackage{wrapfig}
\usepackage[colorinlistoftodos]{todonotes}

\newcommand{\fun}{\pi_1(X(\mathbf{R}), P_{fix})}

\newtheorem*{thma}{Theorem A}

\newtheorem*{thmb}{Theorem B}

\newtheorem*{thmc}{Theorem C}

\newtheorem*{thmIntro}{Theorem 1}

\newtheorem{thm}{Theorem}[section]
\newtheorem{lem}[thm]{Lemma}

\newtheorem{rem}[thm]{Remark}

\newtheorem{prop}[thm]{Proposition}

\newcommand{\cc}{\,}
\newcommand{\bg}{\phantom{ASEFG}}

\title[Entropy of real automorphisms]{Entropy of real rational surface automorphisms: Actions on the Fundamental Groups}

\author{Kyounghee Kim}
\address{Department of Mathematics\\
         Florida State University\\
         Tallahassee, FL 32308}
\email{kim@math.fsu.edu}

\author{Eric P. Klassen}
\address{Department of Mathematics\\
         Florida State University\\
         Tallahassee, FL 32308}
\email{klassen@math.fsu.edu}


\begin{document}

\maketitle

\begin{abstract}
This article discusses a method to compute the induced action on the fundamental group of a real rational surface and provides the induced actions for basic quadratic real automorphisms. Using an invariant set in the fundamental group, we introduce a method to estimate a lower bound of the action's growth rate on the fundamental group. The growth rate of this induced action gives a lower bound for the entropy of real surface automorphism. 
\end{abstract}

\section*{Introduction}
If an automorphism on a complex space preserves a set of real points, it is important to understand how the complex dynamics and the real dynamics are related to each other. In this paper, we consider a family of rational surface automorphisms fixing sets of real points. We compute new lower bounds for the topological entropy of these automorphisms on the corresponding real surfaces. Our main tool is to compute the growth rate of the action of an automorphism on the fundamental group of the surface. This growth rate is known \cite{Bowen} to give a lower bound for the topological entropy. Using ideas from combinatorial group theory, we are able to carry out these computations. In certain cases, our method gives strictly better lower bounds for the topological entropy than those previously computed using, for example, the action of the automorphism on the homology of the surface.

\vspace{1ex}


Let $X(\mathbf{C})$ denote a complex rational surface and let $f:X(\mathbf{C}) \to X(\mathbf{C})$ be a biholomorphism of $X(\mathbf{C})$. The spectral radius $\rho(f^*|_{H^*(X(\mathbf{C}))})$ of the induced linear map $f^*: H^*(X(\mathbf{C})) \to H^*(X(\mathbf{C}))$ on cohomology is related to the topological entropy of $f$, denoted by $h_{top}(f)$, which is an important dynamical invariant. Due to Gromov and Yomdin \cite{Gromov, Yomdin}, we have 
\[ h_{top}(f) \ = \log \, \rho(f^*|_{H^*(X(\mathbf{C}))}).\]  
If $X$ is a compact complex surface that admits an automorphism with positive entropy, Cantat \cite{Cantat:1999} has shown that there are only three possibilities : either $X$ is a torus, a quotient of a K$3$ surface, or a rational surface. Thus, the surfaces admitting automorphisms with positive entropy are quite special. However it has been shown that there are infinite families of rational surfaces that admit automorphisms with positive entropy, as we describe below.

Suppose that $f:X(\mathbf{C}) \to X(\mathbf{C})$ is a rational surface automorphism with positive entropy. Due to a result of Nagata \cite{Nagata, Nagata2, Dolgachev-Ortland}, it is known that $X(\mathbf{C})$ must be equivalent to a blowup of $\mathbb{P}^2(\mathbf{C})$ along a set of points $P \subset \mathbb{P}^2(\mathbf{C})$. In addition, it is well known that $f$ must be conjugate to a composition $f_1 \circ \cdots \circ f_N$, where each $f_j$ is a quadratic map  of the form $f_j = S_j \circ J \circ T_j$ where $S_j$ and $T_j$ are linear and $J$ denotes the Cremona involution 
\[J : \mathbb{P}^2(\mathbf{C})\to \mathbb{P}^2(\mathbf{C}) \ni [x_1:x_2: x_3] \mapsto [x_2 x_3:x_1 x_3: x_1 x_2] \in \mathbb{P}^2(\mathbf{C}). \] 
(See \cite{Cerveau:2008} for further information.)
 We say a quadratic birational map $ f$ is ``\textit{basic}" if $ f= S \circ J \circ T$ for some $S,T \in \text{GL}(3, \mathbf{C})$ and $ f$ is ``\textit{real}" if $S,T \in \text{GL}(3, \mathbf{R})$. Here we focus on rational surfaces with basic real quadratic automorphisms.
 
 \vspace{1ex}
 Each basic quadratic birational map $f$ on $\mathbb{P}^2(\mathbf{C})$ has three exceptional lines $E_i^+, i=1,2,3$ and three points of indeterminacy $e_i^+$. For a basic quadratic birational map, we denote by $n_i \in \mathbf{N} \cup \{ \infty\}$ the length of the orbit of $E_i^+$ and denote by $\sigma \in \Sigma_3$ the permutation determined by the last points (if they exist) of the orbits of exceptional lines: 
 \begin{equation*}
 \begin{aligned}
 & f^{n_i} ( E^+_i \setminus \{ e_1^+, e_2^+, e_3^+\}) = e_{\sigma(i)}^+ , & \\
 &  \text{dim} \, f^{j} ( E^+_i \setminus \{ e_1^+, e_2^+, e_3^+\}) =0 \ \text{for } j = 1,2, \dots, n_i,\ \text{and}& \\
 &   \text{dim} \, f^{n_i+1} ( E^+_i \setminus \{ e_1^+, e_2^+, e_3^+\}) =1,  \quad &\text{if } n_i <\infty \\
   && \\
 & \text{dim} \, f^j ( E^+_i \setminus \{ e_1^+, e_2^+, e_3^+\}) =0 \  \text{for } j = 1,2,3,\dots , \quad &\text{if } n_i = \infty \\
 \end{aligned}
 \end{equation*}
 We call the data given by $(n_1,n_2,n_3)$ and the permutation $\sigma$ the ``\textit{orbit data}" of $ f$. From \cite{McMullen:2002, Bedford-Kim:2004}, we see that if the orbit data associated with a basic quadratic map $ f$ satisfy $n_i < \infty$ for all $i=1,2,3$, then $\check f$ lifts to an automorphism $f$ on a blowup $ X(\mathbf{C}) = B\ell _P \, (\mathbb{P}^2(\mathbf{C}))$ along a finite set of points $P = \{ f^j E_i^+, j = 1, \dots n_i, i = 1,2,3\} \subset \mathbb{P}^2(\mathbf{C})$. If the orbit data of a birational map $\check f$ are given by three positive integers $n_1, n_2, n_3$ with a permutation $\sigma \in \Sigma_3$, we say $ f$ {\em realizes} the orbit data $(n_1, n_2, n_3)$ and $\sigma$. It is still an open question whether every orbit data can be realized. Under the assumption of the existence of an invariant cubic, McMullen \cite{McMullen:2007} and Diller \cite{Diller:2011} gave a systematic way to construct a realization for set of orbit data. Diller's construction in \cite{Diller:2011} provides explicit formulas for those realizations and shows that for each realized orbit data, there are exactly two (up to linear conjugacy) real maps properly fixing a given invariant cubic: a volume increasing map and a volume decreasing map.

\vspace{1ex}
We may consider $\mathbb{P}^2(\mathbf{R})$ as a subset of $\mathbb{P}^2(\mathbf{C})$.
If $X(\mathbf{C})= \text{B}\ell _P (\mathbb{P}^2(\mathbf{C}))$, a blowup of $\mathbb{P}^2(\mathbf{C})$ along a set of points $P \subset \mathbb{P}^2(\mathbf{R})$, then we may consider 
$X(\mathbf{R}) = \text{B}\ell _P(\mathbb{P}^2(\mathbf{R}))$ as the ``\textit{real part}" of $X(\mathbf{C})$.  In this case we know that the entropy of $f_\mathbf{R}$ satisfies 
\[ 0 \le h_{top} (f_\mathbf{R})  \le h_{top} (f).\]
It is of interest to know the relation between two dynamical systems $f:X(\mathbf{C}) \to X(\mathbf{C})$ and $f_\mathbf{R}:X(\mathbf{R}) \to X(\mathbf{R})$. In particular, it is interesting to know if the topological entropy of the real restriction, $h_{top}(f_\mathbf{R})$ is equal to the topological entropy of the complex map, $h_{top}(f)$. Since $h_{top}(f_\mathbf{R}) \le h_{top}(f)$, we say $f_\mathbf{R}$ has \textit{maximal entropy} if equality holds.

\vspace{1ex}

The object of this paper is to study the action $f_{\mathbf{R}*}|_{\pi_1(X(\mathbf{R}))}$ on the fundamental group of $X(\mathbf{R})$. Bowen \cite{Bowen} has considered this and shown that the rate of the growth of  $f_{\mathbf{R}*}|_{\pi_1(X(\mathbf{R}))}$ gives a lower bound of the topological entropy of $f_{\mathbf{R}}$. The growth rate of the action on the fundamental group is defined by \[ \rho (f_{\mathbf{R}*}|_{\pi_1(X(\mathbf{R})}) := \sup_{g \in G} \{ \limsup_{n \to \infty} ( \ell_G (f^n_{\mathbf{R}*} g))^{1/n}\} \]
where $G$ is a set of generators of $\pi_1(X(\mathbf{R}))$, $\ell_G(w)$ is the minimal length among all words representing $w$ with respect to a fixed set of generators of $G$. 

\vspace{1ex}

The spectral radius of the pull-back operator on $H^{1,1}(X(\mathbf{C});\mathbf{C})$ (known as the dynamical degree) is birationally invariant, and it is known that due to Diller and Favre \cite{Diller-Favre:2001} one can always compute the dynamical degree of a birational map on $\mathbb{P}^2(\mathbf{C})$. In \cite{Bedford-Kim:2004}, it is shown that the dynamical degree is determined by the \textit{orbit data} of a birational map. Cantat and Blanc \cite{blancdynamical} showed that a birational map on $\mathbb{P}^2(\mathbf{C})$ is birationally equivalent to a rational surface automorphism if and only if  the dynamical degree is a Salem number.  Computing the induced action of rational surface automorphisms on cohomology classes is now a standard computation. (For a more general setup, an article by Koch and Roeder \cite{Koch-Roeder:2016} is a good source.) We have $\pi_1(X(\mathbf{C})) \equiv \mathbf{Z} / 2 \mathbf{Z}$, but  $\pi_1(X(\mathbf{R})) \not\equiv \mathbf{Z} / 2 \mathbf{Z}$. Thus we have a possibility of a new invariant: the growth rate of $f_{\mathbf{R}*}$ on $\pi_1(X(\mathbf{R}))$. 



\vspace{1ex}
Recent work of Diller and the first author \cite{Diller-Kim} computed the induced action on $H_1(X(\mathbf{R});\mathbf{R})$ for automorphisms obtained from real quadratic birational maps properly fixing an invariant cubic. It has shown that both the forward map and the backward map have the same growth rate of the homology actions:
\begin{thmIntro}[{\cite[Theorem~3.1]{Diller-Kim}}]\label{T:Invhomology} Suppose $f_\mathbf{R}$ is a real diffeomorphism associated a basic quadratic birational map fixing an invariant cubic. Then both induced actions $f_{\mathbf{R}*}|_{H_1(X(\mathbf{R});\mathbf{R})}$ and $f^{-1} _{\mathbf{R}*}|_{H_1(X(\mathbf{R});\mathbf{R})}$ have the same growth rate:
 \[\rho(f_{\mathbf{R}*}|_{H_1(X(\mathbf{R});\mathbf{R})}) = \rho(f^{-1}_{\mathbf{R}*}|_{H_1(X(\mathbf{R});\mathbf{R})}) \]
 where $\rho$ denotes the spectral radius.
 \end{thmIntro}
It is also shown that infinite families of real diffeomorphisms with maximal entropy exist. There are six real diffeomorphisms that have zero homology growth. Since the growth rate of homology classes is merely a lower bound for the topological entropy, zero exponential growth rate on $H_1(X(\mathbf{R});\mathbf{R})$ does not imply zero entropy.

\vspace{1ex}
Since the homology group $H_1(X(\mathbf{R})$ is the abelianization of the fundamental group $\pi_1(X(\mathbf{R}))$, it is clear that the growth rate of the action on the fundamental group is a better estimate of the topological entropy than the growth rate of the action on the homology group
\[\log \rho( f_{\mathbf{R}*}|_{H_1(X(\mathbf{R}))} ) \ \le\  \log \rho( f_{\mathbf{R}*}|_{\pi_1(X(\mathbf{R}))}) \ \le\  h_{top}(f_\mathbf{R}).\]
The fact that the real surface $X(\mathbf{R})$ is non-orientable makes computation of $f_{\mathbf{R}*}|_{H_1(X(\mathbf{R}))}$ complicated.  Since $\pi_1(X(\mathbf{R}))$ is non-abelian, computing the action $f_{\mathbf{R}*}|_{\pi_1(X(\mathbf{R}))}$ is delicate. To deal with $\pi_1(X(\mathbf{R}))$, there are several items to discuss: how to choose generators for $\pi_1(X(\mathbf{R}))$, how to write $\gamma \in \pi_1(X(\mathbf{R}))$, and how to determine the minimal length of each word in a given set of generators. In this paper, we focus on a family of real rational surfaces and their maps. We find sets of generators for which the induced action $ f_{\mathbf{R}*}|_{\pi_1(X(\mathbf{R}))}$ has interesting yet simple behavior.  The behavior of the induced action $ f_{\mathbf{R}*}|_{\pi_1(X(\mathbf{R}))}$ is very close to substitution rules. It suggests a possibility of the existence of a ``\textit{limit}", $\Gamma = \lim_{n \to \infty} f^n_{\mathbf{R}*} \gamma $ for  $\gamma \in \pi_1(X(\mathbf{R}))$. If the limit exists, it would be exciting to know the connection between $\Gamma$ and the boundary of the repelling basin of the invariant cubic and it would provide a better understanding of the boundary of the repelling basin of the invariant cubic.

\vspace{1ex}

This paper's main contributions are introducing a method to compute the induced action on the fundamental group of real diffeomorphisms. Also, for certain rational surface automorphisms we introduce a method to estimate the growth rate of the induced action on the fundamental group and observe a new pheonomenon of the induced action. For every case we investigated (including ones discussed in this paper), we noticed a very striking phenomenon on the iterations of the induced action on the fundamental group: There is a set $\Gamma$ of conjugacy classes in the fundamental group with the property that after the first few iterations their images can be written as products of positive powers of elements in $\Gamma$. We also noticed that only a limited number of successive relations in elements of $\Gamma$ appear in the iterations. This suggests that there are sets of admissible sequences given by positive products of elements in $\Gamma$ and the length growth on the admissible sequences can be obtained by simply counting the number of elements in $\Gamma$.

\vspace{1ex}

In this paper, we use \textit{reading curves} discussed in \cite{Coh} to compute the induced action on the fundamental group for diffeomorphisms associated with quadratic birational map fixing cusp cubic. This induced action on the fundamental group is not linear. To estimate its growth rate, we adapt the idea in Birman and Series \cite{Birman-Series}. We first find a set of invariant conjugacy classes in the fundamental group, and then connect the restriction to this invariant set to linear actions. 
%
%
With our method we have following results:

\begin{thma} There is a non-empty set $\Lambda_1$ of orbit data such that a real rational surface automorphism $f: X \to X$ associated to orbit data in $\Lambda_1$ has the following properties 
\[0= \log \rho( f_{\mathbf{R}*}|_{H_1(X(\mathbf{R});\mathbf{R})} )=\log \rho( f^{-1}_{\mathbf{R}*}|_{H_1(X(\mathbf{R});\mathbf{R})} ) \lneq \log \rho (f^{-1}_{\mathbf{R*}}|_{\pi_1(X(\mathbf{R}))}) = h_{top} (f_\mathbf{R}) = h_{top} (f). \]
In particular $\Lambda_1$ contains the following orbit data 
\begin{itemize} 
\item $\sigma : 1\mapsto 2\mapsto 3 \mapsto 1$, and $n_1 = 1, n_2 = 3, n_3=9$. 
\item $\sigma : 1\mapsto 2\mapsto 3 \mapsto 1$, and $n_1 = 1, n_2 = 4, n_3=8$.
\end{itemize}
\end{thma}

\begin{thmb} There is a non-empty set $\Lambda_2$ of orbit data such that a real rational surface automorphism $f: X \to X$ associated to orbit data in $\Lambda_2$ has the following properties 
\[0 < \log \rho( f_{\mathbf{R}*}|_{H_1(X(\mathbf{R});\mathbf{R})} ) =\log \rho( f^{-1}_{\mathbf{R}*}|_{H_1(X(\mathbf{R});\mathbf{R})} ) \lneq \log \rho (f^{-1}_{\mathbf{R*}}|_{\pi_1(X(\mathbf{R}))}) = h_{top} (f_\mathbf{R}) = h_{top} (f). \]
In particular $\Lambda_2$ contains the following orbit data 
\begin{itemize} 
\item $\sigma : 1\mapsto 2\mapsto 3 \mapsto 1$, and $n_1 = 1, n_2 = 4, n_3=5$. 
\item $\sigma : 1\mapsto 2\mapsto 3 \mapsto 1$, and $n_1 = 1, n_2 = 5, n_3=6$.
\item $\sigma : 1\mapsto 2\mapsto 3 \mapsto 1$, and $n_1 = 1, n_2 = 6, n_3=7$.
\end{itemize}
\end{thmb}

\begin{thmc}
There is a non-empty set $\Lambda_3$ of orbit data such that a real rational surface automorphism $f: X \to X$ associated to orbit data in $\Lambda_3$ has the following properties 
\[0= \log \rho( f_{\mathbf{R}*}|_{H_1(X(\mathbf{R});\mathbf{R})} )=\log \rho( f^{-1}_{\mathbf{R}*}|_{H_1(X(\mathbf{R});\mathbf{R})} ) \lneq \log \rho (f^{-1}_{\mathbf{R*}}|_{\pi_1(X(\mathbf{R}))}) \le h_{top} (f_\mathbf{R}) \lneq h_{top} (f). \]
In particular $\Lambda_3$ contains the following orbit data 
\begin{itemize} 
\item $\sigma : 1\mapsto 2\mapsto 3 \mapsto 1$, and $n_1 = 2, n_2 = 3, n_3=5$. 
\item $\sigma : 1\mapsto 2\mapsto 3 \mapsto 1$, and $n_1 = 3, n_2 = 4, n_3=5$.
\item $\sigma : 1\mapsto 2\mapsto 3 \mapsto 1$, and $n_1 = 3, n_2 = 4, n_3=6$.
\item $\sigma : 1\mapsto 2\mapsto 3 \mapsto 1$, and $n_1 = 3, n_2 = 5, n_3=5$.
\end{itemize}
\end{thmc}

We organize our paper in the following order. Section \ref{S:Birational} of this paper discusses some background about entropy, homology growth rate, and real rational surface automorphisms. In Section \ref{S:Faction}, we discuss some background about reading curves of generators of a fundamental group and then fix sets of generators and sets of reading curves for real diffeomorphisms associated with real quadratic birational maps. In Section \ref{S:Action}, we show how to compute the action $f^{-1}_{\mathbf{R}*}: \pi_1(X(\mathbf{R}), P_{fix}) \to \pi_1(X(\mathbf{R}),P_{fix})$ with respect to a given set of generators. Section \ref{S:Grate} shows how to estimate the growth rate of the induced action on the fundamental group using an invariant set of words. Section \ref{S:coxeter} shows how to apply our method to the real diffeomorphism associated with the orbit data $1,1,8$ and a cyclic permutation. Section \ref{S:max} gives the proof of Theorem $A$ and Theorem $B$, and Section \ref{S:nonmax} gives the proof of Theorem $C$. In Appendix $A$, we list the induced action on the fundamental group for orbit data $n_1, n_2, n_3$ with a cyclic permutation. Appendix $B$ gives necessary data for each case discussed in Theorem $A$, Theorem $B$, and Theorem $C$.


\section{Real Automorphisms and Quadratic Birational Maps}
\label{S:Birational}
Let $\check f: \mathbb{P}^2(\mathbf{C}) \dasharrow \mathbb{P}^2(\mathbf{C})$ be a quadratic birational map of the form $ \check f = T^- \circ J \circ (T^+)^{-1}$ where $J[x_1:x_2:x_3] = [x_2 x_3:x_1 x_3:x_1 x_2]$ is the Cremona involution and $T^\pm \in \text{Aut} (\mathbb{P}^2(\mathbf{C}))$. This family of birational maps has been studied by several authors \cite{Bedford-Diller-K, BedfordKim:2011, Diller:2011, McMullen:2007,Uehara:2010}. In this section, we introduce the notation and basic properties needed for the following sections.

The Cremona involution $J$ has three points of indeterminacy $\mathcal{I}(J) = \{ e_1=[1:0:0],e_2=[0:1:0],e_3=[0:0:1]\}$ and three exceptional lines $E_i = \{ x_i = 0 \}, i = 1,2,3$ such that 
\[ J: E_i \setminus \mathcal{I}(J) \ \to \ e_i. \] The quadratic birational maps $\check f, \check f^{-1}$ also have three distinct points of indeterminacy and three exceptional lines:
\[ \mathcal{I}(\check f) = \{ p_i^+ = T^+ (e_i), i =1,2,3 \} \qquad \text{and}\qquad \mathcal{I}(\check f^{-1}) = \{ p_i^- = T^- (e_i), i=1,2,3 \}, \]
\[ \text{Exc}(\check f) = \{ E_i^+ = T^+ (E_i), i =1,2,3 \} \qquad \text{and}\qquad \text{Exc}(\check f^{-1}) = \{ E_i^- = T^- (E_i), i=1,2,3 \}. \]
For any curve $V$, let us use $\check f(V)$ for the set theoretic strict transformation $\check f(V \setminus \mathcal{I}(\check f))$. We see that 
\[ \check f(E_i^+) = p_i^- \qquad \text{and} \qquad \check f^{-1} (E_i^-) = p_i^+, \ \ \forall i = 1,2,3 \]

For each $i$, we define $n_i \in \mathbf{Z} \cup \{ \infty\}$ by 
\begin{equation*}
n_i = \left \{ \begin{aligned} &\min \{ n : \text{dim} \check f^n ( E_i^+) \lneq \text{dim}\check f^{n+1} ( E_i^+) \}\\ &\infty \ \ \qquad \text{if } \text{dim} \check f^n ( E_i^+) =0\ \ \ \text{for all } n\ge 1\\\end{aligned} \right.
\end{equation*}
If $n_i <\infty$, we have $\check f^{n_i} (E_i^+) \in \mathcal{I}(\check f) = \{ p_1^+,p_2^+,p_3^+ \}$. It is known that $\check f $ lifts to a rational surface automorphism if and only if $n_1, n_2, n_3$ are finite.  
When $n_i < \infty$ for all $i = 1,2,3$, we can define a permutation $\sigma \in \Sigma_3$ given by \[ \check f^{n_i} (E_i^+) = p_{\sigma(i)}^+, \qquad \text{and thus} \qquad \check f^{-n_i}(E_{\sigma(i)}^-) = p_i^-, \ \ i =1,2,3.\]
The numerical data given by the positive integers $n_1,n_2, n_3$ and the permutation $\sigma$ is called \textit{orbit data}.

\begin{thm}\cite{Bedford-Kim:2004}
Suppose $\check f=T^- \circ J \circ (T^+)^{-1}$ is a quadratic birational map. If for all $i =1,2,3$ $n_i <\infty$ then $\check f$ lifts to an automorphism $f:X(\mathbf{C}) \to X(\mathbf{C})$, where $X(\mathbf{C} )= \text{B}\ell_P \mathbb{P}^2(\mathbf{C})$ is a blowup of $\mathbb{P}^2(\mathbf{C})$ along a finite set of (possibly infinitely near) points $P = \{ p_{i,j}=\check f^j(E_i^+), 1 \le j \le n_i, i =1,2,3 \}$.
Furthermore, the characteristic polynomial of the push-forward operator $f_* : H_2(X(\mathbf{C}),\mathbf{R}) \to H_2(X(\mathbf{C}),\mathbf{R})$ is (explicitly) determined by the orbit data. 
\end{thm}

If the linear maps $T^\pm$ restrict to automorphisms of $\mathbb{P}^2(\mathbf{R})$ and the birational map $\check f=T^- \circ J \circ (T^+)^{-1}$ lifts to an automoprhism $f:X(\mathbf{C}) \to X(\mathbf{C})$ then $f$ restricts to a diffeomorphism $f_\mathbf{R}: X(\mathbf{R}) \to X(\mathbf{R})$. In this article, we say  $\check f$ is \textit{real} if the linear maps $T^\pm$ are real. 

Suppose $C$ is a cuspidal cubic. We say a birational map $f$ {\em properly fixes} a curve $C$ if $f(C) = C$ and all points of indeterminacy $p_i^\pm$ are  regular points of $C$. For a quadratic birational map $f$, having all points of indeterminacy on the set $C_\text{reg}$ of regular points is enough to conclude that $f$ properly fixes $C$. 

\begin{prop}\cite{Diller:2011}
Let $f: \mathbb{P}^2(\mathbf{C}) \dasharrow \mathbb{P}^2(\mathbf{C})$ be a quadratic birational map. $f$ properly fixes $C$ if and only if all points of indeterminacy of $f$ and $f^{-1}$ are in $C_\text{reg}$.
\end{prop}
When a birational map $f$ fixes a cuspidal cubic, $f$ fixes the set of regular points $C_\text{reg}$. Thus $f$ induces an automorphism on $C_\text{reg} = \{ \gamma(\omega): \omega \in \mathbf{C} \}$. Hence, we can choose a parametrization $\gamma:\mathbf{C}\to {C_\text{reg}}$  and $\delta \ne 0$ such that \[ f|_{C_\text{reg}} : \gamma(\omega) \to \gamma (\delta \omega), \]$p_{fix}=\gamma(0)$ is the unique (finite) fixed point of $f_{C_{reg}}$, and $p_{cusp}= \gamma(\infty)$ is the cusp point of $C$.  We call this $\delta$ the \textit{determinant} of $f$. 

With few exceptions on the orbit data, McMullen \cite{McMullen:2007} and Diller \cite{Diller:2011} showed that there is a real quadratic birational map $\check f$ properly fixing a cuspidal cubic curve $C$ with the given orbit data $n_1, n_2, n_3 < \infty$ and a permutation $\sigma \in \Sigma_3$ such that $\check f$ lifts to an automorphism $f:X(\mathbf{C}) \to X(\mathbf{C})$ such that the entropy of $f$ is given by the log of the determinant of $f$. In this article, we will focus on the orbit data with a cyclic permutation. 

\begin{thm}\label{T:param}\cite{Diller:2011,McMullen:2007}
Suppose $n_1,n_2,n_3 \in \mathbf{Z}_{>0}$ such that $n_1+n_2+n_3 \ge 10$ and \[ \{ n_1,n_2 ,n_3 \} \not\in \{ \{ 1,n,n\}. \{2,2,n\}, \{n,n,n\} | n \in \mathbf{Z}_{>0} \}. \] 
Then there is a unique (up to linear conjugacy) real quadratic birational map $\check f$ properly fixing a cuspidal cubic curve $C$ with the orbit data $n_1,n_2, n_3$ and a cyclic permutation, $\sigma$ such that $\check f$ lifts to an automorphism $f: X(\mathbf{C}) \to X(\mathbf{C})$ where $X(\mathbf{C})$ is a blowup of $\mathbb{P}^2(\mathbf{C})$ along a finite set of points $\{ p_{i,j}, 1\le j \le n_i, 1 \le i \le 3 \}$ such that 
\begin{itemize}
\item the determinant $\delta \ne \pm1$ of $f$ is a real root of the characteristic polynomial of $f_* : H_2(X(\mathbf{C}), \mathbf{R}) \to H_2(X(\mathbf{C}), \mathbf{R})$, and
\item the indeterminate points of $\check f^{-1}$ are given by $p_i^- = \gamma(t_i)$ where \[ t_i = \frac{1+ \delta^{n_j} + \delta^{n_j+ n_k}}{1- \delta^{n_i+n_j+n_k}}, \ \ \ \sigma: i \mapsto j \mapsto k \mapsto i. \]
\end{itemize}
\end{thm}

\subsection{Real Automorphisms} Let $\check f: \mathbb{P}^2(\mathbf{C}) \dasharrow \mathbb{P}^2(\mathbf{C})$ be a real basic quadratic birational map that properly fixes the cuspidal cubic curve $C$ with the orbit data consisting of three positive integers $n_1, n_2, n_3$ with $n_1 + n_2 + n_3 \ge 10$ and a cyclic permutation $\sigma: 1\to 2\to 3$ such that the determinant $\delta >1$ and the points $p_{i,j}=\check f^j(E_i^+), 1 \le j \le n_i, i =1,2,3 $ are all distinct points in $\mathbb{P}^2(\mathbf{R})$. Then $\check f$ lifts to an automorphism $f: X(\mathbf{C}) \to X(\mathbf{C})$ with positive entropy where $X(\mathbf{C})$ is a blowup of $\mathbb{P}^2(\mathbf{C})$ along a finite set of points $\{p_{i,j}, 1 \le j \le n_i, i =1,2,3\}$. By restriction, we have a diffeomorphism $f_\mathbf{R} : X(\mathbf{R}) \to X(\mathbf{R})$ of the real slice of $X(\mathbf{C})$. Two fixed points of $f_\mathbf{R}$ are on the cuspidal cubic $C$. One fixed point, $P_{cusp}$ on the cusp is an attracting fixed point with multipliers $\delta^{-2}$ and $\delta^{-3}$. The other fixed point, $P_{fix}$ on $C_\text{regular}$ is a saddle point with multipliers $\delta$ and $\delta^{-(n_1+n_2+n_3-3)}$.
To compute the induced action on the fundamental group, it is convenient to have a repelling fixed point. Since $f_\mathbf{R}$ always has an attracting fixed point $P_{cusp}$ at the cusp of $C$, we will work with the inverse, $f_\mathbf{R}^{-1}$. Under $\check f^{-1}$, we have 
\begin{equation*}
\begin{aligned}
\check f^{-1} \ :\ & E_1^- \to p_1^+= p_{3,n_3} \to p_{3,n_3 -1} \to \cdots \to p_3^- = p_{3,1} \\
& E_2^- \to p_2^+ = p_{1,n_1} \to p_{1,n_1-1} \to \cdots \to p_2^- = p_{1,1} \\
&E_3^-\to p_3^+ = p_{2,n_2} \to p_{2,n_2-1} \to \cdots \to p_3^- = p_{2,1} \\
\end{aligned}
\end{equation*}
Thus the orbit data of $f^{-1}_\mathbf{R}$ are given by $n_3,n_1,n_2$ together with a cyclic permutation $1 \to 3\to 2 \to 1$. 

\vspace{1ex}

Due to Gromov \cite{Gromov} and Yomdin \cite{Yomdin}, it is known that the topological entropy of an automorphism $f$ on a complex projective manifold $X(\mathbf{C})$ is given by the logarithm of the spectral radius of the pushforward action on $H_2(X(\mathbf{C});\mathbf{R})$. 
\[ h_{top}(f) = \log \rho(f_*|_{H_2(X(\mathbf{C}); \mathbf{R})}). \] In this article, we let $\rho$ denote the spectral radius. However, the equality does not hold for real diffeomorphisms. For the induced diffeomorphism $f_\mathbf{R}$ on the real slice $X(\mathbf{R})$, we have 
\[ \log \rho(f_{\mathbf{R}*}|_{H_1(X(\mathbf{R}); \mathbf{R})}) \le h_{top}(f_\mathbf{R}) \le h_{top}(f) = \log \rho(f_*|_{H_2(X(\mathbf{C}); \mathbf{R})}) \]

For automorphisms $f:X(\mathbf{C}) \to X(\mathbf{C})$ associated with the orbit data $n_1,n_2,n_3$ with a permutation $\sigma \in \Sigma_3$, the linear operation $ f_*|_{H^{1,1}(X(\mathbf{C}); \mathbf{R})}$ and its characteristic polynomial are given in \cite{Bedford-Kim:2004, Diller:2011}. The characteristic polynomial of  $ f_*|_{H^{1,1}(X(\mathbf{C}); \mathbf{R})}$ associated to orbit data $n_1, n_2, n_3$ with a cyclic permutation $\sigma: 1\to 2 \to 3 \to 1$ is given by 
\begin{equation}\label{E:compchar}
\chi(t) = t- t^{n_1+n_2+n_3} + (t-1) (t^{n_1}+1) (t^{n_2}+1) (t^{n_3}+1) 
\end{equation}

 Recently Diller and Kim \cite{Diller-Kim} computed the induced action on the homology classes for real diffeomorphisms associated with real quadratic birational maps with orbit data $n_1,n_2,n_3$ with a permutation $\sigma \in \Sigma_3$.
\begin{prop} \cite{Diller-Kim}
Let $f: X(\mathbf{R}) \to X(\mathbf{R})$ be a real diffeomorphism associated the orbit data $n_1 \le n_2 \le n_3$ with a permutation $\sigma: 1 \to 2 \to 3 \to 1$, then the characteristic polynomial of the induced action $f_{\mathbf{R}*} |_{H_1(X(\mathbf{R});\mathbf{R})}$ is given by 
\begin{equation}\label{E:realchar}
\phi(t)\ =\  \frac{1}{t+1} \, \left (  g(t)  - (-t)^{n_1+n_2+n_3+1}\, g\left(\frac{1}{t}\right ) \right),
\end{equation}
where
\begin{equation*} 
\begin{aligned}
g(t) \ =\  &t^{1+n_1+n_2+n_3} + \frac{(-1)^{1+n_1} t^{n_2+n_3}  (1+t^2)}{t-1} + \frac{(-1)^{n_1+n_2+n_3} t^{n_1+n_2} (1+t^2)}{t+1} \\&+ \frac{(-1)^{n_1} t^{n_1+n_3} (1-t+3 t^2 +t^3)}{t^2-1}. \\
\end{aligned}
\end{equation*}
\end{prop}

When the homology growth rate is equal to the topological entropy of the corresponding automorphism on a complex projective manifold, one can determine that the real diffeomorphism has the maximum possible entropy. 
 
 \begin{thm} \cite{Diller-Kim}
 Suppose $f: X(\mathbf{C}) \to X(\mathbf{C})$ is an automorphism associated to the real quadratic birational map with one of the following orbit data 
 \begin{itemize}
 \item $\sigma: 1 \to 2 \to 3 \to 1$, $n_1= n_2=1$ and $n_3 \ge 8$.
 \item $\sigma: 1 \to 2 \to 3 \to 1$ , $n_1=2$, and $n_2=n_3 \ge 4$.
 \item $\sigma: 1 \leftrightarrow 2$, $n_1=1, n_2=4$ and $n_3 \ge 6$,
 \item $\sigma: 1 \leftrightarrow 2$, $n_1=1, n_2=5$ and $n_3 \ge 4$,
 \item $\sigma: 1 \leftrightarrow 2$, $n_1=1, n_2 \ge 8$ and $n_3 =2$.
 \item $\sigma$ is the identity permutation, $n_1=2, n_2 =3$ and $n_3 \ge 6$,
 \item$\sigma$ is the identity permutation, $n_1=2, n_2 =4$ and $n_3 \ge 5$
 \end{itemize}
 Then the diffeomorphism $f_\mathbf{R}$ on the real slice $X(\mathbf{R})$ has the maximum entropy: 
 \[ h_{top}(f_\mathbf{R}) = h_{top}(f) >0 \] 
 \end{thm}
 
 \begin{figure}
\centering
\subfloat[ Maximum entropy ]{{\includegraphics[width=1.8in]{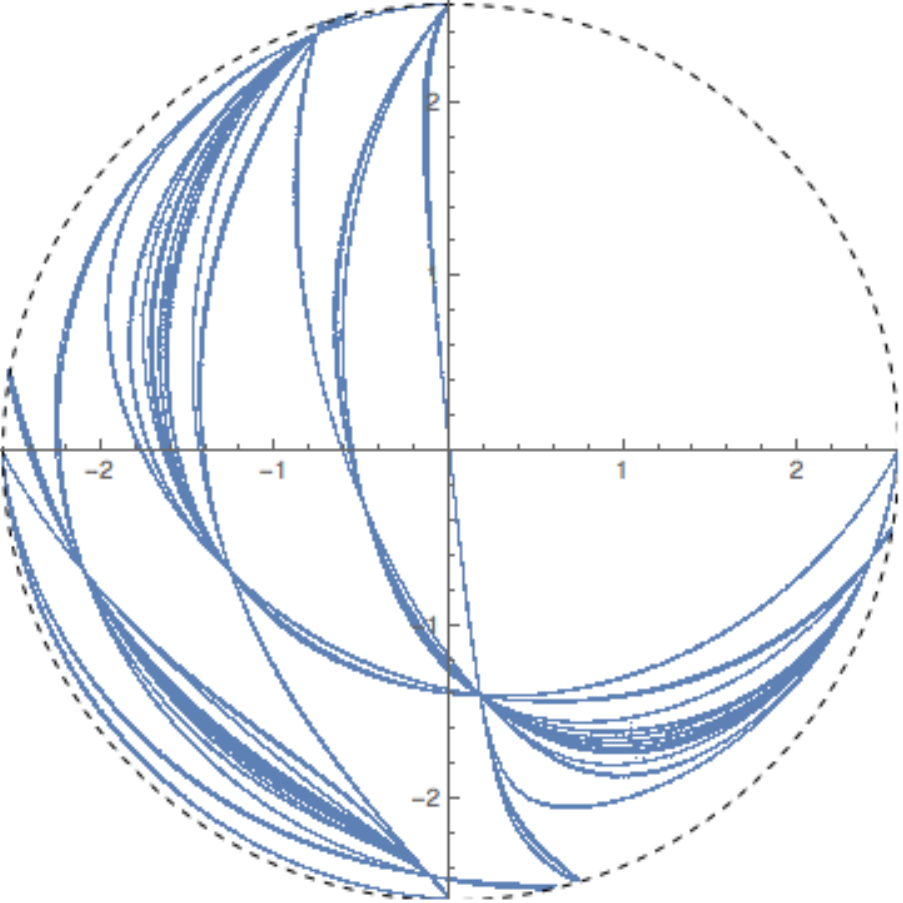} }}
\qquad\qquad
\subfloat[ Homology growth $=0$]{{\includegraphics[width=1.8in]{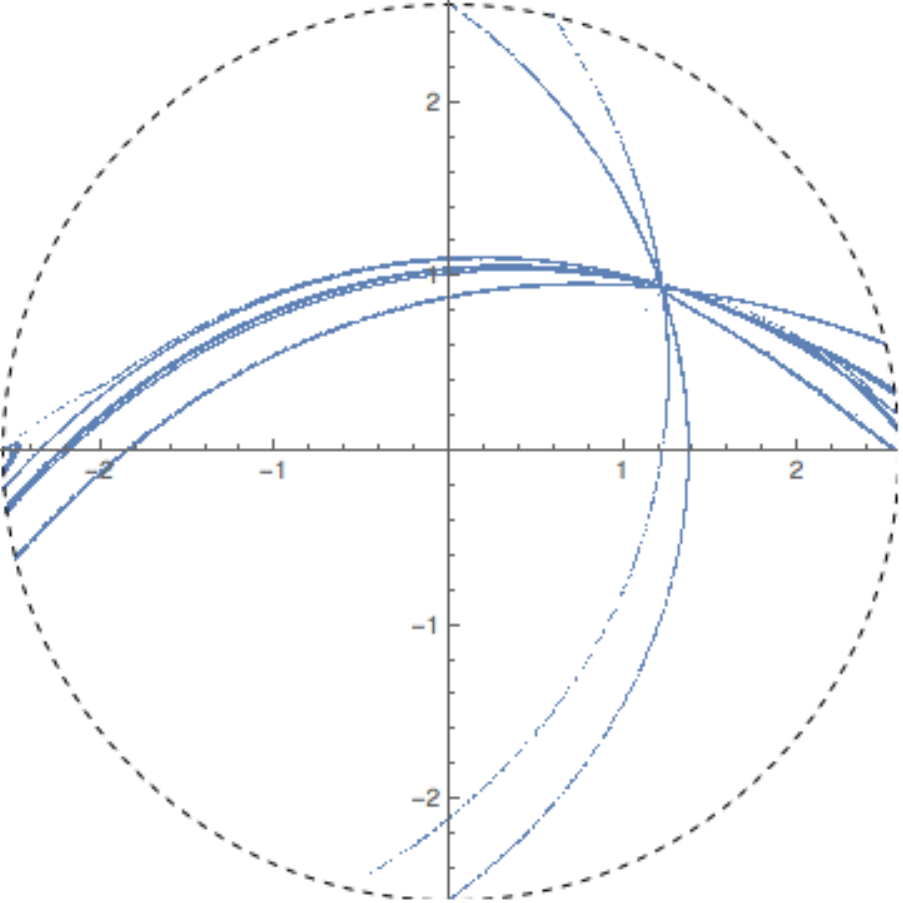} }}
\caption{ The orbit under the real automorphisms. The cusp on the invariant cubic is an attracting fixed point, and the dots represent the images in $\mathbb{P}^2(\mathbf{R})$ under $f_\mathbf{R}^{-50}$ of $600,000$ points evenly distributed on a small circle about the cusp. On the left hand side figure, $f_\mathbf{R}$ is associated to the orbit data $(1,1,8)$ and a cyclic permutation and $f_\mathbf{R}$ has the maximal homology growth. On the right hand side, $f_{\mathbf{R}}$ is associated to the orbit data $(3,4,6)$ and a cyclic permutation and $f_{\mathbf{R}*}$ is periodic with period $60$. }  
\label{fig:homologyCompare}
\end{figure}

There are orbit data such that the topological entropy of the associated real restriction is strictly smaller than the entropy of the complex map, $h_{top}(f_\mathbf{R}) \lneqq h_{top}(f)$. Due to Bedford-Lyubich-Smile \cite{BLS} and Cantat \cite{Cantat:1999, Cantat:2014}, the topological entropy of $f_\mathbf{R}$ can not be maximal if there is a saddle periodic point of $f$ whose stable and unstable manifolds are not algebraic. In \cite{Diller-Kim}, it is shown that an automorphism $f: X(\mathbf{C})\to X(\mathbf{C})$ associated to the orbit data $\sigma:1 \to 2 \to 3 \to 1$ is cyclic, $n_1=n_2=3$, and $n_3 \gg 1$ has complex saddle periodic points. One interesting observation in \cite{Diller-Kim} is that there are orbit data such that the associated automorphism $f$ has 
\[ 0 = \log \rho(f_{\mathbf{R}*}|_{H_1(X(\mathbf{R}); \mathbf{R})}) \le h_{top}(f_\mathbf{R}) \le h_{top}(f), \qquad \text{and } \qquad h_{top}(f)>1\] 

\begin{lem}
 Suppose $f: X(\mathbf{C}) \to X(\mathbf{C})$ is an automorphism associated to the real quadratic birational map with one of the following orbit data. Then $f_{\mathbf{R}*}|_{H_1(X(\mathbf{R}); \mathbf{R})}$ has spectral radius one.
  \begin{itemize}
 \item $\sigma: 1 \to 2 \to 3 \to 1$, $n_1=1,n_2=4, n_3=8$ 
 \item $\sigma: 1 \to 2 \to 3 \to 1$, $n_2=1,n_2=3, n_3=5$ 
  \item $\sigma: 1 \to 2 \to 3 \to 1$, $n_1=3,n_2=4, n_3=5$ 
 \item $\sigma: 1 \to 2 \to 3 \to 1$, $n_1=3,n_2=4, n_3=6$ 
 \item $\sigma: 1 \to 2 \to 3 \to 1$, $n_1=3,n_2=5, n_3=5$ 
 \item $\sigma: 1 \to 2 \to 3 \to 1$, $n_1=1,n_2=3, n_3=9$ 
\end{itemize}
\end{lem}

The log of homology classes' growth rate is merely a lower bound of the topological entropy. To get a better estimate, one can consider the action $f_{\mathbf{R}*} |_{\pi(X(\mathbf{R}))}$ on the fundamental group of $X(\mathbf{R})$.  Due to Bowen \cite{Bowen}, we have 
\[ \log \rho(f_{\mathbf{R}*}|_{H_1(X(\mathbf{R}); \mathbf{R})}) \le  \log \rho(f_{\mathbf{R}*}|_{\pi_1(X(\mathbf{R}))}) \le h_{top}(f_\mathbf{R}) \le h_{top}(f) = \log \rho (f_*|_{H_2(X(\mathbf{C};\mathbf{R})}).\]


\section{Fundamental Group: Generators and Reading Curves}
\label{S:Faction}
Let $X= \text{B}\ell_P \mathbb{P}^2(\mathbf{R})$ be a real rational surface, where $P=\{p_1,p_2, \dots.p_{n-1}\}$ is a set of $n-1$ distinct points in $\mathbb{P}^2(\mathbf{R})$. Let us denote by $\mathcal{E}_i$ the exceptional curve over $p_i$ for $i = 1,2, \dots,n-1$. Since $X$ is the connected sum of $n$ copies of $\mathbb{P}^2(\mathbf{R})$, we can choose a presentation of $\pi_1(X,x)$ with a fixed base point $x \in X$ such that 
\[ \pi_1(X,x) \ = \langle  \alpha_1, \dots, \alpha_n | \alpha_1^2 \alpha_2^2 \cdots \alpha_n^2=1 \rangle \]
where the generators are represented by simple closed curves based at $x$ that are pairwise disjoint in $X \setminus \{ x \}$; i.e.,
\[ \alpha_i \cap \alpha_j = \{ x\} \qquad \text{for all } \ i \ne j. \] Furthermore, for all $i,j\in\{1, \dots, n-1\}$, \[\alpha_i \cap \mathcal{E}_j = \emptyset \ \ \text{if}\ \ i \ne j\qquad \text{and } \qquad |\alpha_i \cap \mathcal{E}_i|=1. \]

To compute the induced action on the fundamental group, we need to determine the $\pi_1$ class of an arbitrary based path  in the pointed set $(X, x)$. We use the method introduced in Cohen and Lustig \cite{Coh}. Let us briefly describe their method. 

\subsection{Reading curves and the $\pi_1$-class of a given curve}\label{S:ReadingCurves}
If we remove an open disk $\Delta$ from $X \setminus \{x\}$, we see $X \setminus \Delta$ is a disk with $n$ twisted handles attached to its boundary. We can define a set of generators $\{\alpha_1,\dots,\alpha_n\}$ of $\pi_1(X,x)$, where each $\alpha_i$ is the homotopy class of an oriented closed curve traversing exactly one handle. 

\begin{figure}
\centering
\def\svgwidth{3.3truein}
 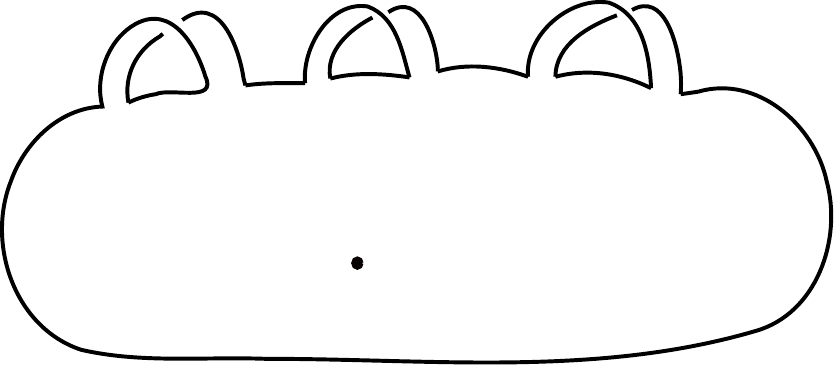
  \def\svgwidth{2truein}
 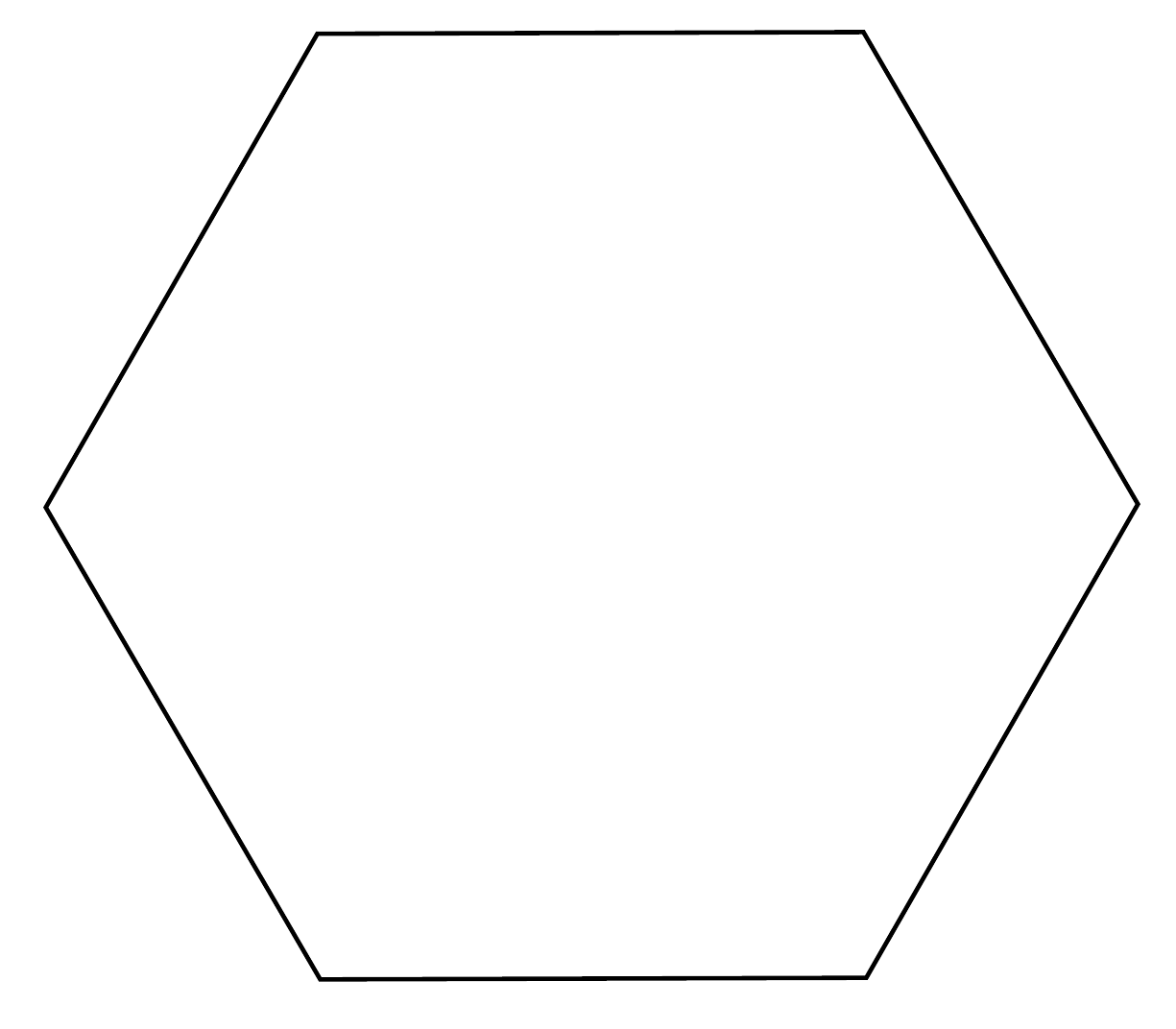

\caption{Generators and Reading curves on $X \equiv \mathbb{P}^2(\mathbf{R})\#\mathbb{P}^2(\mathbf{R})\#\mathbb{P}^2(\mathbf{R})$: The left-hand side is $X \setminus \Delta$ and the right-hand side is a polygon representation of $X$. In both figures, $\alpha_1, \alpha_2, \alpha_3$ are generators of $\pi_1(X,x)$ and $R_{\alpha_1},R_{\alpha_2}, R_{\alpha_3}$ are corresponding reading curves. \label{F:3handles} }
\end{figure}

For each $i$, let $R_{\alpha_i}$ denote a line segment joining the two sides of the boundary of the handle traversed by the generator $\alpha_i$. (See Figure \ref{F:3handles}.) Once we put the removed open disk $\Delta$ back, we can extend each line segment $R_{\alpha_i}$ to a simple closed curve (which we continue to denote by $R_{\alpha_i}$) with base point $\star\in\Delta$ such that 
\begin{itemize}
\item $R_{\alpha_i}$ is a simple closed curve for all $i$,
\item For each $i$, $R_{\alpha_i}$ intersects exactly one generator $\alpha_i$, and 
\item $\{R_{\alpha_i}\}$ are pairwise disjoint on $X \setminus \{ \star\}$.
\end{itemize} 
The curves $\{R_{\alpha_i}\}$ are referred to as {\em reading curves} for the generators $\{\alpha_i\}$.

Given a based oriented simple closed curve $C$ in $(X,x)$, we need to express the element of $\pi_1(X,x)$ represented by $C$ as a product of the generators $\{\alpha_i\}$ and their inverses. The first step is to perturb $C$ by a homotopy relative to its basepoint to make it disjoint from the point $\star$, and transverse to the curves $R_{\alpha_i}$. By compactness, $C$ now intersects $\cup_{i=1}^n R_{\alpha_i}$ in a finite number of points. To identify the element of $\pi_1(X,x)$ represented by $C$, we travel along $C$ in the direction of its orientation, starting and ending at $x$. Each time $C$ crosses $R_{\alpha_i}$, we write $\alpha_i$ if it crosses in the same direction as $\alpha_i$, and $\alpha_i^{-1}$ if it crosses in the opposite direction. This results in a word in the letters $\alpha_i$ and $\alpha_i^{-1}$ which gives the element of $\pi_1(X,x)$ represented by $C$. Note that when we perturbed $C$, we had some freedom. In particular, we could perturb $C$ so that it passes through $\star$ during the perturbation.  It is easy to see that the effect on the word we obtain will be to insert a cyclic permutation of the relator $\alpha_1\alpha_1^{-1}\dots\alpha_n\alpha_n^{-1}$ (or of its inverse) into the original word obtained for $C$. Clearly, this does not affect the element of $\pi_1$ represented by $C$, since in $\pi_1(X,x)$, the relator equals the trivial element.

 In Figure \ref{F:3handles}, we illustrate the generators and the corresponding reading curves of a real rational surface $X$ obtained by blowing up $2$ points. On the left-hand side is showing generators and reading curves in $X \setminus \Delta$ where $\Delta$ is an open disk containing $\star$. On the right hand side we draw a polygon representation of $X$ with the boundary identifications. In both figures, we put $\pm$ sign in the reading curves. If an oriented curve crosses a reading curve $R_{\alpha_i}$ from $+$ to $-$, the corresponding letter would be $\alpha_i$. If the curve crosses $R_{\alpha_i}$ from $-$ to $+$, it would indicate $\alpha_i^{-1}$. In Figure \ref{F:hexagon}, two oriented curves based at $b$ are shown in thick dashed curves. The fundamental group of $X$ with the base point $x$ is given by $\pi_1(X, x) = \langle \alpha_1, \alpha_2, \alpha_3| \alpha_1^2 \alpha_2^2\alpha_3^2=1 \rangle $. On the left-hand side, the curve $C_1$ crosses $R_{\alpha_1}$ from $+$ to $-$ twice and then crosses $R_{\alpha_2}$ from $+$ to $-$. Thus the $\pi_1$ class of $C_1$ is given by $[C_1] = \alpha_1^2\alpha_2^2 = \alpha_3^{-2}$. On the right-hand side, the curve $C_2$ intersects $R_{\alpha_1}$ first from $-$ to $+$ and then intersects $R_{\alpha_3}$ from $-$ to $+$. Thus we have $[C_2] = \alpha_1^{-1} \alpha_3^{-1}$. 

\begin{figure}
\centering
  \def\svgwidth{2.5truein}
 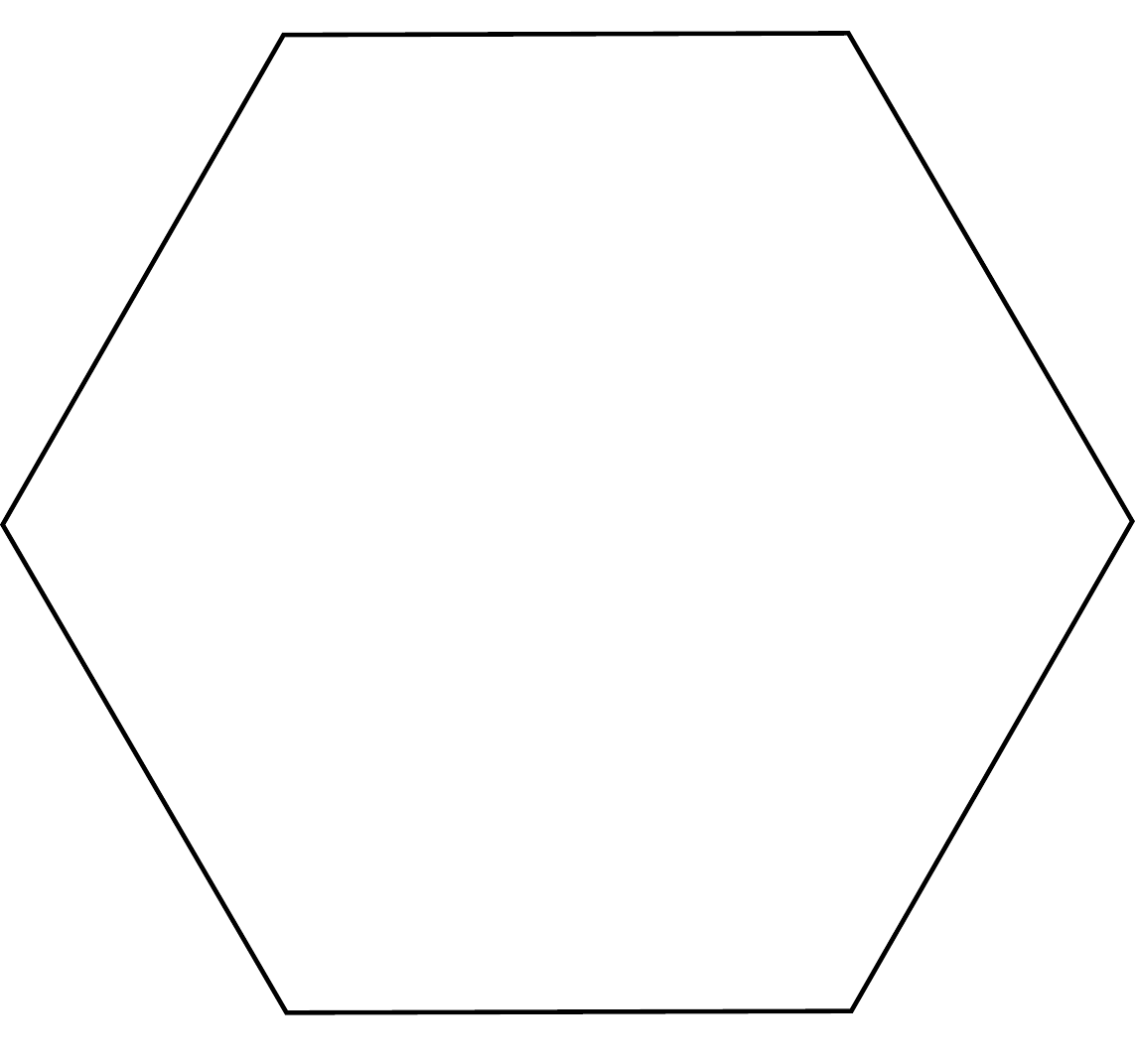\ \ \ \ \ 
  \def\svgwidth{2.5truein}
 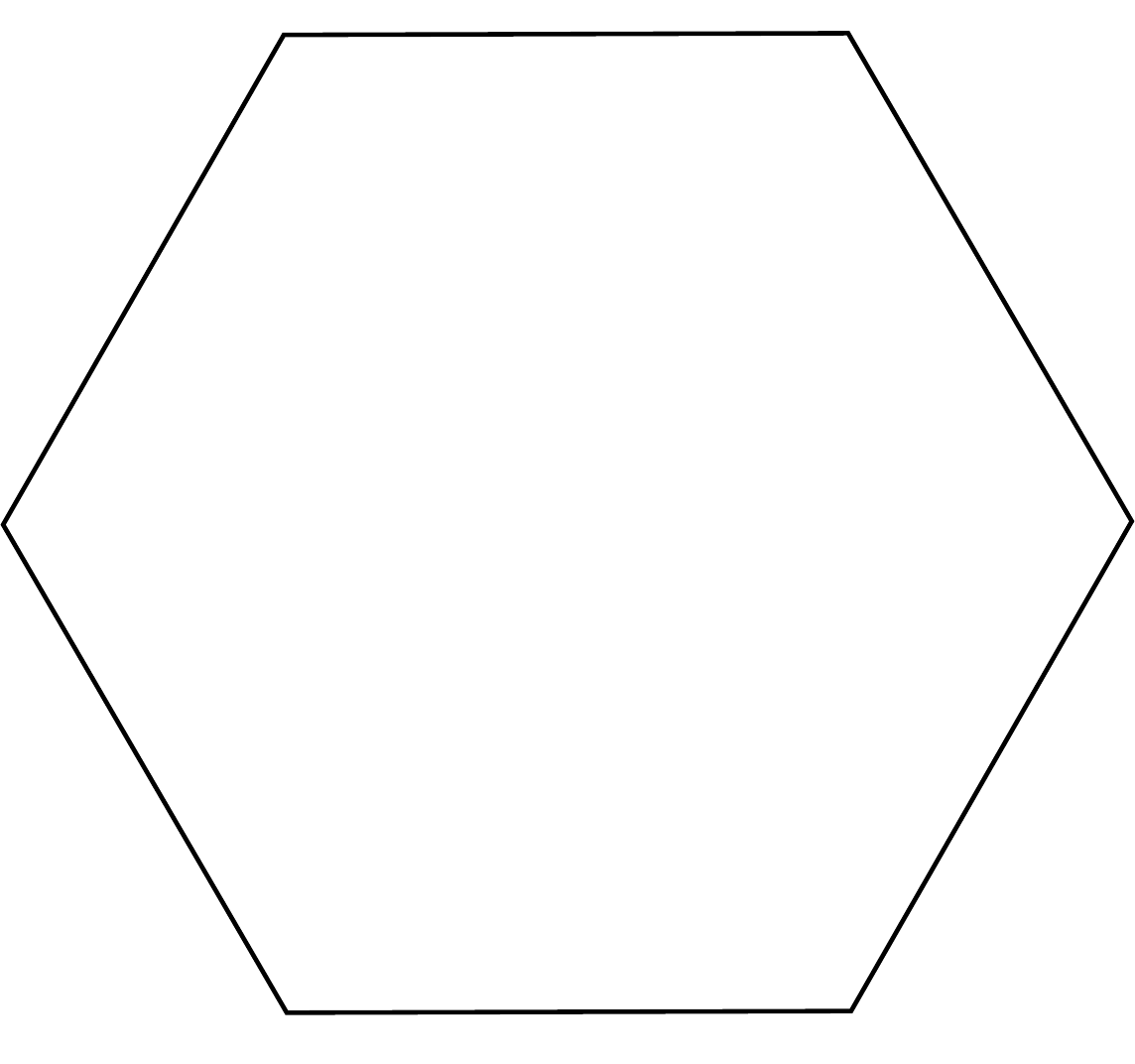

\caption{ A simple closed oriented curve $C$ in $X \equiv \mathbb{P}^2(\mathbf{R})\#\mathbb{P}^2(\mathbf{R})\#\mathbb{P}^2(\mathbf{R})$: On the left-hand side, the $\pi_1$-class of $C_1$ is given by $\alpha_1^2 \alpha_2^2 = \alpha_3^{-2}$. And on the right-hand side the $\pi_1$-class of $C_2$ is given by $\alpha_1^{-1}\alpha_3^{-1} $.   \label{F:hexagon} }
\end{figure}

\subsection{Generators and Reading curves on $X(\mathbf{R})$}\label{SS:gen}
Let $f_\mathbf{R}: X(\mathbf{R}) \to X(\mathbf{R})$ be a diffeomorphism on the real slice $X(\mathbf{R})$ of a blowup $\pi: X(\mathbf{C}) \to \mathbb{P}^2(\mathbf{C})$ associated to a certain basic real birational map $\check f: \mathbb{P}^2(\mathbf{C}) \dasharrow \mathbb{P}^2(\mathbf{C})$ such that (1) $\check f$ properly fixes a cusp cubic $C$, (2) the orbit data is given by three positive integers $n_1, n_2,n_3$ with $n_1+n_2+n_3>10$ and a cyclic permutation $\sigma: 1 \to 2\to 3 \to 1 \in \Sigma_3$, (3) the determinant $\delta >1$, and (4) all center of blowups are distinct points on $C$. The topological entropy of $f$ is given by $\log(\delta)>0$. The purpose of this section is choosing a base point and a set of generators of $\pi_1(X(\mathbf{R}), x)$ such that one can compute the images of generators under the induced action on $\pi_1(X(\mathbf{R}), x)$. 


\begin{figure}
\centering
\def\svgwidth{4.5truein}
\begingroup%
  \makeatletter%
  \providecommand\color[2][]{%
    \errmessage{(Inkscape) Color is used for the text in Inkscape, but the package 'color.sty' is not loaded}%
    \renewcommand\color[2][]{}%
  }%
  \providecommand\transparent[1]{%
    \errmessage{(Inkscape) Transparency is used (non-zero) for the text in Inkscape, but the package 'transparent.sty' is not loaded}%
    \renewcommand\transparent[1]{}%
  }%
  \providecommand\rotatebox[2]{#2}%
  \newcommand*\fsize{\dimexpr\f@size pt\relax}%
  \newcommand*\lineheight[1]{\fontsize{\fsize}{#1\fsize}\selectfont}%
  \ifx\svgwidth\undefined%
    \setlength{\unitlength}{548.64217367bp}%
    \ifx\svgscale\undefined%
      \relax%
    \else%
      \setlength{\unitlength}{\unitlength * \real{\svgscale}}%
    \fi%
  \else%
    \setlength{\unitlength}{\svgwidth}%
  \fi%
  \global\let\svgwidth\undefined%
  \global\let\svgscale\undefined%
  \makeatother%
  \begin{picture}(1,0.43032286)%
    \lineheight{1}%
    \setlength\tabcolsep{0pt}%
    \put(0,0){\includegraphics[width=\unitlength,page=1]{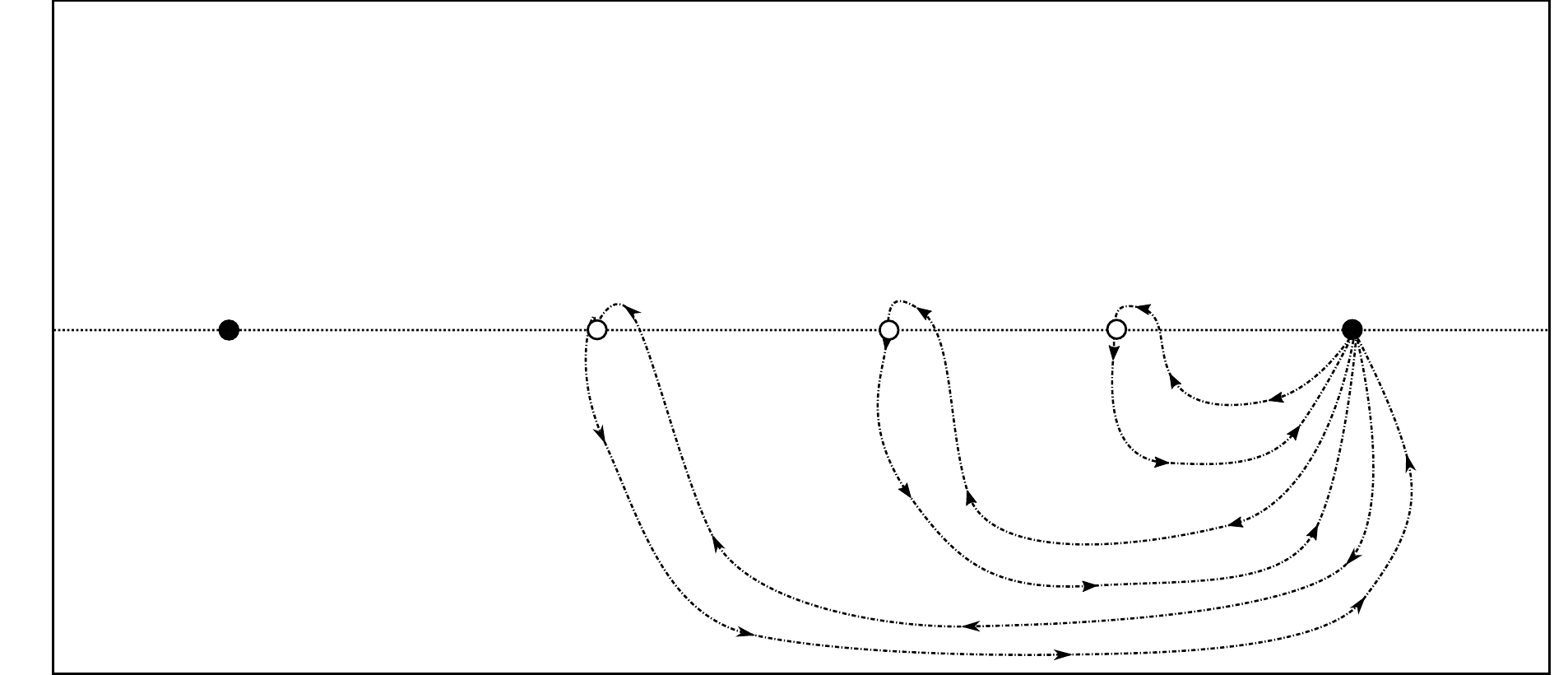}}%
    \put(0.92240175,0.20084721){\makebox(0,0)[lt]{\lineheight{1.25}\smash{\begin{tabular}[t]{l}$\mathcal{C}$\end{tabular}}}}%
    \put(0.8537956,0.23502248){\makebox(0,0)[lt]{\lineheight{1.25}\smash{\begin{tabular}[t]{l}$P_{fix}$\end{tabular}}}}%
    \put(0.12557643,0.19059461){\makebox(0,0)[lt]{\lineheight{1.25}\smash{\begin{tabular}[t]{l}$P_{cusp}$\end{tabular}}}}%
    \put(0.3792715,0.23729937){\makebox(0,0)[lt]{\lineheight{1.25}\smash{\begin{tabular}[t]{l}$a_1$\end{tabular}}}}%
    \put(0.56714631,0.23910858){\makebox(0,0)[lt]{\lineheight{1.25}\smash{\begin{tabular}[t]{l}$b_1$\end{tabular}}}}%
    \put(0.70929902,0.24241766){\makebox(0,0)[lt]{\lineheight{1.25}\smash{\begin{tabular}[t]{l}$c_1$\end{tabular}}}}%
    \put(0,0){\includegraphics[width=\unitlength,page=2]{blowup3V2.pdf}}%
    \put(-0.00121038,0.27928122){\makebox(0,0)[lt]{\lineheight{1.25}\smash{\begin{tabular}[t]{l}$\ell_{\infty}$\end{tabular}}}}%
    \put(0.9933803,0.14984518){\makebox(0,0)[lt]{\lineheight{1.25}\smash{\begin{tabular}[t]{l}$\ell_{\infty}$\end{tabular}}}}%
  \end{picture}%
\endgroup%

\caption{Generators of $\pi_1(X(\mathbf{R}), P_{fix})$. The middle horizontal line represents a cuspidal cubic. Thee generators $a_1,b_1,c_1$ are intersecting exactly one exceptional curves which are given by hollow circles in the middle horizontal line. \label{fig:3BGenerators} }
\end{figure}

\begin{figure}
\centering
 \def\svgwidth{4.5 truein}
 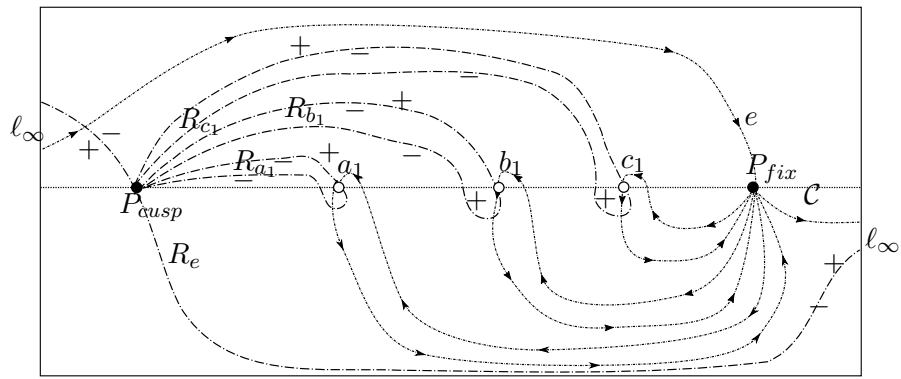
\caption{Generators and Reading Curves: The dashed curves with arrows represent generators in the fundamental group and the densely dashdotted curves with $\pm$ signs represent reading curves. \label{fig:3blowups} }
\end{figure}

To use the repelling fixed point as the base point for the reading curves, we will use $f^{-1}_\mathbf{R}$. To simplify the notation, let us use $\bar a_i = p_{3, n_3+1-i}$, $\bar b_i = p_{1, n_1+1-i}$, $\bar c_i = p_{2, n_2+1-i}$ for both the points in $C$ and the exceptional curves over the associated points. That is, $\bar a_{i+1} = f^{-i} \bar a_1$, $\bar b_{i+1} = f^{-i} \bar b_1$, $\bar c_{i+1} = f^{-i} \bar c_1$ and 
\begin{equation*}
\begin{aligned}
 f^{-1}_\mathbf{R} \ :\ & E_1^- \to \bar a_1 \to \bar a_2 \to \cdots \bar a_{n_3} \to E_3^+ \\
& E_2^- \to \bar b_1 \to \bar b_2 \to \cdots \bar b_{n_1} \to E_1^+ \\
&E_3^- \to \bar c_1 \to \bar c_2 \to \cdots \bar c_{n_2} \to E_2^+ \\
\end{aligned}
\end{equation*}

Recall that we have a parametrization $f_\mathbf{R}|_{C_\text{reg}} : \gamma(\omega) \to \gamma (\delta \omega)$ where $C_\text{reg} = \{ \gamma(t) : t \in \mathbf{R} \}$. For points $p=\gamma(s),q=\gamma(t) \in C_\text{reg}$, we say $p$ precedes $q$ (equivalently, $q$ succeeds $p$) if $s <t$ and we use the notation $p \prec q$. Using the parameterization in Theorem \ref{T:param}, we have $P_{cusp} = \gamma(-\infty)$, $P_{fix} = \gamma(0)$ and all points in the orbit of critical lines $E_i^-$ have  negative parameters in an order determined by the orbit data. 

\subsubsection{Generators of the fundamental group}\label{ss:generators} Let $x=P_{fix}$, a finite fixed point on $C$, be a base point of the fundamental group of $X(\mathbf{R})$. Let us choose closed curves for the set of generators intersecting one and only one exceptional divisor together with a generic curve disjoint from any exceptional divisors. We use $a_i,b_j,c_k$ to denote both the generators and the representing curves intersecting exactly one exceptional curve $\bar a_i,\bar b_j,$ and $\bar c_k$ respectively. Let $e$ be the generator that does not intersect any exceptional curves; that is, $e$ is a generic closed curve in $(X(\mathbf{R}),P_{fix})$. 
 It is convenient to draw $X(\mathbf{R})$ as a rectangle with the invariant cusp cubic curve as the center horizontal line. Since all centers of blowups are located  between two fixed points, let us choose the generators of $\pi_1(X(\mathbf{R}), P_{fix})$ by the fern-leaf shaped closed oriented curves (see Figure \ref{fig:3BGenerators}) in the pointed space $(X(\mathbf{R}),P_{fix})$. A closed curve representing the generator $e$ intersects the invariant cubic only at a regular fixed point and is oriented to the right in Figure \ref{fig:3BGenerators}. All other generators intersect the invariant cubic at two points, the regular fixed point and a point between its center of blowup and the nearest succeeding center of blowup. In a rectangular representation as in Figure \ref{fig:3BGenerators}, these generators are arranged in such a way that if two generators $p$ and $q$ correspond to two points $p=\gamma(s) \prec q=\gamma(t)$ in $C$ then the generator $q$ is above $p$ in the rectangular representation on $X(\mathbf{R})$. In Figure \ref{fig:3BGenerators}, four dashed oriented curves are generators of the fundamental group of a blowup of $\mathbb{P}^2(\mathbf{R})$ along three points $\bar a_1,\bar b_1,\bar c_1$ on a cusp cubic curve. There will be more stacked fern-leaf shaped curves from the fixed point $P_{fix}$ if there are other blowups. 

\subsubsection{Reading Curves}
For the base point of reading curves, we use a repelling fixed point $P_{cusp}$ for $f^{-1}_{\mathbf{R}}$. Due to the arrangement of generators of the fundamental group, a set of reading curves with a base point $P_{cusp}$ can be drawn as a rotation of a set of generators (see Figure \ref{fig:3blowups}). For each generator $p$, we denote by $R_p$ the corresponding reading curve. The reading curve $R_e$ is a curve that intersects the invariant cubic $C$ at $P_{cusp}$ and disjoint from all other generators. For every other generator, the corresponding reading curve intersects the associated exceptional curve, the corresponding generator, and then a point on the invariant cubic between the associate center of blowup and the nearest preceding center of blowup (see Figure \ref{fig:3blowups}). Like the set of generators, the set of reading curves are arranged so that the reading curve for the preceding center is closer to the invariant cubic in the rectangular representation. In Figure \ref{fig:3blowups}, the reading curves are densely dash-dotted curves and are labelled as $R_{\bullet}$ with $\pm$ signs. For any curve in a pointed space $(X(\mathbf{R}), P_{fix}) \setminus \{P_{cusp}\}$, its image under $f_\mathbf{R}^{-1}$ will be disjoint from $\{P_{cusp}\}$. 

\section{The induced action on the Fundamental Group}
\label{S:Action}
Let $f_\mathbf{R}$ be a real diffeomorphism associated to a birational map properly fixing a cusp cubic with an orbit data $n_1, n_2,n_3$ with a cyclic permutation $\sigma:1 \to 2\to 3 \to 1$. The goal of this section is to compute the induced action, $f^{-1}_{\mathbf{R}*}$ on the fundamental group $\pi_1 (X(\mathbf{R}),P_{fix})$ for the orbit data with three positive integers $n_1,n_2,n_3$ and a cyclic permutation. Recall that
\begin{equation*}
\begin{aligned}
 f^{-1}_\mathbf{R} \ :\ & E_1^- \to \bar a_1 \to \bar a_2 \to \cdots \bar a_{n_3} \to E_3^+ \\
& E_2^- \to \bar b_1 \to \bar b_2 \to \cdots \bar b_{n_1} \to E_1^+ \\
&E_3^- \to \bar c_1 \to \bar c_2 \to \cdots \bar c_{n_2} \to E_2^+. \\
\end{aligned}
\end{equation*}
Thus the image of generator under $f^{-1}_\mathbf{R}$ intersects critical lines $E_1^+, E_2^+,E_3^+$ for $\check f$ only if the generator intersects exceptional curves $\bar b_{n_1},\bar c_{n_2},\bar a_{n_3}$ respectively. 
Since each generator in the previous section \ref{S:Faction} is given by an oriented curve starting at $P_{fix}$, one can assign each generator an ordered set of intersections with exceptional curves and an invariant cubic. Using this ordered set, one can get a set of the intersection of a generator's image under $f^{-1}_\mathbf{R}$ together with a set of exceptional curves disjoint from the image. Using these facts, one can compute the image of each generator under $f^{-1}_{\mathbf{R}*} |_{\pi_1 (X(\mathbf{R}),P_{fix})}$.
\begin{figure}
\centering
\def\svgwidth{4.5truein}
 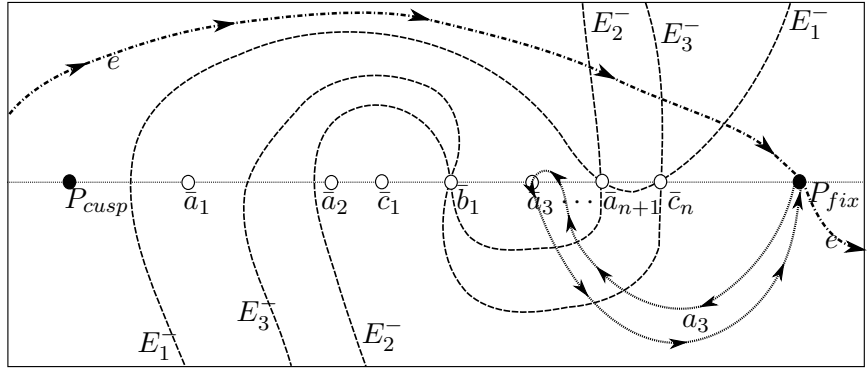
\caption{Generators $a_3, e$ and Exceptional lines, $E_1^-,E_2^-,E_3^-$ for $\check f^{-1}$ with orbit data $1,n,n+1$ and a cyclic permutation. \label{fig:1np1Generator} }
\end{figure}

\subsection{Orbit data $1,n,n+1$ with a cyclic permutation}\label{SS:example} Let us first illustrate the procedure using a real diffeomorphism associated to a birational map properly fixing a cusp cubic with an orbit data $1,n,n+1$, $n\ge 4$ with a cyclic permutation $\sigma:1 \to 2\to 3 \to 1$. Using the parametrization in Theorem \ref{T:param}, we have distinct critical orbits between two fixed points. 
\[ \gamma(-\infty) =P_{cusp} \prec \bar a_1 \prec \bar a_2 \prec \bar c_1 \prec \bar b_1 \prec \bar a_3 \prec \bar c_2 \prec \cdots \prec \bar a_{n+1} \prec \bar c_n \prec P_{fix} = \gamma(0) \]
\subsubsection{Generator $a_3$} In Figure \ref{fig:1np1Generator}, the middle horizontal line represents the cusp cubic curve and hollow circles are critical orbits of a birational map $\check f^{-1}$. Three critical lines $E_1^-,E_2^-,E_3^-$ are given by dashed curves. The dashed arrow curve is a generator $a_3$ of the fundamental group. The way we set up the set of generators in the Section \ref{S:Faction} gives the minimal possible intersections with critical lines. As one can see in Figure \ref{fig:1np1Generator}, the generator $a_3$ is a curve starting at the regular fixed point $P_{fix}$, intersecting $E_3^-$, $E_2^-$, the invariant cubic $C$ , the exceptional curve $\bar a_3$, $E_2^-$, $E_3^-$, and back to the regular fixed point $P_{fix}$. For a closed oriented curve, we will record the important intersections in the following way: \[ a_3 : P_{fix} \to E_3^- \to E_2^- \to C \to \bar a_3 \to E_2^- \to E_3^- \to P_{fix} \]
It follows that the closed oriented curve $f^{-1}_\mathbf{R} a_3$ will intersect exceptional curves and the invariant cubic in the following order
 \[ f^{-1}_\mathbf{R} a_3 : P_{fix} \to \bar c_1 \to \bar b_1 \to C \to \bar a_4 \to \bar b_1 \to \bar c_1 \to P_{fix}. \]
Furthermore, since $a_3$ does not intersect any pre-images of critical lines for $\check f$, the closed curve $f^{-1}_\mathbf{R} a_3$
 will not intersect all three critical lines $E_1^+,E_2^+, E_3^+$ for $\check f$. These critical lines cut $\mathbb{P}^2(\mathbf{R})$ into $4$ subsets. (See Figure \ref{fig:ImageGen}.) Using these critical lines as restrictions, one can obtain the image $f^{-1}_\mathbf{R} a_3$. In Figure \ref{fig:ImageGen}, the curve $f^{-1}_\mathbf{R} a_3$ is given as a thick dashed curve with an arrow. For instance to join the fixed point $P_{fix}$ and the exceptional curve $\bar c_1$, only possible direction is under $E_2^+$ in the bottom half of the rectangle as in the Figure \ref{fig:ImageGen}. Similarly from $\bar c_1$ to $\bar b_1$, the image curve should stay in the region bounded by three critical lines. Following the intersections of $f_\mathbf{R}^{-1}$, one can get the curve (up to a small perturbation) for $f^{-1}_\mathbf{R} a_3$ shown in the Figure \ref{fig:ImageGen}.

The reading curves for four generators $e, a_1, c_1, c_{n+1}$ are shown in the Figure \ref{fig:SReading}. Other generators are similar as $a_1, c_1$ and $c_{n+1}$, that is, they are closed curve starting at the cusp point $P_{cusp}$, intersecting the corresponding exceptional curve, the invariant cubic and back to the cusp point. Reading curves are disjoint on $X(\mathbf{R}) \setminus \{ P_{cusp} \}$. The clockwise direction around the cusp point $P_{cusp}$ is the positive direction and the counter clockwise direction is negative for each reading curves. Since a small circle around $P_{cusp}$ is contractible, one can obtain the unique relator of generators :
\[ e^2 \cc c_n^2 \cc a_{n+1}^2 \cc c_{n-1}^2 \cc a_n^2 \cdots c_2^2 \cc a_3^2 \cc b_1^2 \cc c_1^2 \cc a_2^2 \cc a_1^2 =1 \] To determine the $\pi_1$ class of $f^{-1}_\mathbf{R} a_3$, we follow the curve $f^{-1}_\mathbf{R} a_3$ and check the sequence of intersections of $f^{-1}_\mathbf{R} a_3$ and reading curves. Comparing two Figures \ref{fig:ImageGen} and \ref{fig:SReading}, we see that 
\[ f^{-1}_{\mathbf{R}*} a_3 = e\cc e^{-1} \cc c_1 \cc b_1 \cc a_4^{-1} \cc b_1^{-1} \cc c_1^{-1} \cc e \cc e^{-1} =c_1 \cc b_1 \cc a_4^{-1} \cc b_1^{-1} \cc c_1^{-1} \]

\begin{figure}
\centering
\def\svgwidth{4.5truein}
 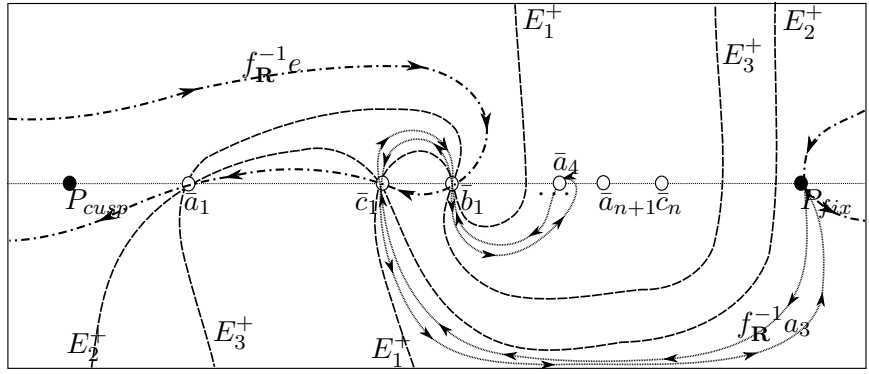
\caption{ The image of the generator $a_3$ under $f^{-1}_\mathbf{R}$. Three dashed curves are critical lines for $\check f$. The gray dashed curve with arrows represents $f^{-1}_\mathbf{R} a_3$ and the dashdotted curve with arrows represents $f^{-1}_\mathbf{R} e$. \label{fig:ImageGen} }
\end{figure}

\begin{figure}
\centering
\def\svgwidth{4.5truein}
 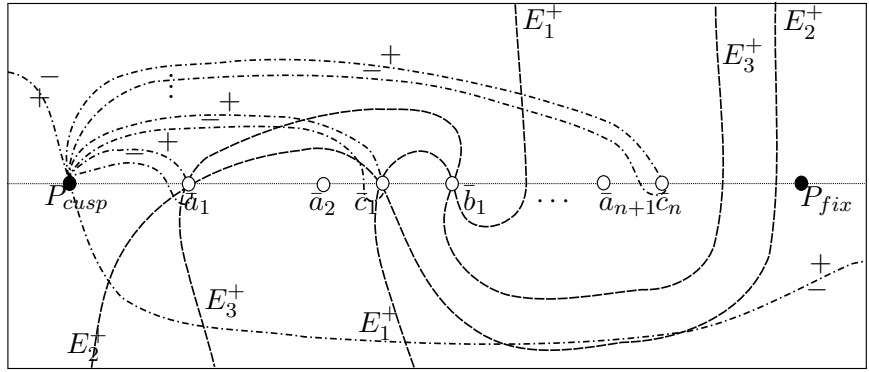
\caption{Reading curves and critical lines for $\check f$. The densely dashdotted curves with $\pm$ signs are reading curves. If a curve crosses a reading curve from $+$ to $-$, then the $\pi_1$ class will pick up the corresponding generator with positive power. The dash cures represent critical lines for $\check f$. \label{fig:SReading} }
\end{figure}

\subsubsection{Generator $e$} In Figure \ref{fig:1np1Generator}, the thick dashdotted curve with arrows represents a generator $e$ and it intersects three critical lines of $\check f^{-1}$ : \[ e: P_{fix} \to E_2^- \to E_3^- \to E_1^- \to P_{fix} \] and thus \[ f^{-1}_\mathbf{R} e : P_{fix} \to \bar b_1 \to \bar c_1 \to \bar a_1 \to P_{fix} .\] Since the oriented curve $e$ intersects the invariant cubic only at the fixed point $P_{fix}$ and $e$ does not intersect any exceptional curves $\bar b_1, \bar a_{n+1}$ and $\bar c_n$, the backward image $f^{-1}_\mathbf{R} e$ does not intersect any of exceptional lines for $\check f^{-1}$, $E_i^+, i=1,2,3$ and $C \setminus \{P_{fix} \}$. In  Figure \ref{fig:ImageGen}, the curve $f^{-1}_\mathbf{R} e$ is shown as the dashdotted curve with arrows. Using the reading curves in the Figure \ref{fig:SReading} and the unique relation on the generators, one see that the $\pi_1$ class of $f^{-1}_\mathbf{R} e$ is written as 
\begin{equation*}
\begin{aligned} f^{-1}_{\mathbf{R}*} e \ &=\ e \cc c_n^2 \cc a_{n+1}^2 \cc \cdots \cc a_3^2 \cc b_1\cc c_1 \cc a_2^2 \cc a_1 \cc e \\
&=\ e^{-1} \cc a_1^{-2} \cc a_2^{-2} \cc c_1^{-2} \cc b_1^{-1} \cc c_1 \cc a_2^2 \cc a_1 \cc e
\end{aligned}
\end{equation*}
Since the number of blowups is $2+2 n \ge 10$, we see that the length of the second expression is shorter than the length of the first. 

\subsection{ Action on the fundamental group} Three exceptional lines for $\check f^{-1}$ intersect the invariant cubic at $6$ points
 \[\{ q_1,\dots,q_6 \} = \{ \bar a_0=f|_C^{-1} \bar a_1,\, \bar b_0=f|_C^{-1} \bar b_1,\, \bar c_0=f|_C^{-1} \bar c_1, \,\bar a_{n_3},\, \bar b_{n_1},\, \bar c_{n_2}\} \] where $ P_{cusp} \prec q_i \prec q_{i+1} \prec P_{fix}$ for all $i =1,2,3,4,5$. For each $i=1,\dots,5$, let us denote $I_i = \{ p \in C | q_i \prec p \prec q_{i+1} \}$ an interval of an invariant cubic cut by exceptional lines for $\check f^{-1}$. To compute the induced action on the fundamental group, one need to determine the order of intersection of each generator and exceptional lines for $\check f^{-1}$. For each generator, these intersections depend on the location of the corresponding base point on the invariant cubic. For example, if the base point is in $I_3$ in Figure \ref{fig:Tinterval}, the corresponding generator can be chosen so that it does not intersect any exceptional lines for $\check f^{-1}$. 
 
Due to the parametrization in Theorem \ref{T:param}, there are four possible cases depending on the orbit data. For the orbit data $n_1 \le n_2 \le n_3$, $n_1+n_2+n_3 \ge 10$ with a cyclic permutation $\sigma : 1\mapsto 2 \mapsto 3 \mapsto 1$, we have 
 \begin{itemize}\addtolength{\itemsep}{1ex}
 \item If $(n_1,n_2,n_3) \in \{ (1,1,n), (1,2,n) : n\in \mathbf{Z} \}$, then $\bar a_0 \prec \bar b_0 \prec \bar b_{n_1} \prec \bar c_0 \prec \bar c_{n_2} \prec \bar a_{n_3}$
 \item If $(n_1,n_2,n_3) \in \{ (1,n,n+1): n \in \mathbf{Z}\}$, then $\bar a_0 \prec \bar c_0 \prec \bar b_{0} \prec \bar b_{n_1} \prec \bar a_{n_3} \prec \bar c_{n_2} $
 \item If either $n_1=1, n_2\ge 3, n_3 \ge n_2 +1$ or $ n_1 \ge 2, n_2 \ne n_3$, then $\bar a_0 \prec \bar b_0 \prec \bar c_{0} \prec \bar b_{n_1} \prec \bar c_{n_2} \prec \bar a_{n_3}$
\item If $n_1 \ge 2, n_2 = n_3$, then $ \bar a_0 \prec \bar c_0 \prec \bar b_{0} \prec \bar b_{n_1} \prec \bar a_{n_3} \prec \bar c_{n_2}.$
 \end{itemize} 
 
Using the location of the base points on the invariant cubic, a set of generators can be selected as a set of simple closed curves in $(X(\mathbf{R}), P_{fix})$ with the minimum number of intersections with exceptional lines for $\check f^{-1}$. For example, if the associated orbit data are $1, n, n+1$ with a cyclic permutation, using the exceptional lines of $\check f^{-1}$ shown in the Figure \ref{fig:Tinterval} we choose the following set of generators with the minimum number of intersections:

\begin{lem}\label{L:1np1Intersection} For an orbit data $1,n,n+1$ with a cylic permutation, both $I_2 $, and $I_5$ do not contain any center of blowups. One can choose a set of curves for generators of the fundamental group such that (1) they are pairwise disjoint in $X(\mathbf{R}) \setminus \{ P_{fix}\}$, (2) each generator intersects at most one exceptional curves, and (3) each generator intersects the invariant cubic and exceptional lines for $\check f^{-1}$ in the following order:
\begin{equation*}
\begin{aligned} 
& e : P_{fix} \to E_2^- \to E_3^- \to E_1^- \to P_{fix} &\\
& a_{n+1} : P_{fix} \to E_3^- \to E_1^- \to C \to \bar a_{n+1} \to E_2^- \to E_3^- \to P_{fix} &\\
& b_1 : P_{fix} \to E_3^- \to E_2^- \to C \to \bar b_1 \to P_{fix} &\\
& c_n: P_{fix} \to C \to E_1^- \to \bar c_n \to E_3^- \to P_{fix} &\\
& w : P_{fix} \to E_2^- \to E_3^- \to C \to \bar w \to E_3^- \to E_2^- \to P_{fix}& \qquad \text{ if } \bar w \in I_1\\
& w: P_{fix} \to C \to \bar w \to P_{fix} & \qquad \text{ if } \bar w \in I_3\\
& w: P_{fix} \to E_3^- \to E_2^- \to C \to \bar w \to E_2^- \to E_3^- \to P_{fix} & \qquad \text{ if } \bar w \in I_4\\
\end{aligned}
\end{equation*}
\end{lem}

\begin{figure}
\centering
\def\svgwidth{4.5truein}
 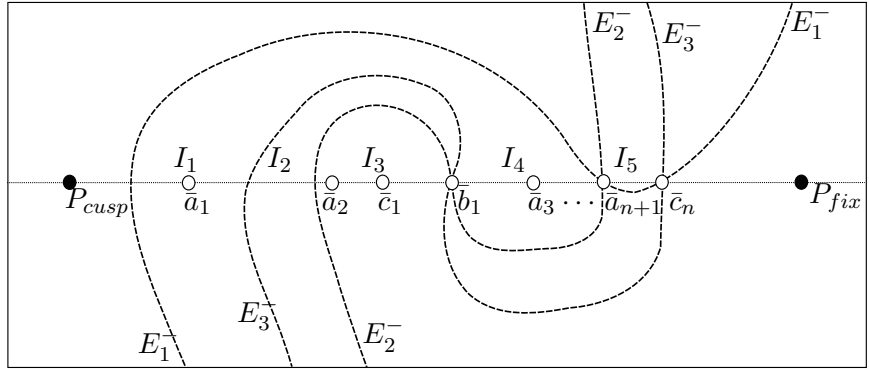
\caption{ Exceptional lines for $\check f^{-1}$ with intervals on the invariant cubic $C$. Each interval determines the order of the intersections of a generator with the exceptional lines. \label{fig:Tinterval} }
\end{figure}

The image of each curve under $f_\mathbf{R}^{-1}$ and the induced action on the fundamental group are determined by the orbits of the exceptional lines for $\check f^\pm$. 

\begin{prop} Let $f_\mathbf{R}:X(\mathbf{R}) \to X(\mathbf{R})$ be a real diffeomorphism associated to the orbit data $1,n,n+1$ and a cyclic permutation. Let $G=\{ e, a_i, b_j, c_k\}$ be a set of generators of the fundamental group of $X(\mathbf{R})$ and 
\[ \pi_1 (X(\mathbf{R}), P_{fix}) = \langle e, a_1, \dots, a_{n+1}, b_1, c_1, \dots, c_n| e^2 \cc c_n^2 \cc a^2_{n+1} \cc \cdots c_2^2 \cc a_3^2 \cc b_1^2 \cc c_1^2 \cc a_2^2 \cc a_1^2 =1\rangle. \]
For each $w=g_1^{\pm 1}\cc g_2^{\pm 1} \cc \cdots g_n^{\pm 1}\in \pi_1(X(\mathbf{R}), P_{fix})$ with $g_1, \dots, g_n \in G$, the image under the induced action $f^{-1}_{\mathbf{R}*} |_{\pi_1(X(\mathbf{R}), P_{fix})}$ is given by
\[ f^{-1}_{\mathbf{R}*} (w=g_1^{\pm 1}\cc g_2 ^{\pm 1}\cc \cdots g_n^{\pm 1}) \ = \ (f^{-1}_{\mathbf{R}*} g_1)^{\pm 1}\cc (f^{-1}_{\mathbf{R}*} g_2)^{\pm 1} \cc \cdots (f^{-1}_{\mathbf{R}*} g_n)^{\pm 1}\]
where 
\begin{equation*}
\begin{aligned}
f^{-1}_{\mathbf{R}*} \ :\ & e &\mapsto\ &e^{-1} \cc a_1^{-2} \cc a_2^{-2} \cc c_1^{-2} \cc b_1^{-1} \cc c_1 \cc a_2^2 \cc a_1 \cc e &\\
& a_{n+1} &\mapsto\ &c_1\cc a_2^2 \cc a_1 \cc e\cc b_1 \cc c_1 &\\ 
& b_1 &\mapsto\ & c_1^{-1} \cc b_1^{-1} \cc c_2^2 \cc a_3^2 \cc b_1^2 \cc c_1^2 \cc a_2^2 \cc a_1^2 \cc e&\\
& c_n &\mapsto\ & e^{-1} \cc a_1^{-1} \cc a_2^{-2} \cc c_1^{-1} &\\
& w_i &\mapsto\ & e^{-1} \cc a_1^{-2}\cc a_2^{-2} \cc c_1^{-2} \cc b_1^{-1} c_1 \cc w_{i+1} \cc c_1^{-1} \cc b_1 \cc c_1^2 \cc a_2^2 \cc a_1^2 \cc e& \text{ if } \bar w_i \in I_1\\
& w_i &\mapsto\ & e^{-1} \cc a_1^{-2}\cc a_2^{-2}\cc \cdots \cc w_{i+1}^{-2} \cc w_{i+1} \cc w_{i+1}^2\cc \cdots \cc a_2^2 \cc a_1^2 \cc e& \text{ if } \bar w_i \in I_2\\
& w_i &\mapsto\ & c_1^{-1} \cc b_1^{-1} \cc w_{i+1} \cc b_1 \cc c_1 & \text{ if } \bar w_i \in I_4\\
\end{aligned}
\end{equation*}
where $w_{i+1}$ is a generator corresponding the exceptional line over $f|_C^{-1} \bar w_i$ and both $a_2^{-2} \cc \cdots \cc w_{i+1}^{-2}$ and $w_{i+1}^2 \cc \cdots \cc a_2^{2}$ are parts of the relator of the group presentation. 
\end{prop}

\begin{proof} Using Lemma \ref{L:1np1Intersection}, we get the order of intersection of the image of each generator and exceptional curves of $\check f$.
\begin{equation}\label{E:intersections}
\begin{aligned} 
& f^{-1}_{\mathbf{R}*}e : P_{fix} \to \bar b_1 \to \bar c_1 \to \bar a_1 \to P_{fix} &\\
& f^{-1}_{\mathbf{R}*}a_{n+1} : P_{fix} \to \bar c_1\to \bar a_1 \to C \to E_3^+ \to \bar b_1 \to \bar c_1 \to P_{fix} &\\
& f^{-1}_{\mathbf{R}*}b_1 : P_{fix} \to \bar c_1 \to \bar b_1 \to C \to E_1^+ \to P_{fix} &\\
& f^{-1}_{\mathbf{R}*}c_n: P_{fix} \to C \to \bar a_1 \to E_2^+ \to \bar c_1 \to P_{fix} &\\
&f^{-1}_{\mathbf{R}*} w : P_{fix} \to \bar b_1 \to \bar c_1 \to C \to \bar w \to \bar c_1 \to \bar b_1 \to P_{fix}& \qquad \text{ if } \bar w \in I_1\\
& f^{-1}_{\mathbf{R}*}w: P_{fix} \to C \to \bar w \to P_{fix} & \qquad \text{ if } \bar w \in I_3\\
&f^{-1}_{\mathbf{R}*} w: P_{fix} \to \bar c_1 \to \bar b_1 \to C \to \bar w \to \bar b_1 \to \bar c_1 \to P_{fix} & \qquad \text{ if } \bar w \in I_4\\
\end{aligned}
\end{equation}
To determine a word in $\pi_1(X(\mathbf{R}), P_{fix})$, we locate the backward image curve of each generator in Figure \ref{fig:SReading} using the intersection information in (\ref{E:intersections}). 
For instance, if a curve starts with $P_{fix} \to \bar b_1$, the only way to join $P_{fix}$ with $\bar b_1$ without intersecting exceptional lines $E_i^+$ is to construct a curve such that it starts at $P_{fix}$, proceeds downward, passes the boundary of the rectangle (a line at infinity), moving downward toward $\bar b_1$ with out touching the hump created by $E_3^+$. It follows that this curve will intersect reading curves corresponding to the word $e\cc c_n^2 \cc a_{n+1}^2 \cc c_{n-1}^2 \cc a_n^2 \cdots c_2^2 \cc a_3^2 \cc b_1$. Since the unique relator is $e^2 \cc c_n^2 \cc a^2_{n+1} \cc \cdots c_2^2 \cc a_3^2 \cc b_1^2 \cc c_1^2 \cc a_2^2 \cc a_1^2 =1$, we see that 
\[ e\cc c_n^2 \cc a_{n+1}^2 \cc c_{n-1}^2 \cc a_n^2 \cdots c_2^2 \cc a_3^2 \cc b_1 = e^{-1} \cc a_1^{-2} \cc a_2^{-2} \cc c_1^{-2} \cc b_1^{-1} \] 
To have an automorphisms on $X$ with positive entropy, we need $1+n+n+1 = 2(n+1) \ge 10$. Thus the right hand side word has the smaller length. 
Continuing this procedure for each of the cases in (\ref{E:intersections}) and reducing the length using the relator, we obtain the desired result. 
\end{proof}

To compute the induced action on the fundamental group, we need to go through case by case. For real diffeomorphisms associated with the orbit data $n_1,n_2,n_3$ with a cyclic permutation, there are $6$ different configurations of exceptional lines. Computations for each case are essentially identical. In Appendix \ref{ApendA}, we describe the induced $\pi_1$ action for each $f^{-1}_\mathbf{R}$ associated with the orbit data $n_1\le n_2\le n_3$ and a cyclic permutation. 

\section{The Growth Rate of the $\pi_1$-Action}
\label{S:Grate}
Let $f_\mathbf{R}$ be the real diffeomorphism with determinant $>1$, associated with the orbit data $n_1, n_2, n_3$ and a cyclic permutation. For the rest of this article, let us fix a set of generators for $\pi_1(X(\mathbf{R}, P_{fix}))$, \[G = \{ e, a_i, b_j, c_k, 1 \le i \le n_3, 1 \le j \le n_1, 1 \le k \le n_2\}\] as defined in the Section \ref{S:Action}. Let $G^{-1} = \{ g^{-1} : g \in G\}$ denote the set of inverses of generators in $G$. 
Let $\ell_G (\omega)$ denote the minimal word length of an element $\omega \in \pi_1(X(\mathbf{R}), P_{fix})$ with respect to the set of generators $G$, and let $\ell_G(1) = 0$. Suppose $\omega= g_{i_1} \cc g_{i_2} \cdots g_{i_n} \in \fun$, $ g_{i_j} \in G \cup G^{-1}$ We say that the word $g_{i_1} \cc g_{i_2} \cdots g_{i_n}$ is \textit{reduced} if it has minimal length among all words representing $\omega$, i.e., if $n= \ell_G(\omega)$. The reduced word representing each group element is not uniquely determined. For example, if the relator has even length, then half of the relator is equal to the inverse of the other half, so both of these words are reduced and represent the same element. The growth rate of the induced action $f^{-1}_{\mathbf{R}*}: \pi_1(X(\mathbf{R}), P_{fix}) \to \pi_1(X(\mathbf{R}), P_{fix})$ is defined by 
\[ \rho (f^{-1}_{\mathbf{R}*}) := \sup_{g \in G} \left \{ \ \limsup_{n \to \infty} \, (\ell_G ( f^{-n}_{\mathbf{R}*} g))^{1/n}\ \right \}. \]
Since $f_\mathbf{R}$ is a diffeomorphism, we have $h_{top} (f_\mathbf{R}) = h_{top}(f^{-1}_\mathbf{R})$. The following theorem by Bowen shows that the growth rate of the induced action on the fundamental group gives the lower estimate for the topological entropy.
\begin{thm}\cite{Bowen}
Let $f:M \to M$ be a map of a compact manifold and let $\mu$ be the growth rate of the induced action $f_*$ on $\pi_1(M)$. Then \[ h_{top}(f) \ge \log \mu \]
\end{thm}

Since the fundamental group of $X(\mathbf{R})$ is non-commutative with one relator, it is hard to compute the growth rate. To overcome this difficulty, we use the conjugacy classes of the fundamental group, which correspond to free homotopy classes of closed curves in $X(\mathbf{R})$. For each $\omega \in \pi_1(X(\mathbf{R}), P_{fix})$, let $\ell (\omega)$ denote the minimal word length of all elements of the conjugacy class $[\omega]$
\[ \ell (\omega) := \min \{ \ell_G (x) : x = y \omega y^{-1}, y \in \pi_1(X(\mathbf{R}), P_{fix}) \}. \] 
We say that a word $w=g_{i_1} \cc g_{i_2} \cc \cdots \cc g_{i_n}$, $g_{i_j} \in G \cup G^{-1}$ for $j=1, \dots n$ is \textit{G-cyclically reduced} if every cyclic permutation of $w$ is a reduced word. In particular, this implies that $g_{i_j} g_{i_{j+1}} \ne 1$ for all $j = 1, \dots n$ modulo $n$. In a free group, the only cyclically reduced conjugates of a cyclically reduced word are cyclic permutations. The unique relator in the group $\fun$ is the ordered product of the squares of generators. The length of the relator is always even, and the half of the relator and the inverse of the other half represent the same element in $\fun$. With a slight modification, we have the following Lemma.

\begin{lem}\label{L:cpconj}
If two $G$-cyclically reduced words represent elements of $\pi_1(X(\mathbf{R}), P_{fix})$ that are in the same conjugacy class, then (after replacing a half relator with the inverse of the other half relator if necessary) one of these words is a cyclic permutation of the other. 
\end{lem}
In \cite{Birman-Series}, Birman and Series considered an automorphism $\phi$ on the fundamental group of a compact surface $M$ and showed there is a family of cyclic words $W$ with respect to a set of generators of $\pi_1(M)$ such that an appropriate restriction of $\phi$ on the positive span of $W$ is linear. To use Whitehead's idea in \cite{whitehead}, Birman and Series construct $W$ such that a group generated by $W$ is free and study the induced map on the positive span $Sp^+(W)$ to show its linearity and the growth rate. To estimate the growth rate of the induced action on the fundamental group, we adapt the idea of Birman and Series \cite{Birman-Series}.

\subsection{Invariant Set of admissible sequences}\label{SS:admissible} Let $\Gamma=\{ \gamma_1, \dots, \gamma_k\} \subset \fun$ be a finite set of reduced words representing different elements in $\pi_1(X(\mathbf{R}), P_{fix})$ with respect to a set of generators $G$ with the following properties:
\begin{itemize}
\item there is $g_\star$ in a symmetric set $ G\cup G^{-1} = \{ g, g^{-1}: g\in G\}$ such that every reduced word $\gamma_i \in \Gamma$ has exactly one $g_\star$ at the end. 
\item $\ell_G(\gamma) = \ell (\gamma) = s \ge 2$ for all $\gamma_i \in \Gamma$
\item Each $\gamma \in \Gamma$ is a $G$-cyclically reduced.
\end{itemize}
It follows that for each $\gamma \in \Gamma$, $\gamma= g_1 \cdots g_s \cc g_\star$, $g_1 \ne g_\star^{-1}$, and $\ell(\gamma) = s+1$.
Let $\Gamma^+$ be a monoid of elements in $\Gamma$ such that each $\gamma \in \Gamma^+$ is written as a product of positive powers of $\gamma_i$'s:
\[ \Gamma^+ = \{ \gamma = \gamma_{i_1}^{\epsilon_1} \cc \dots \gamma_{i_s}^{\epsilon_s}: \gamma_{i_j} \in \Gamma, \text{ and } s, \epsilon_1, \dots \epsilon_s \ge 0\}. \]
For each $\gamma \in \ \Gamma ^+$, let $\ell_\Gamma (\gamma)$ denote the minimal word length in $\Gamma$. Since the unique relator for the set of generator $G$ for the fundamental group of $X(\mathbf{R})$ is given by the product of squares of generators and each $\gamma\in \Gamma^+$ is given by a product of positive powers of $\gamma_i\in \Gamma$, it is clear that there is no relation in $\Gamma^+$. It follows that

\begin{lem}\label{L:Gplus} Every word in $\Gamma ^+$ is $G$-cyclically reduced with respect to a set of generators $G$. Thus for each $\gamma= \gamma_{i_1}\gamma_{i_2} \dots \gamma_{i_n} \in \Gamma ^+$ \[ \ell (\gamma)=\ell_G (\gamma) = \sum_{j=1}^n \ell (\gamma_{i_j}) \quad \text{and} \quad \ell_\Gamma (\gamma)= \sum_{ i =1}^k n_i(\gamma)\] where $n_i (\gamma)$ is the number of $\gamma_i$ in the word $\gamma$.
\end{lem}

\begin{proof}
Every $\gamma_i \in \Gamma$ starts with a letter $\ne g_\star, g_\star^{-1}$ and ends with a letter $g_\star$. Since the unique relation on a set of generators $G$ is given by a product of squares of elements in $G$, any length reduction, if there are any, must happen in each $\gamma_i$. Each $\gamma_i \in \Gamma$ is reduced with respect to $G$. Thus any element in $\gamma \in \Gamma^+$ is reduced with respect to $G$ and the length $\ell(\gamma) = \ell_G(\gamma)$ depends only on the number of $\gamma_i's$ in $\Gamma$.
\end{proof}

Let $\langle \Gamma^+ \rangle := \{ [\gamma] : \gamma \in \Gamma^+\}$ be a set of conjugacy classes of $\Gamma^+$ in $\pi_1(X(\mathbf{R}), P_{fix})$. Using Lemma \ref{L:cpconj} and Lemma \ref{L:Gplus}, we have that for $\gamma, \gamma' \in \Gamma^+$, $\gamma$ and $\gamma'$ are cyclic permutations of each other with respect to $\gamma_i$'s in $\Gamma$ if and only if the $[\gamma] = [\gamma']$. Also we see that the length $\ell_\Gamma$ is well defined in $\langle \Gamma^+ \rangle$, that is for each $[\gamma] \in \langle \Gamma^+ \rangle$ we have
\[ \ell_\Gamma [\gamma] = \ell_\Gamma (\gamma). \]

\subsubsection{Admissible sequences}\label{SS:admissible}
Let $A$ be a set of ordered $n$-tuples in $\{1, 2, \dots, k\}$. We say $\gamma= \gamma_{i_1} \cdots \gamma_{i_n}\in \Gamma^+$ of length $n$ is \textit{an $A$-admissible $n$ block} if the $n$-tuple of indices $(i_1, \dots, i_n) \in A$. We say $\gamma= \gamma_{i_1} \cdots \gamma_{i_s}$ in $\Gamma^+$ is \textit{cyclically admissible with respect to $A$}, (or simply \textit{$A$-admissible}) if each $n$ blocks in $\gamma$ is an $A$-admissible $n$ block where $n$ blocks taken modulo $s$. Let $\Gamma_A \subset \langle \Gamma^+ \rangle$ be a set of conjugacy classes of (cyclically) $A$-admissible elements in $\Gamma^+$.
\[ \Gamma_A = \{ [\gamma] \in \langle \Gamma^+ \rangle: \gamma= \gamma_{i_1} \cdots \gamma_{i_s}, s\ge n, (i_j, \dots i_{j+n}) \in A \} \]
where the indexes are taken modulo $s = \ell_\Gamma (\gamma)$.

When $[f^{-1}_{\mathbf{R}*} \gamma] \in \Gamma_A$ for every $[\gamma] \in \Gamma_A$, i.e. there is $\gamma' \in \Gamma_A$ such that  $f^{-1}_{\mathbf{R}*} \gamma$ is conjugate to $\gamma'$ in $\pi_1(X(\mathbf{R}), P_{fix})$, we say $\Gamma_A$ is $f^{-1}_{\mathbf{R}*}$-invariant. In this case, let us define a map $\phi$ on $\Gamma_A$ by \[ \phi([\gamma]) = [f^{-1}_{\mathbf{R}*} \gamma], \ \ \ \text{for } [\gamma] \in \Gamma_A.\] 


For each $[\gamma] \in \Gamma_A$, the length $\ell_\Gamma [\gamma]$ only depends on the number of each $\gamma_i \in \Gamma$ in any representative $\gamma \in [\gamma]$. Let $\text{Sp}^+ \Gamma$ be the positive linear span of $\Gamma$ and for each $[\gamma] \in \Gamma_A$ define $[\gamma]_\# \in \text{Sp}^+\Gamma$ by 
\begin{equation}\label{E:counting} [\gamma]_\# = \sum_{i=1}^k n_i \gamma_i \quad \text{where } n_i(\gamma) = \text{the number of } \gamma_i \text{ in } [\gamma].\end{equation}
For instance, if $[\gamma] =[ \gamma_1\gamma_3\gamma_2\gamma_1\gamma_4] \in \Gamma_A$ then $[\gamma]_\# = 2 \gamma_1 + \gamma_2 +\gamma_3+\gamma_4$. Let $\Gamma_A^\# = \{ [\gamma]_\#: [\gamma] \in \Gamma_A\}$. 
If $\Gamma_A$ is $f_{\mathbf{R}*}^{-1}$- invariant, for each $[\gamma] \in \Gamma_A$ we have 
\begin{equation} \label{E:cmap}
 \phi [\gamma]_\# \in \Gamma_A^\# 
\end{equation}

\begin{lem}\label{L:linearity} Suppose $\Gamma_A$ is $f_{\mathbf{R}*}^{-1}$-invariant and $ [ \gamma ] ,[\gamma'] \in \Gamma_A$. 
\begin{itemize}
\item For a positive integer $n$, $[\gamma^n]=[ \gamma \gamma \cdots \gamma ]$, the class of the product of $n$ coplies of $\gamma$ is in $\Gamma_A$ and \[ \phi[\gamma^n]_\# = n \, \phi[\gamma]_\#\]
\item If two representatives $\gamma$ and $\gamma'$ have the same $n-1$ successive indices then $[\gamma \cc \gamma'] \in \Gamma_A$ and \[ \phi[\gamma \cc\gamma']_\# = \phi[\gamma]_\# + \phi[\gamma']_\#\]
\end{itemize}
\end{lem}

\begin{proof}
Since for $[\gamma], [\gamma'] \in \Gamma_A$,  \[\phi[\gamma\cc \gamma'] = [ f^{-1}_{\mathbf{R}*} \gamma \cc \gamma'] =[ f^{-1}_{\mathbf{R}*} \gamma \cc f^{-1}_{\mathbf{R}*}\gamma'],\]
it is enough to show that $[\gamma^n] \in \Gamma_A$ for the first part and $[\gamma\cc \gamma'] \in \Gamma_A$ for the second. 
If $[\gamma]= [\gamma_{i_1} \cdots \gamma_{i_s}] \in \Gamma_A$, every cyclic $n$-sub indices of $\gamma$ is in $A$. Thus it is clear that \[[\gamma^n] = [\gamma_{i_1} \cdots \gamma_{i_s} \gamma_{i_1} \cdots \gamma_{i_s} \dots \gamma_{i_1} \cdots \gamma_{i_s}] \in \Gamma_A. \]
If $\gamma$ and $\gamma'$ have the same $n-1$ successive indices, we may write \[ \gamma = \gamma_{i_1} \cdots \gamma_{i_{n-1}} \cc \gamma_{j_1} \cdots \gamma_{j_s} \quad \text{and} \quad \gamma' = \gamma_{i_1} \cdots \gamma_{i_{n-1}} \cc \gamma_{k_1} \cdots \gamma_{k_t} \]
where $(j_s,i_1, \dots, i_{n-1}), (k_t, i_1 \dots i_{n-1}) \in A$. It follows that 
\[ \gamma\cc \gamma' = \gamma_{i_1} \cdots \gamma_{i_{n-1}} \cc \gamma_{j_1} \cdots \gamma_{j_s} \cc \gamma_{i_1} \cdots \gamma_{i_{n-1}} \cc \gamma_{k_1} \cdots \gamma_{k_t}\in \Gamma_A\]
\end{proof}

\subsection{Split and Merge}\label{SS:sandp} 
%
Let us continue to assume that $\Gamma_A$ is $f^{-1}_{\mathbf{R}*}$-invariant. Each $\gamma_i \in \Gamma$ is not necessarily in $\Gamma_A$ and thus a priori $f^{-1}_{\mathbf{R}*} \gamma_i \notin \Gamma^+$. 
However for $[\gamma] = [\gamma_{i_1}\cdots \gamma_{i_n}] \in \Gamma_A$, we have $\phi[\gamma] = [f^{-1}_{\mathbf{R}*} \gamma] \in \Gamma_A$, that is, \[ [f^{-1}_{\mathbf{R}*} \gamma]=[f^{-1}_{\mathbf{R}*} \gamma_{i_1} \cc \cdots \cc f^{-1}_{\mathbf{R}*} \gamma_{i_n}] =[\gamma_{j_1} \cdots \gamma_{j_s}]\] for some $A$-admissible $\gamma_{j_1} \cdots \gamma_{j_s}$. Since there is no cancelation between $\gamma_i$'s, if $f^{-1}_{\mathbf{R}*} \gamma_{i_j} \not\in \Gamma^+$ then the product of the last part of $f^{-1}_{\mathbf{R}*} \gamma_{i_j}$ and the first part of $f^{-1}_{\mathbf{R}*} \gamma_{i_{j+1}}$ should be equal to an element in $\Gamma$. Furthermore since the element $g_\star$ only appears as the last letter for each $\gamma_i$, we have 
\[ [ f^{-1}_{\mathbf{R}*} \gamma_i] = [g_{i_{1,1}} \cdots g_{i_{1,n_1}} \cc g_\star \cc g_{i_{2,1}} \cdots g_{i_{2, n_2}} \cc g_\star \cc g_{i_{3,1}} \cdots g_{i_{j-1, n_{j-1}}} \cc g_\star \cc g_{i_{j,1}} \cdots g_{i_{j, n_j}}]\]
where $g_{i_{j,k}} \ne g_\star$ and $g_{i_{j,k}} \in G \cup G^{-1}$ for all $i,j,k$ and the product $ g_{i_{j,1}} \cdots g_{i_{j, n_j}} \cc g_\star$ must be an element of $\Gamma$. 
It follows that for each $\gamma_i \in \Gamma = \{ \gamma_1, \dots, \gamma_k\}$, we can have $ [f^{-1}_{\mathbf{R}*} \gamma_i] = [\mu_i \cc \gamma_{i_1} \cc \cdots \gamma_{i_{n_i}} \cc \zeta_i]$ or $[ f^{-1}_{\mathbf{R}*} \gamma_i] = [\eta_i]$ where $\gamma_{i_1}, \dots, \gamma_{i_{n_i}} \in \Gamma$, and $ \zeta_i, \mu_i, \eta_i$ are reduced elements (possibly an identity) in $\pi_1(X(\mathbf{R}), P_{fix})$ with the following properties:
\begin{itemize}
\item all three $\zeta_i, \eta_i, \mu_i$ are reduced,
\item the letter $g_\star$ doesn't appear in $\zeta_i$ and $\eta_i$, 
\item the letter $g_\star$ only appears as the last letter of $\mu_i$, and
\item $\zeta_i$ does not start with the letter $g_\star^{-1}$.
\end{itemize}
Let us denote $I_1 \subset \{ 1, \dots, k\}$ a set of indices such that
 \[ \text{if } \ i \in I_1,\quad \text{ then} \ \ f^{-1}_{\mathbf{R}*} \gamma_i = \mu_i \cc \gamma_{i_1} \cc \cdots \gamma_{i_{n_i}} \cc \zeta_i, \quad \text{for some }\ \gamma_{i_1}, \dots \gamma_{i_n} \in \Gamma, \zeta_i, \mu_i \in \pi_1(X(\mathbf{R}), P_{fix})\] 
 and use $I_2$ for the set of indices such that \[ \text{if } \quad i \in I_2,\quad \text{ then} \ \ f^{-1}_{\mathbf{R}*} \gamma_i =\eta_i, \qquad \text{for some } \qquad \eta_i \in \pi_1(X(\mathbf{R}), P_{fix}) \] 
If cyclically successive indices $i,j$ appear in some $A$-admissible elements in $\Gamma_A$, to have $f^{-1}_{\mathbf{R}*}$-invariant $\Gamma_A$ one of the following situations should happen
\begin{itemize}\addtolength{\itemsep}{1ex}
\item If both $i, j \in I_1$, then $\zeta_i \cc \mu_j$ must be reduced to an element of $\Gamma$.
\item If $i \in I_1$ and $j \in I_2$, then there are set of indices $ t_1, \dots, t_M$ such that every successive $n$ indices in $i, j, t_1, \dots, t_M$ belongs to $A$, $t_1, \dots t_{M-1} \in I_2$, $t_M \in I_1$, and $\zeta_i \cc \eta_j \cc \eta_{t_1} \dots \eta_{t_{M-1}} \cc \mu_{t_M}$ must be reduced to an element of $\Gamma$. 
\item If $i \in I_2$ and $j \in I_1$, then there are set of indices $ s_1, \dots, s_N$ such that every successive $n$ indices in $ s_1, \dots, s_N,i,j$ belongs to $A$, $s_1 \in I_1$, $s_2 \dots s_N \in I_2$, and $\zeta_{s_1} \cc \eta_{s_2} \dots \eta_{s_N} \cc \eta_{i } \cc \mu_j$ must be reduced to an element of $\Gamma$. 
\item If both $i, j \in I_2$, then there are set of indices $ s_1, \dots, s_M$ and $ t_1, \dots, t_M$ such that successive indices $ s_1, \dots, s_N,i,j,t_1, \dots t_M$ appear in some $A$-admissible words, $s_1, t_M \in I_1$, $s_2, \dots, s_N, t_2 \dots t_M \in I_2$, and $\zeta_{s_1} \cc \eta_{s_2} \dots \eta_{s_N} \cc \eta_{i }\cc \eta_j \cc \eta_{t_1} \dots \eta_{t_{M-1}} \cc \mu_{t_M}$ must be reduced to an element of $\Gamma$. 
\end{itemize}

Recall that $\Gamma$ is a finite set of reduced elements with respect to a set of generators $G$ with the property that there is $g_\star \in G \cup C^{-1}$ such that every $\gamma_j$ has only one $g_\star$ at the end and does not start with $g_\star^{-1}$. Let $\Delta$ be a finite set of reduced elements in $\pi_1(X(\mathbf{R}), P_{fix})$ with respect to a set of generators $G$
\begin{equation*}
\begin{aligned}
& \Gamma = \{ \gamma_1, \dots, \gamma_k\}\\
& \Delta = \bigcup_{i \in I_1} \{ \zeta_i, \gamma_{i_1}, \dots \gamma_{i_{n_i}}, \mu_i\} \cup \bigcup_{i \in I_2} \{ \eta_i\} . \\
\end{aligned}
\end{equation*}
Let $V$ denote a linear span of $\Gamma$ and $W$ denote a linear span of $\Delta$. We define two linear maps $S,M: V \to W$ in the following way
\begin{itemize}
\item \textit{Splitting map} $S:V \to W$:
\begin{equation*}
S: \left\{ \begin{aligned}
&\gamma_i \mapsto \mu_i + \gamma_{i_1} + \cdots+ \gamma_{i_{n_i}} + \zeta_i, \qquad \text{if } \ i \in I_1\\
&\gamma_i \mapsto \eta_i,\qquad \text{if } \ i \in I_2\\ 
\end{aligned} \right.
\end{equation*}
\item \textit{Merging map} $M:V \to W$:
\begin{equation*}
M: \left\{ \begin{aligned}
&\gamma_i \mapsto \gamma_i, \qquad \text{if }\ \gamma_i \in \Delta\\
&\gamma_i \mapsto \zeta_i + \eta_{i_1} + \cdots \eta_{i_s} + \mu_i, \ \ \text{if }\ \ \gamma_i \not\in \Delta, \ \text{and}\ \gamma_i = \zeta_i \cc \eta_{i_1} \cdots \eta_{i_s} \cc \mu_i\\
\end{aligned}\right.
\end{equation*}
\end{itemize}

\begin{lem}
Both a Splitting map, $S$ and a Merging map $M$ maps a positive span of $\Gamma$, $Sp^+ \Gamma \subset V$ to a positive span of $\Delta$, $Sp^+ \Delta \subset W$.
\end{lem}
\begin{proof}
For both $S$ and $M$, the image of $\gamma_i$ is given by a positive sum of elements in $\Delta$. This Lemma follows from the fact that the matrix representations for $S$ and $M$ are given by matrices with positive entries. 
\end{proof}

\begin{lem}\label{L:Grate}
If $\Gamma_A$ is $f^{-1}_{\mathbf{R}*}$-invariant, then for each $[\gamma] \in \Gamma_A$ we have 
\[ M \phi[\gamma]_\# = S [\gamma]_\# \]
\end{lem}

\begin{proof}
Let $[\gamma] \in \Gamma_A$ with $\gamma = \gamma_{i_1} \cc \cdots \gamma_{i_n}$. Then $\phi[\gamma] = [f^{-1}_{\mathbf{R}*} \gamma] = [^{-1}_{\mathbf{R}*} \gamma_{i_1} \cdots f^{-1}_{\mathbf{R}*} \gamma_{i_n}] $
For each $j=1, \dots n$ we have $f^{-1}_{\mathbf{R}*} \gamma_{i_j} = \prod_{\delta_{s_j} \in \Delta} \delta_{s_j}$, which determines $S \gamma_{i_j}$. On the other hand $\phi[\gamma]$ is determined after we combine $\zeta$'s, $\mu$'s and $\eta$'s. The merging map $M$ is decompose each $\gamma \in \Delta$ into sum of $\zeta$'s, $\mu$'s, and $\eta$'s. Thus by counting the number of each element in $\Delta$, we have $M \phi[\gamma]_\# = S [\gamma]_\#$.
\end{proof}
 
 Since the splitting map, $S$ and the merging map, $M$ are linear, it is easier to connect to the growth rate of $[\gamma]_\#$ under the iterations of $\phi$. 

\begin{thm}\label{T:Grate} Let $\Delta$ is a finite set of reduced elements in $G$. Let $V = \text{Span}\, \Gamma$ be a linear span of $\Gamma$ and $W = \text{Span}\, \Delta$ be a linear span of $\Delta$. Suppose there are two one-to-one linear maps $T_i: V \to W$, $i=1,2$ such that both linear maps $T_i ( \text{Sp}^+ \Gamma) = \text{Sp}^+ \Delta$. Also suppose that \[ T_2 \phi [\gamma]_\# \ = T_1[\gamma]_\#,\qquad \text{for all } \gamma \in \Gamma_A \]
where $\phi$ is the induced action $f^{-1}_{\mathbf{R}*}$ on $\pi_1(X(\mathbf{R}))$ and $[\gamma]_\# = \sum_i n_i(\gamma) \gamma_i$. 
If there is $v_\lambda \in \text{Sp}^+\Gamma$ such that $T_1 v_\lambda = \lambda \, T_2 v_\lambda$ for some $\lambda >1$, then the growth rate of the induced action on the fundamental group \[\rho(f^{-1}_{\mathbf{R}*}|_{\pi_1 (X(\mathbf{R}))}) \ge \lambda.\]
\end{thm}

\begin{proof}
Since the rank of $T_2 = \text{dim} V$, there is a left inverse $T_{2\, left}^{-1}$. And $v_\lambda$ is an eigenvector of $T_{2\, left}^{-1}T_1$ corresponding to an eigenvalue $\lambda>1$. Because both $T_1$ and $T_2$ are one-to-one, for each $[\gamma] \in \Gamma_A$ and $n\ge 0$ there is $[\gamma' ] \in \Gamma_A$ such that \[ \phi[\gamma']_\# = (T_{2\, left}^{-1}T_1)^n [\gamma]_\#. \] 
Since $T_1 v_\lambda = \lambda \, T_2 v_\lambda$, for any left inverse $T_{2\, left}^{-1}$, $v_\lambda$ is an eigenvector associated with an eigenvalue $\lambda$ for $T_{2\, left}^{-1} T_1$. Since $v_\lambda$ is an eigenvector corresponding an eigenvalue $\lambda>1$, we have a sequence $[\gamma]_n \in \Gamma_A$ such that \[ ([\gamma]_n) _\# = (T_{2\, left}^{-1}T_1)^n [\gamma]_\# = c \lambda^n v_\lambda + v_n\] for some positive constant $c$ and $v_n \in \text{Sp}^+ \Gamma \subset V$. 
It follows that \[\sup_{\gamma \in \Gamma_A} \left\{ \limsup_{n \to \infty} (\ell_\Gamma( \phi^n \gamma))^{1/n} \right \} \ge \lambda \] and therefore \[\rho(f^{-1}_{\mathbf{R}*}|_{\pi_1 (X(\mathbf{R}))}) \ge \lambda.\]
\end{proof}

\begin{rem} In this paper, we are not claiming the existence of an invariant set of conjugacy classes of $A$-admissible words for possible orbit data. For those cases we tried, we found the invariant set $\Gamma$ by iterating a generator for about $4 (n_1+n_2+n_3)$ times. For large $n_1+n_2+n_3$, it was computationally hard to iterate on our laptop computer. 

We also want to point out that we always get the same invariant set $\Gamma$ regardless of the choice of an initial word for iterations for those cases we investigated. 
\end{rem}

\section{Coxeter Case}\label{S:coxeter}

The real diffeomorphism associated to the orbit data $1,1,8$ and a cyclic permutation $\sigma: 1\mapsto 2\mapsto 3 \mapsto 1$ is a geometric realization of a Coxeter element in the Wyle group $E_{10}$. This real diffeomorphism has the maximal possible entropy $=$ the logarithm of the Lehmer's number, where the Lehmer's number is the largest real root of $\chi(t) =t^{10} - t^9-t^7+t^6-t^5+t^4-t^3-t+1$, and is the smallest positive entropy that rational surface automorphisms can have. (See \cite{Diller-Kim, BedfordKim:2011, McMullen:2002, McMullen:2007}.)

Let $f_{\mathbf{R}}:X(\mathbf{R}) \to X(\mathbf{R})$ be the real diffeomorphism associated to the orbit data $1,1,8$ and a cyclic permutation $\sigma: 1\mapsto 2\mapsto 3 \mapsto 1$.  The fundamental group of $\pi_1 (X(\mathbf{R}), P_{fix})$ has $11$ generators with one relation. With a set of generator $G=\{ e, a_1, \dots, a_8,b_1,c_1\}$ defined in Section \ref{ss:generators}, we have $e^2 \cc a_8^2\cc a_7^2\cc a_6^2 \cc a_5^2 \cc c_1^2 \cc a_4^2 \cc a_3^2 \cc b_1^2 \cc a_2^2 \cc a_1^2 =1$. As in the Section \ref{S:Action}, the induced action $f_{\mathbf{R}*}^{-1}|_{\pi_1 (X(\mathbf{R}),P_{fix})}$ is determined by the image of each generator under $f_{\mathbf{R}*}^{-1}|_{\pi_1 (X(\mathbf{R}),P_{fix})} $. 
\begin{equation}\label{E:coxeterAction}
\begin{aligned}
f^{-1}_{\mathbf{R}*} \ :\ &e \mapsto  e \cc a_8^2 \cc a_7^2 \cc a_6^2 \cc a_5^2 \cc c_1 \cc b_1 \cc a_2^2 \cc a_1 \cc e\\
&a_8 \mapsto e^{-1} \cc a_1^{-1} \cc a_2^{-2} \cc b_1^{-1}\\
&a_7  \mapsto b_1 \cc a_2^2 \cc a_1 \cc e \cc a_8^{-1} \cc e^{-1} \cc a_1^{-1} \cc a_2^{-2} \cc b_1^{-1} \\
&a_6  \mapsto b_1 \cc a_2^2 \cc a_1 \cc e \cc a_8^2  \cc a_7^{-1} \cc a_8^{-2} \cc e^{-1} \cc a_1^{-1} \cc a_2^{-2} \cc b_1^{-1} \\
&a_5  \mapsto b_1 \cc a_2^2 \cc a_1 \cc e \cc a_8^2\cc a_7^2  \cc a_6^{-1} \cc a_7^{-2} \cc a_8^{-2} \cc e^{-1} \cc a_1^{-1} \cc a_2^{-2} \cc b_1^{-1} \\
&a_4 \mapsto b_1^{-1} \cc a_3^{-2} \cc a_4^{-2} \cc c_1^{-1} \cc a_5 \cc c_1 \cc a_4^2 \cc a_3^2 \cc b_1\\
&a_3 \mapsto b_1^{-1} \cc a_3^{-2} \cc a_4^{-1}  \cc a_3^2 \cc b_1\\
\end{aligned}
\end{equation} 
\begin{equation*}
\begin{aligned}
&a_2 \mapsto e \cc a_8^2 \cc a_7^2 \cc a_6^2 \cc a_5^2 \cc c_1 \cc a_3 \cc c_1^{-1} \cc a_5^{-2} \cc a_6^{-2} \cc a_7^{-2} \cc a_8^{-2} \cc e^{-1}\\
&a_1 \mapsto e \cc a_8^2 \cc a_7^2 \cc a_6^2 \cc a_5^2 \cc c_1 \cc b_1 \cc a_2^{-1} \cc b_1^{-1} \cc c_1^{-1} \cc a_5^{-2} \cc a_6^{-2} \cc a_7^{-2} \cc a_8^{-2} \cc e^{-1}\\
&b_1 \mapsto b_1^{-1} \cc a_3^{-2} \cc c_1^{-1} \cc a_5^{-2} \cc a_6^{-2} \cc a_7^{-2} \cc a_8^{-2} \cc e^{-1}\\
&c_1 \mapsto b_1 \cc a_2^2 \cc a_1 \cc e \cc a_8^2 \cc a_7^2 \cc a_6^2 \cc c_1 \cc a_4^2 \cc a_3^2 \cc b_1. 
\end{aligned}
\end{equation*}

For the real diffeomorphism associated with the orbit data $1,1,n$ with a cyclic permutation, the induced action $f_{\mathbf{R}*}^{-1}|_{\pi_1 (X(\mathbf{R}),P_{fix})}$ is given in Appendix \ref{A:case1}.

\subsection{$A$-Admissible cyclic words $\Gamma_A$}

Let \[ \Gamma = \{ \gamma_1, \dots, \gamma_{10} \} \quad \text{    and    } \quad \Delta = \{ \gamma_1,\gamma_2,\gamma_4,\gamma_7,\gamma_8,\gamma_9,\gamma_{10},\zeta_1,\zeta_2,\mu_1\}\] where 
\begin{equation*}
\begin{aligned}
& \gamma_1 \ &=\ \ &a_2^{-1} \cc b_1^{-1} \cc a_8^{-1} \cc e^{-1} \\
& \gamma_2 \ &=\ \ &a_2^{-1} \cc b_1^{-1} \cc a_7^{-1} \cc a_8^{-2}  \cc e^{-1} \\
& \gamma_3 \ &=\ \ &b_1^{-1} \cc a_3^{-1} \cc a_6^{-1} \cc a_7^{-2} \cc a_8^{-2}  \cc e^{-1} \\
& \gamma_4 \ &=\ \ &a_1^{-1} \cc a_2^{-2} \cc b_1^{-2} \cc a_3^{-1} \cc a_7^{-1} \cc a_8^{-2}  \cc e^{-1} \\
& \gamma_5 \ &=\ \ &a_2^{-1} \cc b_1^{-2} \cc a_3^{-1} \cc a_6^{-1} \cc a_7^{-2} \cc a_8^{-2}  \cc e^{-1} \\
& \gamma_6 \ &=\ \ &b_1^{-1} \cc a_3^{-1} \cc a_5^{-1} \cc a_6^{-2} \cc a_7^{-2} \cc a_8^{-2}  \cc e^{-1} \\
& \gamma_7 \ &=\ \ &a_1^{-1} \cc a_2^{-2} \cc b_1^{-2} \cc a_3^{-1} \cc a_6^{-1} \cc a_7^{-2} \cc a_8^{-2}  \cc e^{-1} \\
& \gamma_8 \ &=\ \ &a_1^{-1} \cc a_2^{-2} \cc b_1^{-2} \cc a_3^{-2} \cc a_4^{-1} \cc a_5^{-1} \cc a_6^{-2} \cc a_7^{-2} \cc a_8^{-2}  \cc e^{-1} \\
& \gamma_9 \ &=\ \ &a_1^{-1} \cc a_2^{-2} \cc b_1^{-2} \cc a_3^{-2} \cc a_4^{-2} \cc c_1^{-1} \cc. a_5^{-1} \cc a_6^{-2} \cc a_7^{-2} \cc a_8^{-2}  \cc e^{-1} \\
& \gamma_{10} \ &=\ \ &a_1^{-1} \cc a_2^{-2} \cc b_1^{-2} \cc a_3^{-2} \cc a_4^{-1} \cc c_1^{-1} \cc a_5^{-2} \cc a_6^{-2} \cc a_7^{-2} \cc a_8^{-2} \cc e^{-1} \\
& \zeta_{1} \ &=\ \ &a_2\cc a_1\\
& \zeta_{2} \ &=\ \ & b_1\cc a_2^2 \cc a_1 \\ 
& \mu_{1} \ &=\ \ &a_1^{-1} \cc a_2^{-2} \cc b_1^{-2} \cc a_3^{-1} \cc a_5^{-1} \cc a_6^{-2} \cc a_7^{-2} \cc a_8^{-2} \cc e^{-1}\\
\end{aligned}
\end{equation*}
Let us also set the admissible pairs 
\begin{equation*}
\begin{aligned} 
A= \{ (1,4), (1,7), & (1,10),(2,8),(2,10),(3,2),(3,8),(3,9),(4,10), \\  
& (5,2),(5,8),(5,9),(6,1),(6,5),(7,2),(8,1),(9,1),(9,2),(10,3),(10,5),(10,6)\}. \\
\end{aligned}
\end{equation*}
With the unique relator, it is easy to see that any cyclic words in $\Gamma$ is cyclically reduced. Let $\Gamma_A$ be the set of $A$-admissible cyclic words in $\Gamma$. For instance $\gamma_1 \cc \gamma_{10}\cc \gamma_6$ and $ \gamma_2\cc \gamma_{10}\cc \gamma_5\cc \gamma_9 \in \Gamma_A$.

To compute the growth rate of the induced $\pi_1$-action for this real diffeomorphism with the orbit data $1,1,8$ with a cyclic permutation, it is convenient to work with $f_{\mathbf{R}*}^{-2}$. Using the formula given in Appendix \ref{A:case1}, we see that 
\begin{equation}\label{E:coxeter}
\begin{aligned}
&f_{\mathbf{R}*}^{-2} \, \gamma_1 \ &=\ \ & (e^{-1} a_1^{-1}) \cc \gamma_{10}\cc \zeta_2 \cc ( a_1 e) ,\ \ 
&f_{\mathbf{R}*}^{-2} \, \gamma_2 \ &=\ \ & (e^{-1} a_1^{-1}) \cc \gamma_{10}\cc \zeta_1 \cc ( a_1 e) \\
&f_{\mathbf{R}*}^{-2} \, \gamma_3 \ &=\ \ & (e^{-1} a_1^{-1}) \cc \gamma_9\cc \gamma_1 \cc ( a_1 e) ,\ \ 
&f_{\mathbf{R}*}^{-2} \, \gamma_4 \ &=\ \ & (e^{-1} a_1^{-1}) \cc \mu_1\cc \zeta_1 \cc ( a_1 e) \\
&f_{\mathbf{R}*}^{-2} \, \gamma_5 \ &=\ \ & (e^{-1} a_1^{-1}) \cc \gamma_8\cc \gamma_1 \cc ( a_1 e) ,\ \ 
&f_{\mathbf{R}*}^{-2} \, \gamma_6 \ &=\ \ & (e^{-1} a_1^{-1}) \cc \gamma_9\cc \gamma_2 \cc ( a_1 e) \\
&f_{\mathbf{R}*}^{-2} \, \gamma_7 \ &=\ \ & (e^{-1} a_1^{-1}) \cc \mu_1\cc \gamma_1 \cc ( a_1 e) ,\ \ 
&f_{\mathbf{R}*}^{-2} \, \gamma_8 \ &=\ \ & (e^{-1} a_1^{-1}) \cc \gamma_7\cc \gamma_2 \cc ( a_1 e) \\
&f_{\mathbf{R}*}^{-2} \, \gamma_9 \ &=\ \ & (e^{-1} a_1^{-1}) \cc \gamma_4 \cc ( a_1 e) ,\ \ 
&f_{\mathbf{R}*}^{-2} \, \gamma_{10} \ &=\ \ & (e^{-1} a_1^{-1}) \cc \gamma_7 \cc ( a_1 e) \\
\end{aligned}
\end{equation}
Also we have 
\begin{equation}\label{E:comb}
\gamma_3 = \zeta_2 \cc \gamma_7, \quad \gamma_5 = \zeta_1 \cc \gamma_7, \quad \text{and} \quad \gamma_6 = \zeta_2 \cc \mu_1 
\end{equation}
Since $\Gamma_A$ is a set of  $A$-admissible cyclic words, we can ignore the extra factors $e^{-1}\cc a_1^{-1}$ and $a_1 e$. Using the admissible pairs, an $A$-admissible cyclic word in $\Gamma_A$ is mapped to another $A$-admissible cyclic words in $\Gamma_A$ under $f_{\mathbf{R}*}^{-2} $ by removing $\zeta_1,\zeta_2$, and $\mu_1$ using (\ref{E:comb}). In fact, the induced action $f_{\mathbf{R}*}^{-2} $ consists of two separate actions, splitting and merging. The action in (\ref{E:coxeter}) determines splitting action and the identities in (\ref{E:comb}) determine merging action. For example, the cyclically reduced image $[f_{\mathbf{R}*}^{-2} \gamma_1\cc\gamma_{10}\cc \gamma_6 ]$ is given by 
\begin{equation*}
f_{\mathbf{R}*}^{-2}: \left\{ \begin{aligned}
&  \text{Splitting:  } \ &\gamma_1\cc\gamma_{10} \cc \gamma_6  \mapsto (\gamma_{10} \cc \zeta_2) \cc (\gamma_7) \cc (\gamma_9 \gamma_2) \\
&  \text{Merging:  } \ & \gamma_{10} \cc (\zeta_2 \cc \gamma_7) \cc \gamma_9 \gamma_2 \mapsto \gamma_{10} \cc \gamma_3 \cc \gamma_9 \cc \gamma_2
\end{aligned} \right.
\end{equation*}
It follows that \[ [f_{\mathbf{R}*}^{-2} (\gamma_1\cc\gamma_{10}\cc \gamma_6) ] = [ \gamma_{10} \cc \gamma_3 \cc \gamma_9 \cc \gamma_2]. \]
Similarly we have 
\begin{equation*}
f_{\mathbf{R}*}^{-2}: \left\{ \begin{aligned}
&  \text{Splitting:  } \ \ \gamma_{10}\cc\gamma_{3} \cc \gamma_9 \cc \gamma_2   \mapsto (\gamma_7) \cc  (\gamma_{9} \cc \gamma_1 ) \cc (\gamma_4) \cc (\gamma_{10} \cc \zeta_1)  \sim  (\gamma_{9} \cc \gamma_1 ) \cc (\gamma_4) \cc (\gamma_{10}\cc \zeta_1)\cc \gamma_7 \\
&  \text{Merging:  } \ \   \gamma_{9} \cc \gamma_1 \cc \gamma_4 \cc \gamma_{10} \cc(\zeta_1\cc \gamma_7)  \mapsto \gamma_{9} \cc \gamma_1 \cc \gamma_4 \cc \gamma_{10} \cc \gamma_5
\end{aligned} \right.
\end{equation*}
and thus $[f_{\mathbf{R}*}^{-2}(\gamma_{10}\cc\gamma_{3} \cc \gamma_9 \cc \gamma_2)] = [\gamma_{9} \cc \gamma_1 \cc \gamma_4 \cc \gamma_{10} \cc \gamma_5]$.
Merging depends only two words in $\Delta$. Thus by checking each pair in $A$, one can see for each $\gamma \in \Gamma_A$, $f_{\mathbf{R}*}^{-2} \gamma \in \Gamma_A$. 
\begin{lem}
The set of $A$-admissible cyclic words $\Gamma_A$ is invariant under $f_{\mathbf{R}*}^{-2}$, that is $f_{\mathbf{R}*}^{-2}(\Gamma_A) \subset \Gamma_A$.
\end{lem}

\begin{proof}
We use equations in (\ref{E:coxeter}). Again since $\Gamma_A$ is a set of $A$-admissible cyclic words, either we consider the conjugation $(a_1 \cc e ) f^{-2}_{\mathbf{R}*} (e^{-1} \cc a_1^{-1})$ or we may simply ignore the extra factors $e^{-1} \cc a_1^{-1}$ and $a_1 \cc e$. To check that $\Gamma_A$ is  $f_{\mathbf{R}*}^{-2}$-invariant, it is sufficient to check each pair in $A$. 
For each $(i,j) \in A$, the image of $\gamma_i \cc \gamma_j$ under  $f_{\mathbf{R}*}^{-2}$ is written in the form $\zeta_i \cc\gamma_{t_1}\cdots \gamma_{t_m} \cc \mu_j$. Since $\zeta_i$ and $\mu_j$ would be combined to create some $\gamma_s \in \Gamma$, we need to see if successive pairs of indices in $\gamma_{t_1}\cdots \gamma_{t_m} $ are in $A$, that is, we need to check wether $(t_i,t_{i+1}) \in A$ for all $i=1, \dots m-1$. 
\begin{itemize}
\item for $(1,4) \in A$, we have $f_{\mathbf{R}*}^{-2} \gamma_1 \cc \gamma_4 =\gamma_{10} \cc \zeta_2 \cc \mu_1 \cc \zeta_1 = \gamma_{10} \cc \gamma_{6} \cc \zeta_1$  and $(10,6) \in A$.
\item for $(1,7) \in A$, we have $f_{\mathbf{R}*}^{-2} \gamma_1 \cc \gamma_7 = \gamma_{10} \cc \zeta_2 \cc \mu_1 \cc \gamma_1 = \gamma_{10} \cc \gamma_{6} \cc \gamma_1$, and  both $(10,6), (6,1) \in A$
\item for $(1,10) \in A$, we have $f_{\mathbf{R}*}^{-2} \gamma_1 \cc \gamma_{10} = \gamma_{10} \cc \zeta_2 \cc  \gamma_7 = \gamma_{10} \cc \gamma_{3} $, and  both $(10,3) \in A$
\item for $(2,8) \in A$, we have $f_{\mathbf{R}*}^{-2} \gamma_2 \cc \gamma_{8} = \gamma_{10} \cc \zeta_1 \cc  \gamma_7\cc \gamma_2 = \gamma_{10} \cc \gamma_{5} \cc \gamma_2 $, and  both $(10,5),(5,2) \in A$
\end{itemize}
and so on. Notice that $\gamma_7$ in the image of $\gamma_{10}$ should be treated as $\mu_{10}$. From $A$, we see that $\gamma_{10}$ is always followed by $\gamma_1, \gamma_2$ or $\gamma_4$.
Thus we need to keep that in mind that $\gamma_7$ in $f_{\mathbf{R}*}^{-2} \gamma_{10}$ will be joined with $\zeta_1$ or $\zeta_2$ and merged into $\gamma_5$ or $\gamma_3$. 
\begin{itemize}
\item for $(1,10,3) \in A$,  $f_{\mathbf{R}*}^{-2} \gamma_1 \cc \gamma_{10} \cc \gamma_3  =\gamma_{10} \cc \zeta_2 \cc \gamma_7 \cc \gamma_9 \cc \gamma_1 = \gamma_{10} \cc \gamma_3 \cc \gamma_9 \cc \gamma_1$  and $(10,3),(3,9),(9,1)\in A$.
\item for $(2,10,3) \in A$,  $f_{\mathbf{R}*}^{-2} \gamma_2 \cc \gamma_{10} \cc \gamma_3  =\gamma_{10} \cc \zeta_1 \cc \gamma_7 \cc \gamma_9 \cc \gamma_1 = \gamma_{10} \cc \gamma_5 \cc \gamma_9 \cc \gamma_1$  and $(10,5),(5,9),(9,1)\in A$.
\item for $(4,10,3) \in A$,  $f_{\mathbf{R}*}^{-2} \gamma_4 \cc \gamma_{10} \cc \gamma_3  =\mu_1 \cc \zeta_1 \cc \gamma_7 \cc \gamma_9 \cc \gamma_1 =\mu_1 \cc \gamma_5 \cc \gamma_9 \cc \gamma_1$  and $(5,9),(9,1)\in A$.
\item for $(1,10,5) \in A$,  $f_{\mathbf{R}*}^{-2} \gamma_1 \cc \gamma_{10} \cc \gamma_5  =\gamma_{10} \cc \zeta_2 \cc \gamma_7 \cc \gamma_8 \cc \gamma_1 = \gamma_{10} \cc \gamma_3 \cc \gamma_8\cc \gamma_1$  and $(10,3),(3,8),(8,1)\in A$.
\item for $(2,10,5) \in A$,  $f_{\mathbf{R}*}^{-2} \gamma_2 \cc \gamma_{10} \cc \gamma_5  =\gamma_{10} \cc \zeta_1 \cc \gamma_7 \cc \gamma_8 \cc \gamma_1 = \gamma_{10} \cc \gamma_5 \cc \gamma_8 \cc \gamma_1$  and $(10,5),(5,8),(8,1)\in A$.
\item for $(4,10,5) \in A$,  $f_{\mathbf{R}*}^{-2} \gamma_4 \cc \gamma_{10} \cc \gamma_5  =\mu_1 \cc \zeta_1 \cc \gamma_7 \cc \gamma_8 \cc \gamma_1 =\mu_1 \cc \gamma_5 \cc \gamma_8 \cc  \gamma_1$  and $(5,8),(8,1)\in A$.
\item for $(1,10,6) \in A$,  $f_{\mathbf{R}*}^{-2} \gamma_1 \cc \gamma_{10} \cc \gamma_6  =\gamma_{10} \cc \zeta_2 \cc \gamma_7 \cc \gamma_9 \cc \gamma_2 = \gamma_{10} \cc \gamma_3 \cc \gamma_9 \cc \gamma_2$  and $(10,3),(3,9),(9,2)\in A$.
\item for $(2,10,6) \in A$,  $f_{\mathbf{R}*}^{-2} \gamma_2 \cc \gamma_{10} \cc \gamma_6  =\gamma_{10} \cc \zeta_1 \cc \gamma_7 \cc \gamma_9 \cc \gamma_2 = \gamma_{10} \cc \gamma_5 \cc \gamma_9 \cc \gamma_2$  and $(10,5),(5,9),(9,2)\in A$.
\item for $(4,10,6) \in A$,  $f_{\mathbf{R}*}^{-2} \gamma_4 \cc \gamma_{10} \cc \gamma_6  =\mu_1 \cc \zeta_1 \cc \gamma_7 \cc \gamma_9 \cc \gamma_2 =\mu_1 \cc \gamma_5 \cc \gamma_9 \cc \gamma_2$  and $(5,9),(9,2)\in A$.
\end{itemize}
By checking all pairs in $A$, we conclude that $\Gamma_A$ is $f_{\mathbf{R}*}^{-2}$-invariant. 
\end{proof}

\subsection{Growth Rate Estimation}
Let $\phi [\gamma] :=[ f_{\mathbf{R}*}^{-2}|_{\Gamma_A}\, (\gamma)]$ denote the induced action on $\Gamma_A$.  We estimate the growth rate of the induced action $f_{\mathbf{R}*}^{-2}$ on $\pi_1( X(\mathbf{R}))$ by the growth rate of $\phi$.  Since there are no relations in $\Gamma$, the length $\ell_\Gamma([\gamma])$  depends only on the number of each $\gamma_i$ in $[\gamma] \in \Gamma_A$.  To compute the length $\ell_\Gamma$, we associate $\gamma\in \Gamma_A$ with an element $\gamma_\#$ in the positive linear span $\text{Sp}^+ \Gamma$ of $\Gamma$ and use $[\gamma]_\#$ and $(\phi([\gamma]))_\#$ defined in (\ref{E:counting}), and (\ref{E:cmap}) .
Thus the equation $[f_{\mathbf{R}*}^{-2} (\gamma_1\cc\gamma_{10}\cc \gamma_6) ] = [ \gamma_{10} \cc \gamma_3 \cc \gamma_9 \cc \gamma_2]$ tells us  
\[ [\gamma_1 \gamma_{10} \gamma_6]_\# =  \gamma_1 + \gamma_6 + \gamma_{10}\ \ \  \text{   and   } \ \ \ \phi[\gamma_1\gamma_{10}\gamma_6]_\#=  \gamma_2 + \gamma_3 +\gamma_9 + \gamma_{10} \]

Let $V = Span \Gamma$ be a linear span of the ordered set  $\Gamma$ and $W= Span \Delta$ be a linear span of the ordered set $\Delta$. 
Using equations (\ref{E:coxeter}), let  $S: V \to W$ be a linear maps such that $S\gamma_1 = \gamma_{10} + \zeta_2$, $S\gamma_2 = \gamma_1+ \gamma_9$, etc. Also with equations (\ref{E:comb}), we define a linear map $M: V \to W$  by $M: \gamma_3 \mapsto  \gamma_7 + \zeta_2$, $M:  \gamma_5 \mapsto \gamma+ \zeta_1$, $M: \gamma_6 \mapsto \zeta_2 + \mu_1$ and $M: \gamma_i \mapsto \gamma_i$ for all other $\gamma_i$'s. 
\[ S = \begin{bmatrix} \ 0&\ 0&\ 1&\ 0&\ 1&\ 0&\ 1&\ 0&\ 0&\ 0 \\  \ 0&\ 0&\ 0&\ 0&\ 0&\ 1&\ 0&\ 1&\ 0&\ 0 \\  \ 0&\ 0&\ 0&\ 0&\ 0&\ 0&\ 0&\ 0&\ 1&\ 0 \\ 
\ 0&\ 0&\ 0&\ 0&\ 0&\ 0&\ 0&\ 1&\ 0&\ 1 \\   \ 0&\ 0&\ 0&\ 0&\ 1&\ 0&\ 0&\ 0&\ 0&\ 0 \\  \ 0&\ 0&\ 1&\ 0&\ 0&\ 1&\ 0&\ 0&\ 0&\ 0 \\ 
 \ 1&\ 1&\ 0&\ 0&\ 0&\ 0&\ 0&\ 0&\ 0&\ 0 \\   \ 0&\ 1&\ 0&\ 1&\ 0&\ 0&\ 0&\ 0&\ 0&\ 0 \\   \ 1&\ 0&\ 0&\ 0&\ 0&\ 0&\ 0&\ 0&\ 0&\ 0 \\ 
  \ 0&\ 0&\ 0&\ 1&\ 0&\ 0&\ 1&\ 0&\ 0&\ 0 \\  \end{bmatrix}, 
  M= \begin{bmatrix}  \ 1&\ 0&\ 0&\ 0&\ 0&\ 0&\ 0&\ 0&\ 0&\ 0 \\   \ 0&\ 1&\ 0&\ 0&\ 0&\ 0&\ 0&\ 0&\ 0&\ 0 \\   \ 0&\ 0&\ 0&\ 1&\ 0&\ 0&\ 0&\ 0&\ 0&\ 0 \\ 
   \ 0&\ 0&\ 1&\ 0&\ 1&\ 0&\ 1&\ 0&\ 0&\ 0 \\  \ 0&\ 0&\ 0&\ 0&\ 0&\ 0&\ 0&\ 1&\ 0&\ 0 \\  \ 0&\ 0&\ 0&\ 0&\ 0&\ 0&\ 0&\ 0&\ 1&\ 0 \\ 
    \ 0&\ 0&\ 0&\ 0&\ 0&\ 0&\ 0&\ 0&\ 0&\ 1 \\  \ 0&\ 0&\ 0&\ 0&\ 1&\ 0&\ 0&\ 0&\ 0&\ 0 \\  \ 0&\ 0&\ 1&\ 0&\ 0&\ 1&\ 0&\ 0&\ 0&\ 0 \\ 
     \ 0&\ 0&\ 0&\ 0&\ 0&\ 1&\ 0&\ 0&\ 0&\ 0 \\ \end{bmatrix} \]
     
 It follows that both $S$ and $M$ map the positive span $\text{Sp}^+ \Gamma $ to $\text{Sp}^+\Delta$ and 
 \begin{equation}\label{E:CoxeterSM} M \circ \phi [\gamma]_\# \ = \ S(\gamma_\#), \qquad \text{for all   } \gamma  \in \Gamma_A 
 \end{equation}
Both $S$ and $M$ are invertible and we see that the characteristic polynomial of $M^{-1} S$ is given by \[ \chi (t) = t^{10} -t^9-t^7+t^6-t^5+t^4-t^3-t+1\]. The largest real eigenvalue is approximatly $1.38314 $, the square of the Lehmer's number and  a corresponding eigenvector $v_\lambda \in \text{Sp}^+\Gamma$
\[ v_\lambda \approx (0.769,0.614,0.178,0.290,0.654,0.376,0.232,0.473,0.402,1). \]
This largest real eignvalue is the unique eigenvalue $\lambda >1$ such that $\log \lambda = h_{top} f^2$, the maximum possible topological entropy. 

\begin{prop}
The growth rate of the induced action on $\pi_1(X(\mathbf{R}))$ is given by the Lehmer's number $=\sqrt{\lambda} >1$.
\[ \log \rho( f_{\mathbf{R}*}|_{H_1(X(\mathbf{R});\mathbf{R})} )=\log \rho( f^{-1}_{\mathbf{R}*}|_{H_1(X(\mathbf{R});\mathbf{R})} )  = \log \rho (f^{-1}_{\mathbf{R*}}|_{\pi_1(X(\mathbf{R}))}) = h_{top} (f_\mathbf{R}) = h_{top} (f). \]
\end{prop}

\begin{proof}
This proposition follows from Theorem \ref{T:Grate}. 
For each $\gamma \in \Gamma$, $\gamma_\# \in \text{Sp}^+ \Gamma$ and we have $(M^{-1} S)^n \gamma_\#  = c \lambda^n v_\lambda + v_n$ for some positive constant $c$ and $v_n \in \text{Sp}^+ \Gamma \subset V$. Furthermore, using (\ref{E:CoxeterSM}) we see  $(M^{-1} S)^n \gamma_\# \in \Gamma_A^\#$ for all $n$. Sincet $\ell_\Gamma (\gamma)=  \gamma_\# [1 \,1 \cdots\, 1]^t$ for $[\gamma] \in \Gamma_A$, we have
\[\sup_{\gamma \in \Gamma_A} \left\{ \limsup_{n \to \infty} (\ell_\Gamma( \phi^n \gamma))^{1/n}  \right \} \ge \lambda \]
Since $\phi$ is induced map from $f_{\mathbf{R}*}^{-2}$, we see that the growth rate $\rho (f^{-1}_{\mathbf{R*}}|_{\pi_1(X(\mathbf{R}))}) \ge \sqrt{\lambda}$.
\end{proof}

\section{Maximum Entropy}\label{S:max}
In this section, we prove two main theorems. Here we present $5$ different orbit data. Two of them give real diffeomorphisms such that the growth rate of the homology classes of $X(\mathbf{R})$  is not exponential and three other cases are associated with real diffeomorphisms with non-maximal possible topological entropy. For each case, a set, $G$ of generators of the fundamental group and the action on the fundamental group $f^{-1}_{\mathbf{R}*}|_{\pi(X(\mathbf{R}), P_{fix})}$ with respect to $G$ are given in the Appendix \ref{A:case3} and \ref{A:case4}. In Appendix \ref{ApendB}, there are the list of sets $\Gamma, \Delta$ of reduced words, Admissible $n$-tuples, Splitting and Merging maps $S,M$, the vector $v_\lambda \in \Gamma^+$ such that $M \lambda v_\lambda = S v_\lambda$  together with $\lambda$. For every case, the procedure is essentially identical :(1) Compute the induced action on the fundamental group, (2) Find the invariant set of equivalent classes $\Gamma_A$ by constructing a set of (non-cyclically) reduced words $\Gamma$ and a set of admissible $n$-pairs $A$, (3) Constructing a Splitting map $S$ and a Merging map $M$, (4) Compute a vector $v_\lambda$ such that $\lambda M v_\lambda = S v_\lambda$ with $\lambda>1$. For this and the next sections, we give a brief sketch for each case. Let us start this section with a remark on the Splitting map,$S$ and Merging map, $M$ in the Appendix \ref{ApendB}. The formulas for the images of $\gamma_i$ under $f^{-1}_{\mathbf{R}*}$ and combinations of $\delta_i$'s to get $\gamma_j$ can be obtained using the Splitting and Merging map in the Appendix \ref{ApendB}. In the Section \ref{S:coxeter}, the formulas in (\ref{E:coxeter}, \ref{E:comb}) are used to show that $\Gamma_A$ is $f^{-1}_{\mathbf{R}*}$-invariant.

\begin{rem} For each cases in the Appendix \ref{ApendB}, there is a reduced (non-cyclic) word $\epsilon$ such that 
\[ \text{If } \quad  S\cc \gamma_i = \omega_{i_1} + \cdots +\omega_{i_{n_i}}, \quad \text{then } \ \   f^{-1}_{\mathbf{R}*} \gamma_ i = \epsilon \cc \omega_{i_1} \cdots \omega_{i_{n_i}} \cc \epsilon^{-1}.\]
Also 
\[ \text{If } \quad M \cc \gamma_i = \delta_{i_1} + \cdots +\delta_{i_{m_i}}, \quad \text{then }\ \   \gamma_ i = \delta_{i_1} \cc \cdots \delta_{i_{m_i}} \in  \pi_1 (X(\mathbf{R}), P_{fix}) \]
\end{rem}

\subsection{$1,3,9$ with a cyclic permutation: zero exponential homology growth rate  but maximal topological entropy}\label{SS:139}
Let us first consider a real diffeomrohpism $f_{\mathbf{R}}$ associated with the orbit data $1,3,9$ with a cyclic permutation. As in the Section \ref{SS:gen}, we choose a set $G$ of generators for the fundamental group of $X(\mathbf{R})$ such that $G=\{ e, a_1, \dots, a_9, b_1, c_1, c_2, c_3 \}$ with one relation \[ e^2 \cc a_9^2 \cc a_8^2 \cc a_7^2 \cc a_6^2 \cc a_5^2 \cc c_3^2 \cc a_4^2 \cc c_2^2 \cc a_3^2 \cc c_1^2 \cc b_1^2 \cc a_2^2 \cc a_1^2 = 1.\] Let  $\epsilon = b_1 \cc c_1^2 \cc a_2^2 \cc a_1^2 \cc e$ be an element in $\pi_1(X(\mathbf{R}),P_{fix})$.   Using the closed form of the induced action on the fundamental group given in Appendix \ref{A:case4} and data given in Appendix \ref{B:139}, 
we see that for each $\gamma_i \in \Gamma$ defined in Appendix \ref{B:139}, \[ f^{-1}_{\mathbf{R}*} \gamma_ i = \epsilon^{-1} \cc \omega_{i_1} \cdots \omega_{i_{n_i}} \cc \epsilon \] where $\omega_{i_j} \in \Delta$ for all $j=1, \dots, n_i$,and  from the splitting map defined in Appendix \ref{B:139}  we see \[S \gamma_i = \omega_{i_1} + \cdots + \omega_{i_{n_i}}.\] 
Also using the unique relation on generators, we have for each $\gamma_i \in \Gamma$, \[ \gamma_ i = \delta_{i_1} \cc \cdots \delta_{i_{m_i}} \ \  \text{as elements in } \ \   \pi_1 (X(\mathbf{R}), P_{fix})\quad \text{ where  } M \gamma_i = \delta_{i_1} + \cdots +\delta_{i_{m_i}}.\] 
For each $\gamma =\gamma_{i_1} \cc \gamma_{i_2} \cdots \gamma_{i_n} \in \Gamma_A$, we have \[ [f^{-1}_{\mathbf{R}*} \gamma] = [f^{-1}_{\mathbf{R}*} \gamma_{i_1} \cc f^{-1}_{\mathbf{R}*} \gamma_{i_2} \cdots f^{-1}_{\mathbf{R}*} \gamma_{i_n}] .\]
Thus we may ignore the extra factor $\epsilon$ and $\epsilon^{-1}$. To see $\Gamma_A$ is $f^{-1}_{\mathbf{R}*}$-invariant, we need to check each quadruple in $A$. Here we show two examples:
\begin{itemize}
\item $(1,15,17,22) \in A$ : $f^{-1}_{\mathbf{R}*} \gamma_1 \cc \gamma_{15} \cc \gamma_{17} \cc \gamma_{22} = \eta_1 \cc \mu_1 \cc \gamma_{12} \cc \gamma_3 \cc \zeta_5 \cc \mu_5 \cc \zeta_2 \cc \mu_1 \cc \gamma_{12} \cc \zeta_4$. Since $\gamma_{23} = \zeta_5 \cc \mu_5$ and $\gamma_6 = \zeta_2 \cc \mu_1$, we have \[ f^{-1}_{\mathbf{R}*} \gamma_1 \cc \gamma_{15} \cc \gamma_{17} \cc \gamma_{22}  = \eta_1 \cc \mu_1 \cc \gamma_{12} \cc \gamma_3 \cc \gamma_{23} \cc \gamma_6 \cc \gamma_{12} \cc \zeta_4\] and two quadruples $(12,3,23,6), (3,23,6,12) \in A$.
\item $(12,3,2,17) \in A$: $f^{-1}_{\mathbf{R}*} \gamma_{12} \cc \gamma_{3} \cc \gamma_{2} \cc \gamma_{17} =\mu_4 \cc \zeta_1\cc   \eta_3 \cc \mu_1 \cc \zeta_6 \cc \mu_5 \cc \zeta_2$. Since $\gamma_{10} = \zeta_1 \cc \eta_3 \cc \mu_1$ and $\gamma_{25} = \zeta_6 \cc \mu_5$, it follows that \[ f^{-1}_{\mathbf{R}*} \gamma_{12} \cc \gamma_{3} \cc \gamma_{2} \cc \gamma_{17} = \mu_4 \gamma_{10} \cc \gamma_{25} \cc \zeta_2.  \] To see this quadruple $(12,3,2,17)$ generates another quadruple in $A$, we need to join other quadruples in $A$. There are three cases:
\begin{itemize}
\item $(9,12,3,2)$ and $(12,3,2,17)$ :  $f^{-1}_{\mathbf{R}*} \gamma_9 \cc \gamma_{12} \cc \gamma_{3} \cc \gamma_{2} \cc \gamma_{17} =\mu_1 \cc \gamma_{17} \cc \zeta_5 \cc \mu_4 \gamma_{10} \cc \gamma_{25} \cc \zeta_2$. Since $\gamma_{22} = \zeta_5 \cc \mu_4$, we have \[ f^{-1}_{\mathbf{R}*} \gamma_9 \cc \gamma_{12} \cc \gamma_{3} \cc \gamma_{2} \cc \gamma_{17} = \mu_1 \cc \gamma_{17} \cc \gamma_{22} \cc \gamma_{10} \cc \gamma_{25} \cc \zeta_2  \quad \text{and }(17,22,10,25) \in A \] 
\item $(10,12,3,2)$ and $(12,3,2,17)$ :  $f^{-1}_{\mathbf{R}*} \gamma_{10} \cc \gamma_{12} \cc \gamma_{3} \cc \gamma_{2} \cc \gamma_{17} =\mu_1 \cc \gamma_{20} \cc \zeta_5 \cc \mu_4 \gamma_{10} \cc \gamma_{25} \cc \zeta_2$. Again since $\gamma_{22} = \zeta_5 \cc \mu_4$, we have \[ f^{-1}_{\mathbf{R}*} \gamma_9 \cc \gamma_{12} \cc \gamma_{3} \cc \gamma_{2} \cc \gamma_{17} = \mu_1 \cc \gamma_{20} \cc \gamma_{22} \cc \gamma_{10} \cc \gamma_{25} \cc \zeta_2 \quad \text{and }(20,22,10,25) \in A  \] 
\item $(12,3,2,17)$ and $(3,2,17,22)$ :  $f^{-1}_{\mathbf{R}*}  \gamma_{12} \cc \gamma_{3} \cc \gamma_{2} \cc \gamma_{17}\cc \gamma_{22}  =  \mu_4 \gamma_{10} \cc \gamma_{25} \cc \zeta_2 \cc \mu_1 \cc \gamma_{12} \cc \zeta_4$. Since $\gamma_{6} = \zeta_2 \cc \mu_1$, we have \[ f^{-1}_{\mathbf{R}*}  \gamma_{12} \cc \gamma_{3} \cc \gamma_{2} \cc \gamma_{17}\cc \gamma_{22} = \mu_4 \cc \gamma_{10} \cc \gamma_{25} \cc \gamma_6 \cc \gamma_{12} \cc \zeta_4 \quad \text{and }(10,25,6,12) \in A \] 
\end{itemize}
\end{itemize}
By checking each individual element (joining couple of elements if necessary) in $A$, we see $\Gamma_A$ is $f^{-1}_{\mathbf{R}*}$-invariant.
\begin{lem} The equivalent classes of $A$-admissible cyclic words $\Gamma_A$ is invariant under $f^{-1}_{\mathbf{R}*}$. \end{lem} 

\begin{prop}\label{P:139} For a real diffeomorphism, $f_\mathbf{R}$ associated with the orbit data $1,3,9$ with a cyclic permutation $\sigma : 1\to 2 \to 3 \to 1$, the growth rate of the induced action on $\pi_1(X(\mathbf{R}))$ is the largest real zero of $t^6-t^4-t^3-t^2+1$ and 
\[ 0=\log \rho( f_{\mathbf{R}*}|_{H_1(X(\mathbf{R});\mathbf{R})} )=\log \rho( f^{-1}_{\mathbf{R}*}|_{H_1(X(\mathbf{R});\mathbf{R})} )  \lneq \log \rho (f^{-1}_{\mathbf{R*}}|_{\pi_1(X(\mathbf{R}))}) = h_{top} (f_\mathbf{R}) = h_{top} (f), \]
that is, $f_\mathbf{R}$ has the maximal possible entropy. 
\end{prop} 

\begin{proof} In \cite{Diller-Kim}, it is shown that he induced action on $H_1(X(\mathbf{R});\mathbf{R})$ grows linearly and thus $0=\log \rho( f_{\mathbf{R}*}|_{H_1(X(\mathbf{R});\mathbf{R})} )=\log \rho( f^{-1}_{\mathbf{R}*}|_{H_1(X(\mathbf{R});\mathbf{R})} ) $. On the other hand the topological entropy of $f:X \to X$ is given by the logarithm of the largest real root of the Salem polynomial $\chi(t) = t^6-t^4-t^3-t^2+1$. (See \cite{Bedford-Kim:2004, Diller:2011}.) Applying the Theorem \ref{T:Grate} together with the vector $v_\lambda$ given in the Appendix \ref{B:139}, we see that the growth rate of the induced action on the fundamental group $ \rho (f^{-1}_{\mathbf{R*}}|_{\pi_1(X(\mathbf{R}))})  \ge \lambda$ where $\lambda$ is given by the largest real root of $\chi(t)$. It follows that a real diffeomosphism $f_\mathbf{R}$ associated with the orbit data $1,3,9$ with a cyclic permutation has the maximal possible entropy. 
\end{proof}

\subsection{$1,4,8$ with a cyclic permutation: zero exponential homology growth rate  but maximal topological entropy}\label{SS:148}
Let $f_\mathbf{R} : X(\mathbf{R}) \to X(\mathbf{R})$ be a real diffeomorphism associated with the orbit data $1,4,8$ and a cyclic permutation. For this case, the induced action on the fundamental group is given in Appendix \ref{A:case4} and all other data ($\Gamma,\Delta$, Splitting and Merging Maps, admissible $n$-blocks, $v_\lambda$, and $\lambda$) are listed in Appendix \ref{B:148}.
The fundamental group of $X(\mathbf(R))$ with respect to a set of generators $G = \{ e, a_1, \dots, a_8, b_1, c_1, \dots, c_4 \} $ has one relation \[ e^2 \cc a_8^2 \cc a_7^2 \cc a_6^2 \cc c_4^2 \cc a_5^2 \cc c_3^2 \cc a_4^2 \cc c_2^2 \cc a_3^2 
\cc c_1^2 \cc b_1^2 \cc a_2^2 \cc a_1^2 \cc = \cc 1.\]
Let $\epsilon = b_1 \cc c_1^2 \cc a_2^2 \cc a_1^2 \cc e \in \pi_1(X(\mathbf{R}),P_{fix})$, then we have 
\[ f^{-1}_{\mathbf{R}*} \gamma_i = \epsilon^{-1} \omega_{i_1} \cdots \omega_{i_{n_i}} \epsilon, \quad \text{where } \omega_{i_j} \in \Delta, S\gamma_i = \omega_{i_1} +\cdots +\omega_{i_{n_i}} \]
and 
\[ \gamma_ i = \delta_{i_1} \cc \cdots \delta_{i_{m_i}} \ \  \text{as elements in } \ \   \pi_1 (X(\mathbf{R}), P_{fix})\quad \text{ where  } M \gamma_i = \delta_{i_1} + \cdots +\delta_{i_{m_i}}.\] 
Repeating the argument for the previous case, we have 
\begin{lem}
The equivalent classes of $A$-admissible cyclic words $\Gamma_A$ is invariant under $f^{-1}_{\mathbf{R}*}$.
\end{lem}

\begin{proof}
Since $\epsilon$ and $\epsilon^{-1}$ will cancelled out, we ignore $\epsilon$ factor in $f^{-1}_{\mathbf{R}*}$. To get the desired result, we have to check each triple in $A$. Here we give one exceptional example. A triple $(8,1,13) \in A$ and $f^{-1}_{\mathbf{R}*} \gamma_8 \cc \gamma_1 \cc \gamma_{13} = \eta_1 \cc \mu_1 \cc \zeta_8 \cc \mu_4 \cc \zeta_1$. Since $\gamma_{21} = \zeta_8 \cc \mu_4$, we have  $f^{-1}_{\mathbf{R}*} \gamma_8 \cc \gamma_1 \cc \gamma_{13} = \eta_1 \cc \mu_1 \cc \gamma_{21} \cc \zeta_1$. It is not clear whether this triple will be mapped to a combination of triples in $A$. Thus we need to join other triples in $A$. 
\begin{itemize}
\item $(23,10,8),(10,8,1)$ and $(8,1,13)$ : $f^{-1}_{\mathbf{R}*} \gamma_{23} \cc \gamma_{10} \cc \gamma_8 \cc \gamma_1 \cc \gamma_{13}  = \mu_9 \cc \gamma_9 \cc \zeta_1 \cc \mu_6 \cc \zeta_1 \cc \eta_1 \cc \mu_1 \cc \gamma_{21} \cc \zeta_1$. Since $\gamma_{18} = \zeta_1\cc \mu_6$, and $\gamma_{20} = \zeta_1\cc \eta_1 \cc \mu_1$, we have \[ f^{-1}_{\mathbf{R}*} \gamma_{23} \cc \gamma_{10} \cc \gamma_8 \cc \gamma_1 \cc \gamma_{13}  = \mu_9 \cc \gamma_9 \cc \gamma_{18} \cc \gamma_{20} \cc \gamma_{21} \zeta_1, \qquad (9,18,20),(18,20,21) \in A\]
\item $(2,12,8),(12,8,1)$ and $(8,1,13)$ : $f^{-1}_{\mathbf{R}*} \gamma_{2} \cc \gamma_{12} \cc \gamma_8 \cc \gamma_1 \cc \gamma_{13}  = \mu_1 \cc \zeta_9 \cc \mu_7 \cc \zeta_1 \cc \eta_1 \cc \mu_1 \cc \gamma_{21} \cc \zeta_1$. Since $\gamma_{25} = \zeta_9\cc \mu_7$, and $\gamma_{20} = \zeta_1\cc \eta_1 \cc \mu_1$, we have \[ f^{-1}_{\mathbf{R}*} \gamma_{2} \cc \gamma_{13} \cc \gamma_8 \cc \gamma_1 \cc \gamma_{13}  = \mu_1 \cc \gamma_{25} \cc \gamma_{20} \cc \gamma_{21} \zeta_1, \qquad (25,20,21) \in A\]
\item $(23,8,1)$ and $(8,1,13)$ : $f^{-1}_{\mathbf{R}*} \gamma_{23}  \cc \gamma_8 \cc \gamma_1 \cc \gamma_{13}  = \mu_9 \cc \gamma_9 \cc \zeta_1 \cc \eta_1 \cc \mu_1 \cc \gamma_{21} \cc \zeta_1$. Since  $\gamma_{20} = \zeta_1\cc \eta_1 \cc \mu_1$, we have \[ f^{-1}_{\mathbf{R}*} \gamma_{23} \cc \gamma_8 \cc \gamma_1 \cc \gamma_{13}  = \mu_9 \cc \gamma_9  \cc \gamma_{20} \cc \gamma_{21} \zeta_1, \qquad (9,20,21) \in A\]
\end{itemize}
Also, we see 
\begin{itemize}
\item $(8,1,13),(1,13,3),(13,3,12) \in A $ gives  $ f^{-1}_{\mathbf{R}*}\gamma_8 \cc \gamma_1 \cc \gamma_{13}  \cc \gamma_3 \cc \gamma_{12}  = \eta_1 \cc \mu_1 \cc \gamma_{21} \cc \gamma_1 \cc \gamma_{26} \cc \zeta_1$
\item $(8,1,13),(1,13,6),(13,6,12) \in A$ gives  $ f^{-1}_{\mathbf{R}*}\gamma_8 \cc \gamma_1 \cc \gamma_{13}  \cc \gamma_6 \cc \gamma_{12}  = \eta_1 \cc \mu_1 \cc \gamma_{21} \cc \gamma_1 \cc \gamma_{27} \cc \zeta_1$
\item $(8,1,13),(1,13,6),(13,6,21) \in A$ gives  $ f^{-1}_{\mathbf{R}*}\gamma_8 \cc \gamma_1 \cc \gamma_{13}  \cc \gamma_6 \cc \gamma_{21}  = \eta_1 \cc \mu_1 \cc \gamma_{21} \cc \gamma_1 \cc \gamma_{28} \cc \zeta_4$
\item $(8,1,13),(1,13,6),(13,6,24) \in A$ gives  $ f^{-1}_{\mathbf{R}*}\gamma_8 \cc \gamma_1 \cc \gamma_{13}  \cc \gamma_6 \cc \gamma_{24}  = \eta_1 \cc \mu_1 \cc \gamma_{21} \cc \gamma_1 \cc \gamma_{28}\cc \gamma_{15} \cc \zeta_1$
\item $(8,1,13),(1,13,7) \in A$ gives  $ f^{-1}_{\mathbf{R}*}\gamma_8 \cc \gamma_1 \cc \gamma_{13}  \cc \gamma_7 = \eta_1 \cc \mu_1 \cc \gamma_{21} \cc \gamma_1 \cc \gamma_{17} \cc \zeta_6$
\item $(8,1,13),(1,13,11) \in A$ gives  $ f^{-1}_{\mathbf{R}*}\gamma_8 \cc \gamma_1 \cc \gamma_{13}  \cc \gamma_{11} = \eta_1 \cc \mu_1 \cc \gamma_{21} \cc \gamma_1 \cc \gamma_{13} \cc \zeta_2$
\end{itemize}
Similarly, by check each triple in $A$, we have this Lemma. 
\end{proof}

With the homology growth computed in \cite{Diller-Kim}, topological entropy of $f:X(\mathbf{C}) \to X(\mathbf{C})$ computed in several articles including \cite{Bedford-Kim:2004, Diller:2011}, and applying the Theorem \ref{T:Grate} with $v_\lambda \in Sp^+ \Gamma $ given in the Appendix \ref{B:148}, we have the following result: 
\begin{prop}\label{P:148} For a real diffeomorphism, $f_\mathbf{R}$ associated with the orbit data $1,4,8$ with a cyclic permutation $\sigma : 1\to 2 \to 3 \to 1$, the growth rate of the induced action on $\pi_1(X(\mathbf{R}))$ is  $\lambda \approx 1.45799 >1$ where $\lambda$ is the largest real zero of $ t^8-t^6-t^5-t^3-t^2+1$ and 
\[ 0=\log \rho( f_{\mathbf{R}*}|_{H_1(X(\mathbf{R});\mathbf{R})} )=\log \rho( f^{-1}_{\mathbf{R}*}|_{H_1(X(\mathbf{R});\mathbf{R})} )  \lneq \log \rho (f^{-1}_{\mathbf{R*}}|_{\pi_1(X(\mathbf{R}))}) = h_{top} (f_\mathbf{R}) = h_{top} (f), \]
that is, $f_\mathbf{R}$ has the maximal possible entropy. 
\end{prop} 

\begin{proof} In \cite{Diller-Kim}, it is shown that he induced action on $H_1(X(\mathbf{R});\mathbf{R})$ is periodic with period $180$ and thus $0=\log \rho( f_{\mathbf{R}*}|_{H_1(X(\mathbf{R});\mathbf{R})} )=\log \rho( f^{-1}_{\mathbf{R}*}|_{H_1(X(\mathbf{R});\mathbf{R})} ) $. On the other hand the topological entropy of $f:X \to X$ is given by the logarithm of the largest real root of the Salem polynomial $\chi(t) = t^8-t^6-t^5-t^3-t^2+1$. (See \cite{Bedford-Kim:2004, Diller:2011}.) Applying the Theorem \ref{T:Grate} together with the vector $v_\lambda$ given in the Appendix \ref{B:148}, we see that the growth rate of the induced action on the fundamental group $ \rho (f^{-1}_{\mathbf{R*}}|_{\pi_1(X(\mathbf{R}))})  \ge \lambda$ where $\lambda$ is given by the largest real root of $\chi(t)$. It follows that a real diffeomosphism $f_\mathbf{R}$ associated with the orbit data $1,4,8$ with a cyclic permutation has the maximal possible entropy. 
\end{proof}

\begin{proof}[Proof of Theorem A]
Proposition \ref{P:139} and Proposition \ref{P:148} give existence of orbit data such that the associated rational surface automorphisms $f:X(\mathbf{C}) \to X(\mathbf{C})$ and the restriction $f_\mathbf{R}$ on the real slice $X(\mathbf{R})$ has the following properties  :  
\[ 0=\log \rho( f_{\mathbf{R}*}|_{H_1(X(\mathbf{R});\mathbf{R})} )=\log \rho( f^{-1}_{\mathbf{R}*}|_{H_1(X(\mathbf{R});\mathbf{R})} )  \lneq \log \rho (f^{-1}_{\mathbf{R*}}|_{\pi_1(X(\mathbf{R}))}) = h_{top} (f_\mathbf{R}) = h_{top} (f). \]
Such orbit data include
\begin{itemize}
\item $n_1 =1, n_2=3, n_3=9$ with a cyclic permutation. 
\item  $n_1 =1, n_2=4, n_3=8$ with a cyclic permutation. 
\end{itemize}
\end{proof}

\subsection{Non-maximal homology growth rate but maximal $\pi_1$ growth rate and the maximal topological entropy: $1,4,5$ with a cyclic permutation, $1,5,6$ with a cyclic permutation, and $1,6,7$ with a cyclic permutation}\label{SS:145} Let us consider three orbit data : (1) $1,4,5$ with a cyclic permutation, (2) $1,5,6$ with a cyclic permutation, and (3) $1,6,7$ with a cyclic permutation. For each orbit data, the necessary information to apply Theorem \ref{T:Grate} is given in Appendix \ref{B:145}, \ref{B:156} and \ref{B:167} respectively. As in the previous two Subsections \ref{SS:139} and \ref{SS:148}, we see that 
 with $\epsilon = b_1 \cc c_1^2 \cc a_2^2 \cc a_1^2 \cc e \in \pi_1(X(\mathbf{R}),P_{fix})$ 
\[ f^{-1}_{\mathbf{R}*} \gamma_i = \epsilon^{-1} \omega_{i_1} \cdots \omega_{i_{n_i}} \epsilon, \quad \text{where } \omega_{i_j} \in \Delta, S\gamma_i = \omega_{i_1} +\cdots +\omega_{i_{n_i}} \]
and 
\[ \gamma_ i = \delta_{i_1} \cc \cdots \delta_{i_{m_i}} \ \  \text{as elements in } \ \   \pi_1 (X(\mathbf{R}), P_{fix})\quad \text{ where  } M \gamma_i = \delta_{i_1} + \cdots +\delta_{i_{m_i}}.\] 
Also by checking each admissible tuples in $A$ given in Appendix \ref{B:145}, \ref{B:156} and \ref{B:167},, we show that $\Gamma_A$ is $f^{-1}_{\mathbf{R}*}$-invariant for each orbit data.
\begin{lem}
For a real diffeomorphism, $f_\mathbf{R}$ associated with one of the orbit data :  $1,4,5$ with a cyclic permutation, $1,5,6$ with a cyclic permutation, and $1,6,7$ with a cyclic permutation $\sigma : 1\to 2 \to 3 \to 1$, the set of cyclically $A$-admissible cyclic words $\Gamma_A$ is $f^{-1}_{\mathbf{R}*}$-invariant. 
\end{lem}

 Using this together with $v_\lambda$ and $\lambda >1$ given in Appendix \ref{B:145}, \ref{B:156} and \ref{B:167}, we have 

\begin{prop}\label{P:145} For a real diffeomorphism, $f_\mathbf{R}$ associated with one of the orbit data :  $1,4,5$ with a cyclic permutation, $1,5,6$ with a cyclic permutation, and $1,6,7$ with a cyclic permutation $\sigma : 1\to 2 \to 3 \to 1$, the topological entropy is maximal. In fact  
\[ 0 \lneq \log \rho( f_{\mathbf{R}*}|_{H_1(X(\mathbf{R});\mathbf{R})} )=\log \rho( f^{-1}_{\mathbf{R}*}|_{H_1(X(\mathbf{R});\mathbf{R})} )  \lneq \log \rho (f^{-1}_{\mathbf{R*}}|_{\pi_1(X(\mathbf{R}))}) = h_{top} (f_\mathbf{R}) = h_{top} (f), \]
\end{prop} 

\begin{proof}
For  a real diffeomorphism $f_\mathbf{R}$ associated with the orbit data $1,4,5$ and a cyclic permutation, the topological entropy of $f:X(\mathbf{C}) \to X(\mathbf{C})$ is the logarithm of the largest real root $\lambda \approx 1.29349$ of $\chi(t) =(t-1)( t^{10}-t^8-t^7+t^5-t^3-t^2+1)$ given in (\ref{E:compchar}). On the other hand, using the formula (\ref{E:realchar}) we see that the exponential growth rate of the induced action on the homology classes is given by the largest absolute value of the roots $\approx 1.12571$ of the roots of a symmetric polynomial $\phi(t) = t^{10} + t^8 +t^7+t^5+t^3 +t^2+1$. Applying Theorem \ref{T:Grate} together with Lemma \ref{L:Grate} using $\Gamma, \Delta, S, M$ and $v_\lambda$ in Appendix \ref{B:145}, we conclude that the growth rate of the induced action on the fundamental group  $\rho (f^{-1}_{\mathbf{R*}}|_{\pi_1(X(\mathbf{R}))})$ is given by the largest possible value $\lambda \approx 1.29349$, the largest real root of $\chi(t)= t^{10}-t^8-t^7+t^5-t^3-t^2+1$. It follows that this proposition holds for the orbit data $1,4,5$ with a cyclic permutation. 

With the orbit data $1,5,6$ with a cyclic permutation, the characteristic polynomial of $f_*|_{H^{1,1}(X(\mathbf{C});\mathbf{R})}$ is given by $\chi(t)= (t-1) (t^{12}-t^{1-} -t^9-t^8+t^6-t^4-t^3-t^2+1)$ and its largest real root $\lambda \approx 1.46048$. The characteristic polynomial of $f_{\mathbf{R}*}|_{H_1(X(\mathbf{R});\mathbf{R})}$ is $\phi(t) = t^{12}+t^{10}+t^9+t^8+t^6+t^4+t^3+t^2+1$  and thus the largest absolute value of the roots is approximately $1.22125$. Again using data in Appendix \ref{B:156}, we apply Theorem \ref{T:Grate} together with Lemma \ref{L:Grate}  to conclude that the growth rate of the induced action on the fundamental group is maximal. 

Similarly, for the orbit data $1,6,7$ with a cyclic permutation we have $\chi(t) = (t-1) (t^4+t^3+t^2+t+1) (t^{10}-t^9-t^8+t^5-t^2-t+1)$ with the largest real root $\lambda \approx 1.53293$ and $\phi(t) = t^{14}+t^{12} +t^{11}+t^{10}+t^9+t^7+t^5+t^4+t^3+t^2+1$ with the spectral radius of $f_{\mathbf{R}*}|_{H_1(X(\mathbf{R});\mathbf{R})} \approx 1.21342$. Again using data in Appendix \ref{B:156}, we apply Theorem \ref{T:Grate} together with Lemma \ref{L:Grate}  to conclude that the growth rate of the induced action on the fundamental group is maximal.  
\end{proof}

\begin{proof}[Proof of Theorem B]
Proposition \ref{P:145} gives existence of orbit data such that the associated rational surface automorphisms $f:X(\mathbf{C}) \to X(\mathbf{C})$ and the restriction $f_\mathbf{R}$ on the real slice $X(\mathbf{R})$ has the following properties  :  
\[ 0\lneq \log \rho( f_{\mathbf{R}*}|_{H_1(X(\mathbf{R});\mathbf{R})} )=\log \rho( f^{-1}_{\mathbf{R}*}|_{H_1(X(\mathbf{R});\mathbf{R})} )  \lneq \log \rho (f^{-1}_{\mathbf{R*}}|_{\pi_1(X(\mathbf{R}))}) = h_{top} (f_\mathbf{R}) = h_{top} (f). \]
Such orbit data include
\begin{itemize}
\item $n_1 =1, n_2=4, n_3=5$ with a cyclic permutation. 
\item  $n_1 =1, n_2=5, n_3=6$ with a cyclic permutation. 
\item  $n_1 =1, n_2=6, n_3=7$ with a cyclic permutation. 
\end{itemize}
\end{proof}



\section{Non-maximum Entropy}\label{S:nonmax}
Based on the work of Bedford, Lyubich and Smilie \cite{BLS:dist, BLS}, Cantat \cite{Cantat:1999}, DeThelin \cite{deThelin} and Dujardin \cite{Dujardin} gives a criterion for real automorphism with maximal entropy.
\begin{thm}[{\cite[Corollary~8.3]{Cantat:2014}}]\label{T:cantat}
Let $f:X(\mathbf{C}) \to X(\mathbf{C})$ be a rational surface automorphism. If the restriction on the real slice $f_\mathbf{R}: X(\mathbf{R}) \to X(\mathbf{R})$ is a diffeomorphism, then $h_{top}(f_\mathbf{R}) = h_{top}(f)$ if and only if all saddle periodic points of $f$ are contained in $X(\mathbf{R})$ or in $f$-invariant algebraic curves. 
\end{thm}

In this section, we show that if an automorphism $f:X(\mathbf{C}) \to X(\mathbf{C})$ is associated to a real birational map $\check f$ with one of the following orbit data $2,3,5, \sigma$, $3,4,5, \sigma$, $3,4,6, \sigma$, and $3,5,5, \sigma$ where $\sigma:1 \to 2\to 3 \to 1$ is a cyclic permutation, then $0\lneq h_{top}(f_\mathbf{R})  \lneq h_{top}(f)$, in fact 
\[ 0= \log \rho( f_{\mathbf{R}*}|_{H_1(X(\mathbf{R});\mathbf{R})} )=\log \rho( f^{-1}_{\mathbf{R}*}|_{H_1(X(\mathbf{R});\mathbf{R})} )  \lneq \log \rho (f^{-1}_{\mathbf{R*}}|_{\pi_1(X(\mathbf{R}))}) \le  h_{top} (f_\mathbf{R}) \lneq h_{top} (f). \]

\begin{lem}\label{L:uniqueC}
Let $\check f$ be a real birational map fixing a cusp cubic $C$ associated with one of the following orbit data $2,3,5, \sigma$, $3,4,5, \sigma$, $3,4,6, \sigma$, and $3,5,5, \sigma$ where $\sigma:1 \to 2\to 3 \to 1$ is a cyclic permutation.The only invariant algebraic curve for the assoicated automrophism $f: X(\mathbf{C}) \to X(\mathbf{C})$ is the invariant cubic $C$ and the only fixed points on $C$ are $P_{cusp}$ and $P_{fix}$. 
\end{lem}
\begin{proof}
Here we are repeating the argument of Lemma 4.3 in \cite{Diller-Kim}.
The cohomology class of an invariant curve must be an eigenvector of $f_*|_{H^{1,1}(X(\mathbf{C});\mathbf{R})}$ associated with eigenvalue $1$. This linear action in the general setup is computed in \cite{Bedford-Kim:2004, Diller:2011} and its characteristic polynomial is given in (\ref{E:compchar}). For all cases, eigenvalue $1$ has multiplicity $1$ and  the class of invariant cubic $C$, $[C]$ is a corresponding eigenvector. The the class of any algebraic invariant curve should be a multiple of $[C]$. Since all the blowup centers are on the invariant cubic $C$, we have $C \cdot C = 3^2 - (\text{\,the number of blowups\,})$. It follows that for each case, the self-intersection $C\cdot C < 0$ thus as a divisor any algebraic invariant curve must be a multiple of $C$. 
Now if a fixed point of $f: X(\mathbf{C}) \to X(\mathbf{C})$ is on the invariant cubic $C$, it must be either a cusp point or a fixed point of the restriction $f|_{C_\text{regular}}$ on a set of regular points on $C$. In \cite{Diller:2011}, it is shown that under an appropriate parametrization, $f|_{C_\text{regular}} : t \mapsto \delta t$ where $\delta$ is the determinant of $\check f$ and there is only one fixed point for $f|_{C_\text{regular}}$. It follows that there are exactly two fixed points on $C$ and we have two fixed points $P_{fix}$ and $P_{cusp}$ on $C$. Thus any other fixed point will be in the complement of the invariant cubic $C$. 
\end{proof}

\begin{lem}\label{L:cpxsaddle}
Let $\check f$ be a real birational map fixing a cusp cubic $C$ associated with one of the following orbit data $2,3,5, \sigma$, $3,4,5, \sigma$, $3,4,6, \sigma$, and $3,5,5, \sigma$ where $\sigma:1 \to 2\to 3 \to 1$ is a cyclic permutation. Then the assoicated automrophism $f: X(\mathbf{C}) \to X(\mathbf{C})$ has two fixed complex saddle fixed points outside the invariant cubic $C$. 
\end{lem}
\begin{proof}
By conjugating a linear map if necessary, we may assume that a real map $\check f = L \circ J$ whith an automorphism $L \in \text{Aut} (\mathbb{P}^2(\mathbf{C}))$. With the formulas given in \cite{Diller:2011}, we have an explicit formula for $\check f$ with given orbit data. Using those formulas, we can directly
compute the fixed points and their multipliers for each case and find two saddle complex fixed points with an identification $(x,y) \leftrightarrow [1:x:y] \in \mathbb{P}^2(\mathbf{C})$:
\begin{itemize}
\item $2,3,5, \sigma$ : $(6.26467\pm 0.544382 i, 4.89414 \pm 0.88372 i)$ with multipliers $1.03623, 0.714322$
\item $3,4,5, \sigma$:  $(1.90652\pm 0.19067 i, 1.35067 \pm 0.0205975 i)$ with multipliers $1.02849, 0.578415$
\item $3,4,6, \sigma$:  $(2.06731 \pm 0.221777 i, 1.51188 \pm 0.0328373 i)$ with multipliers $1.0315,0.562958$
\item $3,5,5, \sigma$: $(1.84572 \pm 0.247174 i, 1.78388\pm 0.236327 i)$ with multipliers $1.79925,0.970414$
\end{itemize}
Due to the previous Lemma, these saddle points are not on the invariant algebraic curve. 
\end{proof}

\begin{prop}\label{P:nonmax}
Let $\check f$ be a real birational map fixing a cusp cubic $C$ associated with one of the following orbit data $2,3,5, \sigma$, $3,4,5, \sigma$, $3,4,6, \sigma$, and $3,5,5, \sigma$ where $\sigma:1 \to 2\to 3 \to 1$ is a cyclic permutation.Then the assoicated real diffeomorphism $f_\mathbf{R}: X(\mathbf{R}) \to X(\mathbf{R})$ has positive non-maximal entropy: 
\[ 0=  \log \rho( f_{\mathbf{R}*}|_{H_1(X(\mathbf{R});\mathbf{R})} )=\log \rho( f^{-1}_{\mathbf{R}*}|_{H_1(X(\mathbf{R});\mathbf{R})} )  \lneq \log \rho (f^{-1}_{\mathbf{R*}}|_{\pi_1(X(\mathbf{R}))}) \le  h_{top} (f_\mathbf{R}) \lneq h_{top} (f). \]
\end{prop}
\begin{proof}
Combining Lemma \ref{L:uniqueC}, Lemma \ref{L:cpxsaddle} and Theorem \ref{T:cantat}, we see that $h_{top} (f_\mathbf{R}) \lneq h_{top}(f)$. Using (\ref{E:realchar}), we have the characteristic polynomial $\phi(t)$ of the induced action on $H_1(X(\mathbf{R});\mathbf{R})$
\begin{itemize}
\item $2,3,5, \sigma$ : $\phi(t) = (t^4-t^2+1) (t^6-t^5+t^4-t^3+t^2-t+1) =(t^6+1)(t^7+1)/((t+1)(t^2+1)) $
\item $3,4,5, \sigma$:  $\phi(t) = (t^6+t^3+1) (t^6-t^5+t^4-t^3+t^2-t+1) =(t^9-1)(t^7+1)/((t+1)(t^3-1)) $
\item $3,4,6, \sigma$:  $\phi(t) =-(t^3+1) (t^5-1) (t^6+1)/(t+1) $
\item $3,5,5, \sigma$:  $\phi(t) =-(t^3+1) (t^4+1) (t^7-1)/(t+1) $
\end{itemize}
Thus for each case, the growth rate of homology classes is equal to zero. On the other hand, from the action on the fundamental group given in Appendix \ref{A:case5} for the first three cases and Appendix \ref{A:case6} for the last case  together with $A$ admissible cyclic words in Appendix \ref{B:235} -- \ref{B:355}, we compute the lower estimate of the growth rate on the induced action on the fundamental group. 
Let $\epsilon_1 = b_1 \cc c_1^2 \cc a_2^2 \cc a_1^2 \cc e$, $\epsilon_2 = c_1 \cc b_1^2 \cc a_1^2 \cc e$ and $\epsilon_3 = b_1 \cc c_1^2 \cc a_1^2 \cc e$. We have 
\begin{itemize}
\item $2,3,5, \sigma$ : $ f^{-1}_{\mathbf{R}*} \gamma_i = \epsilon_1^{-1} \omega_{i_1} \cdots \omega_{i_{n_i}} \epsilon_1$ if $S\gamma_i = \omega_{i_1} +\cdots +\omega_{i_{n_i}}$ and 
\[\rho (f^{-1}_{\mathbf{R*}}|_{\pi_1(X(\mathbf{R}))}) \ge 1.27034 >0\]
\item $3,4,5, \sigma$:  $ f^{-1}_{\mathbf{R}*} \gamma_i = \epsilon_2^{-1} \omega_{i_1} \cdots \omega_{i_{n_i}} \epsilon_2$ if $S\gamma_i = \omega_{i_1} +\cdots +\omega_{i_{n_i}}$ and \[\rho (f^{-1}_{\mathbf{R*}}|_{\pi_1(X(\mathbf{R}))}) \ge 1.54325 >0\]
\item $3,4,6, \sigma$:  $ f^{-1}_{\mathbf{R}*} \gamma_i = \epsilon_2^{-1} \omega_{i_1} \cdots \omega_{i_{n_i}} \epsilon_2$ if $S\gamma_i = \omega_{i_1} +\cdots +\omega_{i_{n_i}}$ and \[\rho (f^{-1}_{\mathbf{R*}}|_{\pi_1(X(\mathbf{R}))}) \ge 1.55943 >0\]
\item $3,5,5, \sigma$: $ f^{-1}_{\mathbf{R}*} \gamma_i = \epsilon_3^{-1} \omega_{i_1} \cdots \omega_{i_{n_i}} \epsilon_3$ if $S\gamma_i = \omega_{i_1} +\cdots +\omega_{i_{n_i}}$ and \[\rho (f^{-1}_{\mathbf{R*}}|_{\pi_1(X(\mathbf{R}))}) \ge 1.58235 >0\]
\end{itemize}
\end{proof}

Theorem C is immediate consequence of Theorem \ref{T:cantat} and Proposition \ref{P:nonmax}.


\appendix
\section{$\pi_1$ actions for quadratic real automorphisms }\label{ApendA}


%
%
%
%

The fundamental group $\pi_1 (X(\mathbf{R}))$ with respect to a set of generators $\{e, a_i, b_j, c_k, 1 \le i \le n_3, 1 \le j \le n_1, 1 \le k \le n_2\}$ is a one relator group  \[ \pi_1 (X(\mathbf{R})) = \langle e, a_i, b_j, c_k \rangle /R \]
where the relator $R$ is a cyclic word given by an ordered product of squares of generators: \[ R\ = e^2 \cdot \prod_{i= 1}^ N x_i^2, \ \ \ x_i \in \{ a_i, b_j, c_k\}, N= n_1+n_2+n_3.\]
For an ordered product in a part of the relator $R$, we use  $Q(x_i, x_j) = \prod_{k=i}^j x_k^2$  if $i\le j$ and $Q(x_i,x_j) =1$ if $i>j$.

\subsection{Case 1:  $n_1=1, n_2 =1, n_3 \ge 8 $} \label{A:case1}

The relator of this case is given by \[ R= e^2 \cc a_{n_3}^2 \cc a_{n_3-1}^2 \cc \cdots a_6^2 \cc a_5^2 \cc c_1^2 \cc a_4^2 \cc a_3^2 \cc b_1^2 \cc a_2^2 \cc a_1^2 \]
\vspace{1ex}
The induced action $f^{-1}_{\mathbf{R} *}$ is given by 

\begin{equation*} 
\begin{aligned}
& e\ \ &\mapsto \ \ &\left\{ \begin{aligned} & e \cc a_8^2 \cc a_7^2 \cc a_6^2 \cc a_5^2 \cc c_1 \cc b_1 \cc a_2^2 \cc a_1 \cc e \ \ & \text{   if   }n_3 =8\\
& e^{-1} \cc a_1^{-2} \cc a_2^{-2} \cc b_1^{-2} \cc a_3^{-2}\cc a_4^{-2} \cc c_1^{-1} \cc b_1 \cc a_2^2 \cc a_1 \cc e \ \ & \text{   if   }n_3 \ge 9\\
\end{aligned} \right.\\
& c_1\ \ &\mapsto \ \ &\left\{ \begin{aligned} & b_1 \cc a_2^2 \cc a_1 \cc e \cc Q(a_{n_3}, a_6) \cc c_1 \cc a_4^2 \cc a_3^2 \cc b_1 \ \ & \text{   if   }n_3 \le 12 \\
& b_1 \cc a_2^2 \cc a_1 \cc e^{-1} \cc Q(a_5, a_1)^{-1} \cc c_1 \cc a_4^2 \cc a_3^2 \cc b_1 \ \ \ \ \ \ \ \ \ & \text{   if   }n_3 > 12 \\
\end{aligned} \right.\\
& b_1\ \ &\mapsto \ \ &\left\{ \begin{aligned} & b_1^{-1} \cc a_3^{-2} \cc c_1^{-1} \cc Q(a_{n_3}, a_5)^{-1}\cc e^{-1} \ \ & \text{   if   }n_3 =8 \\
& b_1^{-1} \cc a_3^{-2} \cc c_1 \cc a_4^2 \cc a_3^2 \cc b_1^2 \cc a_2^2\cc a_1^2 \cc e \ \ \ \ \ \ \ \  \ \  \ \ \ \ \ \ \ \ \ \ & \text{   if   }n_3 \ge 9 \\
\end{aligned} \right.\\
& a_1\ \ &\mapsto \ \ &\left\{ \begin{aligned} & e \cc Q(a_{n_3}, a_5) \cc c_1 \cc b_1 \cc a_2^{-1} \cc b_1^{-1} \cc c_1 ^{-1} \cc Q(a_{n_3}, a_5)^{-1} \cc e^{-1} \ \ & \text{   if   }n_3 =8 \\
& e^{-1} \cc Q(a_4, a_1)^{-1}  \cc c_1^{-1} \cc b_1 \cc a_2^{-1} \cc b_1^{-1} \cc c_1 \cc Q(a_4, a_1)  \cc e  & \text{   if   }n_3 \ge 9 \\
\end{aligned} \right.\\
& a_2\ \ &\mapsto \ \ &\left\{ \begin{aligned} & e \cc Q(a_{n_3}, a_5) \cc c_1 \cc a_3 \cc c_1 ^{-1} \cc Q(a_{n_3}, a_5)^{-1} \cc e^{-1} \ \ & \text{   if   }n_3 =8 \\
& e^{-1} \cc Q(a_4, a_1)^{-1}  \cc c_1^{-1} \cc a_3  \cc c_1 \cc Q(a_4, a_1)  \cc e  & \text{   if   }n_3 \ge 9 \\
\end{aligned} \right.\\
&a_{3}\ \ &\mapsto \ \ &\ \  b_1^{-1} \cc a_3^{-2} \cc a_4^{-1} \cc a_3^2 \cc b_1\\
&a_{4}\ \ &\mapsto \ \ &\ \  b_1^{-1} \cc a_3^{-2} \cc a_4^{-2} \cc c_1^{-1} \cc a_5 \cc c_1 \cc a_4^2  \cc a_3^2 \cc b_1\\
& a_i \ \ &\mapsto \ \ &\left\{ \begin{aligned} & b_1 \cc a_2^2 \cc a_1  \cc e \cc Q(a_{n_3}, a_{i+2})  \cc a_{i+1}^{-1} \cc Q( a_{n_3}, a_{i+2})^{-1} \cc e^{-1}\cc a_1^{-1} \cc a_2^{-2} \cc b_1^{-1} \\  & \hspace{6cm} \text{   if   }\ \ \ n_3 -3 \le 2 i \le 2 (n_3 -1) \\   
& b_1 \cc a_2^2 \cc a_1  \cc e^{-1} \cc Q(a_i, a_1)^{-1}   \cc a_{i+1}^{-1} \cc Q( a_i , a_1)^{-1} \cc e\cc a_1^{-1} \cc a_2^{-2} \cc b_1^{-1}  \\  & \hspace{6cm} \text{   if   }\ \ \ 10 \le 2 i < n_3 -3 \\
\end{aligned} \right.\\
&a_{n_3}\ \ &\mapsto \ \ &\ \  e^{-1} \cc a_1^{-1} \cc a_2^{-2} \cc b_1^{-1}\\
\end{aligned}
\end{equation*}

\subsection{Case 2:  $n_1=1, n_2 =2, n_3 \ge 7 $}\label{A:case2}

The relator of this case is given by \[ R= e^2 \cc a_{n_3}^2 \cc a_{n_3-1}^2 \cc \cdots a_5^2 \cc c_2^2 \cc a_4^2 \cc c_1^2 \cc a_3^2 \cc b_1^2 \cc a_2^2 \cc a_1^2 \]

\vspace{1ex}
The induced action $f^{-1}_{\mathbf{R} *}$ is given by

\begin{equation*} 
\begin{aligned}
&e\ \ &\mapsto \ \ &\ \  e^{-1} \cc a_1^{-2} \cc a_2^{-2} \cc b_1^{-2} \cc a_3^{-2} \cc c_1^{-1} \cc b_1 \cc a_2^2 \cc a_1 \cc e\\
&c_1\ \ &\mapsto \ \ &\ \  b^{-1} \cc a_3^{-2}  \cc c_1^{-1} \cc c_2 \cc c_1 \cc a_3^2 \cc b_1 \\
& c_2 \ \ &\mapsto \ \ &\left\{ \begin{aligned} & b_1 \cc a_2^2 \cc a_1 \cc e \cc Q( a_{n_3}, a_6) \cc c_1 \cc a_3^2 \cc b_1  \ \ & \text{   if   }n_3 \le 13\\
& b_1 \cc a_2^2 \cc a_1 \cc e^{-1} \cc Q(a_5, a_1)^{-1} \cc c_1 \cc a_3^2 \cc b_1 \ \ & \text{   if   }n_3 \ge 14\\
\end{aligned} \right.\\
&b_1\ \ &\mapsto \ \ &\ \   b^{-1} \cc a_3^{-2}  \cc c_1 \cc a_3^2 \cc b_1^2 \cc a_2^2 \cc a_1^2 \cc e \\
&a_1\ \ &\mapsto \ \ &\ \   e^{-1} \cc a_1^{-2} \cc a_2^{-2} \cc b_1^{-2} \cc a_3^{-2} \cc c_1^{-1} \cc b_1 \cc a_2^{-1} \cc b_1^{-1} \cc c_1 \cc a_3^2 \cc b_1^2 \cc a_2^2 \cc a_1^2 \cc e \\ 
&a_2\ \ &\mapsto \ \ &\ \   e^{-1} \cc a_1^{-2} \cc a_2^{-2} \cc b_1^{-2} \cc a_3^{-2} \cc c_1^{-1} \cc a_3 \cc c_1 \cc a_3^2 \cc b_1^2 \cc a_2^2 \cc a_1^2 \cc e \\ 
&a_3\ \ &\mapsto \ \ &\ \  b^{-1} \cc a_3^{-2}  \cc c_1^{-1} \cc a_4 \cc c_1 \cc a_3^2 \cc b_1 \\
&a_4\ \ &\mapsto \ \ &\ \  b^{-1} \cc a_3^{-2}  \cc c_1^{-1} \cc a_5 \cc c_1 \cc a_3^2 \cc b_1 \\
\end{aligned}
\end{equation*}
\begin{equation*} 
\begin{aligned}
& a_i \ \ &\mapsto \ \ &\left\{ \begin{aligned} & b_1 \cc a_2^2 \cc a_1  \cc e \cc Q(a_{n_3}, a_{i+2})  \cc a_{i+1}^{-1} \cc Q( a_{n_3}, a_{i+2})^{-1} \cc e^{-1}\cc a_1^{-1} \cc a_2^{-2} \cc b_1^{-1} \\  & \hspace{6cm} \text{   if   }\ \ \ n_3 -4 \le 2 i \le 2 (n_3 -1) \\   
& b_1 \cc a_2^2 \cc a_1  \cc e^{-1} \cc Q(a_i, a_1)^{-1}   \cc a_{i+1}^{-1} \cc Q( a_i , a_1)^{-1} \cc e\cc a_1^{-1} \cc a_2^{-2} \cc b_1^{-1}  \\  & \hspace{6cm} \text{   if   } \ \ \ 10 \le 2 i < n_3 -4 \\
\end{aligned} \right.\\
&a_{n_3}\ \ &\mapsto \ \ &\ \  e^{-1} \cc a_1^{-1} \cc a_2^{-2} \cc b_1^{-1}\\
\end{aligned}
\end{equation*}


\subsection{Case 3:  $n_1=1, n_2 = n_3-1 $}\label{A:case3}

The relator of this case is given by \[ R= e^2 \cdot c_{n_2}^2 \cc a_{n_3}^2 \cc c_{n_2-1}^2 \cc a_{n_3-1}^2 \cc \cdots c_2^2 \cc a_3^2 \cc b_1^2 \cc c_1^2 \cc a_2^2 \cc a_1^2 \]

\vspace{1ex}
The induced action $f^{-1}_{\mathbf{R} *}$ is given by 

\begin{equation*} 
\begin{aligned}
&e\ \ &\mapsto \ \ &\ \  e^{-1} \cc a_1^{-2} \cc a_2^{-2} \cc c_1^{-2} \cc b_1^{-1} \cc c_1 \cc a_2^2 \cc a_1 \cc e\\ 
& c_1 \ \ &\mapsto \ \ &\left\{ \begin{aligned} & e^{-1} \cc Q(a_3, a_1)^{-1} \cc c_2^{-1} \cc Q(a_3, a_1) \cc e \ \ & \text{   if   }\ \ n_2 \ge 5\\
& e \cc Q(c_3, a_4) \cc c_2^{-1} \cc Q(c_3, a_4)^{-1} \cc e^{-1} \ \ & \text{   if   }\ \ n_2 = 4\\
\end{aligned} \right.\\
&c_{i}\ \ &\mapsto \ \ &\ \  c_1^{-1} \cc b_1^{-1} \cc c_{i+1} \cc b_1 \cc c_1 \ \ \ \ \ \ \ \ \ \ \ \text{   if   } \ \  2 \le i \le n_2 -1\\
&c_{n_2}\ \ &\mapsto \ \ &\ \  e^{-1} \cc a_1^{-1} \cc a_2^{-2} \cc c_1^{-1}\\
& b_1 \ \ &\mapsto \ \ &\left\{ \begin{aligned} & c_1^{-1} \cc b_1^{-1}\cc c_2^2 \cc a_3^2 \cc b_1^2 \cc c_1^2 \cc a_2^2 \cc a_1^2 \cc e  \ \ & \text{   if   }\ \ n_2 \ge 5\\
& c_1^{-1} \cc b_1^{-1}\cc a_4^{-2} \cc c_3^{-2} \cc a_5^{-2} \cc c_4^{-2} \cc e^{-1} \ \ & \text{   if   }\ \ n_2 = 4\\
\end{aligned} \right.\\
&a_1\ \ &\mapsto \ \ &\ \  e^{-1} \cc Q( c_1, a_1)^{-1} \cc b_1^{-1} \cc c_1 \cc a_2^{-1}\cc c_1^{-1} \cc b_1 \cc Q(c_1, a_1)\cc e \\
&a_2\ \ &\mapsto \ \ &\ \  e^{-1} \cc Q( b_1, a_1)^{-1} \cc a_3^{-1} \cc Q(b_1, a_1)\cc e \\
&a_{i}\ \ &\mapsto \ \ &\ \  c_1^{-1} \cc b_1^{-1} \cc a_{i+1} \cc b_1 \cc c_1 \ \ \ \ \ \ \ \ \ \ \ \text{   if   } \ \  3 \le i \le n_2 =n_3 -1\\
&a_{n_3}\ \ &\mapsto \ \ &\ \  c_1 \cc a_2^2 \cc a_1 \cc e \cc b_1 \cc c_1 \\
\end{aligned}
\end{equation*}


\subsection{Case 4:  $n_1=1, n_2 < n_3-1 $}\label{A:case4} 

The relator of this case is given by \[ R= e^2 \cdot a_{n_3}^2 \cc a_{n_3-1}^2 \cdots a_{n_2+2}^2  \cc c_{n_2}^2 \cc a_{n_2+1}^2 \cc c_{n_2-1}^2 \cc \cdots a_4^2 \cc c_2^2 \cc a_3^2 \cc c_1^2 \cc b_1^2 \cc a_2^2 \cc a_1^2\]

\vspace{1ex}
The induced action $f^{-1}_{\mathbf{R} *}$ is given by 

\begin{equation*} 
\begin{aligned}
&e\ \ &\mapsto \ \ &\ \  e^{-1} \cc a_1^{-2} \cc a_2^{-2} \cc b_1^{-2} \cc c_1^{-1} \cc b_1 \cc a_2^2 \cc a_1 \cc e\\
&c_{i}\ \ &\mapsto \ \ &\ \  b_1^{-1} \cc c_1^{-1} \cc c_{i+1} \cc c_1 \cc b_1 \ \ \ \ \ \ \ \ \ \ \ \text{   if   } \ \  1 \le i \le n_2  - 1\\
& c_{n_2} \ \ &\mapsto \ \ &\left\{ \begin{aligned} & b_1 \cc a_2^2 \cc a_1 \cc e \cc Q(a_{n_3}, a_{n_2+3}) \cc c_1 \cc b_1 \ \ & \text{   if   }\ \ n_3 \le 3 n_2 + 5\\
& b_1 \cc a_2^2 \cc a_1 \cc e^{-1}  \cc Q(a_{n_2+2}, a_{1})^{-1}  \cc c_1 \cc b_1 \ \ & \text{   if   }\ \ n_3 > 3 n_2  +5 \\
\end{aligned} \right.\\
&b_1\ \ &\mapsto \ \ &\ \  b_1^{-1}\cc c_1^{-1} \cc a_3^2 \cc c_1^2 \cc b_1^2 \cc a_2^2\cc a_1^2 \cc e\\
&a_1\ \ &\mapsto \ \ &\ \  e^{-1}\cc a_1^{-2} \cc a_2^{-2} \cc b_1^{-2} \cc c_1^{-1} \cc b_1 \cc a_2^{-1} \cc b_1^{-1} \cc c_1 \cc b_1^2 \cc a_2^2 \cc a_1^2 \cc e\\
&a_2\ \ &\mapsto \ \ &\ \  e^{-1}\cc a_1^{-2} \cc a_2^{-2} \cc b_1^{-2} \cc c_1^{-2}  \cc a_3^{-1} \cc c_1^2 \cc b_1^2 \cc a_2^2 \cc a_1^2 \cc e\\
&a_{j}\ \ &\mapsto \ \ &\ \  b_1^{-1} \cc c_1^{-1} \cc a_{j+1} \cc c_1 \cc b_1 \ \ \ \ \ \ \ \ \ \ \ \text{   if   } \ \  3 \le j \le n_2  +1\\
& a_i \ \ &\mapsto \ \ &\left\{ \begin{aligned} & b_1 \cc a_2^2 \cc a_1 \cc e^{-1} \cc Q(a_i, a_1)^{-1} \cc a_{i+1}^{-1} \cc Q(a_i, a_1) \cc e \cc a_1^{-1} \cc a_2^{-2} \cc b_1^{-1} \\  & \hspace{6cm} \text{   if   }\ \ \ 2 (n_2 +2) \le 2 i < n_3 -n_2 \\   
&  b_1 \cc a_2^2 \cc a_1 \cc e \cc Q(a_{n_3}, a_{i+2}) \cc a_{i+1}^{-1} \cc Q(a_{n_3},a_{i+2})^{-1} \cc e^{-1} \cc a_1^{-1} \cc a_2^{-2} \cc b_1^{-1}   \\  & \hspace{6cm} \text{   if   } \ \ \ 2i  \ge n_3-n_2 \text{, and }\   n_2+2 \le i \le n_3 -1 \\
\end{aligned} \right.\\
\end{aligned}
\end{equation*}
\begin{equation*} 
\begin{aligned}
&a_{n_3}\ \ &\mapsto \ \ &\ \  e^{-1} \cc a_1^{-1} \cc a_2^{-2} \cc b_1^{-1}\\
\end{aligned}
\end{equation*}

\vspace{1ex}

\subsection{Case 5: $2 \le n_1 \le n_2 < n_3$}\label{A:case5}

The relator of this case is given by \[ R=e^2\cdot  a_{n_3}^2 \cc a_{n_3-1}^2 \cc \cdots a_{n_2+1}^2 \cc c_{n_2}^2 \cc a_{n_2}^2 \cc c_{n_2-1}^2 \cc \cdots a_{n_1+1}^2  \cc c_{n_1}^2 \cc b_{n_1}^2 \cc a_{n_1}^2 \cc \cdots c_2^2 \cc b_2^2 \cc a_2^2 \cc c_1^2 \cc b_1^2\cc a_1^2\]

\vspace{1ex}
The induced action $f^{-1}_{\mathbf{R} *}$ is given by

\begin{equation*} 
\begin{aligned}
&e\ \ &\mapsto \ \ &\ \  e^{-1} \cc a_1^{-2} \cc  b_1^{-2} \cc c_1^{-1} \cc b_1  \cc a_1 \cc e\\ 
& c_i \ \ &\mapsto \ \ &\left\{ \begin{aligned} &  e^{-1}  \cc Q(b_{i+1},a_1)^{-1} \cc c_{i+1}^{-1}  \cc Q(b_{i+1},a_1) \cc e \\  & \hspace{3cm} \text{   if   }\ \ \ 1 \le i \le n_1-1 \ \text{ and }\ 6i+5 \le n_1+n_2+n_3 \\   
&  e \cc Q(a_{n_3}, a_{i+2}) \cc c_{i+1}^{-1} \cc Q(a_{n_3}, a_{i+2})^{-1} \cc e^{-1}   \\  & \hspace{3cm} \text{   if   } \ \ \ 1 \le i \le n_1-1 \ \text{ and }\ 6i+5 >  n_1+n_2+n_3  \\
\end{aligned} \right.\\
&c_j\ \ &\mapsto \ \ &\ \  b_1^{-1} \cc c_1^{-1} \cc c_{j+1}\cc c_1 \cc b_1 \hspace{4cm} \text{  if  }\ \ n_1 \le j \le n_2-1 \\
& c_{n_2} \ \ &\mapsto \ \ &\left\{ \begin{aligned} & b_1 \cc a_1 \cc e  \cc Q(a_{n_3}, a_{n_2+2}) \cc c_1 \cc b_1 \  &\text{   if   }\ \ \ n_3 \le n_1 + 3 n_2 +2  \\   
&   b_1 \cc a_1 \cc e^{-1}  \cc Q(a_{n_2+1}, a_{1})^{-1} \cc c_1 \cc b_1   \ \hspace{0.9cm} &\text{   if   } \ \ \ n_3 > n_1 + 3 n_2 +2\\
\end{aligned} \right.\\
& b_i \ \ &\mapsto \ \ &\left\{ \begin{aligned} &  e^{-1}  \cc Q(a_{i+1},a_1)^{-1} \cc b_{i+1}^{-1}  \cc Q(a_{i+1},a_1) \cc e \\  & \hspace{3cm} \text{   if   }\ \ \ 1 \le i \le n_1-1 \ \text{ and }\ 6i+3 \le n_1+n_2+n_3 \\   
&  e \cc Q(a_{n_3}, c_{i+1}) \cc b_{i+1}^{-1} \cc Q(a_{n_3}, c_{i+1})^{-1} \cc e^{-1}   \\  & \hspace{3cm} \text{   if   } \ \ \ 1 \le i \le n_1-1 \ \text{ and }\ 6i+3 >  n_1+n_2+n_3  \\
\end{aligned} \right.\\
& b_{n_1} \ \ &\mapsto \ \ &\left\{ \begin{aligned} & b_1^{-1} \cc c_1^{-1} \cc Q(a_{n_3}, c_{n_1+1})^{-1} \cc e^{-1} \  &\text{   if   }\ \ \ 5 n_1 +2 \ge n_2 + n_3  \\   
&  b_1^{-1} \cc c_1^{-1} \cc Q(a_{n_1+1}, a_1) \cc e   \ \hspace{2.7cm} &\text{   if   } \ \ \ 5 n_1 +2 < n_2 + n_3 \\
\end{aligned} \right.\\
& a_i \ \ &\mapsto \ \ &\left\{ \begin{aligned} &  e^{-1}  \cc Q(c_{i},a_1)^{-1} \cc a_{i+1}^{-1}  \cc Q(c_{i},a_1) \cc e \\  & \hspace{3cm} \text{   if   }\ \ \ 1 \le i \le n_1-1 \ \text{ and }\ 6i+1 \le n_1+n_2+n_3 \\   
&  e \cc Q(a_{n_3}, b_{i+1}) \cc a_{i+1}^{-1} \cc Q(a_{n_3}, b_{i+1})^{-1} \cc e^{-1}   \\  & \hspace{3cm} \text{   if   } \ \ \ 1 \le i \le n_1-1 \ \text{ and }\ 6i+1 >  n_1+n_2+n_3  \\
\end{aligned} \right.\\
& a_{n_1} \ \ &\mapsto \ \ &\left\{ \begin{aligned} & e^{-1}  \cc Q(c_{n_1},a_1)^{-1} \cc a_{n_1+1}^{-1}  \cc Q(c_{n_1},a_1) \cc e \  &\text{   if   }\ \ \ 5 n_1  \le  n_2 + n_3  \\   
&   e \cc Q(a_{n_3}, c_{n_1+1}) \cc a_{n_1+1}^{-1} \cc Q(a_{n_3}, c_{n_1+1})^{-1} \cc e^{-1}     \ \hspace{2.7cm} &\text{   if   } \ \ \ 5 n_1 >  n_2 + n_3 \\
\end{aligned} \right.\\
&a_j\ \ &\mapsto \ \ &\ \  b_1^{-1} \cc c_1^{-1} \cc a_{j+1}\cc c_1 \cc b_1 \hspace{4cm} \text{  if  }\ \ n_1+1 \le j \le n_2 \\
& a_k \ \ &\mapsto \ \ &\left\{ \begin{aligned} & b_1 \cc a_1 \cc e \cc Q(a_{n_3}, a_{k+2}) \cc a_{k+1}^{-1} \cc Q(a_{n_3}, a_{k+2})^{-1} \cc e^{-1} \cc a_1^{-1} \cc b_1^{-1} \\  & \hspace{6cm} \text{   if   }\ \ \ 2k \ge n_3 - n_1-n_2 -1 \text{ and } k \ge n_2+1 \\   
&  b_1 \cc a_1 \cc e^{-1} \cc Q(a_{1}, a_{k})^{-1} \cc a_{k+1}^{-1} \cc Q(a_{k}, a_{1}) \cc e \cc a_1^{-1} \cc b_1^{-1}   \\  & \hspace{6cm} \text{   if   } \ \ \ 2 n_2 +2 \le 2k < n_3 - n_1-n_2 -1 \\
\end{aligned} \right.\\
&a_{n_3}\ \ &\mapsto \ \ &\ \  e^{-1} \cc a_1^{-1} \cc b_1^{-1}\\
\end{aligned}
\end{equation*}


\subsection{Case 6: $2 \le n_1 < n_2 = n_3$}\label{A:case6}

The relator of this case is given by \[ R=e^2\cdot  c_{n_2}^2 \cc a_{n_2}^2 \cc c_{n_2-1}^2 \cc \cdots  \cc c_{n_1+1}^2\cc a_{n_1+1}^2  \cc b_{n_1}^2 \cc c_{n_1}^2\cc a_{n_1}^2 \cc \cdots b_2^2 \cc c_2^2 \cc a_2^2 \cc b_1^2 \cc c_1^2\cc a_1^2\]

\vspace{1ex}
The induced action $f^{-1}_{\mathbf{R} *}$ is given by

\begin{equation*} 
\begin{aligned}
&e\ \ &\mapsto \ \ &\ \  e^{-1} \cc a_1^{-2} \cc  c_1^{-2} \cc b_1^{-1} \cc c_1  \cc a_1 \cc e\\ 
& c_i \ \ &\mapsto \ \ &\left\{ \begin{aligned} &  e^{-1}  \cc Q(a_{i+1},a_1)^{-1} \cc c_{i+1}^{-1}  \cc Q(a_{i+1},a_1) \cc e \\  & \hspace{3cm} \text{   if   }\ \ \ 1 \le i \le n_1-1 \ \text{ and }\ 6i+3 \le n_1+2 n_2 \\   
&  e \cc Q(c_{n_2}, b_{i+1}) \cc c_{i+1}^{-1} \cc Q(c_{n_2}, b_{i+1})^{-1} \cc e^{-1}   \\  & \hspace{3cm} \text{   if   } \ \ \ 1 \le i \le n_1-1 \ \text{ and }\ 6i+3 >  n_1+2 n_2  \\
\end{aligned} \right.\\
& c_{n_1} \ \ &\mapsto \ \ &\left\{ \begin{aligned} &  e^{-1}  \cc Q(a_{n_1+1},a_1)^{-1} \cc c_{n_1+1}^{-1}  \cc Q(a_{n_1+1},a_1) \cc e \\  & \hspace{5cm} \text{   if   }\ \ \ 5 n_1+3 \le 2 n_2 \\   
&  e \cc Q(c_{n_2}, a_{n_1+2}) \cc c_{n_1+1}^{-1} \cc Q(c_{n_2}, a_{n_1+2})^{-1} \cc e^{-1}   \\  & \hspace{5cm} \text{   if   } \ \ \ 5 n_1+3 >  2 n_2  \\
\end{aligned} \right.\\
&c_j\ \ &\mapsto \ \ &\ \  c_1^{-1} \cc b_1^{-1} \cc c_{j+1}\cc b_1 \cc c_1 \hspace{4cm} \text{  if  }\ \ n_1+1 \le j \le n_2-1 \\
& c_{n_2} \ \ &\mapsto \ \ &\ \  e^{-1} \cc a_1^{-1} \cc c_1^{-1} \\
& b_i \ \ &\mapsto \ \ &\left\{ \begin{aligned} &  e^{-1}  \cc Q(c_{i+1},a_1)^{-1} \cc b_{i+1}^{-1}  \cc Q(c_{i+1},a_1) \cc e \\  & \hspace{3cm} \text{   if   }\ \ \ 1 \le i \le n_1-1 \ \text{ and }\ 6i+5 \le n_1+2 n_2 \\   
&  e \cc Q(c_{n_2}, a_{i+2}) \cc b_{i+1}^{-1} \cc Q(c_{n_2}, a_{i+2})^{-1} \cc e^{-1}   \\  & \hspace{3cm} \text{   if   } \ \ \ 1 \le i \le n_1-1 \ \text{ and }\ 6i+5 >  n_1+2 n_2  \\
\end{aligned} \right.\\
& b_{n_1} \ \ &\mapsto \ \ &\left\{ \begin{aligned} & c_1^{-1} \cc b_1^{-1} \cc Q(c_{n_2}, a_{n_1+2})^{-1} \cc e^{-1} \  &\text{   if   }\ \ \ 5 n_1 +3 \ge 2 n_2   \\   
&  c_1^{-1} \cc b_1^{-1} \cc Q(c_{n_1+1}, a_1) \cc e   \ \hspace{2.7cm} &\text{   if   } \ \ \ 5 n_1 +3 < 2 n_2  \\
\end{aligned} \right.\\
& a_i \ \ &\mapsto \ \ &\left\{ \begin{aligned} &  e^{-1}  \cc Q(b_{i},a_1)^{-1} \cc a_{i+1}^{-1}  \cc Q(b_{i},a_1) \cc e \\  & \hspace{3cm} \text{   if   }\ \ \ 1 \le i \le n_1-1 \ \text{ and }\ 6i+1 \le n_1+2 n_2 \\   
&  e \cc Q(c_{n_2}, c_{i+1}) \cc a_{i+1}^{-1} \cc Q(c_{n_2}, c_{i+1})^{-1} \cc e^{-1}   \\  & \hspace{3cm} \text{   if   } \ \ \ 1 \le i \le n_1-1 \ \text{ and }\ 6i+1 >  n_1+2 n_2  \\
\end{aligned} \right.\\
& a_{n_1} \ \ &\mapsto \ \ &\left\{ \begin{aligned} & e^{-1}  \cc Q(b_{n_1},a_1)^{-1} \cc a_{n_1+1}^{-1}  \cc Q(b_{n_1},a_1) \cc e \  &\text{   if   }\ \ \ 5 n_1  \le 2  n_2   \\   
&   e \cc Q(c_{n_2}, c_{n_1+1}) \cc a_{n_1+1}^{-1} \cc Q(c_{n_2}, c_{n_1+1})^{-1} \cc e^{-1}     \ \hspace{2.7cm} &\text{   if   } \ \ \ 5 n_1 >  2 n_2  \\
\end{aligned} \right.\\
&a_j\ \ &\mapsto \ \ &\ \  c_1^{-1} \cc b_1^{-1} \cc a_{j+1}\cc b_1 \cc c_1 \hspace{4cm} \text{  if  }\ \ n_1+1 \le j \le n_2-1 \\
&a_{n_2}\ \ &\mapsto \ \ &\ \ c_1\cc a_1\cc e \cc b_1 \cc c_1\\
\end{aligned}
\end{equation*}

 \section{Invariant $\Gamma_A$ Sets}\label{ApendB}
 For each case discussed in this article, we list the invariant set of admissible cyclic words, splitting and merging maps $S,M$, and a vector $v_\lambda$ such that $S v_\lambda = \lambda M v_\lambda$ together with $\lambda>1$. 

\subsection{Orbit data $1,3,9$ with a cyclic permutation}\label{B:139}
Let $f_\mathbf{R}: X(\mathbf{R}) \to X(\mathbf{R})$ be a real diffeomorphism associated with the orbit data $1,3,9$ and a cyclic permutation. 
Let \[\Gamma = \{ \gamma_1, \dots, \gamma_{28} \},\ \Delta = \{ \gamma_1, \gamma_3,\gamma_4,\gamma_5,\gamma_{11},\gamma_{12},\gamma_{13},\gamma_{17},\gamma_{20},\gamma_{21}, \zeta_1,\dots,\zeta_{10}, \mu_1, \dots, \mu_6, \eta_1, \dots \eta_4\}\]
be two ordered sets of non-cyclic words in $\pi_1( X(\mathbf{R}),P_{fix})$ where 
\begin{equation*}
\begin{aligned}
\gamma_1  \cc &=\cc a_2^{-1}\cc b_1^{-1}\cc a_7^{-1}\cc a_8^{-2}\cc a_9^{-2}\cc e^{-1}, \quad
&\gamma_2 \cc&=\cc a_1^{-1}\cc a_2^{-2} \cc b_1^{-2} \cc c_1^{-1} \cc b_1\cc a_2 \cc e^{-1} \\
\gamma_3 \cc&=\cc a_2^{-1} \cc b_1^{-1} \cc a_6^{-1} \cc a_7^{-2} \cc a_8^{-2} \cc a_9^{-2} \cc e^{-1}, \quad 
&\gamma_4 \cc&=\cc a_1^{-1} \cc a_2^{-2} \cc b_1^{-2} \cc c_1^{-2} \cc a_3^{-1} \cc c_1 \cc b_1 \cc e^{-1} \\
\gamma_5 \cc&=\cc a_2^{-1} \cc b_1^{-1} \cc a_5^{-1} \cc a_6^{-2} \cc a_7^{-2} \cc a_8^{-2} \cc a_9^{-2} \cc e^{-1}, \quad
& \gamma_6 \cc&=\cc a_1^{-1} \cc a_2^{-2} \cc b_1^{-2} \cc c_1^{-2} \cc a_3^{-1} \cc c_1 \cc b_1^2 \cc a_2 \cc e^{-1} \\
\end{aligned}
\end{equation*}
\begin{equation*} 
\begin{aligned}
&\gamma_7 \cc=\cc a_1\cc a_2^{-2} \cc b_1^{-2} \cc c_1^{-2} \cc a_3^{-1} \cc c_1 \cc b_1^2 \cc a_2^2 \cc a_1 \cc e \cc a_9 \cc b_1 \cc a_2 \cc e^{-1}, \\
& \gamma_8 \cc=\cc a_2^{-2} \cc b_1^{-1} \cc c_1\cc b_1^2 \cc a_2^2 \cc a_1 \cc e\cc a_9^2 \cc a_8^2 \cc a_7^2 \cc a_6 \cc b_1 \cc a_2 \cc e^{-1}\\
&\gamma_9\cc = \cc a_1^{-1} \cc a_2^{-2} \cc b_1^{-2} \cc c_1^{-2} \cc a_3^{-1} \cc c_1 \cc b_1^2 \cc a_2^2 \cc a_1 \cc e \cc a_9^2 \cc a_8\cc b_1 \cc a_2 \cc e^{-1}\\ 
&\gamma_{10} \cc  = \cc a_1^{-1} \cc a_2^{-2} \cc b_1^{-2} \cc c_1^{-2} \cc a_3^{-1} \cc c_1 \cc b_1^2 \cc a_2^2 \cc a_1 \cc e\cc a_9^2 \cc a_8 \cc a_7 \cc b_1 \cc a_2 \cc e^{-1} \\
&\gamma_{11} \cc = \cc a_1^{-1} \cc a_2^{-2} \cc b_1^{-2} \cc c_1^{-2} \cc a_3^{-2} \cc c_2^{-1} \cc a_4^{-1} \cc c_3^{-2} \cc a_5^{-2} \cc a_6^{-2} \cc a_7^{-2} \cc a_8^{-2} \cc a_9^{-2} \cc e^{-1} \\
&\gamma_{12} \cc = \cc a_1^{-2} \cc a_2^{-2} \cc b_1^{-2} \cc c_1^{-2} \cc a_3^{-2} \cc c_2^{-1} \cc a_4^{-1} \cc c_3^{-2} \cc a_5^{-2} \cc a_6^{-2} \cc a_7^{-2} \cc a_8^{-2} \cc a_9^{-2} \cc e^{-1} \\
&\gamma_{13} \cc = \cc a_1^{-2} \cc a_2^{-2} \cc b_1^{-2} \cc c_1^{-2} \cc a_3^{-2} \cc c_2^{-2} \cc a_4^{-1} \cc c_3^{-1} \cc a_5^{-2} \cc a_6^{-2} \cc a_7^{-2} \cc a_8^{-2} \cc a_9^{-2} \cc e^{-1} \\
&\gamma_{14} \cc = \cc a_1^{-1} \cc a_2^{-2} \cc b_1^{-2} \cc c_1^{-2} \cc a_3^{-1}\cc c_1\cc b_1^2 \cc a_2^2 \cc a_1 \cc e \cc a_9^2 \cc a_8^2 \cc a_7^2 \cc a_6 \cc b_1 \cc a_2 \cc e^{-1} \\
&\gamma_{15} \cc = \cc a_1^{-1} \cc a_2^{-2} \cc b_1^{-2} \cc c_1^{-1} \cc a_3\cc c_1^2\cc b_1^2 \cc a_2^2 \cc a_1 \cc e \cc a_9^2 \cc a_8^2 \cc a_7^2 \cc a_6^2 \cc a_5 \cc b_1 \cc a_2 \cc e^{-1} \\
&\gamma_{16} \cc = \cc a_1^{-1} \cc a_2^{-2} \cc b_1^{-2} \cc c_1^{-2} \cc a_3^{-1}\cc c_2 \cc a_3^2 \cc c_1^2 \cc b_1^2 \cc a_2^2 \cc a_1^2 \cc e \cc a_2^{-1} \cc b_1^{-1} \cc a_7^{-1} \cc a_8^{-2} \cc a_9^{-2} \cc e^{-1} \\
&\gamma_{17} \cc = \cc a_1^{-2} \cc a_2^{-2} \cc b_1^{-2} \cc c_1^{-2} \cc a_3^{-2} \cc c_2^{-2} \cc a_4^{-1}\cc c_2 \cc a_3^2 \cc c_1^2 \cc b_1^2 \cc a_2^2 \cc a_1^2 \cc e \cc a_2^{-1} \cc b_1^{-1} \cc a_9^{-1} \cc e^{-1} \\
&\gamma_{18} \cc = \cc a_1^{-2} \cc a_2^{-2} \cc b_1^{-2} \cc c_1^{-2} \cc a_3^{-2} \cc c_2^{-1} \cc a_3 \cc c_1^2 \cc b_1^2 \cc a_2^2 \cc a_1 \cc e \cc a_9^2 \cc a_8^2 \cc a_7^2 \cc a_6 \cc b_1 \cc a_2 \cc e^{-1} \\
&\gamma_{19} \cc=\cc a_1^{-1} \cc a_2^{-2} \cc b_1^{-2} \cc c_1^{-1} \cc b_1 \cc a_2\cc e \cc a_9^2 \cc a_8^2 \cc a_7^2 \cc a_6^2 \cc a_5^2 \cc c_3 \cc a_5^{-1} \cc a_6^{-2} \cc a_7^{-2} \cc a_8^{-2} \cc a_9^{-2} \cc e^{-1} \\
&\gamma_{20} \cc = \cc a_1^{-2} \cc a_2^{-2} \cc b_1^{-2} \cc c_1^{-2} \cc a_3^{-2} \cc c_2^{-2} \cc a_4^{-1}\cc c_2 \cc a_3^2 \cc c_1^2 \cc b_1^2 \cc a_2^2 \cc a_1^2 \cc e \cc a_2^{-1} \cc b_1^{-1} \cc a_8^{-1}\cc a_9^{-2} \cc e^{-1} \\
&\gamma_{21} \cc = \cc a_1^{-2} \cc a_2^{-2} \cc b_1^{-2} \cc c_1^{-2} \cc a_3^{-2} \cc c_2^{-2} \cc a_4^{-1}\cc c_2 \cc a_3^2 \cc c_1^2 \cc b_1^2 \cc a_2^2 \cc a_1^2 \cc e \cc a_2^{-1} \cc b_1^{-1} \cc a_7^{-1} \cc a_8^{-2}\cc a_9^{-2} \cc e^{-1} \\
& \gamma_{22} \cc= \cc a_1^{-1} \cc a_2^{-2} \cc b_1^{-2} \cc c_1^{-1} \cc a_3\cc c_1^2\cc b_1^2 \cc a_2^2 \cc a_1 \cc e \cc a_9^2 \cc a_8^2 \cc a_7^2 \cc a_6^2 \cc a_5^2 \cc c_3 \cc a_5^{-1} \cc a_6^{-2} \cc a_7^{-2} \cc a_8^{-2} \cc a_9^{-2} \cc e^{-1}\\
& \gamma_{23} \cc= \cc a_1^{-1} \cc a_2^{-2} \cc b_1^{-2} \cc c_1^{-1} \cc a_3\cc c_1^2\cc b_1^2 \cc a_2^2 \cc a_1 \cc e \cc a_9^2 \cc a_8^2 \cc a_7^2 \cc a_6^2 \cc a_5 \cc c_3^{-1} \cc a_5^{-2} \cc a_6^{-2} \cc a_7^{-2} \cc a_8^{-2} \cc a_9^{-2} \cc e^{-1}\\
& \gamma_{24} \cc= \cc a_1^{-1} \cc a_2^{-2} \cc b_1^{-2} \cc c_1^{-1} \cc a_3\cc c_1^2\cc b_1^2 \cc a_2^2 \cc a_1 \cc e \cc a_9^2 \cc a_8^2 \cc a_7^2 \cc a_6^2 \cc a_5^2 \cc c_3 \cc a_4^{-1} \cc a_3^{-2} \cc a_5^{-2} \cc a_6^{-2} \cc a_7^{-2} \cc a_8^{-2} \cc a_9^{-2} \cc e^{-1}\\
& \gamma_{25} \cc= \cc a_1^{-2} \cc a_2^{-2} \cc b_1^{-2} \cc c_1^{-2} \cc a_3^{-2}\cc c_2^{-1} \cc a_3\cc c_1^2\cc b_1^2 \cc a_2^2 \cc a_1 \cc e \cc a_9^2 \cc a_8^2 \cc a_7^2 \cc a_6^2 \cc a_5 \cc c_3^{-1} \cc a_5^{-2} \cc a_6^{-2} \cc a_7^{-2} \cc a_8^{-2} \cc a_9^{-2} \cc e^{-1}\\
& \gamma_{26} \cc= \cc a_1^{-2} \cc a_2^{-2} \cc b_1^{-2} \cc c_1^{-1} \cc a_3^{-2}\cc c_2^{-2} \cc a_4^{-1} \cc c_2 \cc a_3^2 \cc c_1^2\cc b_1^2 \cc a_2^2 \cc a_1 \cc e \cc a_9^2 \cc a_8^2 \cc a_7^2 \cc a_6^2 \cc a_5^2 \cc c_3 \cc  a_5^{-1} \cc a_6^{-2} \cc a_7^{-2} \cc a_8^{-2} \cc a_9^{-2} \cc e^{-1}\\
&\gamma_{27} \cc = \cc a_1^{-2} \cc a_2^{-2} \cc b_1^{-2} \cc c_1^{-2} \cc a_3^{-2} \cc c_2^{-2} \cc a_4^{-1} \cc c_2 \cc a_3^2 \cc c_1^2 \cc b_1^2 \cc a_2^2 \cc a_1^2 \cc e \cc b_1^{-1} \cc c_1^{-1} \cc a_3 \cc c_1^2 \cc b_1^2 \cc a_2^2 \cc a_1\cc e \\
& \hspace{3cm} \cc a_9^2 \cc a_8 \cc a_7^2 \cc a_6^2 \cc a_5^2 \cc c_3 \cc a_5^{-1} \cc a_6^{-2} \cc a_7^{-2} \cc a_8^{-2} \cc a_9^{-2} \cc e^{-1} \\
&\gamma_{28} \cc = \cc a_1^{-2} \cc a_2^{-2} \cc b_1^{-2} \cc c_1^{-2} \cc a_3^{-2} \cc c_2^{-2} \cc a_4^{-1} \cc c_2 \cc a_3^2 \cc c_1^2 \cc b_1^2 \cc a_2^2 \cc a_1^2 \cc e \cc a_2^{-1} \cc b_1^{-2} \cc c_1^{-1} \cc a_3 \cc c_1^2 \cc b_1^2 \cc a_2^2 \cc a_1\cc e \\
& \hspace{3cm} \cc a_9^2 \cc a_8 \cc a_7^2 \cc a_6^2 \cc a_5^2 \cc c_3 \cc a_5^{-1} \cc a_6^{-2} \cc a_7^{-2} \cc a_8^{-2} \cc a_9^{-2} \cc e^{-1} \\
\end{aligned}
\end{equation*}
and 
\begin{equation*}
\begin{aligned}
& \zeta_1 \cc = \cc a_1^{-1} \cc a_2^{-2} \cc c_1^{-2} \cc b_1 \\
& \zeta_2 \cc= \cc a_1^{-1} \cc a_2^{-2} \cc b_1^{-2} \cc c_1^{-2} \cc a_3^{-1} \cc c_1 \cc b_1 \cc c_1 \cc b_1^2 \cc c_1^{-2} \cc b_1^{-1} \\
& \zeta_3 \cc= \cc a_1^{-1} \cc a_2^{-2} \cc b_1^{-2} \cc c_1^{-2} \cc a_3^{-1} \cc c_1 \cc b_1^2 \cc a_2^2 \cc a_1 \cc e \cc a_9 \cc c_1 \cc b_1^2 \cc c_1^{-2} \cc b_1^{-1}\\
&\zeta_4 \cc=\cc a_2^{-1} \cc b_1^{-1} \cc c_1 \cc b_1^2 \cc a_2^2 \cc a_1 \cc e \cc a_9^2 \cc a_8^2 \cc a_7^2 \cc a_6 \cc c_1 \cc b_1^2 \cc c_1^{-2} \cc b_1^{-1}\\ 
&\zeta_5 \cc=\cc a_1^{-1} \cc a_2^{-2} \cc b_1^{-2} \cc c_1^{-1} \cc a_3 \cc c_1^2 \cc b_1^2 \cc a_2^2 \cc a_1 \cc a_2^{-2} \cc b_1^{-1} \cc c_1 \cc b_1^2 \cc c_1^{-2} \cc b_1^{-1} \\
&\zeta_6\cc = \cc a_1^{-2} \cc a_2^{-2} \cc b_1^{-2} \cc c_1^{-2} \cc a_3^{-2} \cc c_2^{-1} \cc a_3 \cc c_1^2 \cc b_1^2 \cc a_2^2 \cc a_1 \cc a_2^{-2} \cc b_1^{-1} \cc c_1 \cc b_1^2 \cc c_1^{-2} \cc b_1^{-1}\\
&\zeta_7 \cc = \cc a_1^{-2} \cc a_2^{-2} \cc b_1^{-2} \cc c_1^{-2} \cc a_3^{-2} \cc c_2^{-2}\cc a_4^{-1} \cc c_2  \cc a_3^2 \cc c_1^2 \cc b_1^2 \cc a_2^2 \cc a_1 \cc a_2^{-2} \cc b_1^{-1} \cc c_1 \cc b_1^2 \cc c_1^{-2} \cc b_1^{-1}\\
&\zeta_8 \cc = \cc a_1^{-2} \cc a_2^{-2} \cc b_1^{-2} \cc c_1^{-2} \cc a_3^{-2} \cc c_2^{-1}  \cc a_3 \cc c_1^2 \cc b_1^2 \cc a_2^2 \cc a_1\cc e \cc a_9^2 \cc a_8^2 \cc a_7^2 \cc a_6 \cc c_1 \cc b_1^2 \cc c_1^{-2} \cc b_1^{-1}\\
&\zeta_9 \cc = \cc a_1^{-2} \cc a_2^{-2} \cc b_1^{-2} \cc c_1^{-2} \cc a_3^{-2} \cc c_2^{-2}\cc a_4^{-1} \cc c_2  \cc a_3^2 \cc c_1^2 \cc b_1^2 \cc a_2^2 \cc a_1^2\cc e  \cc b_1^{-1} \cc c_1^{-1} \cc a_3 \cc c_1^2 \cc b_1^2 \cc a_2^2 \cc a_1 \cc a_2^{-2} \cc b_1^{-1} \cc c_1 \cc b_1^2 \cc c_1^{-2} \cc b_1^{-1}\\
&\zeta_{10} \cc = \cc a_1^{-2} \cc a_2^{-2} \cc b_1^{-2} \cc c_1^{-2} \cc a_3^{-2} \cc c_2^{-2}\cc a_4^{-1} \cc c_2  \cc a_3^2 \cc c_1^2 \cc b_1^2 \cc a_2^2 \cc a_1^2\cc e \cc a_2^{-1} \cc \cc b_1^{-2} \cc c_1^{-1} \cc a_3 \cc c_1^2 \cc b_1^2 \cc a_2^2 \cc a_1 \cc a_2^{-2} \cc b_1^{-1} \cc c_1 \cc b_1^2 \cc c_1^{-2} \cc b_1^{-1}\\
\end{aligned}
\end{equation*}
\begin{equation*} 
\begin{aligned}
& \mu_1 \cc=\cc b_1 \cc c_1^2 \cc b_1^{-2} \cc c_1^{-1} \cc b_1 \cc a_2 \cc e^{-1} \\
& \mu_2 \cc = \cc  b_1 \cc c_1^2 \cc b_1^{-2} \cc c_1^{-2}\cc a_3^{-1} \cc c_2 \cc a_3^2 \cc c_1^2  \cc b_1^2 \cc a_2^2\cc a_1^2 \cc e\cc a_2^{-1} \cc b_1^{-1} \cc a_7^{-1} \cc a_8^{-2} \cc a_9^{-2}  \cc e^{-1} \\
& \mu_3 \cc = \cc  b_1 \cc c_1^2 \cc b_1^{-2} \cc c_1^{-1} \cc b_1 \cc a_2 \cc e \cc a_9^2 \cc a_8^2 \cc a_7^2 \cc a_6^2 \cc a_5^2 \cc c_3 \cc a_5^{-1} \cc a_6^{-2} \cc a_7^{-2} \cc a_8^{-2} \cc a_9^{-2} \cc e^{-1}\\
& \mu_4 \cc = \cc  b_1 \cc c_1^2 \cc b_1^{-2} \cc c_1^{-1} \cc b_1 \cc a_2^2 \cc e \cc a_9^2 \cc a_8^2 \cc a_7^2 \cc a_6^2 \cc a_5^2 \cc c_3 \cc a_5^{-1} \cc a_6^{-2} \cc a_7^{-2} \cc a_8^{-2} \cc a_9^{-2} \cc e^{-1}\\
& \mu_5 \cc = \cc  b_1 \cc c_1^2 \cc b_1^{-2} \cc c_1^{-1} \cc b_1 \cc a_2^2 \cc e \cc a_9^2 \cc a_8^2 \cc a_7^2 \cc a_6^2 \cc a_5 \cc c_3^{-1} \cc a_5^{-2} \cc a_6^{-2} \cc a_7^{-2} \cc a_8^{-2} \cc a_9^{-2} \cc e^{-1}\\
& \mu_6 \cc = \cc  b_1 \cc c_1^2 \cc b_1^{-2} \cc c_1^{-1} \cc b_1 \cc a_2^2 \cc e \cc a_9^2 \cc a_8^2 \cc a_7^2 \cc a_6^2 \cc a_5^2 \cc c_3\cc a_4^{-1} \cc c_3^{-2}  \cc a_5^{-2} \cc a_6^{-2} \cc a_7^{-2} \cc a_8^{-2} \cc a_9^{-2} \cc e^{-1}\\
&\eta_1\cc = \cc b_1 \cc c_1^2 \cc b_1^{-2} \cc c_1^{-2} \cc a_3^{-1} \cc c_1 \cc b_1^2 \cc a_2^2 \cc a_1 \cc e \cc a_9^2 \cc a_8 \cc c_1 \cc b_1^2 \cc c_1^{-2} \cc b_1^{-1} \\
&\eta_2\cc = \cc b_1 \cc c_1^2 \cc b_1^{-2} \cc c_1^{-1}  \cc b_1 \cc a_2^2  \cc e \cc a_9^2 \cc a_8^2 \cc a_7^2 \cc a_6^2 \cc a_5 \cc c_1 \cc b_1^2 \cc c_1^{-2} \cc b_1^{-1} \\
&\eta_3\cc = \cc b_1 \cc c_1^2 \cc b_1^{-2} \cc c_1^{-2}\cc a_3^{-1}\cc c_1  \cc b_1^2 \cc a_2^2\cc a_1   \cc e \cc a_9^2 \cc a_8^2 \cc a_7 \cc c_1 \cc b_1^2 \cc c_1^{-2} \cc b_1^{-1} \\
&\eta_4\cc = \cc b_1 \cc c_1^2 \cc b_1^{-2} \cc c_1^{-2}\cc a_3^{-1}\cc c_1  \cc b_1^2 \cc a_2^2\cc a_1   \cc e \cc a_9^2 \cc a_8^2 \cc a_7^2 \cc a_6 \cc c_1 \cc b_1^2 \cc c_1^{-2} \cc b_1^{-1} \\
\end{aligned}
\end{equation*}
The admissible quadruples are
\begin{equation*}
\begin{aligned}
A= &\{ (1,15,17,22),(2,17,22,10),(2,26,16,15),(2,26,16,23),(2,27,16,23),(2,28,10,12),\\
&(2,28,14,17),(2,28,14,20),(2,28,16,23),(3,2,17,22),(3,2,26,16),(3,2,27,16),(3,2,28,10),\\
&(3,2,28,14),(3,2,28,16),(3,16,15,17),(3,19,16,15),(3,23,6,12),(4,12,8,13),(4,12,8,20),\\
&(4,12,8,21),(5,9,12,3),(5,10,13,9),(5,10,18,13 ),(5,10,25,8),(5,10,25,11),(6,12,8,20),\\
&(7,12,3,23),(7,12,5,9),(7,12,5,10),(7,12,8,20),(8,13,9,12),(8,20,22,10),(8,20,24,3),\\
&(8,21,23,6),(8,21,23,7),(9,12,3,2),(9,12,3,23),(10,12,3,2),(10,13,9,12),(10,18,13,9),\\
&(10,25,4,12),(10,25,6,12),(10,25,8,13),(10,25,11,8),(11,3,2,17),(11,3,2,27),(11,3,2,28),\\
&(11,3,23,6),(11,5,9,12),(11,5,10,13),(11,5,10,18),(11,5,10,25),(11,8,13,9),(11,8,20,22),\\
&(11,8,20,24),(11,8,21,23),(12,3,2,17),(12,3,2,27),(12,3,2,28),(12,3,23,6),(12,5,9,12),\\
&(12,5,10,13),(12,5,10,18),(12,5,10,25),(12,8,13,9),(12,8,20,22),(12,8,20,24),\\
&(12,8,21,23),(13,9,12,3),(14,17,22,10),(14,20,15,17),(14,20,24,1),(14,20,24,3),\\
&(15,17,22,10),(16,15,17,22),(16,23,7,12),(16,23,9,12),(17,22,10,12),(17,22,10,25),\\
&(18,13,9,12),(19,16,15,17),(20,15,17,22),(20,22,10,25),(20,24,1,15),(20,24,3,2),\\
&(20,24,3,16),(20,24,3,19),(21,23,6,12),(21,23,7,12),(22,10,12,3),(22,10,25,4),\\
&(22,10,25,6),(22,10,25,11),(23,6,12,8),(23,7,12,3),(23,7,12,5),(23,7,12,8),\\
&(23,9,12,3),(24,1,15,17),(24,3,2,26),(24,3,2,27),(24,3,16,15),(24,3,19,16),(25,4,12,8),\\
&(25,6,12,8),(25,8,13,9),(25,11,8,13),(26,16,15,17),(26,16,23,7),(26,16,23,9 ),\\
&(27,16,23,7),(28,10,12,3),(28,14,17,22),(28,14,20,15),(28,14,20,24),(28,16,23,7)\} 
\end{aligned}
\end{equation*}
Let $V= Span \, \Gamma$ and $W= Span\, \Delta$. The splitting map $S: V\to W$ is given by 
\begin{equation*}
S : \left\{ \begin{aligned} &\gamma_1 \mapsto \eta_1,\quad \gamma_2 \mapsto \mu_1 + \zeta_6, \quad \gamma_3 \mapsto \eta_3,\quad \gamma_4 \mapsto \mu_1 + \zeta_7, \quad \gamma_5 \mapsto \eta_4,\\
&\gamma_6 \mapsto \mu_1 + \zeta_9, \quad \gamma_7 \mapsto \mu_1 + \zeta_{10},\quad \gamma_8 \mapsto \mu_2 + \zeta_5,\quad \gamma_9 \mapsto \mu_1+ \gamma_{17} + \zeta_5, \\
&\gamma_{10} \mapsto \mu_1 +\gamma_{20}+ \zeta_5, \quad \gamma_{11} \mapsto \mu_3 + \zeta_{1},\quad \gamma_{12} \mapsto \mu_4 + \zeta_1,\quad \gamma_{13} \mapsto \eta_2,\\
&\gamma_{14} \mapsto \mu_1 +\gamma_{21}+ \zeta_5, \quad \gamma_{15} \mapsto \mu_1 + \gamma_{12} + \gamma_3+\zeta_{5}, \quad \gamma_{16} \mapsto \mu_1+\gamma_{13} + \zeta_1+\eta_1,\\
&\gamma_{17} \mapsto \mu_5+\zeta_2,\quad \gamma_{18} \mapsto \mu_6 +\gamma_{1}+ \zeta_5, \quad \gamma_{19} \mapsto \mu_1 + \zeta_{8},\quad \gamma_{20} \mapsto \mu_5 + \zeta_3,\\
&\gamma_{21} \mapsto \mu_5+ \zeta_1+\eta_1,\quad \gamma_{22} \mapsto \mu_1 +\gamma_{12}+ \zeta_4, \quad \gamma_{23} \mapsto \mu_1+\gamma_{12} + \gamma_3 + \zeta_{1},\\
&\gamma_{24} \mapsto \mu_1+\gamma_{12}+ \gamma_5 + \zeta_1, \quad \gamma_{25} \mapsto \mu_6 +\gamma_{3}+ \zeta_1, \quad \gamma_{26} \mapsto \mu_5 + \zeta_{4},\\
&\gamma_{27} \mapsto \mu_5+\gamma_{11}+  \zeta_4,\quad \gamma_{28} \mapsto \mu_5 +\gamma_4+\gamma_{12}+ \zeta_4.
\end{aligned} \right.
\end{equation*}
The merging map $M:V \to W$ is given by
\begin{equation*}
M: \left\{ \begin{aligned} &\gamma_1 \mapsto \gamma_1, \quad \gamma_2 \mapsto \zeta_1+ \mu_1, \quad \gamma_3 \mapsto \gamma_3, \quad \gamma_4 \mapsto \gamma_4, \quad \gamma_5 \mapsto \gamma_5, \quad \gamma_6 \mapsto \zeta_2 + \mu_1, \quad \gamma_7 \mapsto \zeta_3 + \mu_1 \\
&\gamma_8 \mapsto \zeta_4 + \mu_1, \quad \gamma_9 \mapsto \zeta_1+\eta_1 + \mu_1,\quad \gamma_{10} \mapsto \zeta_1+ \eta_3 + \mu_1, \quad \gamma_{11} \mapsto \gamma_{11}, \quad \gamma_{12} \mapsto \gamma_{12}, \quad \gamma_{13} \mapsto \gamma_{13}\\
& \gamma_{14} \mapsto \zeta_1 + \eta_4 + \mu_1, \quad \gamma_{15} \mapsto \zeta_5 + \eta_2 + \mu_1, \quad \gamma_{16} \mapsto \zeta_1 + \mu_2, \quad \gamma_{17} \mapsto \gamma_{17}, \quad \gamma_{18} \mapsto \zeta_8 + \mu_1\\
 &\gamma_{19} \mapsto \zeta_1 + \mu_3, \quad \gamma_{20} \mapsto \gamma_{20}, \quad \gamma_{21} \mapsto \gamma_{21}, \quad \gamma_{22} \mapsto \zeta_5 + \mu_4, \quad \gamma_{23} \mapsto \zeta_5 + \mu_5, \quad \gamma_{24} \mapsto \zeta_5+ \mu_6\\
 &\gamma_{25} \mapsto \zeta_6+ \mu_5, \quad \gamma_{26} \mapsto \zeta_7 + \mu_4, \quad \gamma_{27} \mapsto \zeta_9 + \mu_4, \quad \gamma_{28} \mapsto \zeta_{10} + \mu_4 \\
\end{aligned}
\right.
\end{equation*}
There is a $v_\lambda \in V$ such that $Sv_\lambda = \lambda Mv_\lambda$ where $\lambda\approx 1.40127$ is the largest real root of $\chi(t) = t^6-t^4-t^3-t^2+1$
and 
\begin{equation*}
\begin{aligned}
 v_\lambda \approx ( & 0.185,2.492,3.856,0.714,1.038,1,1.401,2.856,1.964,2.751,0.509,5.607,1.454,0.741,1.038 \\
 &2.038,1.401,0.259,0.363,1.964,0.529,1.778,2.586,1.454,1.778,0.509,0.714,1)\\
 \end{aligned}
 \end{equation*}

\subsection{Orbitdata $1,4,8$ with a cyclic permutation}\label{B:148}
Let $f_\mathbf{R}: X(\mathbf{R}) \to X(\mathbf{R})$ be a real diffeomorphism associated with the orbit data $1,4,8$ and a cyclic permutation. 
Let \[\Gamma = \{ \gamma_1, \dots, \gamma_{28} \},\ \Delta = \{ \gamma_2, \gamma_8,\gamma_9,\gamma_{10}, \gamma_{12},\gamma_{13}, \gamma_{15}, \gamma_{17}, \gamma_{19}, \zeta_1,\dots,\zeta_{11}, \mu_1, \dots, \mu_{9}, \eta_1\} \]
be two ordered sets of non-cyclic words in $\pi_1( X(\mathbf{R}),P_{fix})$ where
\begin{equation*}
\begin{aligned}
&\gamma_1 = a_1^{-1} \cc a_2^{-2} \cc b_1^{-2} \cc c_1^{-1} \cc b_1 \cc a_2 \cc e^{-1}, \quad \quad \gamma_2 = a_1^{-1} \cc a_2^{-2} \cc b_1^{-2} \cc c_1^{-2}\cc a_3^{-1} \cc c_1 \cc b_1 \cc e^{-1},\\
&\gamma_3 = a_1^{-1} \cc a_2^{-2} \cc b_1^{-2} \cc c_1^{-2}\cc a_3^{-1} \cc c_1 \cc b_1^2\cc a_2 \cc e^{-1},\quad \quad \gamma_4 = a_2^{-1} \cc b_1^{-1} \cc c_1 \cc b_1^2 \cc a_2^2 \cc a_1 \cc e \cc a_8^2 \cc a_7 \cc b_1 \cc a_2 \cc e^{-1},\\
&\gamma_5 = a_2^{-1} \cc b_1^{-1} \cc c_1 \cc b_1^2 \cc a_2^2 \cc a_1 \cc e \cc a_8^2 \cc a_7^2 \cc a_6 \cc b_1 \cc a_2 \cc e^{-1}, \qquad \gamma_6 = a_1^{-1} \cc a_2^{-2} \cc b_1^{-2} \cc c_1^{-2}\cc a_3^{-1} \cc c_1 \cc b_1^2\cc a_2^2 \cc a_1 \cc e \cc a_8 \cc b_1 \cc a_2 \cc e^{-1},\\
&\gamma_7 = a_1^{-1} \cc a_2^{-2} \cc b_1^{-2} \cc c_1^{-2}\cc a_3^{-1} \cc c_1 \cc b_1^2\cc a_2^2 \cc a_1 \cc e \cc a_8^2 \cc a_7 \cc b_1 \cc a_2 \cc e^{-1},\\
&\gamma_8 = a_2^{-1} \cc b_1^{-1} \cc c_1 \cc b_1^2 \cc a_2^2 \cc a_1 \cc e \cc a_8^2 \cc a_7^2 \cc a_6^2 \cc c_4 \cc a_6^{-1} \cc a_7^{-2} \cc a_8^{-2} \cc e^{-1},\\
&\gamma_9 = a_2^{-1} \cc b_1^{-1} \cc c_1 \cc b_1^2 \cc a_2^2 \cc a_1 \cc e \cc a_8^2 \cc a_7^2 \cc a_6 \cc c_4^{-1} \cc a_6^{-2} \cc a_7^{-2} \cc a_8^{-2} \cc e^{-1},\\
&\gamma_{10} = a_1^{-1} \cc a_2^{-2} \cc b_1^{-2} \cc c_1^{-2} \cc a_3^{-2} \cc c_2^{-1} \cc a_4^{-1} \cc c_3^{-2} \cc a_5^{-2} \cc c_4^{-2} \cc a_6^{-2} \cc a_7^{-2} \cc a_8^{-2} \cc e^{-1},\\
&\gamma_{11}= a_1^{-1} \cc a_2^{-2} \cc b_1^{-2} \cc c_1^{-2}\cc a_3^{-1} \cc c_2\cc a_3^2 \cc c_1^2 \cc b_1^2\cc a_2^2 \cc a_1^2 \cc e  \cc a_2^{-1} \cc b_1^{-1} \cc a_8^{-1} \cc e^{-1},\\
&\gamma_{12} = a_1^{-2} \cc a_2^{-2} \cc b_1^{-2} \cc c_1^{-2} \cc a_3^{-2} \cc c_2^{-1} \cc a_4^{-1} \cc c_3^{-2} \cc a_5^{-2} \cc c_4^{-2} \cc a_6^{-2} \cc a_7^{-2} \cc a_8^{-2} \cc e^{-1},\\
&\gamma_{13} = a_1^{-2} \cc a_2^{-2} \cc b_1^{-2} \cc c_1^{-2} \cc a_3^{-2} \cc c_2^{-2} \cc a_4^{-1} \cc c_3^{-1} \cc a_5^{-2} \cc c_4^{-2} \cc a_6^{-2} \cc a_7^{-2} \cc a_8^{-2} \cc e^{-1},\\
&\gamma_{14}= a_1^{-1} \cc a_2^{-2} \cc b_1^{-2} \cc c_1^{-2}\cc a_3^{-1} \cc c_2\cc a_3^2 \cc c_1^2 \cc b_1^2\cc a_2^2 \cc a_1^2 \cc e  \cc a_2^{-1} \cc b_1^{-1} \cc a_7^{-1} \cc a_8^{-2} \cc e^{-1},\\
&\gamma_{15} = a_2^{-1} \cc b_1^{-1} \cc c_1 \cc b_1^2 \cc a_2^2 \cc a_1 \cc e \cc a_8^2 \cc a_7^2 \cc a_6\cc c_4 \cc a_5^{-1}  \cc c_4^{-2} \cc a_6^{-2} \cc a_7^{-2} \cc a_8^{-2} \cc e^{-1},\\
&\gamma_{16}= a_1^{-1} \cc a_2^{-2} \cc b_1^{-2} \cc c_1^{-2}\cc a_3^{-1} \cc c_2\cc a_3^2 \cc c_1^2 \cc b_1^2\cc a_2^2 \cc a_1^2 \cc e  \cc a_2^{-1} \cc b_1^{-1}\cc a_6^{-1} \cc a_7^{-2} \cc a_8^{-2} \cc e^{-1},\\
&\gamma_{17}= a_1^{-2} \cc a_2^{-2} \cc b_1^{-2} \cc c_1^{-2}\cc a_3^{-2} \cc c_2^{-2}\cc a_4^{-1} \cc c_2 \cc a_3^2 \cc c_1^2 \cc b_1^2\cc a_2^2 \cc a_1^2 \cc e  \cc a_2^{-1} \cc b_1^{-1}\cc a_8^{-1} \cc e^{-1},\\
&\gamma_{18}= a_1^{-1} \cc a_2^{-2} \cc b_1^{-2} \cc c_1^{-1} \cc b_1 \cc a_2 \cc e  \cc a_8^2 \cc a_7^2 \cc a_6^2 \cc c_4^2 \cc a_5^2 \cc c_3 \cc a_5^{-1} \cc c_4^{-2} \cc a_6^{-2} \cc a_7^{-2} \cc a_8^{-2} \cc e^{-1},\\
&\gamma_{19}= a_1^{-2} \cc a_2^{-2} \cc b_1^{-2} \cc c_1^{-2}\cc a_3^{-2} \cc c_2^{-1} \cc a_3 \cc c_1^2 \cc b_1^2\cc a_2^2 \cc a_1 \cc e \cc a_8^2 \cc a_7^2 \cc a_6^2 \cc c_4 \cc a_6^{-1} \cc a_7^{-2} \cc a_8^{-2} \cc e^{-1},\\
&\gamma_{20} = a_1^{-1} \cc a_2^{-2} \cc b_1^{-2} \cc c_1^{-2}\cc a_3^{-1} \cc c_2 \cc a_3^2 \cc c_1^2 \cc b_1^2 \cc a_2^2 \cc a_1^2 \cc e \cc a_2^{-1} \cc b_1^{-1} \cc c_1 \cc b_1^2 \cc a_2^2\cc a_1 \cc e \cc a_8^2 \cc a_7 \cc b_1 \cc a_2 \cc e^{-1},\\
&\gamma_{21}= a_1^{-2} \cc a_2^{-2} \cc b_1^{-2} \cc c_1^{-2}\cc a_3^{-2} \cc c_2^{-1} \cc a_3 \cc c_1^2 \cc b_1^2\cc a_2^2 \cc a_1 \cc e \cc a_8^2 \cc a_7^2 \cc a_6^2 \cc c_4^2 \cc a_5 \cc c_4^{-1} \cc a_6^{-2} \cc a_7^{-2} \cc a_8^{-2} \cc e^{-1},\\
&\gamma_{22}= a_1^{-1} \cc a_2^{-2} \cc b_1^{-2} \cc c_1^{-1} \cc a_3 \cc c_1^2 \cc b_1^2 \cc a_2\cc a_1 \cc e  \cc a_8^2 \cc a_7^2 \cc a_6^2 \cc c_4^2 \cc a_5^2 \cc c_3\cc a_4^{-1} \cc c_3^{-2} \cc a_5^{-2} \cc c_4^{-2} \cc a_6^{-2} \cc a_7^{-2} \cc a_8^{-2} \cc e^{-1},\\
&\gamma_{23}= a_1^{-2} \cc a_2^{-2} \cc b_1^{-2} \cc c_1^{-2}\cc a_3^{-2} \cc c_2^{-1} \cc a_3 \cc c_1^2 \cc b_1^2\cc a_2^2 \cc a_1 \cc e \cc a_8^2 \cc a_7^2 \cc a_6^2 \cc c_4^2 \cc a_5 \cc c_3^{-1} \cc a_5^{-2} \cc c_4^{-2} \cc a_6^{-2} \cc a_7^{-2} \cc a_8^{-2} \cc e^{-1},\\
\end{aligned}
\end{equation*}
\begin{equation*} 
\begin{aligned}
&\gamma_{24}= a_1^{-2} \cc a_2^{-2} \cc b_1^{-2} \cc c_1^{-2}\cc a_3^{-2} \cc c_2^{-1} \cc a_3 \cc c_1^2 \cc b_1^2\cc a_2^2 \cc a_1 \cc e \cc a_8^2 \cc a_7^2 \cc a_6^2 \cc c_4^2 \cc a_5^2 \cc c_3 \cc a_4^{-1}  \cc c_3^{-2} \cc a_5^{-2} \cc c_4^{-2} \cc a_6^{-2} \cc a_7^{-2} \cc a_8^{-2} \cc e^{-1},\\
&\gamma_{25}= a_1^{-2} \cc a_2^{-2} \cc b_1^{-2} \cc c_1^{-2}\cc a_3^{-2} \cc c_2^{-1}  \cc a_4^{-1}  \cc c_2 \cc a_3^2 \cc c_1^2 \cc b_1^2\cc a_2^2 \cc a_1 \cc e \cc a_8^2 \cc a_7^2 \cc a_6^2 \cc c_4^2 \cc a_5^2 \cc c_3 \cc  a_5^{-1} \cc c_4^{-2} \cc a_6^{-2} \cc a_7^{-2} \cc a_8^{-2} \cc e^{-1},\\
&\gamma_{26} = a_1^{-2} \cc a_2^{-2} \cc b_1^{-2} \cc c_1^{-2}\cc a_3^{-2}\cc c_2^{-2} \cc a_4^{-1} \cc c_2 \cc a_3^2 \cc c_1^2 \cc b_1^2 \cc a_2^2 \cc a_1^2 \cc e  \cc b_1^{-1} \cc c_1^{-1} \cc a_3 \cc c_1^2 \cc b_1^2 \cc a_2^2\cc a_1 \cc e\\
& \phantom{ABCDE}  \cc a_8^2 \cc a_7^2 \cc a_6^2 \cc c_4^2 \cc a_5^2 \cc c_3 \cc a_5^{-1} \cc c_4^{-2} \cc a_6^{-2} \cc a_7^{-2} \cc a_8^{-2} \cc e^{-1},\\
&\gamma_{27} = a_1^{-2} \cc a_2^{-2} \cc b_1^{-2} \cc c_1^{-2}\cc a_3^{-2}\cc c_2^{-2} \cc a_4^{-1} \cc c_2 \cc a_3^2 \cc c_1^2 \cc b_1^2 \cc a_2^2 \cc a_1^2 \cc e\cc a_2^{-1}   \cc b_1^{-2} \cc c_1^{-1} \cc a_3 \cc c_1^2 \cc b_1^2 \cc a_2^2\cc a_1 \cc e\\
& \phantom{ABCDE}  \cc a_8^2 \cc a_7^2 \cc a_6^2 \cc c_4^2 \cc a_5^2 \cc c_3 \cc a_5^{-1} \cc c_4^{-2} \cc a_6^{-2} \cc a_7^{-2} \cc a_8^{-2} \cc e^{-1},\\
&\gamma_{28} = a_1^{-2} \cc a_2^{-2} \cc b_1^{-2} \cc c_1^{-2}\cc a_3^{-2}\cc c_2^{-2} \cc a_4^{-1} \cc c_2 \cc a_3^2 \cc c_1^2 \cc b_1^2 \cc a_2^2 \cc a_1^2 \cc e\cc a_2^{-1}   \cc b_1^{-2} \cc c_1^{-1} \cc a_3 \cc c_1^2 \cc b_1^2 \cc a_2^2\cc a_1 \cc e\\
& \phantom{ABCDE}  \cc a_8^2 \cc a_7^2 \cc a_6^2 \cc c_4^2 \cc a_5^2 \cc c_3\cc a_4^{-1} \cc c_3^{-2} \cc a_5^{-2} \cc c_4^{-2} \cc a_6^{-2} \cc a_7^{-2} \cc a_8^{-2} \cc e^{-1},\\
\end{aligned}
\end{equation*}
and
\begin{equation*}
\begin{aligned}
&\zeta_1 = a_1^{-1} \cc a_2^{-2} \cc c_1^{-2} \cc b_1, \qquad \quad \zeta_2 =  a_1^{-1} \cc a_2^{-2} \cc b_1^{-2} \cc c_1^{-2} \cc a_3^{-1} \cc c_1 \cc b_1 \cc c_1 \cc b_1^2 \cc c_1^{-2} \cc b_1^{-1},\\
&\zeta_3 = a_2^{-1} \cc b_1^{-1} \cc c_1 \cc b_1^2 \cc a_2^2 \cc a_1 \cc e \cc a_8^2 \cc a_7 \cc c_1 \cc b_1^2 \cc c_1^{-2} \cc b_1^{-1}, \quad \zeta_4 = a_2^{-1} \cc b_1^{-1} \cc c_1 \cc b_1^2 \cc a_2^2 \cc a_1 \cc e \cc a_8^2 \cc a_7^2 \cc a_6 \cc c_1 \cc b_1^2 \cc c_1^{-2} \cc b_1^{-1},\\
&\zeta_5= a_1^{-1} \cc a_2^{-2} \cc b_1^{-2} \cc c_1^{-2} \cc a_3^{-1} \cc c_1 \cc b_1^2 \cc a_2^2 \cc a_1 \cc e \cc a_8 \cc c_1 \cc b_1^2 \cc c_1^{-2} \cc b_1^{-1},\\
&\zeta_6= a_1^{-1} \cc a_2^{-2} \cc b_1^{-2} \cc c_1^{-1} \cc a_3 \cc c_1^2 \cc b_1^2 \cc a_2^2 \cc a_1 \cc a_2^{-2} \cc b_1^{-1} \cc c_1 \cc b_1^2 \cc c_1^{-2} \cc b_1^{-1},\\
&\zeta_7= a_1^{-1} \cc a_2^{-2} \cc b_1^{-2} \cc c_1^{-2} \cc a_3^{-1} \cc c_1 \cc b_1^2 \cc a_2^2 \cc a_1 \cc e \cc a_8^2 \cc a_7 \cc c_1 \cc b_1^2 \cc c_1^{-2} \cc b_1^{-1},\\
&\zeta_8= a_1^{-2} \cc a_2^{-2} \cc b_1^{-2} \cc c_1^{-2}\cc a_3^{-2} \cc c_2^{-1} \cc a_3 \cc c_1^2 \cc b_1^2 \cc a_2^2 \cc a_1 \cc a_2^{-2} \cc b_1^{-1} \cc c_1 \cc b_1^2 \cc c_1^{-2} \cc b_1^{-1},\\
&\zeta_9= a_1^{-2} \cc a_2^{-2} \cc b_1^{-2} \cc c_1^{-2}\cc a_3^{-2} \cc c_2^{-2}\cc a_4^{-1} \cc c_2 \cc a_3^2 \cc c_1^2 \cc b_1^2 \cc a_2^2 \cc a_1 \cc a_2^{-2} \cc b_1^{-1} \cc c_1 \cc b_1^2 \cc c_1^{-2} \cc b_1^{-1},\\
&\zeta_{10}= a_1^{-2} \cc a_2^{-2} \cc b_1^{-2} \cc c_1^{-2}\cc a_3^{-2} \cc c_2^{-2}\cc a_4^{-1} \cc c_2 \cc a_3^2 \cc c_1^2 \cc b_1^2 \cc a_2^2\cc a_1^2 \cc e \cc b_1^{-1} \cc c_1^{-1} \cc a_3 \cc c_1^2 \cc b_1^2 \cc a_2^2 \cc a_1 \cc a_2^{-2} \cc b_1^{-1} \cc c_1 \cc b_1^2 \cc c_1^{-2} \cc b_1^{-1},\\
&\zeta_{11}= a_1^{-2} \cc a_2^{-2} \cc b_1^{-2} \cc c_1^{-2}\cc a_3^{-2} \cc c_2^{-2}\cc a_4^{-1} \cc c_2 \cc a_3^2 \cc c_1^2 \cc b_1^2 \cc a_2^2\cc a_1^2 \cc e \cc a_2^{-1}\cc b_1^{-2} \cc c_1^{-1} \cc a_3 \cc c_1^2 \cc b_1^2 \cc a_2^2 \cc a_1 \cc a_2^{-2} \cc b_1^{-1} \cc c_1 \cc b_1^2 \cc c_1^{-2} \cc b_1^{-1},\\
&\mu_1 = b_1 \cc c_1^2 \cc b_1^{-2} \cc c_1^{-1} \cc b_1 \cc a_2 \cc e^{-1},\\
&\mu_2 = b_1 \cc c_1^2 \cc b_1^{-2} \cc c_1^{-2}\cc a_3^{-1} \cc c_2 \cc a_3^2 \cc c_1^2 \cc b_1^2 \cc a_2^2 \cc a_1^2 \cc e \cc a_2^{-1} \cc b_1^{-1} \cc a_8^{-1} \cc e^{-1},\\
&\mu_3 = b_1 \cc c_1^2 \cc b_1^{-2} \cc c_1^{-2}\cc a_3^{-1} \cc c_2 \cc a_3^2 \cc c_1^2 \cc b_1^2 \cc a_2^2 \cc a_1^2 \cc e \cc a_2^{-1} \cc b_1^{-1}\cc a_7^{-1} \cc a_8^{-2} \cc e^{-1},\\
&\mu_4= b_1 \cc c_1^2 \cc b_1^{-2} \cc c_1^{-1} \cc b_1 \cc a_2^2 \cc e \cc a_8^2 \cc a_7^2 \cc a_6^2 \cc c_4^2 \cc a_5 \cc c_4^{-1} \cc a_6^{-2} \cc a_7^{-2} \cc a_8^{-2} \cc e^{-1},\\
&\mu_5 = b_1 \cc c_1^2 \cc b_1^{-2} \cc c_1^{-2}\cc a_3^{-1} \cc c_2 \cc a_3^2 \cc c_1^2 \cc b_1^2 \cc a_2^2 \cc a_1^2 \cc e \cc a_2^{-1} \cc b_1^{-1}\cc a_6^{-1} \cc a_7^{-2} \cc a_8^{-2} \cc e^{-1},\\
&\mu_6= b_1 \cc c_1^2 \cc b_1^{-2} \cc c_1^{-1} \cc b_1 \cc a_2 \cc e \cc a_8^2 \cc a_7^2 \cc a_6^2 \cc c_4^2 \cc a_5^2 \cc c_3 \cc a_5^{-1} \cc c_4^{-2} \cc a_6^{-2} \cc a_7^{-2} \cc a_8^{-2} \cc e^{-1},\\
&\mu_7= b_1 \cc c_1^2 \cc b_1^{-2} \cc c_1^{-1} \cc b_1 \cc a_2^2 \cc e \cc a_8^2 \cc a_7^2 \cc a_6^2 \cc c_4^2 \cc a_5^2 \cc c_3 \cc a_5^{-1} \cc c_4^{-2} \cc a_6^{-2} \cc a_7^{-2} \cc a_8^{-2} \cc e^{-1},\\
&\mu_8= b_1 \cc c_1^2 \cc b_1^{-2} \cc c_1^{-1} \cc b_1 \cc a_2^2 \cc e \cc a_8^2 \cc a_7^2 \cc a_6^2 \cc c_4^2 \cc a_5 \cc c_3^{-1} \cc a_5^{-2} \cc c_4^{-2} \cc a_6^{-2} \cc a_7^{-2} \cc a_8^{-2} \cc e^{-1},\\
&\mu_9= b_1 \cc c_1^2 \cc b_1^{-2} \cc c_1^{-1} \cc b_1 \cc a_2^2 \cc e \cc a_8^2 \cc a_7^2 \cc a_6^2 \cc c_4^2 \cc a_5^2\cc c_3 \cc a_4^{-1} \cc c_3^{-2} \cc a_5^{-2} \cc c_4^{-2} \cc a_6^{-2} \cc a_7^{-2} \cc a_8^{-2} \cc e^{-1},\\
&\eta_1 = b_1 \cc c_1^2 \cc b_1^{-2} \cc c_1^{-2} \cc a_3^{-1} \cc c_2 \cc a_3^2 \cc c_1^2 \cc b_1^2 \cc a_2^2 \cc a_1^2 \cc e \cc a_2^{-1} \cc b_1^{-1} \cc c_1 \cc b_1^2 \cc a_2^2 \cc a_1 \cc e \cc a_8^2 \cc a_7 \cc c_1 \cc b_1^2 \cc c_1^{-2} \cc b_1^{-1}.
\end{aligned}
\end{equation*}
The admissible triples are 
\begin{equation*}
\begin{aligned}
A=& \{ (1,13,3),(1,13,6),(1,13,7),(1,13,11),(1,17,22),(1,19,1),(1,23,8),(1,23,10),\\
&(1,25,16),(1,25,20),(1,26,16),(1,27,16),(1,28,5),(1,28,15),(2,12,8),(2,12,15),(3,12,15),\\
&(4,21,1),(5,21,1),(5,23,2),(5,23,3),(5,23,10),(6,12,15),(6,21,1),(6,24,4),(6,24,9),\\
&(7,21,1),(8,1,13),(9,1,25),(9,1,26),(9,18,20),(9,20,21),(10,8,1),(10,15,14),(10,15,16),\\
&(11,22,5),(12,8,1),(12,15,14),(12,15,16),(13,3,12),(13,6,12),(13,6,21),(13,6,24),\\
&(13,7,21),(13,11,22),(14,1,13),(14,1,19),(14,1,23),(14,22,5),(14,22,9),(15,11,22),\\
&(15,14,1),(15,14,22),(15,16,1),(16,1,13),(17,22,5),(18,20,21),(19,1,13),(20,21,1),\\
\end{aligned}
\end{equation*}
\begin{equation*} 
\begin{aligned}
&(21,1,13),(21,1,17),(21,1,26),(21,1,27),(21,1,28),(22,5,21),(22,5,23),(22,9,1),(22,9,18),\\
&(23,2,12),(23,3,12),(23,8,1),(23,10,8),(24,4,21),(24,9,18),(24,9,20),(25,16,1),(25,20,21),\\
&(26,16,1),(27,16,1),(28,5,21),(28,15,11),(28,15,14)\}
\end{aligned}
\end{equation*}
Let $V= Span \, \Gamma$ and $W= Span\, \Delta$. The splitting map $S: V\to W$ is given by 
\begin{equation*}
S: \left\{ \begin{aligned} &\gamma_1 \mapsto \mu_1 + \zeta_8, \quad \gamma_2 \mapsto \mu_1 + \zeta_9, \quad \gamma_3 \mapsto \mu_1 + \zeta_{10}, \quad \gamma_4 \mapsto \mu_2 + \zeta_6, \quad \gamma_5\mapsto \mu_3 + \zeta_6, \quad \gamma_6 \mapsto \mu_1 + \zeta_{11},\\
&\gamma_7 \mapsto \mu_1 + \gamma_{17} + \zeta_6,\quad \gamma_8 \to \eta_1,\quad \gamma_9 \mapsto \mu_3 + \zeta_1,\quad \gamma_{10} \mapsto \mu_6+ \zeta_1,\quad \gamma_{11} \mapsto \mu_1 + \gamma_{13} + \zeta_2,\\
&\gamma_{12} \mapsto \mu_7+ \zeta_1,\quad \gamma_{13} \mapsto \mu_4 + \zeta_1,\quad \gamma_{14} \mapsto \mu_1+ \gamma_{13} + \zeta_5,\quad \gamma_{15} \mapsto \mu_5 + \zeta_1,\quad \gamma_{16} \mapsto \mu_1 + \gamma_{13} + \zeta_7,\\
&\gamma_{17} \mapsto \mu_8 + \zeta_2,\quad \gamma_{18} \mapsto \mu_1 + \gamma_{19} + \zeta_1,\quad \gamma_{19} \mapsto \mu_9 + \zeta_3,\quad \gamma_{20} \mapsto \mu_1 + \gamma_{13} + \zeta_1 + \mu_2 + \zeta_6,\\
&\gamma_{21} \mapsto \mu_9 + \zeta_4,\quad \gamma_{22} \mapsto \mu_1 + \gamma_{12} + \gamma_{15} + \zeta_1,\quad \gamma_{23} \mapsto \mu_9 + \gamma_9 + \zeta_1,\quad \gamma_{24} \mapsto \mu_9 + \gamma_{15} + \zeta_1,\\
&\gamma_{25} \mapsto \mu_8 + \gamma_8 + \zeta_1,\quad \gamma_{26} \mapsto \mu_8 + \gamma_{10} + \gamma_8 + \zeta_1,\quad \gamma_{27} \mapsto \mu_8 + \gamma_2 + \gamma_{12} + \gamma_8 + \zeta_1,\\
&\gamma_{28} \mapsto \mu_8 + \gamma_2 + \gamma_{12} + \gamma_{15} + \zeta_1.
 \end{aligned}\right.
\end{equation*}
and the merging map $M:V \to W$ is
\begin{equation*}
M: \left\{ \begin{aligned} &\gamma_1 \mapsto \zeta_1 + \mu_1,\quad \gamma_2 \mapsto \gamma_2,\quad \gamma_3 \mapsto \zeta_2 + \mu_1,\quad \gamma_4 \mapsto \zeta_3 + \mu_1,\quad \gamma_5 \mapsto \zeta_4 + \mu_1,\quad \gamma_6 \mapsto \zeta_5 + \mu_1,\\
&\gamma_7  \mapsto \zeta_7 + \mu_1,\quad \gamma_8 \mapsto \gamma_8,\quad \gamma_9 \mapsto \gamma_9,\quad \gamma_{10} \mapsto \gamma_{10},\quad \gamma_{11} \mapsto \zeta_1+ \mu_2,\quad \gamma_{12} \mapsto \gamma_{12},\\
&\gamma_{13} \mapsto \gamma_{13},\quad \gamma_{14} \mapsto \zeta_1 + \mu_3,\quad \gamma_{15} \mapsto \gamma_{15},\quad \gamma_{16} \mapsto \zeta_1+ \mu_5,\quad \gamma_{17} \mapsto \gamma_{17},\quad \gamma_{18} \mapsto \zeta_1 + \mu_6,\\
&\gamma_{19} \mapsto \gamma_{19},\quad \gamma_{20} \mapsto \zeta_1 + \eta_1 + \mu_1,\quad \gamma_{21} \mapsto \zeta_8+ \mu_4,\quad \gamma_{22} \mapsto \zeta_6+ \mu_9,\quad \gamma_{23} \mapsto \zeta_8+ \mu_8,\\
&\gamma_{24} \mapsto \zeta_8+ \mu_9,\quad \gamma_{25} \mapsto \zeta_9 + \mu_7,\quad \gamma_{26} \mapsto \zeta_{10} + \mu_7,\quad \gamma_{27} \mapsto \zeta_{11} + \mu_7,\quad \gamma_{28} \mapsto \zeta_{11}+ \mu_9.
\end{aligned}\right.
\end{equation*}
There is a $v_\lambda \in V$ such that $Sv_\lambda = \lambda Mv_\lambda$ where $\lambda\approx 1.45799$ is the largest real root of $\chi(t) =t^8-t^6-t^5-t^3-t^2+1$
and
\begin{equation*}
\begin{aligned}
 v_\lambda \approx ( &16.455,1.527,2.226,0.338,3.800,3.246,2.607,2.607,3.099  \\
 &1.047,1.458,5.541,8.079,4.732,5.541,3.800,1.788,0.718\\
 &0.493,1.788,5.541,5.852,4.519,1.226,1.047,1.527,1.226,1 )\\
 \end{aligned}
 \end{equation*}

\subsection{Orbit data $1,4,5$ with a cyclic permutation}\label{B:145}
Let $f_\mathbf{R}: X(\mathbf{R}) \to X(\mathbf{R})$ be a real diffeomorphism associated with the orbit data $1,4,5$ and a cyclic permutation. 
Let \[\Gamma = \{ \gamma_1, \dots, \gamma_{21} \},\ \Delta = \{ \gamma_1, \gamma_2,\gamma_3,\gamma_4, \gamma_{12},\gamma_{13}, \zeta_1,\zeta_2,\zeta_3, \mu_1, \dots, \mu_{11}, \eta_1\} \]
be two ordered sets of non-cyclic words in $\pi_1( X(\mathbf{R}),P_{fix})$ where 
\begin{equation*}
\begin{aligned}
& \gamma_1 = a_2^{-1} \cc c_1^{-1} \cc c_4^{-1} \cc e^{-1}, \qquad \qquad \gamma_2 = a_2^{-1} \cc c_1^{-1} \cc b_1\cc c_1^2 \cc a_2^2 \cc a_1 \cc e \cc c_4\cc a_5^{-1} \cc c_4^{-2} \cc e^{-1},  \\
&\gamma_3 = c_1^{-1} \cc b_1^{-1} \cc a_3 \cc b_1^2 \cc c_1^2 \cc a_2^2 \cc a_1 \cc e \cc c_4^2 \cc a_5 \cc c_4^{-1} \cc e^{-1}, \quad \gamma_4 = c_1 \cc a_2 \cc e \cc c_4^2 \cc a_5^2 \cc c_3^2 \cc a_4 \cc c_3^{-1} \cc a_5^{-2} \cc c_4^{-2} \cc e^{-1}, \\
&\gamma_5 = a_1^{-1} \cc a_2^{-2} \cc c_1^{-2} \cc b_1^{-2} \cc a_3^{-2} \cc c_2^{-1} \cc a_4^{-1} \cc c_3^{-2} \cc a_5^{-2} \cc c_4^{-2} \cc e^{-1}, \quad \gamma_6=a_1^{-1} \cc a_2^{-2} \cc c_1^{-2} \cc b_1^{-2} \cc a_3^{-1} \cc c_2^{-1} \cc a_4^{-2} \cc c_3^{-2} \cc a_5^{-2} \cc c_4^{-2} \cc e^{-1}\\
&\gamma_7 = a_1^{-1} \cc a_2^{-2} \cc c_1^{-2} \cc b_1^{-1} \cc c_1 \cc a_2 \cc e \cc c_4^2 \cc a_5^2 \cc c_3 \cc a_5^{-1} \cc c_4^{-2} \cc e^{-1}, \quad \gamma_8 = a_2^{-1} \cc c_1^{-1} \cc b_1 \cc c_1^2 \cc a_2 \cc e \cc c_4^2 \cc a_5^2 \cc c_3^2 \cc a_4 \cc c_3^{-1} \cc a_5^{-2} \cc c_4^{-2} \cc e^{-1}, \\
&\gamma_9 = a_2^{-1} \cc c_1^{-2} \cc b_1^{-1} \cc c_1 \cc a_2 \cc e \cc c_4^2 \cc a_5^2 \cc c_3^2 \cc a_4 \cc c_3^{-1} \cc a_5^{-2} \cc c_4^{-2} \cc e^{-1}, \\
&\gamma_{10} = a_1^{-1} \cc a_2^{-2} \cc c_1^{-2} \cc b_1^{-1} \cc c_1 \cc a_2 \cc e \cc c_4^2 \cc a_5^2 \cc c_3 \cc a_4^{-1} \cc c_3^{-2} \cc a_5^{-2} \cc c_4^{-2} \cc e^{-1},\\
&\gamma_{11} = a_1^{-1} \cc a_2^{-2} \cc c_1^{-2} \cc b_1^{-2} \cc a_3^{-2} \cc c_2^{-1} \cc a_3 \cc b_1^2 \cc c_1^2 \cc a_2^2 \cc a_1 \cc e \cc c_4^2 \cc a_5 \cc c_4^{-1} \cc e^{-1},\\
&\gamma_{12} = c_1^{-1} \cc b_1^{-1} \cc a_3 \cc b_1^2 \cc c_1^2 \cc a_2^2 \cc a_1 \cc e \cc c_4^2 \cc a_5^2 \cc c_3 \cc a_4^{-1} \cc c_3^{-2} \cc a_5^{-2} \cc c_4^{-2} \cc e^{-1},\\
&\gamma_{13} = a_2^{-1} \cc c_1^{-1} \cc e \cc c_4^2 \cc a_5^2 \cc c_3^2 \cc a_4^2 \cc c_2 \cc a_3 \cc b_1^2 \cc c_1^2 \cc a_2^2 \cc a_1 \cc e \cc c_4^2 \cc a_5 \cc c_4^{-1} \cc e^{-1},\\
\end{aligned}
\end{equation*}
\begin{equation*} 
\begin{aligned}
&\gamma_{14} = a_2^{-1} \cc c_1^{-1} \cc b_1 \cc c_1^2 \cc a_2 \cc e \cc c_4^2 \cc a_5^2 \cc c_3^2 \cc a_4^2 \cc c_2 \cc a_3 \cc b_1^2 \cc c_1^2 \cc a_2^2 \cc a_1 \cc e \cc c_4 \cc a_5^{-1} \cc c_4^{-2} \cc e^{-1},\\
&\gamma_{15} = a_1^{-1} \cc a_2^{-2} \cc c_1^{-2} \cc b_1^{-2} \cc a_3^{-1} \cc c_2 \cc a_3^2 \cc b_1^2 \cc c_1^2 \cc a_2^2 \cc a_1 \cc e \cc c_4^2 \cc a_5^2 \cc c_3 \cc a_4^{-1} \cc c_3^{-2} \cc a_5^{-2} \cc c_4^{-2} \cc e^{-1},\\
&\gamma_{16} = a_1^{-1} \cc a_2^{-2} \cc c_1^{-2} \cc b_1^{-2} \cc a_3^{-1} \cc b_1 \cc c_1 \cc e \cc c_4^2 \cc a_5^2 \cc c_3^2 \cc a_4^2 \cc c_2 \cc a_3 \cc b_1^2 \cc c_1^2 \cc a_2^2 \cc a_1 \cc e \cc c_4^2 \cc a_5 \cc c_4^{-1} \cc e^{-1},\\
&\gamma_{17} = a_2^{-1} \cc c_1^{-1} \cc b_1 \cc c_1^2 \cc a_2 \cc e \cc c_4^2 \cc a_5^2 \cc c_3^2 \cc a_4^2 \cc c_2 \cc a_3 \cc b_1^2 \cc c_1^2 \cc a_2^2 \cc a_1 \cc e \cc c_4^2 \cc a_5 \cc c_3^{-1} \cc a_5^{-2} \cc c_4^{-2} \cc e^{-1},\\
&\gamma_{18} =  a_1^{-1} \cc a_2^{-2} \cc c_1^{-2} \cc b_1^{-2} \cc a_3^{-1} \cc b_1 \cc c_1 \cc e \cc c_4^2 \cc a_5^2 \cc c_3^2 \cc a_4^2 \cc c_2 \cc a_3 \cc b_1^2 \cc c_1^2 \cc a_2^2 \cc a_1 \cc e \cc c_4^2 \cc a_5\cc c_3^{-1} \cc a_5^{-2}  \cc c_4^{-2} \cc e^{-1},\\
&\gamma_{19} = a_2^{-1} \cc c_1^{-2} \cc b_1^{-1}  \cc c_1 \cc a_2 \cc e \cc c_4^2 \cc a_5^2 \cc c_3^2 \cc a_4 \cc c_2 \cc a_3^2 \cc b_1^2 \cc c_1^2 \cc a_2^2 \cc a_1 \cc e \cc c_4^2 \cc a_5^2 \cc c_3 \cc a_4^{-1} \cc c_3^{-2} \cc a_5^{-2}  \cc c_4^{-2} \cc e^{-1},\\
&\gamma_{20} = a_1^{-1} \cc a_2^{-2} \cc c_1^{-2} \cc b_1^{-1} \cc c_1 \cc a_2 \cc e \cc c_4^2 \cc a_5^2 \cc c_3^2 \cc a_4 \cc c_2 \cc a_3^2 \cc b_1^2 \cc c_1^2 \cc a_2^2 \cc a_1 \cc e \cc c_4^2 \cc a_5^2 \cc c_3 \cc a_4^{-1} \cc c_3^{-2} \cc a_5^{-2} \cc c_4^{-2} \cc e^{-1},\\
&\gamma_{21} = a_2^{-1}\cc c_1^{-1} \cc b_1 \cc c_1^2 \cc a_2 \cc e \cc c_4^2 \cc a_5^2 \cc c_3^2 \cc a_4^2 \cc c_2 \cc a_3 \cc b_1^2 \cc c_1^2 \cc a_2^2 \cc a_1 \cc e \cc c_4 \cc c_1 \cc a_2 \cc e \cc c_4^2 \cc a_5^2 \cc c_3^2 \cc a_4 \cc c_3^{-1} \cc a_5^{-2} \cc c_4^{-2} \cc e^{-1}
\end{aligned}
\end{equation*}
and
\begin{equation*}
\begin{aligned}
&\zeta_1 = a_2^{-1} \cc c_1 \cc b_1 \cc c_1, \quad \zeta_2 = a_2^{-1} \cc c_1^{-2} \cc b_1^{-1}, \quad \zeta_3 = a_1^{-1} \cc a_2^{-2} \cc c_1^{-2} \cc b_1^{-1},\\
&\mu_1 = b_1^{-1}\cc a_3^{-2}\cc c_2^{-1} \cc a_4^{-1} \cc c_3^{-2} \cc a_5^{-2} \cc c_4^{-2} \cc e^{-1}, \quad \mu_2 = b_1^{-1}\cc a_3^{-1}\cc c_2^{-1} \cc a_4^{-2} \cc c_3^{-2} \cc a_5^{-2} \cc c_4^{-2} \cc e^{-1},\\
&\mu_3 = c_1 \cc a_2 \cc e \cc c_4^2\cc a_5^2 \cc c_3 \cc a_5^{-1} \cc c_4^{-2} \cc e^{-1}, \quad \mu_4 = c_1\cc a_2 \cc e \cc c_4^2 \cc a_5^2 \cc c_3 \cc a_4^{-1} \cc c_3^{-2} \cc a_5^{-2} \cc c_4^{-2} \cc e^{-1},\\
&\mu_5= b_1^{-1} \cc a_3^{-2} \cc c_2^{-1} \cc a_3 \cc b_1^2 \cc c_1^2 \cc a_2^2\cc a_1\cc e \cc c_4^2 \cc a_5 \cc c_4^{-1} \cc e^{-1},\\
& \mu_6= b_1^{-1} \cc a_3^{-1} \cc c_2 \cc a_3^2 \cc b_1^2 \cc c_1^2 \cc a_2^2 \cc a_1 \cc  e\cc c_4^2 \cc a_5^2 \cc c_3 \cc a_4^{-1} \cc c_3^{-2} \cc a_5^{-2} \cc c_4^{-2} \cc e^{-1}, \\
&\mu_7=c_1 \cc a_2 \cc e \cc c_4^2 \cc a_5^2 \cc c_3^2 \cc a_4^2 \cc c_2 \cc a_3 \cc b_1^2 \cc c_1^2 \cc a_2^2 \cc a_1 \cc e \cc c_4 \cc a_5^{-1} \cc c_4^{-2} \cc e^{-1},\\
&\mu_8= b_1^{-1} \cc a_3^{-1}\cc b_1 \cc c_1  \cc  e\cc c_4^2 \cc a_5^2 \cc c_3^2 \cc a_4^2 \cc c_2 \cc a_3 \cc b_1^2 \cc c_1^2 \cc a_2^2 \cc a_1 \cc e \cc c_4^2 \cc a_5 \cc c_4^{-1} \cc e^{-1}, \\
&\mu_9 = c_1 \cc a_2 \cc e \cc c_4^2 \cc a_5^2 \cc c_3^2 \cc a_4^2 \cc c_2 \cc a_3 \cc b_1^2 \cc c_1^2 \cc a_2^2 \cc a_1 \cc e \cc c_4^2 \cc a_5 \cc c_3^{-1} \cc a_5^{-2} \cc c_4^{-2} \cc e^{-1},\\
&\mu_{10} = b_1^{-1} \cc a_3^{-1}\cc b_1 \cc c_1  \cc  e\cc c_4^2 \cc a_5^2 \cc c_3^2 \cc a_4^2 \cc c_2 \cc a_3 \cc b_1^2 \cc c_1^2 \cc a_2^2 \cc a_1 \cc e \cc c_4^2 \cc a_5\cc c_3^{-1} \cc a_5^{-2}  \cc c_4^{-2} \cc e^{-1}, \\
&\mu_{11}= c_1 \cc a_2 \cc e \cc c_4^2 \cc a_5^2 \cc c_3^2 \cc a_4 \cc c_2 \cc a_3^2 \cc b_1^2 \cc c_1^2 \cc a_2^2 \cc a_1 \cc e \cc c_4^2 \cc a_5^2 \cc c_3 \cc a_4^{-1} \cc c_3^{-2} \cc a_5^{-2} \cc c_4^{-2} \cc e^{-1},\\
&\eta_1= c_1 \cc a_2 \cc e \cc c_4^2 \cc a_5^2 \cc c_3^2 \cc a_4^2 \cc c_2 \cc a_3 \cc b_1^2 \cc c_1^2 \cc a_2^2 \cc a_1\cc e \cc c_4.
\end{aligned}
\end{equation*}
The admissible quadruples are 
\begin{equation*}
\begin{aligned}
A= &\{ (1,6,9,5),(1,6,9,11),(2,6,3,20),(2,6,4,16),(2,6,12,1),(2,6,12,9),(2,6,12,13),\\
&(2,15,17,2),(2,15,21,16),(2,16,7,6),( 3,20,21,16),(4,16,6,9),(4,16,7,6),(4,16,10,17),\\
&(5,2,6,4),(6,2,15,17),(5,2,15,21),(5,2,16,7),(5,17,2,6),(6,3,20,21),(6,4,16,7),\\
&(6,4,16,10),(6,9,2,6),(6,9,5,2),(6,9,5,17),(6,9,11,20),(6,12,1,6),(6,12,9,5),\\
&(6,12,13,7),(6,12,13,20),(6,19,21,16),(7,6,9,2),(7,6,9,11),(7,6,12,1),(7,6,9,21),\\
&(8,16,7,6),(9,2,6,12),(9,5,2,6),(9,5,2,15),(9,5,2,16),(9,5,17,2),(9,11,20,14),\\
&(9,11,20,21),(10.17,2,6),(11,20,14,6),(11,20,21,5),(11,20,21,16),(11,20,21,18),\\
&(12,1,6,9),(12,9,5,2),(12,13,7,6),(12,13,20,1),(12,13,20,8),(12,13,20,14),(13,7,6,9),\\
&(13,20,1,6),(13,20,8,16),(13,20,14,6),(14,6,12,1),(15,17,2,6),(15,21,16,10),\\
&(16,6,9,5),(16,7,6,9),(16,7,6,12),(16,7,6,19),(16,19,17,2),(17,2,6,3),(17,2,6,4),\\
&(17,2,6,12),(18,2,6,12),(19,21,5,2),(19,21,16,6),(19,21,16,10),(19,21,18,2),\\
&(20,1,6,9),(20,8,16,7),(20,14,6,12),(20,21,5,2),(20,21,16,6),(20,21,16,10),(20,21,18,2),\\
&(21,5,2,6),(21,16,6,9),(21,16,10,17),(21,18,2,6)\}
\end{aligned}
\end{equation*}
Let $V= Span \, \Gamma$ and $W= Span\, \Delta$. The splitting map $S: V\to W$ is given by 
\begin{equation*}
S: \left\{ \begin{aligned} & \gamma_1 \mapsto \mu_2, \quad \gamma_2 \mapsto \mu_2 + \zeta_2, \quad \gamma_3 \mapsto \mu_1 + \zeta_1, \quad \gamma_4 \mapsto \mu_5 + \zeta_3, \quad \gamma_5 \mapsto \mu_3 + \zeta_3, \quad \gamma_6 \mapsto \gamma_4 + \zeta_3 \\
& \gamma_7 \mapsto \eta_1, \quad \gamma_8 \mapsto \mu_2 + \gamma_3 + \zeta_3, \quad \gamma_9 \mapsto \mu_8 + \zeta_3, \quad \gamma_{10} \mapsto \mu_7 + \zeta_3, \quad \gamma_{11} \mapsto \mu_4 + \zeta_1, \\
& \gamma_{12} \mapsto \mu_1 + \gamma_2 + \zeta_3, \quad \gamma_{13} \mapsto \mu_6 + \zeta_1, \quad \gamma_{14} \mapsto \mu_2 + \gamma_{12} + \zeta_2, \quad \gamma_{15} \mapsto \gamma_4 + \gamma_2 + \zeta_3, \\
&\gamma_{16} \mapsto \mu_{11} + \zeta_1,\quad \gamma_{17} \mapsto \mu_2 + \gamma_{12} + \gamma_1+ \zeta_3, \quad \gamma_{18} \mapsto \mu_{11} + \gamma_1 + \zeta_3, \quad \gamma_{19} \mapsto \mu_{10} + \gamma_2 + \zeta_3\\
&\gamma_{20} \mapsto \mu_9 + \gamma_2 + \zeta_3, \quad \gamma_{21} \mapsto \mu_2 + \gamma_{12} + \gamma_{13} + \zeta_3 \end{aligned}\right.
\end{equation*}
and the merging map $M:V \to W$ is
\begin{equation*}
M: \left\{ \begin{aligned} &\gamma_1 \mapsto \gamma_1,\quad \gamma_2 \mapsto \gamma_2, \quad \gamma_3 \mapsto \gamma_3,, \quad \gamma_4 \mapsto \gamma_4, \quad \gamma_5 \mapsto \zeta_3 + \mu_1, \quad \gamma_6 \mapsto \zeta_3 + \mu_2\\
&\gamma_7 \mapsto \zeta_3 + \mu_3, \quad \gamma_8 \mapsto \zeta_1 + \gamma_4, \quad \gamma_9 \mapsto \zeta_2 + \gamma_4, \quad \gamma_{10} \mapsto \zeta_3 + \mu_4, \quad \gamma_{11} \mapsto \zeta_3+ \mu_5,\\
&\gamma_{12} \mapsto \gamma_{12}, \quad \gamma_{13} \mapsto \gamma_{13}, \quad \gamma_{14} \mapsto \zeta_1 + \mu_7, \quad \gamma_{15} \mapsto \zeta_3+ \mu_6, \quad \gamma_{16} \mapsto \zeta_3+ \mu_8,\\
&\gamma_{17} \mapsto \zeta_1 + \mu_9, \quad \gamma_{18} \mapsto \zeta_3 + \mu_{10}, \quad \gamma_{19} \mapsto \zeta_2 + \mu_{11}, \quad \gamma_{20} \mapsto \zeta_3 + \mu_{11}, \quad \gamma_{21} \mapsto \zeta_1+ \eta_1 + \gamma_4
\end{aligned}\right.
\end{equation*}
There is a $v_\lambda \in V$ such that $Sv_\lambda = \lambda Mv_\lambda$ where $\lambda\approx 1.29349$ is the largest real root of $\chi(t) =t^{10} -t^8-t^7+t^5-t^3-t^2+1$
and 
\begin{equation*}
\begin{aligned}
 v_\lambda \approx ( &1,3.263,0.214,0.773,1.673,5.460,1.293,0.276,2.634,0.462,0.598,  \\
 &1.951,0.773,0.357,0.598,2.037,1.166,0.128,0.165,1.501,1 )\\
 \end{aligned}
 \end{equation*}

\subsection{Orbitdata $1,5,6$ with a cyclic permutation}\label{B:156}
Let $f_\mathbf{R}: X(\mathbf{R}) \to X(\mathbf{R})$ be a real diffeomorphism associated with the orbit data $1,5,6$ and a cyclic permutation. 
Let \[\Gamma = \{ \gamma_1, \dots, \gamma_{26} \},\ \Delta = \{ \gamma_2, \gamma_4,\gamma_5,\gamma_{6}, \gamma_{7},\gamma_{8}, \gamma_{15}, \gamma_{18}, \gamma_{19}, \gamma_{21},\zeta_1,\dots,\zeta_{7}, \mu_1, \dots, \mu_{8}, \eta_1,\eta_2,\eta_3\} \]
be two ordered sets of non-cyclic words in $\pi_1( X(\mathbf{R}),P_{fix})$ where
\begin{equation*}
\begin{aligned}
& \gamma_1 = a_1^{-1} \cc a_2^{-2} \cc c_1^{-2} \cc b_1^{-1} \cc c_1 \cc a_2 \cc e^{-1}, \qquad \quad \gamma_2 = a_1^{-1} \cc a_2^{-2} \cc c_1^{-2} \cc b_1^{-2} \cc a_3^{-1} \cc b_1\cc c_1\cc e^{-1},\\
&\gamma_3 = a_2^{-1} \cc c_1^{-1} \cc b_1 \cc c_1^2 \cc a_2^2 \cc a_1 \cc e \cc c_5 \cc c_1 \cc a_2 \cc e^{-1}, \qquad \gamma_4= a_2^{-1} \cc c_1^{-1} \cc b_1 \cc c_1^2 \cc a_2^2 \cc a_1 \cc e \cc c_5 \cc a_6^{-1} \cc c_5^{-2}  \cc e^{-1},\\
&\gamma_5= a_2^{-1} \cc c_1^{-1} \cc b_1 \cc c_1^2 \cc a_2^2 \cc a_1 \cc e \cc c_5^2 \cc a_6 \cc c_5^{-1} \cc e^{-1}, \qquad \gamma_6= a_2^{-1} \cc c_1^{-1} \cc b_1 \cc c_1^2 \cc a_2^2 \cc a_1 \cc e \cc c_5^2 \cc a_6 \cc c_4^{-1} \cc a_6^{-2} \cc c_5^{-2} \cc e^{-1},\\
&\gamma_7 = a_1^{-1} \cc a_2^{-2} \cc c_1^{-2} \cc b_1^{-2} \cc a_3^{-2} \cc c_2^{-1} \cc a_4^{-1} \cc c_3^{-2} \cc a_5^{-2} \cc c_4^{-2} \cc a_6^{-2} \cc c_5^{-2} \cc e^{-1},\\
&\gamma_8 = a_2^{-1} \cc c_1^{-1} \cc b_1 \cc c_1^2 \cc a_2^2 \cc a_1 \cc e \cc c_5^2 \cc a_6^2 \cc c_4 \cc a_5^{-1}  \cc c_4^{-2} \cc a_6^{-2} \cc c_5^{-2} \cc e^{-1},\\
&\gamma_9 = a_1^{-1} \cc a_2^{-2} \cc c_1^{-2} \cc b_1^{-2} \cc a_3^{-1} \cc c_2 \cc a_3^2 \cc b_1^2 \cc c_1^2 \cc a_2^2 \cc a_1^2 \cc e \cc a_2^{-1} \cc c_1^{-1} \cc e^{-1},\\
&\gamma_{10} = a_1^{-1} \cc a_2^{-2} \cc c_1^{-2} \cc b_1^{-2} \cc a_3^{-1} \cc c_2 \cc a_3^2 \cc b_1^2 \cc c_1^2 \cc a_2^2 \cc a_1^2 \cc e c_1\cc a_2 \cc e^{-1},\\
&\gamma_{11} = a_1^{-1} \cc a_2^{-2} \cc c_1^{-2} \cc b_1^{-2} \cc a_3^{-1} \cc c_2 \cc a_3^2 \cc b_1^2 \cc c_1^2 \cc a_2^2 \cc a_1^2 \cc e \cc a_2^{-1} \cc c_1^{-1} \cc c_5^{-1} \cc e^{-1},\\
&\gamma_{12} = a_1^{-1} \cc a_2^{-2} \cc c_1^{-2} \cc b_1^{-1} \cc c_1 \cc a_2 \cc e \cc c_5^2 \cc a_6^2 \cc c_4^2 \cc a_5^2 \cc c_3 \cc a_5^{-1} \cc c_4^{-2} \cc a_6^{-2} \cc c_5^{-2} \cc e^{-1},\\
&\gamma_{13} = a_1^{-1} \cc a_2^{-2} \cc c_1^{-2} \cc b_1^{-2} \cc a_3^{-1} \cc c_2 \cc a_3^2 \cc b_1^2 \cc c_1^2 \cc a_2^2 \cc a_1^2 \cc e \cc a_2^{-1} \cc c_1^{-1} \cc b_1 \cc c_1^2 \cc a_2 \cc e^{-1},\\
&\gamma_{14} = a_1^{-1} \cc a_2^{-2} \cc c_1^{-2} \cc b_1^{-2} \cc a_3^{-1} \cc c_2 \cc a_3^2 \cc b_1^2 \cc c_1^2 \cc a_2^2 \cc a_1^2 \cc e \cc a_2^{-1} \cc c_1^{-2} \cc b_1^{-1} \cc c_1 \cc a_2 \cc e^{-1},\\
&\gamma_{15} =  a_1^{-2} \cc a_2^{-2} \cc c_1^{-2} \cc b_1^{-2}\cc a_3^{-2}  \cc c_2^{-1} \cc a_3 \cc b_1^2 \cc c_1^2 \cc a_2^2 \cc a_1 \cc e \cc c_5^2 \cc a_6^2 \cc c_4\cc a_6^{-2} \cc c_5^{-2} \cc e^{-1},\\
&\gamma_{16} =  a_1^{-1} \cc a_2^{-2} \cc c_1^{-2} \cc b_1^{-1} \cc c_1 \cc a_2 \cc e\cc c_5^2 \cc a_6^2 \cc c_4^2 \cc a_5^2 \cc c_3\cc a_4^{-1} \cc c_3^{-2}  \cc a_5^{-2} \cc c_4^{-2} \cc a_6^{-2} \cc c_5^{-2} \cc e^{-1},\\
&\gamma_{17} =  a_1^{-1} \cc a_2^{-2} \cc c_1^{-2} \cc b_1^{-1} \cc c_1 \cc a_2 \cc e\cc c_5^2 \cc a_6^2 \cc c_4^2 \cc a_5^2 \cc c_3\cc a_4 \cc c_3^{-1}  \cc a_5^{-2} \cc c_4^{-2} \cc a_6^{-2} \cc c_5^{-2} \cc e^{-1},\\
&\gamma_{18}=a_1^{-2} \cc a_2^{-2} \cc c_1^{-2} \cc b_1^{-2}\cc a_3^{-2} \cc c_2^{-1} \cc a_3 \cc b_1^2 \cc c_1^2 \cc a_2^2\cc a_1 \cc e\cc c_5^2 \cc a_6^2 \cc c_4 \cc a_5^{-1} \cc c_4^{-2} \cc a_6^{-2} \cc c_5^{-2} \cc e^{-1},\\
\end{aligned}
\end{equation*}
\begin{equation*} 
\begin{aligned}
&\gamma_{19}=a_1^{-2} \cc a_2^{-2} \cc c_1^{-2} \cc b_1^{-2}\cc a_3^{-2} \cc c_2^{-1} \cc a_3 \cc b_1^2 \cc c_1^2 \cc a_2^2\cc a_1 \cc e\cc c_5^2 \cc a_6^2 \cc c_4^2 \cc a_5\cc c_4^{-1} \cc a_6^{-2} \cc c_5^{-2} \cc e^{-1},\\
&\gamma_{20}=a_1^{-1} \cc a_2^{-2} \cc c_1^{-2} \cc b_1^{-2}\cc a_3^{-1} \cc c_2 \cc a_3^2 \cc b_1^2 \cc c_1^2 \cc a_2^2\cc a_1^2 \cc e\cc a_2^{-1} \cc c_1^{-1} \cc b_1 \cc c_1^2 \cc a_2^2 \cc a_1\cc e \cc c_5 \cc a_6^{-1} \cc c_5^{-2} \cc e^{-1},\\
&\gamma_{21}=a_1^{-2} \cc a_2^{-2} \cc c_1^{-2} \cc b_1^{-2}\cc a_3^{-2} \cc c_2^{-1} \cc a_3 \cc b_1^2 \cc c_1^2 \cc a_2^2\cc a_1 \cc e\cc c_5^2 \cc a_6^2 \cc c_4^2 \cc a_5\cc c_3^{-1} \cc a_5^{-2} \cc c_4^{-2} \cc a_6^{-2} \cc c_5^{-2} \cc e^{-1},\\
&\gamma_{22}=a_1^{-1} \cc a_2^{-2} \cc c_1^{-2} \cc b_1^{-2}\cc a_3^{-2} \cc c_2^{-1} \cc a_3 \cc b_1^2 \cc c_1^2 \cc a_2^2\cc a_1 \cc e\cc c_5^2 \cc a_6^2 \cc c_4^2 \cc a_5^2 \cc c_3\cc a_4^{-1} \cc c_3^{-2} \cc a_5^{-2} \cc c_4^{-2} \cc a_6^{-2} \cc c_5^{-2} \cc e^{-1},\\
&\gamma_{23}=a_1^{-1} \cc a_2^{-2} \cc c_1^{-2} \cc b_1^{-2}\cc a_3^{-1} \cc c_2 \cc a_3^2 \cc b_1^2 \cc c_1^2 \cc a_2^2\cc a_1 \cc e\cc c_5^2 \cc a_6^2 \cc c_4^2 \cc a_5^2\cc c_3\cc a_4^{-1} \cc c_3^{-2} \cc a_5^{-2} \cc c_4^{-2} \cc a_6^{-2} \cc c_5^{-2} \cc e^{-1},\\
&\gamma_{24}=a_1^{-2} \cc a_2^{-2} \cc c_1^{-2} \cc b_1^{-2}\cc a_3^{-2} \cc c_2^{-1} \cc a_3 \cc b_1^2 \cc c_1^2 \cc a_2^2\cc a_1 \cc e\cc c_5^2 \cc a_6^2 \cc c_4^2 \cc a_5^2\cc c_3\cc a_4^{-1} \cc c_3^{-2} \cc a_5^{-2} \cc c_4^{-2} \cc a_6^{-2} \cc c_5^{-2} \cc e^{-1},\\
&\gamma_{25}=a_1^{-1} \cc a_2^{-2} \cc c_1^{-2} \cc b_1^{-2}\cc a_3^{-1} \cc c_2 \cc a_3^2 \cc b_1^2 \cc c_1^2 \cc a_2^2\cc a_1^2 \cc e \cc c_1^{-1} \cc b_1^{-1} \cc a_3\cc b_1^2 \cc c_1^2 \cc a_2^2 \cc a_1\cc e \\
&\phantom{ASDFF} \cc  c_5^2 \cc a_6^2 \cc c_4^2 \cc a_5^2\cc c_3\cc a_4^{-1} \cc c_3^{-2} \cc a_5^{-2} \cc c_4^{-2} \cc a_6^{-2} \cc c_5^{-2} \cc e^{-1},\\
&\gamma_{26}=a_1^{-1} \cc a_2^{-2} \cc c_1^{-2} \cc b_1^{-1}\cc c_1 \cc a_2 \cc e \cc  c_5^2 \cc a_6^2 \cc c_4^2 \cc a_5^2 \cc c_3^2 \cc a_4 \cc c_2 \cc a_3^2\cc b_1^2 \cc c_1^2 \cc a_2^2 \cc a_1\cc e\\
&\phantom{ASDFF}  \cc  c_5^2 \cc a_6^2 \cc c_4^2 \cc a_5^2\cc c_3\cc a_4^{-1} \cc c_3^{-2} \cc a_5^{-2} \cc c_4^{-2} \cc a_6^{-2} \cc c_5^{-2} \cc e^{-1},\\
\end{aligned}
\end{equation*}
and
\begin{equation*}
\begin{aligned}
& \zeta_1 = a_1^{-1} \cc a_2^{-2} \cc c_1^{-2} \cc b_1^{-1}, \qquad \qquad\qquad \qquad  \zeta_2 = a_2^{-1} \cc c_1^{-1} \cc b_1 \cc c_1^2 \cc a_2^2 \cc a_1 \cc e \cc c_5,\\
&\zeta_3 = a_1^{-1} \cc a_2^{-2} \cc c_1^{-2} \cc b_1^{-2} \cc a_3^{-1} \cc c_2 \cc a_3^2 \cc b_1^2 \cc c_1^2 \cc a_2^2 \cc a_1^2 \cc e, \quad \zeta_4 = a_1^{-1} \cc a_2^{-2} \cc c_1^{-2} \cc b_1^{-2} \cc a_3^{-2} \cc c_2^{-1}  \cc a_3 \cc b_1^2 \cc c_1^2 \cc a_2^2 \cc a_1 \cc a_2^{-2} \cc c_1^{-1},\\
&\zeta_5 = a_1^{-1} \cc a_2^{-2} \cc c_1^{-2} \cc b_1^{-2} \cc a_3^{-1} \cc c_2  \cc a_3^2 \cc b_1^2 \cc c_1^2 \cc a_2^2 \cc a_1 \cc a_2^{-2} \cc c_1^{-1}, \quad \zeta_6 = a_1^{-2} \cc a_2^{-2} \cc c_1^{-2} \cc b_1^{-2} \cc a_3^{-2} \cc c_2^{-1}  \cc a_3 \cc b_1^2 \cc c_1^2 \cc a_2^2 \cc a_1 \cc a_2^{-2} \cc c_1^{-1},\\
&\zeta_7= a_1^{-1} \cc a_2^{-2} \cc c_1^{-2} \cc b_1^{-2} \cc a_3^{-1} \cc c_2  \cc a_3^2 \cc b_1^2 \cc c_1^2 \cc a_2^2 \cc a_1^2\cc e \cc c_1^{-1} \cc b_1^{-1} \cc a_3 \cc b_1^2 \cc c_1^2 \cc a_2^2 \cc a_1 \cc a_2^{-2} \cc c_1^{-1},\\
&\mu_1 = c_1 \cc a_2 \cc e^{-1}, \qquad \qquad \qquad \qquad \qquad  \mu_2 = b_1^{-1} \cc a_3^{-1} \cc c_2 \cc a_3^2 \cc b_1^2 \cc c_1^2 \cc a_2^2 \cc a_1^2 \cc e\cc a_2^{-1} \cc c_1^{-1} \cc e^{-1},\\
& \mu_3 = b_1^{-1} \cc a_3^{-1} \cc c_2 \cc a_3^2 \cc b_1^2 \cc c_1^2 \cc a_2^2 \cc a_1^2 \cc e\cc a_2^{-1} \cc c_1^{-1} \cc  c_5^{-1} \cc e^{-1}, \quad \mu_4 = c_1 \cc a_2 \cc e \cc c_5^2 \cc a_6^2 \cc c_4^2 \cc a_5^2 \cc c_3 \cc a_5^{-1} \cc c_4^{-2} \cc a_6^{-2} \cc c_5^{-2} \cc e^{-1},\\
&\mu_5 =  c_1 \cc a_2 \cc e \cc c_5^2 \cc a_6^2 \cc c_4^2 \cc a_5^2 \cc c_3\cc a_4^{-1} \cc c_3^{-2}  \cc a_5^{-2} \cc c_4^{-2} \cc a_6^{-2} \cc c_5^{-2} \cc e^{-1},\\
 & \mu_6 =  c_1 \cc a_2 \cc e \cc c_5^2 \cc a_6^2 \cc c_4^2 \cc a_5^2 \cc c_3^2 \cc a_4 \cc c_3^{-1}  \cc a_5^{-2} \cc c_4^{-2} \cc a_6^{-2} \cc c_5^{-2} \cc e^{-1},\\
&\mu_7 = c_1 \cc a_2^2 \cc e \cc c_5^2 \cc a_6^2 \cc c_4^2 \cc a_5^2 \cc c_3 \cc a_4^{-1} \cc c_3^{-2}  \cc a_5^{-2} \cc c_4^{-2} \cc a_6^{-2} \cc c_5^{-2} \cc e^{-1},\\
&\mu_8 = b_1^{-1} \cc a_3^{-1} \cc c_2 \cc a_3^2 \cc b_1^2 \cc c_1^2 \cc a_2^2 \cc a_1^2 \cc e \cc a_2^{-1} \cc c_1^{-1} \cc b_1\cc c_1^2 \cc a_2^2 \cc a_1 \cc e \cc c_5 \cc a_6^{-1} \cc c_5^{-2} \cc e^{-1},\\
&\eta_1 = b_1^{-1} \cc a_3^{-1} \cc c_2 \cc a_3^2 \cc b_1^2 \cc c_1^2 \cc a_2^2 \cc a_1^2 \cc e\cc a_2^{-1} \cc c_1^{-1} \cc b_1 \cc c_1, \quad  \eta_2 = b_1^{-1} \cc a_3^{-1} \cc c_2 \cc a_3^2 \cc b_1^2 \cc c_1^2 \cc a_2^2 \cc a_1^2 \cc e\cc a_2^{-1} \cc c_1^{-2} \cc b_1^{-1},\\
&\eta_3 = c_1 \cc a_2 \cc e \cc c_5^2 \cc a_6^2 \cc c_4^2 \cc a_5^2 \cc c_3^2 \cc a_4\cc c_2 \cc a_3^2 \cc b_1^2 \cc c_1^2 \cc a_2^2 \cc a_1 \cc a_2^{-2} \cc c_1^{-1}.\\
\end{aligned}
\end{equation*}
The admissible triples are 
\begin{equation*}
\begin{aligned}
A= \{ & (1,21,8),(1,19,12),(1,1916),(1,15,17),(1,18,20),(1,19,20),(1,19,26),(2,24,5),(3,19,17),\\
&(4,17,2),(4,17,10),(4,17,14),(4,17,23),(4,17,25),(5,17,14),(5,1,15),(5,1,18),(5,1,19),(5,1,21),\\
&(6,20,17),(7,8,13),(7,8,20),(8,11,17),(8,20,17),(8,13,19),(8,14,19),(8,9,24),(8,13,24),(9,24,5),\\
&(10,19,17),(11,17,14),(12,20,17),(13,24,3),(13,24,4),(13,24,5),(13,24,6),(13,19,17),(14,19,17),\\
&(14,19,26),(15,17,14),(16,8,13),(17,22,5),(17,25,5),(17,7,8),(17,23,8),(17,25,8),(17,10,19),\\
&(17,14,19),(17,8,20),(17,2,24),(18,20,17),(19,17,2),(19,17,7),(19,16,8),(19,17,8),(19,26,8),\\
&(19,20,17),(19,12,20),(19,17,22),(20,17,2),(20,17,10),(20,17,14),(20,17,23),(20,17,25),\\
&(21,8,20),(22,5,1),(22,4,17),(22,5,17),(22,3,19),(22,6,20),(23,8,13),(24,5,1),(24,4,17),\\
&(24,5,17),(24,3,19),(24,6,20),(25,5,1),(25,8,9),(25,8,11),(25,8,13),(25,8,14),(26,8,13)\}
\end{aligned}
\end{equation*}
Let $V= Span \, \Gamma$ and $W= Span\, \Delta$. The splitting map $S: V\to W$ is given by 
\begin{equation*}
S: \left\{ \begin{aligned} & \gamma_1 \mapsto \mu_1 + \zeta_6, \quad \gamma_2 \mapsto \eta_3, \quad \gamma_3 \mapsto \mu_2 + \zeta_6, \quad \gamma_4 \mapsto \eta_2, \quad \gamma_5 \mapsto \eta_1, \quad \gamma_6 \mapsto \mu_3 + \zeta_1,\\
&\gamma_7 \mapsto \mu_4 + \zeta_1, \quad \gamma_8 \mapsto \mu_8 + \zeta_1, \quad \gamma_9 \mapsto \mu_6+ \zeta_5, \quad \gamma_{10} \mapsto \mu_6+ \zeta_4, \quad \gamma_{11} \mapsto \mu_6 + \zeta_3,\\
& \gamma_{12} \mapsto \mu_1 + \gamma_{15} + \zeta_1, \quad \gamma_{13} \mapsto \mu_6 + \zeta_7, \quad \gamma_{14} \mapsto \mu_6+ \gamma_2 + \zeta_6, \quad \gamma_{15} \mapsto \mu_7+ \zeta_2, \\
&\gamma_{16} \mapsto \mu_1 + \gamma_{18} + \zeta_1 , \quad \gamma_{17} \mapsto \mu_1 +\gamma_{19} + \zeta_1, \quad \gamma_{18} \mapsto \mu_7 + \gamma_{4} + \zeta_1, \quad \gamma_{19} \mapsto \mu_7 + \gamma_5 + \zeta_1,\\
&\gamma_{20} \mapsto \mu_6 + \zeta_1 + \eta_2, \quad \gamma_{21} \mapsto \mu_7 + \gamma_6+ \zeta_1, \quad \gamma_{22} \mapsto \mu_5 + \gamma_8 + \zeta_1, \quad \gamma_{23} \mapsto \mu_6 + \gamma_8+ \zeta_1,\\
&\gamma_{24} \mapsto \mu_7 + \gamma_8 + \zeta_1, \quad \gamma_{25} \mapsto \mu_6+ \gamma_7+ \gamma_8 + \zeta_1, \quad \gamma_{26} \mapsto \mu_1 + \gamma_{21} + \gamma_8+ \zeta_1.
 \end{aligned}\right.
\end{equation*}
and the merging map $M:V \to W$ is
\begin{equation*}
M: \left\{ \begin{aligned} &\gamma_1 \mapsto \zeta_1 + \mu_1, \quad \gamma_2 \mapsto \gamma_2, \quad \gamma_3 \mapsto \zeta_2 + \mu_1, \quad \gamma_4 \mapsto \gamma_4, \quad \gamma_5 \mapsto \gamma_5, \quad \gamma_6 \mapsto \gamma_6,\\
&\gamma_7 \mapsto \gamma_7, \quad \gamma_8 \mapsto \gamma_8, \quad \gamma_9 \mapsto \zeta_1 + \mu_2, \quad \gamma_{10} \mapsto \zeta_3 + \mu_1, \quad \gamma_{11} \mapsto \zeta_1 + \mu_3,\\
&\gamma_{12} \mapsto \zeta_1+ \mu_4, \quad \gamma_{13} \mapsto \zeta_1 + \eta_1 + \mu_1, \quad \gamma_{14} \mapsto \zeta_1+ \eta_2 + \mu_1, \quad \gamma_{15}  \mapsto \gamma_{15}, \\
&\gamma_{16} \mapsto \zeta_1 + \mu_5, \quad \gamma_{17} \mapsto \zeta_1 + \mu_6, \quad \gamma_{18} \mapsto \gamma_{18}, \quad \gamma_{19} \mapsto \gamma_{19}, \quad \gamma_{20} \mapsto \zeta_1 + \mu_8,\\
&\gamma_{21} \mapsto \gamma_{21}, \quad \gamma_{22}  \mapsto \zeta_4 + \mu_7 , \quad \gamma_{23} \mapsto \zeta_5 + \mu_7, \quad \gamma_{24} \mapsto \zeta_6 + \mu_7, \\
&\gamma_{25} \mapsto \zeta_7+ \mu_7, \quad \gamma_{26} \mapsto \zeta_1 + \eta_3 + \mu_7.
\end{aligned}\right.
\end{equation*}
There is a $v_\lambda \in V$ such that $Sv_\lambda = \lambda Mv_\lambda$ where $\lambda\approx 1.46048 $ is the largest real root of $\chi(t) =t^{12} - t^{10} -t^9-t^8+t^6-t^4-t^3-t^2+1$
and
\begin{equation*}
\begin{aligned}
 v_\lambda \approx ( &3.087,1.460,0.321,0.048,3.115,0.469,1,4.479,0.220,0.220,\\
 &0.321,0.685,2.133,2.133,0.469,0.103,6.645,0.071,4.550,3.067,\\
 &0.685,0.150,0.150,3.780,1.460,1 )\\
 \end{aligned}
 \end{equation*}

\subsection{Orbitdata $1,6,7$ with a cyclic permutation}\label{B:167}
Let $f_\mathbf{R}: X(\mathbf{R}) \to X(\mathbf{R})$ be a real diffeomorphism associated with the orbit data $1,6,7$ and a cyclic permutation. 
Let \[\Gamma = \{ \gamma_1, \dots, \gamma_{30} \},\ \Delta = \{ \gamma_2, \gamma_3,\gamma_5,\gamma_{6}, \gamma_{9},\gamma_{11},\gamma_{12},\gamma_{14}, \gamma_{15}, \gamma_{16}, \gamma_{20},\zeta_1,\dots,\zeta_{10}, \mu_1, \dots, \mu_{11}\} \]
be two ordered sets of non-cyclic words in $\pi_1( X(\mathbf{R}),P_{fix})$ where
\begin{equation*}
\begin{aligned}
&\gamma_1 = a_1^{-1} \cc a_2^{-2} \cc c_1^{-2} \cc b_1^{-1} \cc c_1 \cc a_2 \cc e^{-1}, \qquad \qquad \gamma_2 =  a_1^{-1} \cc a_2^{-2} \cc c_1^{-2} \cc b_1^{-2} \cc a_3^{-1} \cc b_1 \cc c_1 \cc e^{-1} \\
&\gamma_3 = a_1^{-1} \cc a_2^{-2} \cc c_1^{-2} \cc b_1^{-2} \cc a_3^{-1}\cc c_2 \cc a_3^2  \cc b_1^2 \cc c_1^2 \cc a_2^2 \cc a_1^2 \cc e \cc a_2^{-1} \cc c_1^{-1} \cc e^{-1},\\
 &\gamma_4 = a_1^{-1} \cc a_2^{-2} \cc c_1^{-2} \cc b_1^{-2} \cc a_3^{-1}\cc c_2 \cc a_3^2  \cc b_1^2 \cc c_1^2 \cc a_2^2 \cc a_1^2 \cc e  \cc c_1\cc a_2 \cc e^{-1},\\
 &\gamma_5 =  a_1^{-1} \cc a_2^{-2} \cc c_1^{-2} \cc b_1^{-2} \cc a_3^{-1}\cc c_2 \cc a_3^2  \cc b_1^2 \cc c_1^2 \cc a_2^2 \cc a_1^2 \cc e \cc a_2^{-1} \cc c_1^{-1} \cc c_6^{-1} \cc e^{-1},\\
 &\gamma_6 =a_1^{-2} \cc a_2^{-2} \cc c_1^{-2} \cc b_1^{-2} \cc a_3^{-2} \cc c_2^{-2} \cc a_4^{-1} \cc c_3^{-1} \cc a_5^{-2} \cc c_4^{-2} \cc a_6^{-2} \cc c_5^{-2} \cc a_7^{-2} \cc c_6^{-2} \cc e^{-1},\\
 &\gamma_7 =  a_1^{-1} \cc a_2^{-2} \cc c_1^{-2} \cc b_1^{-2} \cc a_3^{-1}\cc c_2 \cc a_3^2  \cc b_1^2 \cc c_1^2 \cc a_2^2 \cc a_1^2 \cc e \cc a_2^{-1} \cc c_1^{-1} \cc b_1 \cc c_1^2 \cc a_2 \cc e^{-1},\\
 &\gamma_8 = a_1^{-1} \cc a_2^{-2} \cc c_1^{-2} \cc b_1^{-2} \cc a_3^{-1}\cc c_2 \cc a_3^2  \cc b_1^2 \cc c_1^2 \cc a_2^2 \cc a_1^2 \cc e \cc a_2^{-1} \cc c_1^{-2} \cc b_1^{-1}  \cc c_1 \cc a_2 \cc e^{-1},\\
&\gamma_9 = a_2^{-1} \cc c_1^{-1} \cc b_1 \cc c_1^2 \cc a_2^2 \cc a_1 \cc e \cc c_6^2 \cc a_7^2 \cc c_5^2 \cc a_6^2 \cc c_4 \cc a_5^{-1} \cc c_4^{-2} \cc a_6^{-2} \cc c_5^{-2} \cc a_7^{-2} \cc c_6^{-2} \cc e^{-1},\\
&\gamma_{10} =  a_1^{-1} \cc a_2^{-2} \cc c_1^{-2} \cc b_1^{-2} \cc a_3^{-1}\cc c_2 \cc a_3^2  \cc b_1^2 \cc c_1^2 \cc a_2^2 \cc a_1^2 \cc e \cc a_2^{-1} \cc c_1^{-1} \cc b_1 \cc c_1^2 \cc a_2^2 \cc a_1 \cc e \cc c_6 \cc c_1 \cc a_2 \cc e^{-1},\\
&\gamma_{11} = a_1^{-1} \cc a_2^{-2} \cc c_1^{-2} \cc b_1^{-2} \cc a_3^{-1}\cc c_2 \cc a_3^2  \cc b_1^2 \cc c_1^2 \cc a_2^2 \cc a_1^2 \cc e \cc a_2^{-1} \cc c_1^{-1} \cc b_1 \cc c_1^2 \cc a_2^2 \cc a_1 \cc e \cc c_6 \cc a_7^{-1} \cc c_6^{-2} \cc e^{-1},\\
&\gamma_{12} = a_1^{-1} \cc a_2^{-2} \cc c_1^{-2} \cc b_1^{-2} \cc a_3^{-1}\cc c_2 \cc a_3^2  \cc b_1^2 \cc c_1^2 \cc a_2^2 \cc a_1^2 \cc e \cc a_2^{-1} \cc c_1^{-1} \cc b_1 \cc c_1^2 \cc a_2^2 \cc a_1 \cc e \cc c_6^2  \cc a_7\cc c_6^{-1}  \cc e^{-1},\\
&\gamma_{13}= a_1^{-1} \cc a_2^{-2} \cc c_1^{-2} \cc b_1^{-1}  \cc c_1 \cc a_2 \cc e \cc c_6^2 \cc a_7^2 \cc c_5^2 \cc a_6^2 \cc c_4^2 \cc a_5^2 \cc c_3 \cc a_5^{-1} \cc c_4^{-2} \cc a_6^{-2} \cc c_5^{-2} \cc a_7^{-2} \cc c_6^{-2} \cc e^{-1},\\
\end{aligned}
\end{equation*}
\begin{equation*} 
\begin{aligned}
&\gamma_{14} =a_1^{-2} \cc a_2^{-2} \cc c_1^{-2} \cc b_1^{-2} \cc a_3^{-2} \cc c_2^{-1} \cc a_3 \cc b_1^2 \cc c_1^2  \cc a_2^2 \cc a_1  \cc e \cc c_6^2 \cc a_7^2 \cc c_5^2 \cc a_6^2 \cc c_4  \cc a_6^{-1} \cc c_5^{-2} \cc a_7^{-2} \cc c_6^{-2} \cc e^{-1},\\
&\gamma_{15} = a_1^{-1} \cc a_2^{-2} \cc c_1^{-2} \cc b_1^{-2} \cc a_3^{-1}\cc c_2 \cc a_3^2  \cc b_1^2 \cc c_1^2 \cc a_2^2 \cc a_1^2 \cc e \cc a_2^{-1} \cc c_1^{-1} \cc b_1 \cc c_1^2 \cc a_2^2 \cc a_1 \cc e \cc c_6^2  \cc a_7\cc c_5^{-1} \cc a_7^{-2} \cc c_6^{-2}  \cc e^{-1},\\
&\gamma_{16} = a_1^{-2} \cc a_2^{-2} \cc c_1^{-2} \cc b_1^{-2} \cc a_3^{-2} \cc c_2^{-1} \cc a_3 \cc b_1^2 \cc c_1^2  \cc a_2^2 \cc a_1  \cc e \cc c_6^2 \cc a_7^2 \cc c_5^2 \cc a_6^2 \cc c_4\cc a_5^{-1} \cc c_4^{-2}   \cc a_6^{-2} \cc c_5^{-2} \cc a_7^{-2} \cc c_6^{-2} \cc e^{-1},\\
&\gamma_{17} = a_1^{-2} \cc a_2^{-2} \cc c_1^{-2} \cc b_1^{-2} \cc a_3^{-2} \cc c_2^{-1} \cc a_3 \cc b_1^2 \cc c_1^2  \cc a_2^2 \cc a_1  \cc e \cc c_6^2 \cc a_7^2 \cc c_5^2 \cc a_6^2 \cc c_4^2 \cc a_5 \cc c_4^{-1}   \cc a_6^{-2} \cc c_5^{-2} \cc a_7^{-2} \cc c_6^{-2} \cc e^{-1},\\
&\gamma_{18} =  a_1^{-1} \cc a_2^{-2} \cc c_1^{-2} \cc b_1^{-2} \cc a_3^{-1}\cc c_2 \cc a_3^2  \cc b_1^2 \cc c_1^2 \cc a_2^2 \cc a_1^2 \cc e \cc a_2^{-1} \cc c_1^{-1} \cc b_1 \cc c_1^2 \cc a_2^2 \cc a_1 \cc e \cc c_6^2  \cc a_7^2 \cc c_5 \cc a_6^{-1} \cc c_5^{-2} \cc a_7^{-2} \cc c_6^{-2}  \cc e^{-1},\\
&\gamma_{19} = a_1^{-2} \cc a_2^{-2} \cc c_1^{-2} \cc b_1^{-2} \cc a_3^{-2} \cc c_2^{-1} \cc a_3 \cc b_1^2 \cc c_1^2  \cc a_2^2 \cc a_1  \cc e \cc c_6^2 \cc a_7^2 \cc c_5^2 \cc a_6^2 \cc c_4^2 \cc a_5 \cc c_3^{-1} \cc a_5^{-2}  \cc c_4^{-2}   \cc a_6^{-2} \cc c_5^{-2} \cc a_7^{-2} \cc c_6^{-2} \cc e^{-1},\\
&\gamma_{20} = a_1^{-1} \cc a_2^{-2} \cc c_1^{-2} \cc b_1^{-2} \cc a_3^{-2}\cc c_2^{-1} \cc a_4 \cc c_2^2 \cc a_3^2  \cc b_1^2 \cc c_1^2 \cc a_2^2 \cc a_1^2 \cc e \cc a_2^{-1} \cc c_1^{-1} \cc b_1 \cc c_1^2 \cc a_2^2 \cc a_1 \cc e\\
&\phantom{ASDFEW}  \cc c_6^2  \cc a_7 \cc c_5^2 \cc a_6^2 \cc c_4 \cc a_5^{-1} \cc c_4^{-2} \cc a_6^{-2} \cc c_5^{-2} \cc a_7^{-2} \cc c_6^{-2}  \cc e^{-1},\\
&\gamma_{21} =a_1^{-2} \cc a_2^{-2} \cc c_1^{-2} \cc b_1^{-2} \cc a_3^{-2} \cc c_2^{-1} \cc a_3 \cc b_1^2 \cc c_1^2  \cc a_2^2 \cc a_1  \cc e \cc c_6^2 \cc a_7^2 \cc c_5^2 \cc a_6^2 \cc c_4^2 \cc a_5^2\cc c_3 \cc a_4 \cc c_2^2 \cc a_3^2 \cc b_1^2 \cc c_1^2 \cc a_2^2 \cc a_1^2 \cc e \\
&\phantom{ASDFEW}  \cc a_2^{-1} \cc c_1^{-1} \cc b_1 \cc c_1^2 \cc a_2^2 \cc a_1 \cc e \cc c_6^2 \cc a_7^2 \cc c_5 \cc a_7^{-1} \cc c_6^{-2} \cc e^{-1},\\
&\gamma_{22} = a_1^{-1} \cc a_2^{-2} \cc c_1^{-2} \cc b_1^{-1}  \cc c_1  \cc a_2  \cc e \cc c_6^2 \cc a_7^2 \cc c_5^2 \cc a_6^2 \cc c_4^2 \cc a_5^2\cc c_3 \cc a_4 \cc c_2^2 \cc a_3^2 \cc b_1^2 \cc c_1^2 \cc a_2^2 \cc a_1^2 \cc e \\
&\phantom{ASDFEW}  \cc a_2^{-1} \cc c_1^{-1} \cc b_1 \cc c_1^2 \cc a_2^2 \cc a_1 \cc e \cc c_6^2 \cc a_7^2 \cc c_5^2 \cc a_6^2 \cc c_4 \cc a_5^{-1} \cc c_4^{-2} \cc a_6^{-2} \cc c_5^{-2} \cc a_7^{-2} \cc c_6^{-2} \cc e^{-1},\\
&\gamma_{23} = a_1^{-2} \cc a_2^{-2} \cc c_1^{-2} \cc b_1^{-2} \cc a_3^{-2} \cc c_2^{-1} \cc a_3 \cc b_1^2 \cc c_1^2  \cc a_2^2 \cc a_1  \cc e \cc c_6^2 \cc a_7^2 \cc c_5^2 \cc a_6^2 \cc c_4^2 \cc a_5^2\cc c_3 \cc a_4 \cc c_2^2 \cc a_3^2 \cc b_1^2 \cc c_1^2 \cc a_2^2 \cc a_1^2 \cc e \\
&\phantom{ASDFEW}  \cc a_2^{-1} \cc c_1^{-1} \cc b_1 \cc c_1^2 \cc a_2^2 \cc a_1 \cc e \cc c_6^2 \cc a_7^2 \cc c_5 \cc a_6^{-1} \cc c_5^{-2} \cc a_7^{-2} \cc c_6^{-2} \cc e^{-1},\\
&\gamma_{24} = a_1^{-2} \cc a_2^{-2} \cc c_1^{-2} \cc b_1^{-2} \cc a_3^{-2} \cc c_2^{-1} \cc a_3 \cc b_1^2 \cc c_1^2  \cc a_2^2 \cc a_1  \cc e \cc c_6^2 \cc a_7^2 \cc c_5^2 \cc a_6^2 \cc c_4^2 \cc a_5^2\cc c_3 \cc a_4 \cc c_2^2 \cc a_3^2 \cc b_1^2 \cc c_1^2 \cc a_2^2 \cc a_1^2 \cc e \\
&\phantom{ASDFEW}  \cc a_2^{-1} \cc c_1^{-1} \cc b_1 \cc c_1^2 \cc a_2^2 \cc a_1 \cc e \cc c_6^2 \cc a_7^2 \cc c_5^2 \cc a_6 \cc c_5^{-1} \cc a_7^{-2} \cc c_6^{-2} \cc e^{-1},\\
&\gamma_{25} =  a_1^{-2} \cc a_2^{-2} \cc c_1^{-2} \cc b_1^{-2} \cc a_3^{-2} \cc c_2^{-1} \cc a_3 \cc b_1^2 \cc c_1^2  \cc a_2^2 \cc a_1  \cc e \cc c_6^2 \cc a_7^2 \cc c_5^2 \cc a_6^2 \cc c_4^2 \cc a_5^2\cc c_3 \cc a_4 \cc c_2^2 \cc a_3^2 \cc b_1^2 \cc c_1^2 \cc a_2^2 \cc a_1^2 \cc e \\
&\phantom{ASDFEW}  \cc a_2^{-1} \cc c_1^{-1} \cc b_1 \cc c_1^2 \cc a_2^2 \cc a_1 \cc e \cc c_6^2 \cc a_7^2 \cc c_5^2 \cc a_6 \cc c_4^{-1} \cc a_6^{-2} \cc c_5^{-2} \cc a_7^{-2} \cc c_6^{-2} \cc e^{-1},\\
&\gamma_{26} =  a_1^{-1} \cc a_2^{-2} \cc c_1^{-2} \cc b_1^{-2} \cc a_3^{-2} \cc c_2^{-1} \cc a_3 \cc b_1^2 \cc c_1^2  \cc a_2^2 \cc a_1  \cc e \cc c_6^2 \cc a_7^2 \cc c_5^2 \cc a_6^2 \cc c_4^2 \cc a_5^2\cc c_3 \cc a_4 \cc c_2^2 \cc a_3^2 \cc b_1^2 \cc c_1^2 \cc a_2^2 \cc a_1^2 \cc e \\
&\phantom{ASDFEW}  \cc a_2^{-1} \cc c_1^{-1} \cc b_1 \cc c_1^2 \cc a_2^2 \cc a_1 \cc e \cc c_6^2 \cc a_7^2 \cc c_5^2 \cc a_6^2 \cc c_4 \cc a_5^{-1} \cc c_4^{-2} \cc a_6^{-2} \cc c_5^{-2} \cc a_7^{-2} \cc c_6^{-2} \cc e^{-1},\\
&\gamma_{27} = a_1^{-1} \cc a_2^{-2} \cc c_1^{-2} \cc b_1^{-2} \cc a_3^{-1} \cc c_2 \cc a_3^2 \cc b_1^2 \cc c_1^2  \cc a_2^2 \cc a_1  \cc e \cc c_6^2 \cc a_7^2 \cc c_5^2 \cc a_6^2 \cc c_4^2 \cc a_5^2\cc c_3 \cc a_4 \cc c_2^2 \cc a_3^2 \cc b_1^2 \cc c_1^2 \cc a_2^2 \cc a_1^2 \cc e \\
&\phantom{ASDFEW}  \cc a_2^{-1} \cc c_1^{-1} \cc b_1 \cc c_1^2 \cc a_2^2 \cc a_1 \cc e \cc c_6^2 \cc a_7^2 \cc c_5^2 \cc a_6^2 \cc c_4 \cc a_5^{-1} \cc c_4^{-2} \cc a_6^{-2} \cc c_5^{-2} \cc a_7^{-2} \cc c_6^{-2} \cc e^{-1},\\
&\gamma_{28} =  a_1^{-2} \cc a_2^{-2} \cc c_1^{-2} \cc b_1^{-2} \cc a_3^{-2} \cc c_2^{-1}  \cc a_3 \cc b_1^2 \cc c_1^2  \cc a_2^2 \cc a_1  \cc e \cc c_6^2 \cc a_7^2 \cc c_5^2 \cc a_6^2 \cc c_4^2 \cc a_5^2\cc c_3 \cc a_4 \cc c_2^2 \cc a_3^2 \cc b_1^2 \cc c_1^2 \cc a_2^2 \cc a_1^2 \cc e \\
&\phantom{ASDFEW}  \cc a_2^{-1} \cc c_1^{-1} \cc b_1 \cc c_1^2 \cc a_2^2 \cc a_1 \cc e \cc c_6^2 \cc a_7^2 \cc c_5^2 \cc a_6^2 \cc c_4 \cc a_5^{-1} \cc c_4^{-2} \cc a_6^{-2} \cc c_5^{-2} \cc a_7^{-2} \cc c_6^{-2} \cc e^{-1},\\
&\gamma_{29} =  a_1^{-2} \cc a_2^{-2} \cc c_1^{-2} \cc b_1^{-2} \cc a_3^{-2} \cc c_2^{-2}\cc a_4^{-1} \cc c_2  \cc a_3^2  \cc b_1^2 \cc c_1^2  \cc a_2^2 \cc a_1  \cc e \cc c_6^2 \cc a_7^2 \cc c_5^2 \cc a_6^2 \cc c_4^2 \cc a_5^2\cc c_3 \cc a_4 \cc c_2^2 \cc a_3^2 \cc b_1^2 \cc c_1^2 \cc a_2^2 \cc a_1^2 \cc e \\
&\phantom{ASDFEW}  \cc a_2^{-1} \cc c_1^{-1} \cc b_1 \cc c_1^2 \cc a_2^2 \cc a_1 \cc e \cc c_6^2 \cc a_7^2 \cc c_5^2 \cc a_6^2 \cc c_4 \cc a_5^{-1} \cc c_4^{-2} \cc a_6^{-2} \cc c_5^{-2} \cc a_7^{-2} \cc c_6^{-2} \cc e^{-1},\\
&\gamma_{30} =  a_1^{-1} \cc a_2^{-2} \cc c_1^{-2} \cc b_1^{-2} \cc a_3^{-1} \cc c_2  \cc a_3^2 \cc b_1^2 \cc c_1^2  \cc a_2^2 \cc a_1^2  \cc e \cc c_1^{-1} \cc b_1^{-1} \cc a_3 \cc b_1^2 \cc c_1^2 \cc a_2^2 \cc a_1 \cc e \cc c_6^2 \cc a_7^2 \cc c_5^2 \cc a_6^2 \cc c_4^2 \cc a_5^2\cc c_3 \cc a_4 \cc c_2^2 \cc a_3^2 \cc b_1^2 \cc c_1^2 \cc a_2^2 \cc a_1^2 \cc e \\
&\phantom{ASDFEW}  \cc a_2^{-1} \cc c_1^{-1} \cc b_1 \cc c_1^2 \cc a_2^2 \cc a_1 \cc e \cc c_6^2 \cc a_7^2 \cc c_5^2 \cc a_6^2 \cc c_4 \cc a_5^{-1} \cc c_4^{-2} \cc a_6^{-2} \cc c_5^{-2} \cc a_7^{-2} \cc c_6^{-2} \cc e^{-1},\\
\end{aligned}
\end{equation*}
and 
\begin{equation*}
\begin{aligned}
&\zeta _1 = a_1^{-1} \cc a_2^{-2} \cc c_1^{-2} \cc b_1^{-1}, \qquad \qquad \qquad \zeta_2 =   a_1^{-1} \cc a_2^{-2} \cc c_1^{-2} \cc b_1^{-2} \cc a_3^{-1} \cc c_2 \cc a_3^2 \cc b_1^2 \cc c_1^2 \cc a_2^2 \cc a_1^2 \cc e,\\
&\zeta_3 = a_1^{-1} \cc a_2^{-2} \cc c_1^{-2} \cc b_1^{-2} \cc a_3^{-2} \cc c_2 ^{-1}\cc a_3 \cc b_1^2 \cc c_1^2 \cc a_2^2 \cc a_1 \cc a_2^{-2} \cc c_1^{-1},\\ 
&\zeta_4=  a_1^{-1} \cc a_2^{-2} \cc c_1^{-2} \cc b_1^{-2} \cc a_3^{-1} \cc c_2 \cc a_3^2 \cc b_1^2 \cc c_1^2 \cc a_2^2 \cc a_1 \cc a_2^{-2} \cc c_1^{-1},\\
&\zeta_5 =  a_1^{-2} \cc a_2^{-2} \cc c_1^{-2} \cc b_1^{-2} \cc a_3^{-2} \cc c_2 ^{-1}\cc a_3 \cc b_1^2 \cc c_1^2 \cc a_2^2 \cc a_1 \cc a_2^{-2} \cc c_1^{-1},\\
&\zeta_6 = a_1^{-1} \cc a_2^{-2} \cc c_1^{-2} \cc b_1^{-2} \cc a_3^{-1} \cc c_2 \cc a_3^2 \cc b_1^2 \cc c_1^2 \cc a_2^2 \cc a_1^2 \cc e  \cc a_2^{-1} \cc c_1^{-1} \cc b_1 \cc c_1,\\
&\zeta_7 = a_1^{-1} \cc a_2^{-2} \cc c_1^{-2} \cc b_1^{-2} \cc a_3^{-2} \cc c_2 \cc a_3^2 \cc b_1^2 \cc c_1^2 \cc a_2^2 \cc a_1^2 \cc e \cc a_2^{-1} \cc c_1^{-2} \cc b_1^{-1},\\
\end{aligned}
\end{equation*}
\begin{equation*} 
\begin{aligned}
&\zeta_8 = a_1^{-2} \cc a_2^{-2} \cc c_1^{-2} \cc b_1^{-2} \cc a_3^{-2} \cc c_2 ^{-2} \cc a_4^{-1} \cc c_2 \cc a_3^2 \cc b_1^2 \cc c_1^2 \cc a_2^2 \cc a_1 \cc a_2^{-2} \cc c_1^{-1},\\
&\zeta_9 = a_1^{-1} \cc a_2^{-2} \cc c_1^{-2} \cc b_1^{-2} \cc a_3^{-1} \cc c_2 \cc a_3^2 \cc b_1^2 \cc c_1^2 \cc a_2^2 \cc a_1^2 \cc e  \cc a_2^{-1} \cc c_1^{-1} \cc b_1 \cc c_1^2 \cc a_2^2 \cc a_1 \cc e \cc c_6,\\
&\zeta_{10} = a_1^{-1} \cc a_2^{-2} \cc c_1^{-2} \cc b_1^{-2} \cc a_3^{-1} \cc c_2 \cc a_3^2 \cc b_1^2 \cc c_1^2 \cc a_2^2 \cc a_1^2 \cc e \cc c_1^{-1} \cc b_1^{-1} \cc a_3 \cc b_1^2 \cc c_1^2 \cc a_2^2 \cc a_1 \cc a_2^{-2} \cc c_1^{-1},\\
&\mu_1 = c_1 \cc a_2 \cc e^{-1},\\
&\mu_2 = c_1 \cc a_2^2 \cc e \cc c_6^2 \cc a_7^2 \cc c_5^2 \cc a_6^2 \cc c_4^2 \cc a_5 \cc c_4^{-1} \cc a_6^{-2} \cc c_5^{-2} \cc a_7^{-2} \cc c_6^{-2} \cc e^{-1},\\
&\mu_3 = c_1 \cc a_2 \cc e \cc c_6^2 \cc a_7^2 \cc c_5^2 \cc a_6^2 \cc c_4^2 \cc a_5^2 \cc c_3 \cc a_5^{-1} \cc c_4^{-2} \cc a_6^{-2} \cc c_5^{-2} \cc a_7^{-2} \cc c_6^{-2} \cc e^{-1},\\
&\mu_4 = c_1 \cc a_2^2 \cc e \cc c_6^2 \cc a_7^2 \cc c_5^2 \cc a_6^2 \cc c_4^2 \cc a_5 \cc  c_3^{-1} \cc a_5^{-2} \cc c_4^{-2} \cc a_6^{-2} \cc c_5^{-2} \cc a_7^{-2} \cc c_6^{-2} \cc e^{-1},\\
&\mu_5=b_1^{-1} \cc a_3^{-1} \cc c_2 \cc a_3^2 \cc b_1^2 \cc c_1^2 \cc a_2^2\cc a_1^2 \cc e \cc a_2^{-1} \cc c_1^{-1} \cc b_1 \cc c_1^2 \cc a_2^2 \cc a_1 \cc e \cc c_6^2 \cc a_7^2 \cc c_5 \cc a_6^{-1} \cc c_5^{-2} \cc a_7^{-2} \cc c_6^{-2} \cc e^{-1},\\
&\mu_6 = c_1 \cc a_2^2 \cc e \cc c_6^2 \cc a_7^2 \cc c_5^2 \cc a_6^2 \cc c_4^2 \cc a_5^2 \cc c_3 \cc a_4 \cc c_2^2 \cc a_3^2 \cc b_1^2 \cc c_1^2 \cc a_2^2 \cc a_1^2 \cc e \cc a_2^{-1} \cc c_1^{-1} \cc b_1 \cc c_1^2 \cc a_2^2 \cc a_1 \cc e \cc c_6^2 \cc a_7^2 \cc c_5 \cc a_7^{-1} \cc c_6^{-2} \cc e^{-1},\\
&\mu_7= c_1 \cc a_2^2 \cc e \cc c_6^2 \cc a_7^2 \cc c_5^2 \cc a_6^2 \cc c_4^2 \cc a_5^2 \cc c_3 \cc a_4 \cc c_2^2 \cc a_3^2 \cc b_1^2 \cc c_1^2 \cc a_2^2 \cc a_1^2 \cc e \cc a_2^{-1} \cc c_1^{-1} \cc b_1 \cc c_1^2 \cc a_2^2 \cc a_1 \cc e \\
&\phantom{ASDEFG}  \cc c_6^2 \cc a_7^2 \cc c_5 \cc a_6^{-1} \cc c_5^{-2}  \cc a_7^{-2} \cc c_6^{-2} \cc e^{-1},\\
&\mu_8=c_1 \cc a_2^2 \cc e \cc c_6^2 \cc a_7^2 \cc c_5^2 \cc a_6^2 \cc c_4^2 \cc a_5^2 \cc c_3 \cc a_4 \cc c_2^2 \cc a_3^2 \cc b_1^2 \cc c_1^2 \cc a_2^2 \cc a_1^2 \cc e \cc a_2^{-1} \cc c_1^{-1} \cc b_1 \cc c_1^2 \cc a_2^2 \cc a_1 \cc e \\
&\phantom{ASDEFG}  \cc c_6^2 \cc a_7^2 \cc c_5^2 \cc a_6 \cc c_5^{-1}  \cc a_7^{-2} \cc c_6^{-2} \cc e^{-1},\\
&\mu_9 = c_1 \cc a_2^2 \cc e \cc c_6^2 \cc a_7^2 \cc c_5^2 \cc a_6^2 \cc c_4^2 \cc a_5^2 \cc c_3 \cc a_4 \cc c_2^2 \cc a_3^2 \cc b_1^2 \cc c_1^2 \cc a_2^2 \cc a_1^2 \cc e \cc a_2^{-1} \cc c_1^{-1} \cc b_1 \cc c_1^2 \cc a_2^2 \cc a_1 \cc e \\
&\phantom{ASDEFG}  \cc c_6^2 \cc a_7^2 \cc c_5^2 \cc a_6 \cc c_4^{-1}  \cc a_6^{-2} \cc c_5^{-2}  \cc a_7^{-2} \cc c_6^{-2} \cc e^{-1},\\
&\mu_{10} = c_1 \cc a_2 \cc e \cc c_6^2 \cc a_7^2 \cc c_5^2 \cc a_6^2 \cc c_4^2 \cc a_5^2 \cc c_3 \cc a_4 \cc c_2^2 \cc a_3^2 \cc b_1^2 \cc c_1^2 \cc a_2^2 \cc a_1^2 \cc e \cc a_2^{-1} \cc c_1^{-1} \cc b_1 \cc c_1^2 \cc a_2^2 \cc a_1 \cc e \\
&\phantom{ASDEFG}  \cc c_6^2 \cc a_7^2 \cc c_5^2 \cc a_6^2 \cc c_4 \cc a_5^{-1} \cc c_4^{-2}  \cc a_6^{-2} \cc c_5^{-2}  \cc a_7^{-2} \cc c_6^{-2} \cc e^{-1},\\
&\mu_{11} = c_1 \cc a_2^2 \cc e \cc c_6^2 \cc a_7^2 \cc c_5^2 \cc a_6^2 \cc c_4^2 \cc a_5^2 \cc c_3 \cc a_4 \cc c_2^2 \cc a_3^2 \cc b_1^2 \cc c_1^2 \cc a_2^2 \cc a_1^2 \cc e \cc a_2^{-1} \cc c_1^{-1} \cc b_1 \cc c_1^2 \cc a_2^2 \cc a_1 \cc e \\
&\phantom{ASDEFG}  \cc c_6^2 \cc a_7^2 \cc c_5^2 \cc a_6^2 \cc c_4 \cc a_5^{-1} \cc c_4^{-2}  \cc a_6^{-2} \cc c_5^{-2}  \cc a_7^{-2} \cc c_6^{-2} \cc e^{-1}.\\
\end{aligned}
\end{equation*}
The admissible triples are 
\begin{equation*}
\begin{aligned}
A= \{ & (1,6,2,28),(1,6,3,28),(1,6,4,24),(1,6,5,1),(1,6,7,24),(1,6,7,28),(1,6,8,24),(1,6,9,18),\\
&(1,6,11,1),(1,6,20,18),(1,9,26,12),(1,6,27,18),(1,14,1,6),(1,16,18,1),(1,17,1,6),(1,17,1,29),\\
&(1,17,13,18),(1,17,18,1),(1,17,22,18),(1,19,9,18),(1,21,1,17),(1,23,1,6),(1,24,1,6),\\
&(1,24,1,14),(1,24,1,16),(1,24,1,17),(1,24,1,19),(1,25,18,1),(1,29,18,1),(2,28,12,1),\\
&(3,28,12,1),(4,24,1,17),(5,1,17,1),(6,2,28,12),(6,3,28,12),(6,4,24,1),(6,5,1,17),(6,7,24,1),\\
&(6,7,28,10),(6,7,28,12),(6,7,28,15),(6,8,24,1),(6,9,18,1),(6,11,1,17),(6,20,18,1),\\
&(6,26,12,1),(6,27,18,1),(6,30,18,1),(7,24,1,17),(7,28,10,24),(7,28,11,1),(7,28,15,1),\\
&(8,24,1,17),(9,10,24,1),(9,12,1,17),(9,12,1,21),(9,12,1,23),(9,12,1,24),(9,12,1,25),\\
&(9,15,1,6),(9,18,1,6),(10,24,1,17),(11,1,17,1),(12,1,17,1),(12,1,21,1),(12,1,23,1),\\
&(12,1,24,1),(12,1,25,18),(13,18,1,6),(14,1,6,11),(15,1,6,11),(16,18,1,6),(17,1,6,2),\\
&(17,1,6,4),(17,1,6,8),(17,1,6,9),(17,1,6,20),(17,1,6,26),(17,1,6,27),(17,1,6,30),\\
&(17,1,29,18),(17,13,18,1),(17,18,1,6),(17,22,18,1),(18,1,6,3),(18,1,6,5),(18,1,6,7),\\
&(18,1,6,11),(19,9,18,1),(20,18,1,6),(21,1,17,1),(22,18,1,6),(23,1,6,11),(24,1,6,11),\\
&(24,1,14,1),(24,1,16,18),(24,1,17,1),(24,1,17,13),(24,1,17,18),(24,1,17,22),(24,1,19,9),\\
&(25,18,1,6),(26,10,24,1),(26,12,1,17),(26,12,1,21),(26,12,1,23),(26,12,1,24),\\
&(26,12,1,25),(26,15,1,6),(27,18,1,6),(28,10,24,1),(28,12,1,17),(28,12,1,21),\\
&(28,12,1,23),(28,12,1,24),(28,12,1,25),(28,15,1,6),(29,18,1,6),(30,18,1,6)\}
\end{aligned}
\end{equation*}
Let $V= Span \, \Gamma$ and $W= Span\, \Delta$. The splitting map $S: V\to W$ is given by 
\begin{equation*}
S: \left\{ \begin{aligned} & \gamma_1 \mapsto \mu_1 + \zeta_5, \quad \gamma_2 \mapsto \mu_1 + \zeta_8, \quad \gamma_3 \mapsto \mu_1 + \gamma_6 + \zeta_4, \quad \gamma_4 \mapsto \mu_1 + \gamma_6 + \zeta_3, \quad \gamma_5 \mapsto \mu_1 + \gamma_6 + \zeta_2,\\ 
&\gamma_6 \mapsto \mu_2 + \zeta_1,\quad \gamma_7 \mapsto \mu_1 + \gamma_6 + \zeta_{10}, \quad \gamma_8 \mapsto \mu_1 + \gamma_6 + \gamma_2 + \zeta_5, \quad \gamma_9 \mapsto \mu_5 + \zeta_1, \\ &\gamma_{10} \mapsto \mu_1 + \gamma_6 + \gamma_3 + \zeta_5,\quad \gamma_{11} \mapsto \mu_1 + \gamma_6 + \zeta_7, \quad \gamma_{12} \mapsto \mu_1+ \gamma_6 + \zeta_6, \quad \gamma_{13} \mapsto \mu_1 + \gamma_{14} + \zeta_1,\\
&\gamma_{14} \mapsto \mu_6 + \zeta_1, \quad \gamma_{15} \mapsto \mu_1 + \gamma_6 + \gamma_5 + \zeta_1, \quad \gamma_{16} \mapsto \mu_7 + \zeta_1, \quad \gamma_{17} \mapsto \mu_8 + \zeta_1, \\
&\gamma_{18} \mapsto \mu_1 + \gamma_6  +\gamma_{11} +\zeta_1, \quad \gamma_{19} \mapsto \mu_9+ \zeta_1, \quad \gamma_{20} \mapsto \mu_3 +  \zeta_1 + \mu_5+\zeta_1, \quad \gamma_{21} \mapsto \mu_{11} + \zeta_9,\\
&\gamma_{22} \mapsto \mu_1 + \gamma_{16} + \zeta_1 + \mu_5+\zeta_1, \quad \gamma_{23} \mapsto \mu_{11} + \gamma_{11} + \zeta_1, \quad \gamma_{24} \mapsto \mu_{11} + \gamma_{12} + \zeta_1,\\
&\gamma_{25} \mapsto \mu_{11} +\gamma_{15} + \zeta_1, \quad \gamma_{26} \mapsto \mu_{10} + \zeta_1 + \mu_5+\zeta_1, \quad \gamma_{27} \mapsto \mu_1 + \gamma_6 + \gamma_9 +  \zeta_1 + \mu_5+\zeta_1,\\
&\gamma_{28} \mapsto \mu_{11} +  \zeta_1 + \mu_5+\zeta_1, \quad \gamma_{29} \mapsto \mu_4 + \gamma_9 + \zeta_1 + \mu_5+\zeta_1, \quad \gamma_{30} \mapsto \mu_1 + \gamma_6 + \gamma_{20} +  \zeta_1 + \mu_5+\zeta_1.\\
 \end{aligned}\right.
\end{equation*}
and the merging map $M:V \to W$ is
\begin{equation*}
M: \left\{ \begin{aligned} &\gamma_1 \mapsto \zeta_1 + \mu_1, \quad \gamma_2 \mapsto \gamma_2 , \quad \gamma_3 \mapsto \gamma_3, \quad \gamma_4 \mapsto \zeta_2 + \mu_1, \quad \gamma_5 \mapsto \gamma_5, \quad \gamma_6 \mapsto \gamma_6,\quad \gamma_7 \mapsto \zeta_6+ \mu_1 \\
&\gamma_8 \mapsto \zeta_7 + \mu_1 , \quad \gamma_9 \mapsto \gamma_9, \quad \gamma_{10} \mapsto \zeta_9 + \mu_1, \quad \gamma_{11} \mapsto \gamma_{11}, \quad \gamma_{12} \mapsto \gamma_{12}, \quad \gamma_{13} \mapsto \zeta_1 + \mu_3,\\
&\gamma_{14} \mapsto \gamma_{14}, \quad \gamma_{15} \mapsto \gamma_{15}, \quad \gamma_{16} \mapsto \gamma_{16}, \quad \gamma_{17} \mapsto \zeta_5 + \mu_2, \quad \gamma_{18} \mapsto \zeta_1 + \mu_5, \quad \gamma_{19} \mapsto \zeta_5 + \mu_4,\\
&\gamma_{20} \mapsto \gamma_{20}, \quad \gamma_{21}  \mapsto \zeta_5 + \mu_6, \quad \gamma_{22} \mapsto \zeta_1 + \mu_{10}, \quad \gamma_{23} \mapsto \zeta_5 + \mu_7, \quad \gamma_{24} \mapsto \zeta_5 + \mu_8, \quad \gamma_{25} \mapsto \zeta_5 + \mu_9,\\
&\gamma_{26} \mapsto \zeta_3 + \mu_{11}, \quad \gamma_{27} \mapsto \zeta_4+ \mu_{11}, \quad \gamma_{28} \mapsto \zeta_5 + \mu_{11}, \quad \gamma_{29} \mapsto \zeta_8 + \mu_{11}, \quad \gamma_{30} \mapsto \zeta_{10} + \mu_{11}.\\
\end{aligned}\right.
\end{equation*}
There is a $v_\lambda \in V$ such that $Sv_\lambda = \lambda Mv_\lambda$ where $\lambda\approx 1.53293 $ is the largest real root of $\chi(t) =t^{10} -t^9-t^8+t^5-t^2-t+1$
and
\begin{equation*}
\begin{aligned}
 v_\lambda \approx ( &17.704,1,0.077,0.077,0.118,8.465,1.533,1.533,0.458,0.118,\\
 &2.350,2.350,0.426,0.278,0.282,0.021,5.522,3.588,0.426,0.652,0.181,0.033, \\
 &0.014,3.602,0.278,0.050,0.050,2.604,0.652,1)\\
 \end{aligned}
 \end{equation*}

\subsection{Orbitdata $2,3,5$ with a cyclic permutation}\label{B:235}
Let $f_\mathbf{R}: X(\mathbf{R}) \to X(\mathbf{R})$ be a real diffeomorphism associated with the orbit data $2,3,5$ and a cyclic permutation. 
Let \[\Gamma = \{ \gamma_1, \dots, \gamma_{35} \},\ \Delta = \{ \gamma_1, \gamma_2,\gamma_{11},\gamma_{15}, \gamma_{16},\gamma_{17},\gamma_{24},\gamma_{25}, \gamma_{28}, \zeta_1,\dots,\zeta_{4}, \mu_1, \dots, \mu_{21}, \eta_1, \eta_2 \} \]
be two ordered sets of non-cyclic words in $\pi_1( X(\mathbf{R}),P_{fix})$ where
\begin{equation*}
\begin{aligned}
& \gamma_1 = b_1^{-1} \cc c_1^{-1} \cc a_5^{-1} \cc e^{-1}, \qquad \qquad \qquad \gamma_2 = b_1^{-1} \cc c_1^{-1} \cc a_2 \cc c_1 \cc a_5^{-1} \cc e^{-1},\\
&\gamma_3 = c_1^{-1} \cc a_2^{-1} \cc c_3^{-1} \cc a_4^{-1} \cc a_2 \cc c_1 \cc a_5^{-1} \cc e^{-1}, \quad \gamma_4 = b_1^{-1} \cc c_1^{-2} \cc a_2^{-1} \cc c_1 \cc b_1 \cc e \cc a_5^2 \cc a_4^2 \cc c_3 \cc a_4^{-1} \cc a_5^{-2} \cc e^{-1},\\
&\gamma_5 = a_1^{-1} \cc b_1^{-2} \cc c_1^{-2} \cc a_2^{-2} \cc b_2^{-2} \cc c_2^{-1} \cc a_3^{-1} \cc c_3^{-2} \cc a_4^{-2} \cc a_5^{-2} \cc e^{-1}, \\ 
&\gamma_6 = a_1^{-1} \cc b_1^{-2} \cc c_1^{-2} \cc a_2^{-2} \cc b_2^{-1} \cc c_2^{-1} \cc a_3^{-2} \cc c_3^{-2} \cc a_4^{-2} \cc a_5^{-2} \cc e^{-1},\\
&\gamma_7 =a_1^{-1} \cc  b_1^{-2} \cc c_1^{-2} \cc a_2^{-1} \cc c_1 \cc b_1 \cc e \cc a_5^2 \cc a_4^2 \cc c_3 \cc a_4^{-1} \cc a_5^{-2} \cc e^{-1},\\
&\gamma_8 = b_1^{-1} \cc c_1^{-1} \cc a_2 \cc c_1^2 \cc b_1^2 \cc a_1 \cc e \cc a_5 \cc c_1^{-1} \cc a_2^{-1} \cc a_4 \cc c_3 \cc a_2 \cc c_1 \cc a_5^{-1} \cc e^{-1},\\
&\gamma_9 = b_1^{-1} \cc c_1^{-1} \cc a_2 \cc c_1^2 \cc b_1^2 \cc a_1 \cc e \cc a_5 \cc c_1^{-1} \cc a_2^{-1} \cc c_3^{-1} \cc a_4^{-2} \cc a_5^{-2} \cc e^{-1},\\
&\gamma_{10} = a_1^{-1} \cc b_1^{-2} \cc c_1^{-2} \cc a_2^{-2} \cc b_2^{-1} \cc c_2^{-1}\cc a_3^{-1} \cc c_2^{-1} \cc a_3^{-2} \cc c_3^{-2} \cc a_4^{-2} \cc a_5^{-2} \cc e^{-1},\\
&\gamma_{11} = a_1^{-1} \cc b_1^{-1} \cc e  \cc a_5^2\cc a_4^2  \cc c_3^2 \cc a_3 \cc c_2^{-1} \cc a_3^{-2} \cc c_3^{-2} \cc a_4^{-2} \cc a_5^{-2} \cc e^{-1},\\
\end{aligned}
\end{equation*}
\begin{equation*} 
\begin{aligned}
&\gamma_{12} = c_1^{-1} \cc a_2^{-1} \cc b_2 \cc a_2^2 \cc c_1^2 \cc b_1^2 \cc a_1\cc e \cc a_5^2 \cc a_4^2 \cc c_3 \cc a_2 \cc c_1 \cc a_5^{-1} \cc e^{-1},\\
&\gamma_{13} = c_1^{-1} \cc a_2^{-1} \cc b_2 \cc a_2^2 \cc c_1^2 \cc b_1^2 \cc a_1\cc e \cc a_5^2 \cc a_4 \cc c_3^{-1} \cc a_4^{-2} \cc a_5^{-2} \cc e^{-1},\\
&\gamma_{14} = c_1^{-1} \cc a_2^{-1} \cc a_3 \cc c_2 \cc b_2 \cc a_2^2 \cc c_1^2 \cc b_1^2 \cc a_1\cc e \cc a_5 \cc c_1^{-1} \cc a_2^{-1} \cc c_3^{-1} \cc a_4^{-2} \cc a_5^{-2} \cc e^{-1},\\
&\gamma_{15} = b_1^{-1} \cc c_1^{-1} \cc a_2 \cc c_1^2 \cc b_1 \cc e \cc a_5^2 \cc a_4^2 \cc c_3^2 \cc a_3 \cc c_2^{-1} \cc a_3^{-2} \cc c_3^{-2} \cc a_4^{-2} \cc a_5^{-2} \cc e^{-1},\\
&\gamma_{16} = b_1^{-1} \cc c_1^{-1} \cc a_4\cc c_3 \cc a_2 \cc c_1 \cc e \cc a_5^2 \cc a_4^2 \cc c_3^2 \cc a_3^2 \cc c_2 \cc a_3^{-1} \cc c_3^{-2} \cc a_4^{-2} \cc a_5^{-2} \cc e^{-1},\\
&\gamma_{17} = b_1 \cc a_1 \cc e \cc a_5\cc c_1 \cc b_1 \cc e \cc a_5^2 \cc a_4^2 \cc c_3^2 \cc a_3 \cc c_2^{-1} \cc a_3^{-2} \cc c_3^{-2} \cc a_4^{-2} \cc a_5^{-2} \cc e^{-1},\\
&\gamma_{18} =a_1^{-1} \cc  b_1^{-2} \cc c_1^{-2} \cc a_2^{-2}\cc b_2^{-1} \cc a_3 \cc c_2 \cc b_2 \cc a_2^2  \cc c_1^2 \cc b_1^2 \cc a_1  \cc e \cc a_5^2 \cc a_4 \cc c_3^{-1} \cc a_4^{-2} \cc a_5^{-2} \cc e^{-1},\\
&\gamma_{19} =a_1^{-1} \cc  b_1^{-2} \cc c_1^{-2} \cc a_2^{-2}\cc b_2^{-1} \cc a_3 \cc c_2 \cc b_2 \cc a_2^2  \cc c_1^2 \cc b_1^2 \cc a_1  \cc e \cc a_5 \cc c_1^{-1} \cc a_2^{-1} \cc c_3^{-1} \cc a_4^{-1} \cc a_2 \cc c_1 \cc a_5^{-1} \cc e^{-1},\\
&\gamma_{20} =a_1^{-1} \cc  b_1^{-2} \cc c_1^{-2} \cc a_2^{-2}\cc b_2^{-1} \cc a_3 \cc c_2 \cc b_2 \cc a_2^2  \cc c_1^2 \cc b_1^2 \cc a_1  \cc e \cc a_5 \cc c_1^{-1} \cc a_2^{-1} \cc c_3^{-1} \cc a_4^{-2}  \cc a_5^{-2} \cc e^{-1},\\
&\gamma_{21} =a_1^{-1} \cc  b_1^{-2} \cc c_1^{-2} \cc a_2^{-2}\cc b_2^{-1} \cc c_2^{-1} \cc a_3^{-1}  \cc b_2 \cc a_2^2  \cc c_1^2 \cc b_1^2 \cc a_1  \cc e \cc a_5^2 \cc a_4 \cc c_3^{-1} \cc a_4^{-2} \cc a_5^{-2} \cc e^{-1},\\
&\gamma_{22} =a_1^{-1} \cc  b_1^{-2} \cc c_1^{-2} \cc a_2^{-2}\cc b_2^{-1} \cc a_3 \cc c_2 \cc b_2 \cc a_2^2  \cc c_1^2 \cc b_1^2 \cc a_1  \cc e \cc a_5 \cc c_1^{-1} \cc a_2^{-1} \cc c_3^{-1} \cc a_4^{-1}  \cc c_3^{-1} \cc a_4^{-2} \cc a_5^{-2} \cc e^{-1},\\
&\gamma_{23} =a_1^{-1} \cc  b_1^{-2} \cc c_1^{-2} \cc a_2^{-2}\cc b_2^{-1} \cc c_2^{-1} \cc a_3^{-1}  \cc a_2 \cc c_1   \cc e \cc a_5^2 \cc a_4^2 \cc c_3^2 \cc a_3^2 \cc c_2 \cc a_3^{-1}  \cc c_3^{-2} \cc a_4^{-2} \cc a_5^{-2} \cc e^{-1},\\
&\gamma_{24} = a_1^{-1} \cc b_1^{-1} \cc e  \cc a_5^2\cc a_4^2  \cc c_3^2 \cc a_3^2 \cc c_2 \cc b_2 \cc a_2^2 \cc c_1^2 \cc b_1^2 \cc a_1 \cc e \cc a_5^2 \cc a_4 \cc c_3^{-1} \cc a_4^{-2} \cc a_5^{-2} \cc e^{-1},\\
&\gamma_{25} = b_1^{-1} \cc c_1^{-1} \cc a_2 \cc c_1^2 \cc b_1^2 \cc a_1 \cc e \cc a_5 \cc c_1 \cc b_1 \cc e  \cc a_5^2\cc a_4^2  \cc c_3^2 \cc a_3 \cc c_2^{-1} \cc a_3^{-2}  \cc c_3^{-2} \cc a_4^{-2} \cc a_5^{-2} \cc e^{-1},\\
&\gamma_{26} = b_1^{-1} \cc c_1^{-1} \cc a_2 \cc c_1^2 \cc b_1^2 \cc a_1 \cc e \cc a_5 \cc c_1^{-1} \cc a_2^{-1} \cc a_4 \cc c_3 \cc a_2 \cc c_1 \cc e  \cc a_5^2\cc a_4^2  \cc c_3^2 \cc a_3 \cc c_2 \cc a_3^{-1}  \cc c_3^{-2} \cc a_4^{-2} \cc a_5^{-2} \cc e^{-1},\\
&\gamma_{27} = b_1^{-1} \cc c_1^{-2} \cc a_2^{-1}  \cc c_1 \cc b_1 \cc e \cc  a_5^2 \cc a_4^2 \cc c_3^2\cc a_3 \cc c_2 \cc b_2^2 \cc a_2^2 \cc c_1^2 \cc b_1^2  \cc a_1  \cc e  \cc a_5^2\cc a_4^2  \cc c_3 \cc a_2 \cc c_1 a_5^{-1} \cc e^{-1},\\
&\gamma_{28} = b_1 \cc a_1 \cc e \cc a_5 \cc c_1 \cc b_1 \cc e \cc a_5^2 \cc a_4^2 \cc c_3^2 \cc a_3\cc c_2 \cc b_2^2 \cc a_2^2 \cc c_1^2 \cc b_1^2 \cc a_1 \cc e \cc a_5^2 \cc a_4^2 \cc c_3 \cc a_4^{-1} \cc a_5^{-2} \cc e^{-1},\\
&\gamma_{29} = c_1^{-1} \cc a_2^{-1} \cc c_3^{-1} \cc a_4^{-1} \cc c_1 \cc b_1 \cc e  \cc a_5^2 \cc a_4^2 \cc c_3^2 \cc a_3\cc c_2 \cc b_2^2 \cc a_2^2 \cc c_1^2 \cc b_1^2 \cc a_1 \cc e \cc a_5^2 \cc a_4^2 \cc c_3\cc a_2 \cc c_1\cc a_5^{-1} \cc e^{-1},\\
&\gamma_{30} = c_1^{-1} \cc a_2^{-1} \cc c_3^{-1} \cc a_4^{-1} \cc c_1 \cc b_1 \cc e  \cc a_5^2 \cc a_4^2 \cc c_3^2 \cc a_3\cc c_2 \cc b_2^2 \cc a_2^2 \cc c_1^2 \cc b_1^2 \cc a_1 \cc e \cc a_5^2 \cc a_4^2 \cc c_3\cc a_4^{-1} \cc a_5^{-2} \cc e^{-1},\\
&\gamma_{31} = c_1^{-1} \cc a_2^{-1} \cc c_3^{-1} \cc a_4^{-1} \cc c_1 \cc b_1 \cc e  \cc a_5^2 \cc a_4^2 \cc c_3^2 \cc a_3\cc c_2 \cc b_2^2 \cc a_2^2 \cc c_1^2 \cc b_1^2 \cc a_1 \cc e \cc a_5^2 \cc a_4^2 \cc c_3\cc a_4 \cc c_3 \cc a_2 \cc c_1\cc a_5^{-1} \cc e^{-1},\\
&\gamma_{32} =a_1^{-1} \cc b_1^{-2} \cc c_1^{-2} \cc a_2^{-2} \cc b_2^{-1} \cc a_2 \cc c_1  \cc e  \cc a_5^2 \cc a_4^2 \cc c_3^2 \cc a_3^2\cc c_2 \cc b_2 \cc a_2^2 \cc c_1^2 \cc b_1^2 \cc a_1 \cc e \cc a_5^2 \cc a_4 \cc c_3^{-1}\cc a_4^{-2} \cc a_5^{-2} \cc e^{-1},\\
&\gamma_{33} =a_1^{-1} \cc b_1^{-2} \cc c_1^{-2} \cc a_2^{-2} \cc b_2^{-1}\cc c_2^{-1} \cc a_3^{-1} \cc a_2 \cc c_1  \cc e  \cc a_5^2 \cc a_4^2 \cc c_3^2 \cc a_3^2\cc c_2 \cc b_2 \cc a_2^2 \cc c_1^2 \cc b_1^2 \cc a_1 \cc e \cc a_5^2 \cc a_4 \cc c_3^{-1}\cc a_4^{-2} \cc a_5^{-2} \cc e^{-1},\\
&\gamma_{34} =a_1^{-1} \cc b_1^{-2} \cc c_1^{-2} \cc a_2^{-2} \cc b_2^{-1}\cc c_2^{-1} \cc a_3^{-1} \cc a_2 \cc c_1  \cc e  \cc a_5^2 \cc a_4^2 \cc c_3^2 \cc a_3^2\cc c_2 \cc a_3 \cc c_2 \cc b_2 \cc a_2^2 \cc c_1^2 \cc b_1^2 \cc a_1 \cc e \cc a_5^2 \cc a_4 \cc c_3^{-1}\cc a_4^{-2} \cc a_5^{-2} \cc e^{-1},\\
&\gamma_{35} = b_1^{-1} \cc c_1^{-1} \cc a_2 \cc c_1^2 \cc b_1^2 \cc a_1 \cc e \cc a_5 \cc c_1^{-1} \cc a_2^{-1} \cc c_1 \cc b_1 \cc e \cc a_5^2 \cc a_4^2 \cc c_3^2 \cc a_3 \cc c_2 \cc b_2^2 \cc a_2^2 \cc c_1^2 \cc b_1^2 \cc a_1 \cc e \cc a_5^2 \cc a_4^2 \cc c_3 \cc a_2 \cc c_1 \cc a_5^{-1} \cc e^{-1},
\end{aligned}
\end{equation*}
and 
\begin{equation*}
\begin{aligned}
& \zeta_1 = a_1^{-1} \cc a_2^{-2} \cc c_1^{-2} \cc b_1^{-1}, \qquad \qquad \zeta_2 = b_1 \cc a_2^{-2} \cc c_1^{-2} \cc b_1^{-1},\\
&\zeta_3 = c_1 \cc b_1^2 \cc a_2^{-2} \cc c_1^{-2} \cc b_1^{-1}, \ \quad \zeta_4 = b_1^{-1} \cc c_1^{-1} \cc a_2 \cc c_1^2 \cc b_1^2 \cc a_1\cc e \cc a_5 \cc c_1 \cc b_1^2 \cc a_2^{-2} \cc c_1^{-2} \cc b_1^{-1},\\
&\mu_1 = b_1 \cc c_1^2 \cc a_2^2 \cc b_1^{-2} \cc c_1^{-2} \cc a_2^{-1} \cc a_4 \cc  c_3 \cc a_2 \cc c_1 \cc a_5^{-1} \cc e^{-1}
, \qquad \qquad\mu_2 = b_1 \cc c_1^2 \cc a_2^2 \cc b_1^{-2} \cc c_1^{-2} \cc a_2^{-1} \cc c_3^{-1} \cc a_4^{-1} \cc a_2 \cc c_1 \cc a_5^{-1} \cc e^{-1},\\
&\mu_3 = b_1 \cc c_1^2 \cc a_2^2 \cc b_1^{-2} \cc c_1^{-2} \cc a_2^{-1} \cc c_3^{-1} \cc a_4^{-2} \cc a_5^{-2} \cc e^{-1}
, \qquad \qquad\mu_4 = b_1 \cc c_1^2 \cc a_2^2 \cc b_1^{-2} \cc c_1^{-2} \cc a_2^{-1} \cc c_3^{-1}\cc a_4^{-1} \cc c_3^{-1} \cc a_4^{-2} \cc a_5^{-2} \cc e^{-1},\\
&\mu_5 = b_1 \cc c_1^2 \cc a_2^2 \cc b_1^{-2} \cc c_1^{-2} \cc a_2^{-2} \cc b_2^{-2} \cc c_2^{-1} \cc a_3^{-1} \cc c_3^{-2} \cc a_4^{-2} \cc a_5^{-2} \cc e^{-1},\\
&\mu_6 = b_1 \cc c_1^2 \cc a_2^2 \cc b_1^{-2} \cc c_1^{-2} \cc a_2^{-2} \cc b_2^{-1} \cc c_2^{-1} \cc a_3^{-2} \cc c_3^{-2} \cc a_4^{-2} \cc a_5^{-2} \cc e^{-1},\\
&\mu_7 = b_1 \cc c_1^2 \cc a_2^2 \cc b_1^{-2} \cc c_1^{-2} \cc a_2^{-1} \cc c_1 \cc b_1 \cc e \cc a_5^2 \cc a_4^2 \cc c_3 \cc a_4^{-1} \cc a_5^{-2} \cc e^{-1},\\
&\mu_8 = b_1 \cc c_1^2 \cc a_2^2 \cc b_1^{-2} \cc c_1^{-2} \cc a_2^{-2} \cc b_2^{-1} \cc c_2^{-1} \cc a_3^{-1} \cc c_2^{-1} \cc a_3^{-2} \cc c_3^{-2} \cc a_4^{-2} \cc a_5^{-2} \cc e^{-1},\\
&\mu_9 = b_1 \cc c_1^2 \cc a_2^2 \cc b_1^{-2} \cc c_1^{-2} \cc a_2^{-1}\cc b_2 \cc a_2^2  \cc c_1^2 \cc b_1^2 \cc a_1 \cc e \cc a_5^2 \cc a_4^2 \cc c_3\cc a_2 \cc c_1 \cc a_5^{-1} \cc e^{-1},\\
&\mu_{10} = b_1 \cc c_1^2 \cc a_2^2 \cc b_1^{-2} \cc c_1^{-2} \cc a_2^{-1}\cc b_2 \cc a_2^2  \cc c_1^2 \cc b_1^2 \cc a_1 \cc e \cc a_5^2 \cc a_4 \cc c_3^{-1}\cc a_4^{-2} \cc a_5^{-2} \cc e^{-1},\\
&\mu_{11} = b_1 \cc c_1^2 \cc a_2^2 \cc b_1^{-2} \cc c_1^{-2} \cc a_2^{-2}\cc b_2^{-1} \cc a_3 \cc c_2 \cc b_2 \cc a_2^2  \cc c_1^2 \cc b_1^2 \cc a_1 \cc e \cc a_5^2 \cc a_4 \cc c_3^{-1}\cc a_4^{-2} \cc a_5^{-2} \cc e^{-1},\\
\end{aligned}
\end{equation*}
\begin{equation*}
\begin{aligned}
&\mu_{12} = b_1 \cc c_1^2 \cc a_2^2 \cc b_1^{-2} \cc c_1^{-2} \cc a_2^{-2}\cc b_2^{-1}\cc c_2^{-1} \cc a_3^{-1}  \cc b_2 \cc a_2^2  \cc c_1^2 \cc b_1^2 \cc a_1 \cc e \cc a_5^2 \cc a_4 \cc c_3^{-1}\cc a_4^{-2} \cc a_5^{-2} \cc e^{-1},\\
&\mu_{13} = b_1 \cc c_1^2 \cc a_2^2 \cc b_1^{-2} \cc c_1^{-2} \cc a_2^{-1}\cc a_4 \cc c_3 \cc a_2 \cc c_1 \cc e \cc a_5^2 \cc a_4^2 \cc c_3^2 \cc a_3^2 \cc c_2 \cc a_3^{-1} \cc c_3^{-2}\cc a_4^{-2} \cc a_5^{-2} \cc e^{-1},\\
&\mu_{14} = b_1 \cc c_1^2 \cc a_2^2 \cc b_1^{-2} \cc c_1^{-2} \cc a_2^{-2}\cc b_2^{-1} \cc c_2^{-1} \cc a_3^{-1}  \cc a_2 \cc c_1 \cc e \cc a_5^2 \cc a_4^2 \cc c_3^2 \cc a_3^2 \cc c_2 \cc a_3^{-1} \cc c_3^{-2}\cc a_4^{-2} \cc a_5^{-2} \cc e^{-1},\\
&\mu_{15} = b_1 \cc c_1^2 \cc a_2^2 \cc b_1^{-2} \cc c_1^{-2} \cc a_2^{-1}\cc c_1 \cc b_1 \cc e \cc a_5^2 \cc a_4^2 \cc c_3^2 \cc a_3 \cc c_2 \cc b_2^2 \cc a_2^2  \cc c_1^2 \cc b_1^2 \cc a_1 \cc e \cc a_5^2 \cc a_4^2 \cc c_3\cc a_2 \cc c_1 \cc a_5^{-1} \cc e^{-1},\\
&\mu_{16} = b_1 \cc c_1^2 \cc a_2^2 \cc b_1^{-2} \cc c_1^{-2} \cc a_2^{-2}\cc b_2^{-1} \cc a_2 \cc c_1 \cc e \cc a_5^2 \cc a_4^2 \cc c_3^2 \cc a_3^2 \cc c_2 \cc b_2 \cc a_2^2  \cc c_1^2 \cc b_1^2 \cc a_1 \cc e \cc a_5^2 \cc a_4 \cc c_3^{-1}\cc a_4^{-2} \cc a_5^{-2} \cc e^{-1},\\
&\mu_{17} = b_1 \cc c_1^2 \cc a_2^2 \cc b_1^{-2} \cc c_1^{-2} \cc a_2^{-1}\cc c_3^{-1} \cc a_4^{-1} \cc c_1 \cc b_1 \cc e \cc a_5^2 \cc a_4^2 \cc c_3^2 \cc a_3 \cc c_2 \cc b_2^2 \cc a_2^2  \cc c_1^2 \cc b_1^2 \cc a_1 \cc e \cc a_5^2 \cc a_4^2 \cc c_3\cc a_2 \cc c_1 \cc a_5^{-1} \cc e^{-1},\\
&\mu_{18} = b_1 \cc c_1^2 \cc a_2^2 \cc b_1^{-2} \cc c_1^{-2} \cc a_2^{-1}\cc c_3^{-1} \cc a_4^{-1} \cc c_1 \cc b_1 \cc e \cc a_5^2 \cc a_4^2 \cc c_3^2 \cc a_3 \cc c_2 \cc b_2^2 \cc a_2^2  \cc c_1^2 \cc b_1^2 \cc a_1 \cc e \cc a_5^2 \cc a_4^2 \cc c_3\cc a_4^{-1} \cc a_5^{-2} \cc e^{-1},\\
&\mu_{19} = b_1 \cc c_1^2 \cc a_2^2 \cc b_1^{-2} \cc c_1^{-2} \cc a_2^{-2}\cc b_2^{-1}\cc c_2^{-1} \cc a_3^{-1} \cc a_2 \cc c_1 \cc e \cc a_5^2 \cc a_4^2 \cc c_3^2 \cc a_3^2 \cc c_2 \cc b_2 \cc a_2^2  \cc c_1^2 \cc b_1^2 \cc a_1 \cc e \cc a_5^2 \cc a_4 \cc c_3^{-1}\cc a_4^{-2} \cc a_5^{-2} \cc e^{-1},\\
&\mu_{20} = b_1 \cc c_1^2 \cc a_2^2 \cc b_1^{-2} \cc c_1^{-2} \cc a_2^{-1}\cc c_3^{-1} \cc a_4^{-1} \cc c_1 \cc b_1 \cc e \cc a_5^2 \cc a_4^2 \cc c_3^2 \cc a_3 \cc c_2 \cc b_2^2 \cc a_2^2  \cc c_1^2 \cc b_1^2 \cc a_1 \cc e \cc a_5^2 \cc a_4^2 \cc c_3\cc a_4 \cc c_3\cc a_2 \cc c_1 \cc a_5^{-1} \cc e^{-1},\\
&\mu_{21} = b_1 \cc c_1^2 \cc a_2^2 \cc b_1^{-2} \cc c_1^{-2} \cc a_2^{-2}\cc b_2^{-1}\cc c_2^{-1} \cc a_3^{-1} \cc a_2 \cc c_1 \cc e \cc a_5^2 \cc a_4^2 \cc c_3^2 \cc a_3^2 \cc c_2\cc a_3 \cc c_2  \cc b_2 \cc a_2^2  \cc c_1^2 \cc b_1^2 \cc a_1 \cc e \cc a_5^2 \cc a_4 \cc c_3^{-1}\cc a_4^{-2} \cc a_5^{-2} \cc e^{-1},\\
&\eta_1 = b_1 \cc c_1^2 \cc a_2^2 \cc b_1^{-2} \cc c_1^{-2} \cc a_2^{-1} \cc a_3 \cc c_2 \cc b_2 \cc a_2^{2} \cc c_1^{2} \cc b_1^{2} \cc a_1 \cc e \cc a_5 \cc c_1 \cc b_1^{2} \cc a_2^{-2} \cc c_1^{-2} \cc b_1^{-1},\\
&\eta_2 = b_1 \cc c_1^2 \cc a_2^2 \cc b_1^{-2} \cc c_1^{-2} \cc a_2^{-2}\cc b_2^{-1}  \cc a_3 \cc c_2 \cc b_2 \cc a_2^{2} \cc c_1^{2} \cc b_1^{2} \cc a_1 \cc e \cc a_5 \cc c_1 \cc b_1^{2} \cc a_2^{-2} \cc c_1^{-2} \cc b_1^{-1}.\\
\end{aligned}
\end{equation*}
The admissible $3$-tuples are 
\begin{equation*}
\begin{aligned}
A= \{ & (1,7,6),(1,11,14),(1,24,15),(1,24,25),(1,24,26),(1,24,35),(2,7,6),(3,7,6),(4,21,5),\\
&(4,10,14),(4,6,28),(5,15,3),(5,16,4),(5,2,7),(5,1,11),(5,1,24),\\
&(5,16,27),(5,15,29),(5,15,30),(5,15,31),(6,13,5),(6,28,18),(6,28,19),(6,12,21),\\
&(6,28,22),(6,12,23),(6,14,32),(6,12,33),(6,12,34),(7,6,12),(7,6,13),(7,6,28),\\
&(8,21,5),(9,5,2),(9,5,16),(10,14,32),(11,14,5),(11,14,32),(12,21,5),(12,33,5),\\
&(12,34,15),(12,23,28),(13,5,16),(14,5,2),(14,32,8),(14,5,15),(14,32,26),\\
&(15,14,15),(15,3,7),(15,29,7),(15,30,20),(15,31,21),(15,30,22),(15,30,32),\\
&(16,4,10),(16,4,21),(16,27,23),(17,3,7),(18,15,30),(19,7,6),(20,32,8),(21,5,1),\\
&(21,5,16),(22,5,16),(23,17,3),(23,4,6),(23,28,6),(23,4,10),(23,28,18),(24,26,1),\\
&(24,25,3),(24,35,7),(24,15,14),(24,35,23),(24,26,27),(25,3,7),(26,1,7),(26,1,11),\\
&(26,27,23),(27,23,17),(27,23,28),(28,22,5),(28,19,7),(28,6,13),(28,6,14),\\
&(28,18,15),(29,7,6),(30,22,5),(30,32,5),(30,32,8),(30,32,9),(30,20,32),(31,21,5),\\
&(32,5,1),(32,26,1),(32,9,5),(32,8,21),(33,5,1),(34,15,14),(34,15,30),\\
&(35,23,4),(35,7,6),(35,23,17)\}
\end{aligned}
\end{equation*}
Let $V= Span \, \Gamma$ and $W= Span\, \Delta$. The splitting map $S: V\to W$ is given by 
\begin{equation*}
S: \left\{ \begin{aligned} & \gamma_1 \mapsto \mu_6 + \zeta_3, \quad \gamma_2 \mapsto \mu_8+ \zeta_3, \quad \gamma_3 \mapsto \mu_5 + \gamma_1 + \gamma_{11} + \zeta_3, \quad \gamma_4 \mapsto \eta_2, \quad \gamma_5 \mapsto \mu_7 + \zeta_1,\\
&\gamma_6  \mapsto \mu_3 + \zeta_1, \quad \gamma_7 \mapsto \eta_1, \quad \gamma_8 \mapsto \mu_{14} +\gamma_{17} + \zeta_3, \quad \gamma_9 \mapsto \mu_{14} + \zeta_2, \quad \gamma_{10} \mapsto \mu_4 + \zeta_1,\\
&\gamma_{11} \mapsto \mu_{10} + \zeta_1, \quad \gamma_{12} \mapsto \mu_5 + \gamma_{15} + \zeta_3, \quad \gamma_{13} \mapsto \mu_5 + \gamma_2 + \zeta_1, \quad \gamma_{14} \mapsto \mu_5 + \gamma_{16} + \zeta_2,\\
&\gamma_{15} \mapsto \mu_{12} + \zeta_1, \quad \gamma_{16} \mapsto \mu_6 + \gamma_{28} + \zeta_1, \quad \gamma_{17} \mapsto \mu_{16} + \zeta_1, \quad \gamma_{18} \mapsto \mu_1 + \zeta_1, \\
& \gamma_{19} \mapsto \mu_{13} + \gamma_1 + \gamma_{11} + \zeta_3, \quad \gamma_{20} \mapsto \mu_{13} + \zeta_2, \quad \gamma_{21} \mapsto \mu_2+ \zeta_1, \quad \gamma_{22} \mapsto \mu_{13} + \gamma_1 + \zeta_1,\\
&\gamma_{23} \mapsto \mu_{18} + \zeta_1, \quad \gamma_{24} \mapsto \mu_9 + \zeta_1, \quad \gamma_{25} \mapsto \mu_{19} + \zeta_1,  \quad \gamma_{26} \mapsto \mu_{14} + \gamma_{28} + \zeta_1,\\
&\gamma_{27} \mapsto \mu_{11} + \gamma_{15} + \zeta_3, \quad \gamma_{28} \mapsto \mu_{16} + \zeta_4, \quad \gamma_{29} \mapsto \mu_5 + \gamma_1 + \gamma_{24} + \gamma_{15} + \zeta_3, \\
&\gamma_{30} \mapsto \mu_5 + \gamma_1 + \gamma_{24}+ \zeta_4, \quad \gamma_{31} \mapsto \mu_5 + \gamma_1 + \gamma_{24} + \gamma_{25} + \zeta_3, \quad \gamma_{32} \mapsto \mu_{15} + \zeta_1,\\
&\gamma_{33} \mapsto \mu_{17} + \zeta_1, \quad \gamma_{34} \mapsto \mu_{20} + \zeta_1, \quad \gamma_{35} \mapsto \mu_{21} + \gamma_{15} + \zeta_3.\\
 \end{aligned}\right.
\end{equation*}
and the merging map $M:V \to W$ is
\begin{equation*}
M: \left\{ \begin{aligned} &\gamma_1 \mapsto \gamma_1, \quad \gamma_2 \mapsto \gamma_2, \quad \gamma_3\mapsto \zeta_3 + \mu_2, \quad \gamma_4 \mapsto \zeta_2 + \mu_7, \quad \gamma_5 \mapsto \zeta_1 + \mu_5, \quad \gamma_6 \mapsto \zeta_1 + \mu_6,\\
&\gamma_7 \mapsto \zeta_1+ \mu_7, \quad \gamma_8 \mapsto \zeta_4+\mu_1, \quad \gamma_9 \mapsto \zeta_4+ \mu_3, \quad \gamma_{10} \mapsto \zeta_1 + \mu_8, \quad \gamma_{11} \mapsto \gamma_{11}, \quad \gamma_{12} \mapsto \zeta_3 + \mu_9,\\
&\gamma_{13} \mapsto \zeta_3 + \mu_{10}, \quad \gamma_{14} \mapsto \zeta_3 + \eta_1 + \mu_3, \quad \gamma_{15} \mapsto \gamma_{15}, \quad \gamma_{16} \mapsto \gamma_{16}, \quad \gamma_{17} \mapsto \gamma_{17},\\
&\gamma_{18} \mapsto \zeta_1 + \mu_{11}, \quad \gamma_{19} \mapsto \zeta_1 + \eta_2 + \mu_2, \quad \gamma_{20} \mapsto \zeta_1+ \eta_2 + \mu_3, \quad \gamma_{21} \mapsto \zeta_1 + \mu_{12},\\
&\gamma_{22} \mapsto \zeta_1 + \eta_2 + \mu_4, \quad \gamma_{23} \mapsto \zeta_1 + \mu_{14}, \quad \gamma_{24} \mapsto \gamma_{24}, \quad \gamma_{25} \mapsto \gamma_{25}, \quad \gamma_{26} \mapsto \zeta_4 + \mu_{13},\\
&\gamma_{27} \mapsto \zeta_2 + \mu_{15}, \quad \gamma_{28} \mapsto \gamma_{28}, \quad \gamma_{29} \mapsto \zeta_3 + \mu_{17}, \quad \gamma_{30} \mapsto \zeta_3 + \mu_{18}, \quad \gamma_{31} \mapsto \zeta_3 + \mu_{20},\\
&\gamma_{32} \mapsto \zeta_1 + \mu_{16}, \quad \gamma_{33} \mapsto \zeta_1 + \mu_{19}, \quad \gamma_{34} \mapsto \zeta_1 + \mu_{21}, \quad \gamma_{35} \mapsto \zeta_4 + \mu_{15}.\\
\end{aligned}\right.
\end{equation*}

There is a $v_\lambda \in V$ such that $Sv_\lambda = \lambda Mv_\lambda$ where $\lambda\approx 1.27034 $ is the largest real root of $\chi(t) =t^{14} - t^{11}-t^{10} -t^4-t^3 +1$
and
\begin{equation*}
\begin{aligned}
 v_\lambda \approx ( &4.163,1,1.258,2.031,7.949,5.339,4.226,0.792,0.688,0.787,1.614,1.623,\\
 &1.270,3.327,3.308,2.612,0.623,1.006,0.792,0.187,2.604,0.620,2.156,\\
 &2.062,0.488,1.258,1.277,3.052,0.302,1.697,0.620,2.893,0.384,0.787,1)\\
 \end{aligned}
 \end{equation*}

\subsection{Orbitdata $3,4,5$ with a cyclic permutation}\label{B:345}
Let $f_\mathbf{R}: X(\mathbf{R}) \to X(\mathbf{R})$ be a real diffeomorphism associated with the orbit data $3,4,5$ and a cyclic permutation. 
Let 
\begin{equation*}
\begin{aligned}
&\Gamma = \{ \gamma_1, \dots, \gamma_{33} \},\\
&\Delta = \{ \gamma_1, \gamma_2,\gamma_{3},\gamma_{4}, \gamma_{5},\gamma_{8},\gamma_{11},\gamma_{12}, \gamma_{16},\gamma_{17},\gamma_{18},\gamma_{21},\gamma_{27}, \zeta_1,\dots,\zeta_{8}, \mu_1, \dots, \mu_{11}, \eta_1,\dots. \eta_4 \} 
\end{aligned}
\end{equation*}
be two ordered sets of non-cyclic words in $\pi_1( X(\mathbf{R}),P_{fix})$ where
\begin{equation*}
\begin{aligned}
&\gamma_1 = a_1^{-1} \cc b_1^{-1} \cc e^{-1}, \qquad  \gamma_2= b_1^{-1} \cc c_1^{-1} \cc a_4^{-1} \cc c_4^{-1} \cc c_1 \cc b_1 \cc e^{-1}, \qquad  \gamma_3= b_1^{-1} \cc c_1^{-1} \cc c_4^{-1} \cc a_5^{-1} \cc c_1 \cc b_1 \cc e^{-1}\\
&\gamma_4 = a_1^{-2} \cc b_1^{-2} \cc a_2^{-2} \cc b_2^{-2} \cc c_2^{-1} \cc b_2 \cc a_2^2 \cc c_1^2 \cc b_1^2 \cc a_1 \cc b_1^{-1} \cc c_1^{-1} \cc a_4^{-1} \cc c_4^{-1} \cc c_1 \cc b_1 \cc e^{-1},\\ 
&\gamma_5 = a_1^{-2} \cc b_1^{-2} \cc c_1^{-2} \cc a_2^{-2} \cc b_2^{-2} \cc c_2^{-1} \cc b_2 \cc a_2 \cc b_3^{-1} \cc c_3^{-1} \cc c_1 \cc b_1 \cc a_1^{-1} \cc b_1^{-2} \cc c_1^{-2} \cc a_2^{-2} \cc b_2^{-1} \cc a_2\cc c_1 \cc e^{-1},\\
&\gamma_6 =a_1^{-2} \cc b_1^{-2} \cc c_1^{-2} \cc a_2^{-2} \cc b_2^{-2} \cc c_2^{-1} \cc b_2 \cc a_2^2 \cc c_1^2 \cc b_1^2 \cc a_1 \cc b_1^{-1} \cc c_1^{-1} \cc c_3 \cc b_3 \cc a_3^{-1} \cc b_3^{-2} \cc c_3^{-2} \cc a_4^{-2} \cc c_4^{-1} \cc c_1 \cc b_1 \cc e^{-1},\\
&\gamma_7 =a_1^{-2} \cc b_1^{-2} \cc c_1^{-2} \cc a_2^{-2} \cc b_2^{-2} \cc c_2^{-1} \cc b_2 \cc a_2^2 \cc c_1^2 \cc b_1^2 \cc a_1 \cc b_1^{-1} \cc c_1^{-1} \cc c_3 \cc b_3 \cc a_3^{-1} \cc b_3^{-2} \cc c_3^{-2} \cc a_4^{-2} \cc c_4^{-2} \cc a_5^{-2} \cc e^{-1},\\
&\gamma_8 = a_1^{-2} \cc b_1^{-2} \cc c_1^{-2} \cc a_2^{-2} \cc b_2^{-2} \cc c_2^{-1} \cc b_2 \cc a_2 \cc b_3^{-1} \cc c_3^{-1} \cc c_1 \cc b_1 \cc a_1^{-1} \cc b_1^{-2} \cc c_1^{-2} \cc a_2^{-2} \cc b_2^{-1} \cc c_2 \cc b_2^2 \cc a_2^2 \cc c_1^2 \cc b_1^2 \cc a_1^2 \cc e \cc c_1 \cc b_1 \cc e^{-1},\\
&\gamma_9 = a_1^{-2} \cc b_1^{-2} \cc c_1^{-2} \cc a_2^{-2} \cc b_2^{-2} \cc c_2^{-1} \cc b_2 \cc a_2^2 \cc c_1^2 \cc b_1^2 \cc a_1 \cc b_1^{-1} \cc c_1^{-1} \cc c_3 \cc b_3 \cc a_2^{-1} \cc b_2^{-1} \cc c_2 \cc b_2^2 \cc a_2^2 \cc c_1^2 \cc b_1^2 \cc a_1^2 \cc e \cc b_1^{-1}\cc c_1^{-1}\cc a_5^{-1} \cc e^{-1},\\
&\gamma_{10} =a_1^{-2} \cc b_1^{-2} \cc c_1^{-2} \cc a_2^{-2} \cc b_2^{-2} \cc c_2^{-1} \cc b_2 \cc a_2 \cc b_2 \cc a_2^2 \cc c_1^2 \cc b_1^2 \cc a_1 \cc b_1^{-1} \cc c_1^{-1} \cc c_3 \cc b_3 \cc a_2^{-1} \cc b_2^{-1} \cc c_2 \\
&\phantom{ASEFG} \cc b_2^2 \cc a_2^2 \cc c_1^2 \cc b_1^2 \cc a_1^2 \cc e \cc b_1^{-1}\cc c_1^{-1}\cc a_5^{-1} \cc e^{-1},\\
&\gamma_{11} = a_1^{-2} \cc b_1^{-2} \cc c_1^{-2} \cc a_2^{-2} \cc b_2^{-2} \cc c_2^{-1} \cc b_2 \cc a_2 \cc b_3^{-1} \cc c_3^{-1} \cc c_1 \cc b_1 \cc a_1^{-1} \cc b_1^{-2} \cc c_1^{-2} \cc a_2^{-2} \cc b_2^{-1}\cc a_2^{-1}  \cc b_2^{-1}  \cc c_2  \\
&\phantom{ASEFG} \cc b_2^2 \cc a_2^2 \cc c_1^2 \cc b_1^2 \cc a_1^2 \cc e \cc c_1 \cc b_1 \cc e^{-1},\\
&\gamma_{12} = b_1^{-1} \cc c_1^{-1} \cc c_4 \cc a_4^2 \cc c_3^2 \cc b_3^2 \cc a_3 \cc b_3^{-1} \cc c_3^{-1} \cc c_1 \cc b_1 \cc a_1^{-1} \cc b_1^{-2} \cc c_1^{-2} \cc a_2^{-2} \cc b_2^{-1} \cc c_2 \cc b_2^2 \cc a_2^2 \cc c_1^2 \cc b_1^2 \cc a_1^2 \cc e\\
&\phantom{ASEFG} \cc b_1 \cc a_1 \cc b_1^{-1} \cc c_1^{-1} \cc a_4^{-1} \cc c_4^{-1} \cc c_1 \cc b_1 \cc e^{-1},\\
&\gamma_{13}=a_1^{-2} \cc b_1^{-2} \cc c_1^{-2} \cc a_2^{-2} \cc b_2^{-2} \cc c_2^{-1} \cc b_2 \cc a_2 \cc b_2 \cc a_2^2 \cc c_1^2 \cc b_1^2 \cc a_1 \cc b_1^{-1} \cc c_1^{-1} \cc c_3 \cc b_3 \cc a_2^{-1} \cc b_2^{-1} \cc c_2 \\
&\phantom{ASEFG} \cc b_2^2 \cc a_2^2 \cc c_1^2 \cc b_1^2 \cc a_1^2 \cc e \cc b_1^{-1}\cc c_1^{-1}\cc c_4 \cc a_4 \cc c_1 \cc b_1 \cc a_1^{-1}\cc b_1^{-1} \cc e^{-1},\\
&\gamma_{14}=a_1^{-2} \cc b_1^{-2} \cc c_1^{-2} \cc a_2^{-2} \cc b_2^{-2} \cc c_2^{-1} \cc b_2 \cc a_2 \cc b_3\cc a_3 \cc c_2 \cc b_2^2 \cc a_2^2 \cc c_1^2 \cc b_1^2 \cc a_1 \cc b_1^{-1} \cc c_1^{-1} \cc c_3 \cc b_3 \cc a_2^{-1} \cc b_2^{-1} \cc c_2 \\
&\phantom{ASEFG} \cc b_2^2 \cc a_2^2 \cc c_1^2 \cc b_1^2 \cc a_1^2 \cc e \cc b_1^{-1}\cc c_1^{-1}\cc c_4 \cc a_4 \cc c_1 \cc b_1 \cc a_1^{-1}\cc b_1^{-1} \cc e^{-1},\\
\end{aligned}
\end{equation*}
\begin{equation*}
\begin{aligned}
&\gamma_{15} = a_1^{-2} \cc b_1^{-2} \cc c_1^{-2} \cc a_2^{-2} \cc b_2^{-2} \cc c_2^{-1} \cc b_2 \cc a_2 \cc b_3^{-1} \cc c_3^{-1} \cc c_1 \cc b_1 \cc a_1^{-1} \cc b_1^{-2} \cc c_1^{-2} \cc a_2^{-2} \cc b_2^{-2}\cc c_2^{-1}  \cc b_2  \cc a_2^2 \cc  c_1^2 \cc b_1^2 \cc a_1  \\
&\phantom{ASEFG} \cc b_1^{-1} \cc c_1^{-1} \cc c_3 \cc b_3 \cc a_2^{-1} \cc b_2^{-1} \cc c_2 \cc b_2^2 \cc a_2^2\cc c_1^2 \cc b_1^2 \cc a_1^2 \cc e \cc b_1^{-1} \cc c_1^{-1} \cc a_5^{-1} \cc e^{-1},\\
&\gamma_{16} = b_1^{-1} \cc c_1^{-1} \cc a_4 \cc c_3 \cc a_2^{-1} \cc b_2^{-1} \cc c_2 \cc b_2^2 \cc a_2^2 \cc c_1^2 \cc b_1^2 \cc a_1^2 \cc e \cc b_1^{-1} \cc c_1^{-1} \cc c_4 \cc a_4^2 \cc c_3^2 \cc b_3^2 \cc a_3 \cc b_3^{-1} \cc c_3^{-1} \cc c_1 \cc b_1 \cc a_1^{-1} \cc b_1^{-2} \cc c_1^{-2} \cc a_2^{-2} \cc b_2^{-1} \\
&\phantom{ASEFG}\cc c_2 \cc b_2^2 \cc a_2^2 \cc c_1^2 \cc b_1^2 \cc a_1^2 \cc e \cc b_1 \cc a_1 \cc e \cc a_5  \cc c_1 \cc b_1 \cc e^{-1},\\
&\gamma_{17} = b_1^{-1} \cc c_1^{-1} \cc a_4 \cc c_3 \cc a_2^{-1} \cc b_2^{-1} \cc c_2 \cc b_2^2 \cc a_2^2 \cc c_1^2 \cc b_1^2 \cc a_1^2 \cc e \cc b_1^{-1} \cc c_1^{-1} \cc c_4 \cc a_4^2 \cc c_3^2 \cc b_3^2 \cc a_3 \cc b_3^{-1} \cc c_3^{-1} \cc c_1 \cc b_1 \cc a_1^{-1} \cc b_1^{-2} \cc c_1^{-2} \cc a_2^{-2} \cc b_2^{-1} \\
&\phantom{ASEFG}\cc c_2 \cc b_2^2 \cc a_2^2 \cc c_1^2 \cc b_1^2 \cc a_1^2 \cc e \cc b_1 \cc a_1 \cc b_1^{-1} \cc c_1^{-1} \cc a_4^{-1}\cc c_4^{-1}  \cc c_1 \cc b_1 \cc e^{-1},\\
&\gamma_{18} = a_1^{-2} \cc b_1^{-2} \cc c_1^{-2} \cc a_2^{-2} \cc b_2^{-2} \cc c_2^{-1} \cc b_2 \cc a_2 \cc c_3^{-1} \cc a_4^{-1} \cc c_1 \cc b_1 \cc e \cc a_5^2 \cc c_4^2 \cc a_4^2 \cc c_3^2 \cc b_3^2 \cc a_3 \cc b_3^{-1} \cc c_3^{-1} \cc c_1 \cc b_1 \cc a_1^{-1} \cc b_1^{-2} \cc c_1^{-2} \cc a_2^{-2} \cc b_2^{-2}\\
&\phantom{ASEFG} \cc c_2 \cc b_2^2 \cc a_2^2 \cc c_1^2 \cc b_1^2 \cc a_1^2 \cc e \cc b_1 \cc a_1 \cc b_1^{-1} \cc c_1^{-1} \cc a_4^{-1} \cc c_4^{-1} \cc c_1 \cc b_1 \cc e^{-1},\\
&\gamma_{19} = a_1^{-2} \cc b_1^{-2} \cc c_1^{-2} \cc a_2^{-2} \cc b_2^{-2} \cc c_2^{-1} \cc b_2 \cc a_2\cc b_2 \cc a_2^2\cc c_1^2 \cc b_1^2 \cc a_1 \cc b_1^{-1} \cc c_1^{-1} \cc c_3 \cc b_3 \cc a_2^{-1} \cc b_2^{-1} \cc c_2 \cc b_2^2 \cc a_2^2 \cc c_1^2 \cc b_1^2 \cc a_1^2 \cc e \cc b_1^{-1} \cc c_1^{-1} \\
&\phantom{ASEFG} \cc c_4 \cc a_4 \cc c_1 \cc b_1 \cc a_1^{-1} \cc b_1^{-2} \cc c_1^{-2} \cc a_2^{-2} \cc b_2^{-1} \cc c_2 \cc b_2^2 \cc a_2^2 \cc c_1^2 \cc b_1^2 \cc a_1^2 \cc e \cc b_1^{-1} \cc c_1^{-1} \cc e^{-1},\\
&\gamma_{20} =  a_1^{-2} \cc b_1^{-2} \cc c_1^{-2} \cc a_2^{-2} \cc b_2^{-2} \cc c_2^{-1} \cc b_2 \cc a_2\cc b_2 \cc a_2^2\cc c_1^2 \cc b_1^2 \cc a_1 \cc b_1^{-1} \cc c_1^{-1} \cc c_3 \cc b_3 \cc a_2^{-1} \cc b_2^{-1} \cc c_2 \cc b_2^2 \cc a_2^2 \cc c_1^2 \cc b_1^2 \cc a_1^2 \cc e \cc b_1^{-1} \cc c_1^{-1} \\
&\phantom{ASEFG} \cc c_4 \cc a_4 \cc c_1 \cc b_1 \cc a_1^{-1} \cc b_1^{-2} \cc c_1^{-2} \cc a_2^{-2} \cc b_2^{-1} \cc c_2 \cc b_2^2 \cc a_2^2 \cc c_1^2 \cc b_1^2 \cc a_1^2 \cc e \cc b_1^{-1} \cc c_1^{-1} \cc a_5^{-1} \cc e^{-1},\\
&\gamma_{21} =  a_1^{-2} \cc b_1^{-2} \cc c_1^{-2} \cc a_2^{-2} \cc b_2^{-2} \cc c_2^{-1} \cc b_2 \cc a_2\cc b_3^{-1} \cc c_3^{-1} \cc a_2^{-1} \cc b_2^{-1} \cc c_2 \cc b_2^2 \cc a_2^2 \cc c_1^2 \cc b_1^2 \cc a_1^2 \cc e \cc b_1^{-1} \cc c_1^{-1} \cc c_4 \cc a_4^2 \cc c_3^2 \cc b_3^2 \cc a_3 \cc b_3^{-1} \cc c_3^{-1}\\
&\phantom{ASEFG} \cc c_1 \cc b_1 \cc  a_1^{-1} \cc b_1^{-2} \cc c_1^{-2} \cc a_2^{-2} \cc b_2^{-1} \cc c_2 \cc b_2^2 \cc a_2^2 \cc c_1^2 \cc b_1^2 \cc a_1^2 \cc e\cc b_1 \cc a_1 \cc b_1^{-1} \cc c_1^{-1} \cc a_4^{-1} \cc c_4^{-1} \cc c_1 \cc b_1 \cc e^{-1},\\
&\gamma_{22} =  a_1^{-2} \cc b_1^{-2} \cc c_1^{-2} \cc a_2^{-2} \cc b_2^{-2} \cc c_2^{-1} \cc b_2 \cc a_2\cc b_2 \cc a_2^2\cc c_1^2 \cc b_1^2 \cc a_1 \cc b_1^{-1} \cc c_1^{-1} \cc c_3 \cc b_3 \cc a_2^{-1} \cc b_2^{-1} \cc c_2 \cc b_2^2 \cc a_2^2 \cc c_1^2 \cc b_1^2 \cc a_1^2 \cc e \cc b_1^{-1} \cc c_1^{-1} \\
&\phantom{ASEFG} \cc c_4 \cc a_4 \cc c_1 \cc b_1 \cc a_1^{-1} \cc b_1^{-2} \cc c_1^{-2} \cc a_2^{-2} \cc b_2^{-1} \cc c_2 \cc b_2^2 \cc a_2^2 \cc c_1^2 \cc b_1^2 \cc a_1^2 \cc e \cc b_1^{-1} \cc c_1^{-1} \cc a_4^{-1}\cc c_4^{-1} \cc c_1 \cc b_1 \cc e^{-1},\\
&\gamma_{23} =  a_1^{-2} \cc b_1^{-2} \cc c_1^{-2} \cc a_2^{-2} \cc b_2^{-2} \cc c_2^{-1} \cc b_2 \cc a_2\cc b_3^{-1} \cc c_3^{-1} \cc c_1 \cc b_1 \cc a_1^{-1} \cc b_1^{-2} \cc c_1^{-2} \cc a_2^{-2} \cc b_2^{-1} \cc c_2 \cc b_2^2 \cc a_2^2 \cc c_1^2 \cc b_1^2 \cc a_1^2 \cc e\cc c_1^{-1} \cc a_2^{-1} \\
&\phantom{ASEFG} \cc b_2 \cc a_2^2 \cc c_1^2 \cc b_1^2 \cc a_1 \cc b_1^{-1} \cc c_1^{-1} \cc c_3 \cc b_3 \cc a_2^{-1} \cc b_2^{-1} \cc c_2 \cc b_2^2 \cc a_2^2 \cc c_1^2 \cc b_1^2 \cc a_1^2 \cc e \cc b_1^{-1} \cc c_1^{-1} \cc c_4 \cc a_4 \cc c_1 \cc b_1 \cc a_1^{-1} \cc b_1^{-1} \cc e^{-1},\\
&\gamma_{24} =  a_1^{-2} \cc b_1^{-2} \cc c_1^{-2} \cc a_2^{-2} \cc b_2^{-2} \cc c_2^{-1} \cc b_2 \cc a_2\cc b_3 \cc a_3 \cc b_3^{-1} \cc c_3^{-1} \cc c_1 \cc b_1 \cc a_1^{-1} \cc b_1^{-2} \cc c_1^{-2} \cc a_2^{-2} \cc b_2^{-1} \cc c_2 \cc b_2^2 \cc a_2^2 \cc c_1^2 \cc b_1^2 \cc a_1^2 \cc e\cc c_1^{-1} \cc a_2^{-1} \\
&\phantom{ASEFG} \cc b_2 \cc a_2^2 \cc c_1^2 \cc b_1^2 \cc a_1 \cc b_1^{-1} \cc c_1^{-1} \cc c_3 \cc b_3 \cc a_2^{-1} \cc b_2^{-1} \cc c_2 \cc b_2^2 \cc a_2^2 \cc c_1^2 \cc b_1^2 \cc a_1^2 \cc e \cc b_1^{-1} \cc c_1^{-1} \cc c_4 \cc a_4 \cc c_1 \cc b_1 \cc a_1^{-1} \cc b_1^{-1} \cc e^{-1},\\
&\gamma_{25} = a_1^{-2} \cc b_1^{-2} \cc c_1^{-2} \cc a_2^{-2} \cc b_2^{-2} \cc c_2^{-1} \cc b_2 \cc a_2^2 \cc c_1^2 \cc b_1^2 \cc a_1\cc  b_1^{-1} \cc c_1^{-1} \cc c_3 \cc b_3 \cc a_3^{-1} \cc b_3^{-2} \cc c_3^{-2} \cc a_4^{-2} \cc c_4^{-1} \cc c_1 \cc b_1 \cc e \cc a_5^2 \cc c_4^2 \cc a_4^2 \cc c_3^2 \cc b_3^2 \cc a_3 \\
&\phantom{ASEFG} \cc b_3^{-1} \cc c_3^{-1} \cc c_1 \cc b_1 \cc a_1^{-1} \cc b_1^{-2} \cc c_1^{-2} \cc a_2^{-2} \cc b_2^{-1} \cc c_2 \cc b_2^2 \cc a_2^2 \cc c_1^2 \cc b_1^2 \cc a_1^2 \cc e \cc b_1 \cc a_1 \cc b_1^{-1} \cc c_1^{-1} \cc a_4^{-1} \cc c_4^{-1} \cc c_1 \cc b_1 \cc e^{-1},\\
&\gamma_{26} = a_1^{-2} \cc b_1^{-2} \cc c_1^{-2} \cc a_2^{-2} \cc b_2^{-2} \cc c_2^{-1} \cc b_2 \cc a_2^2 \cc c_1^2 \cc b_1^2 \cc a_1\cc  b_1^{-1} \cc c_1^{-1} \cc c_3 \cc b_3 \cc a_3^{-1} \cc b_3^{-1}\cc a_2^{-1} \cc b_2^{-1} \cc c_2 \cc b_2^2 \cc a_2^2 \cc c_1^2 \cc b_1^2 \cc a_1^2 \cc e \cc b_1 \cc a_1\\
&\phantom{ASEFG} \cc b_1^{-1} \cc c_1^{-1} \cc c_4 \cc a_4^2 \cc c_3^2 \cc b_3^2 \cc a_3  \cc b_3^{-1} \cc c_3^{-1} \cc c_1 \cc b_1 \cc a_1^{-1} \cc b_1^{-2} \cc c_1^{-2} \cc a_2^{-2} \cc b_2^{-1} \cc c_2 \cc b_2^2 \cc a_2^2 \cc c_1^2 \cc b_1^2 \cc a_1^2 \cc e \cc b_1 \cc a_1 \cc e \cc a_5 \cc c_1 \cc b_1 \cc e^{-1},\\
&\gamma_{27} = a_1^{-2} \cc b_1^{-2} \cc c_1^{-2} \cc a_2^{-2} \cc b_2^{-2} \cc c_2^{-1} \cc b_2 \cc a_2 \cc b_3^{-1} \cc c_3^{-1} \cc c_1 \cc b_1 \cc  a_1^{-1} \cc b_1^{-2} \cc c_1^{-2} \cc a_2^{-2} \cc b_2^{-2} \cc c_2^{-1} \cc a_3^{-1} \cc b_3^{-1} \cc a_2^{-1} \cc b_2^{-1} \cc c_2 \cc b_2^2 \cc a_2^2 \\
&\phantom{ASEFG}  \cc c_1^2 \cc b_1^2 \cc a_1^2 \cc e \cc b_1^{-1} \cc c_1^{-1} \cc c_4 \cc a_4^2 \cc c_3^2 \cc b_3^2 \cc a_3 \cc b_3^{-1} \cc c_3^{-1} \cc c_1 \cc b_1 \cc a_1^{-1} \cc b_1^{-2} \cc c_1^{-2} \cc a_2^{-2} \cc b_2^{-1} \cc c_2 \cc b_2^2 \cc a_2^2 \cc c_1^2 \cc b_1^2 \cc a_1^2 \cc e\\
&\phantom{ASEFG} \cc b_1 \cc a_1 \cc e \cc a_5 \cc c_1 \cc b_1 \cc e^{-1},\\
&\gamma_{28} =  a_1^{-2} \cc b_1^{-2} \cc c_1^{-2} \cc a_2^{-2} \cc b_2^{-2} \cc c_2^{-1} \cc b_2 \cc a_2 \cc b_3 \cc a_3 \cc c_2 \cc b_2^2 \cc a_2^2 \cc c_1^2 \cc b_1^2 \cc a_1 \cc b_1^{-1} \cc c_1^{-1} \cc c_3 \cc b_3 \cc a_2^{-1} \cc b_2^{-1} \cc c_2 \cc b_2^2 \cc a_2^2 \cc c_1^2 \cc b_1^2 \\
&\bg \cc a_1^2 \cc e \cc b_1^{-1} \cc c_1^{-1} \cc c_4 \cc a_4 \cc c_1 \cc b_1 \cc e \cc a_5^2 \cc c_4^2 \cc a_4^2 \cc c_3^2 \cc b_3^2 \cc a_3 \cc b_3^{-1} \cc c_3^{-1} \cc c_1 \cc b_1 \cc a_1^{-1} \cc b_1^{-2} \cc c_1^{-2} \cc a_2^{-2} \cc b_2^{-1} \cc c_2 \cc b_2^2 \\
&\bg \cc a_2^2 \cc c_1^2 \cc b_1^2 \cc a_1^2 \cc e \cc b_1\cc a_1 \cc b_1^{-1} \cc c_1^{-1} \cc a_4^{-1} \cc c_4^{-1} \cc c_1 \cc b_1 \cc e^{-1},\\
&\gamma_{29} =  a_1^{-2} \cc b_1^{-2} \cc c_1^{-2} \cc a_2^{-2} \cc b_2^{-2} \cc c_2^{-1} \cc b_2 \cc a_2 \cc c_3 \cc b_3 \cc a_2^{-1} \cc b_2^{-1}  \cc c_2 \cc b_2^2 \cc a_2^2 \cc c_1^2 \cc b_1^2 \cc a_1^2 \cc e \cc b_1^{-1} \cc c_1^{-1} \cc c_4 \cc a_4^2 \cc c_3^2 \cc b_3^2 \cc a_3 \\
&\bg \cc b_3^{-1} \cc c_3^{-1} \cc c_1 \cc b_1 \cc a_1^{-1} \cc b_1^{-2} \cc c_1^{-2} \cc a_2^{-2} \cc b_2^{-1} \cc c_2 \cc b_2^2\cc a_2^2 \cc c_1^2 \cc b_1^2 \cc a_1^2 \cc e \cc c_1^{-1} \cc a_2^{-1} \cc b_2 \cc a_2^2 \cc c_1^2 \cc b_1^2 \cc a_1 \cc b_1^{-1} \cc c_1^{-1}\\
&\bg \cc c_3 \cc b_3 \cc a_2^{-1} \cc b_2^{-1} \cc c_2 \cc b_2^2 \cc a_2^2 \cc c_1^2 \cc b_1^2 \cc a_1^2 \cc e \cc b_1^{-1} \cc c_1^{-1} \cc c_4 \cc a_4 \cc c_1 \cc b_1 \cc a_1^{-1} \cc b_1^{-1} \cc e^{-1},\\
&\gamma_{30} =  a_1^{-2} \cc b_1^{-2} \cc c_1^{-2} \cc a_2^{-2} \cc b_2^{-2} \cc c_2^{-1} \cc b_2 \cc a_2^2 \cc c_1^2 \cc b_1^2 \cc a_1 \cc b_1^{-1} \cc c_1^{-1} \cc c_3 \cc b_3 \cc a_3^{-1} \cc b_3^{-1} \cc a_2^{-1} \cc b_2^{-1} \cc c_2 \cc b_2^2 \cc a_2^2 \cc c_1^2 \cc b_1^2 \cc a_1^2 \cc e\\
&\bg \cc b_1^{-1} \cc c_1^{-1} \cc c_4 \cc a_4^2 \cc c_3^2 \cc b_3^2 \cc a_3 \cc b_3^{-1} \cc c_3^{-1} \cc c_1 \cc b_1 \cc a_1^{-1} \cc b_1^{-2} \cc c_1^{-2} \cc a_2^{-2} \cc b_2^{-1} \cc c_2 \cc b_2^2 \cc a_2^2 \cc c_1^2 \cc b_1^2 \cc a_1^2 \cc e \cc c_1^{-1} \cc a_2^{-1} \\
&\bg \cc b_2 \cc a_2^2 \cc c_1^2 \cc b_1^2 \cc a_1 \cc b_1^{-1} \cc c_1^{-1} \cc c_3 \cc b_3\cc a_2^{-1} \cc b_2^{-1} \cc c_2 \cc b_2^2 \cc a_2^2 \cc c_1^2 \cc b_1^2 \cc a_1^2 \cc e \cc b_1^{-1} \cc c_1^{-1} \cc c_4 \cc a_4 \cc c_1 \cc b_1 \cc a_1^{-1} \cc b_1^{-1} \cc e^{-1},\\
\end{aligned}
\end{equation*}
\begin{equation*}
\begin{aligned}
&\gamma_{31} = a_1^{-2} \cc b_1^{-2} \cc c_1^{-2} \cc a_2^{-2} \cc b_2^{-2} \cc c_2^{-1} \cc b_2 \cc a_2\cc b_3 \cc a_3 \cc c_2 \cc b_2^2 \cc a_2^2  \cc c_1^2 \cc b_1^2 \cc a_1 \cc b_1^{-1} \cc c_1^{-1} \cc c_3 \cc b_3 \cc a_2^{-1} \cc b_2^{-1} \cc c_2 \cc b_2^2 \cc a_2^2 \cc c_1^2 \\
&\bg \cc b_1^2 \cc a_1^2 \cc e \cc b_1^{-1} \cc c_1^{-1} \cc c_4 \cc a_4^2 \cc c_3^2 \cc b_3^2 \cc a_3 \cc b_3^{-1} \cc c_3^{-1} \cc c_1 \cc b_1 \cc a_1^{-1} \cc b_1^{-2} \cc c_1^{-2} \cc a_2^{-2} \cc b_2^{-1} \cc c_2 \cc b_2^2 \cc a_2^2 \cc c_1^2 \cc b_1^2 \cc a_1^2 \cc e\\
&\bg \cc c_1^{-1} \cc a_2^{-1} \cc b_2 \cc a_2^2 \cc c_1^2 \cc b_1^2 \cc a_1 \cc b_1^{-1} \cc c_1^{-1} \cc c_3\cc b_3 \cc a_2^{-1} \cc b_2^{-1} \cc c_2 \cc b_2^2 \cc a_2^2 \cc c_1^2 \cc b_1^2 \cc a_1^2 \cc e \cc b_1^{-1} \cc c_1^{-1}\\
&\bg  \cc c_4 \cc a_4 \cc c_1 \cc b_1\cc a_1^{-1} \cc b_1^{-1} \cc e^{-1},\\
&\gamma_{32} = a_1^{-2} \cc b_1^{-2} \cc c_1^{-2} \cc a_2^{-2} \cc b_2^{-2} \cc c_2^{-1} \cc b_2 \cc a_2\cc \cc b_2 \cc a_2^2  \cc c_1^2 \cc b_1^2 \cc a_1 \cc b_1^{-1} \cc c_1^{-1} \cc c_3 \cc b_3 \cc a_2^{-1} \cc b_2^{-1} \cc c_2 \cc b_2^2 \cc a_2^2 \cc c_1^2\cc b_1^2 \cc a_1^2 \cc e \\
&\bg \cc b_1^{-1} \cc c_1^{-1} \cc c_4 \cc a_4 \cc c_1 \cc b_1 \cc a_1^{-1} \cc b_1^{-2} \cc c_1^{-2} \cc a_2^{-2} \cc b_2^{-1} \cc c_2 \cc b_2^2 \cc a_2^2 \cc c_1^2 \cc b_1^2 \cc a_1^2 \cc e \cc b_1^{-1} \cc c_1^{-1} \cc a_5 \cc c_4\cc c_1 \cc b_1 \cc e\cc a_5^2 \\
&\bg \cc c_4^2 \cc a_4^2 \cc c_3^2 \cc b_3^2 \cc a_3 \cc b_3^{-1} \cc c_3^{-1} \cc c_1 \cc b_1 \cc a_1^{-1} \cc b_1^{-2} \cc c_1^{-2} \cc a_2^{-1} \cc b_2^{-1} \cc c_2 \cc b_2^2 \cc a_2^2 \cc c_1^2 \cc b_1^2 \cc a_1^2 \cc e \cc b_1 \cc a_1\\
&\bg  \cc b_1^{-1} \cc c_1^{-1} \cc a_4^{-1} \cc c_4^{-1} \cc c_1 \cc b_1 \cc e^{-1},\\
&\gamma_{33} = a_1^{-2} \cc b_1^{-2} \cc c_1^{-2} \cc a_2^{-2} \cc b_2^{-2} \cc c_2^{-1} \cc b_2 \cc a_2 \cc b_3^{-1} \cc c_3^{-1} \cc c_1 \cc b_1 \cc a_1^{-1} \cc b_1^{-2} \cc c_1^{-2} \cc a_2^{-2} \cc b_2^{-1} \cc c_2 \cc b_2^2 \cc a_2^2 \cc c_1^2 \cc b_1^2 \cc a_1 \cc b_1^{-1} \cc c_1^{-1}\\
&\bg \cc c_3 \cc b_3 \cc a_2^{-1} \cc b_2^{-1} \cc c_2 \cc b_2^2 \cc a_2^2 \cc c_1^2 \cc b_1^2 \cc a_1^2 \cc e \cc b_1^{-1} \cc c_1^{-1} \cc c_4 \cc a_4^2 \cc c_3^2 \cc b_3^2 \cc a_3 \cc b_3^{-1} \cc c_3^{-1} \cc c_1 \cc b_1 \cc a_1^{-1} \cc b_1^{-2} \cc c_1^{-2} \\
&\bg \cc a_2^{-2} \cc b_2^{-1} \cc c_2 \cc b_2^2 \cc a_2^2 \cc c_1^2 \cc b_1^2 \cc a_1^2 \cc e \cc c_1^{-1} \cc a_2^{-1} \cc b_2 \cc a_2^2 \cc c_1^2 \cc b_1^2 \cc a_1 \cc b_1^{-1} \cc c_1^{-1} \cc  c_3 \cc b_3 \cc a_2^{-1} \cc b_2^{-1} \cc c_2 \cc b_2^2 \\
&\bg \cc a_2^2 \cc c_1^2 \cc b_1^2 \cc a_1^2 \cc e \cc b_1^{-1} \cc c_1^{-1} \cc c_4 \cc a_4 \cc c_1 \cc b_1 \cc a_1^{-1} \cc b_1 ^{-1} \cc e^{-1},\\
\end{aligned}
\end{equation*}
and
\begin{equation*}
\begin{aligned}
&\zeta_1= a_1^{-2} \cc b_1^{-2} \cc c_1^{-2} \cc a_2^{-2} \cc b_2^{-2} \cc c_2^{-1} \cc b_2 \cc a_2^2 \cc c_1^2 \cc b_1^2 \cc a_1 \cc b_1^{-1},\\
&\zeta_2= a_1^{-2} \cc b_1^{-2} \cc c_1^{-2} \cc a_2^{-2} \cc b_2^{-2} \cc c_2^{-1} \cc b_2 \cc a_2\cc b_2 \cc a_2^2 \cc c_1^2 \cc b_1^2 \cc a_1 \cc b_1^{-1},\\
&\zeta_3= a_1^{-2} \cc b_1^{-2} \cc c_1^{-2} \cc a_2^{-2} \cc b_2^{-2} \cc c_2^{-1} \cc b_2 \cc a_2\cc b_3 \cc a_3 \cc c_2 \cc b_2^2 \cc a_2^2 \cc c_1^2 \cc b_1^2 \cc a_1 \cc b_1^{-1},\\
&\zeta_4= a_1^{-2} \cc b_1^{-2} \cc c_1^{-2} \cc a_2^{-2} \cc b_2^{-2} \cc c_2^{-1} \cc b_2 \cc a_2 \cc b_3^{-1} \cc c_3^{-1} \cc c_1 \cc b_1 \cc a_1^{-1} \cc b_1^{-2} \cc c_1^{-2} \cc a_2^{-2} \cc b_2^{-2} \cc c_2^{-1} \cc b_2 \cc a_2^2 \cc c_1^2 \cc b_1^2 \cc a_1 \cc b^{-1},\\
&\zeta_5= a_1^{-2} \cc b_1^{-2} \cc c_1^{-2} \cc a_2^{-2} \cc b_2^{-2} \cc c_2^{-1} \cc b_2 \cc a_2 \cc b_3^{-1} \cc c_3^{-1} \cc c_1 \cc b_1 \cc a_1^{-1} \cc b_1^{-2} \cc c_1^{-2} \cc a_2^{-2} \cc b_2^{-1} \cc c_2 \cc b_2^2 \cc a_2^2 \cc c_1^2 \cc b_1^2 \cc a_1 \cc b^{-1},\\
&\zeta_6= a_1^{-2} \cc b_1^{-2} \cc c_1^{-2} \cc a_2^{-2} \cc b_2^{-2} \cc c_2^{-1} \cc b_2 \cc a_2 \cc b_3^{-1} \cc c_3^{-1} \cc c_1 \cc b_1 \cc a_1^{-1} \cc b_1^{-2} \cc c_1^{-2} \cc a_2^{-2} \cc b_2^{-1} \cc c_2 \cc b_2^2 \cc a_2^2 \cc c_1^2 \cc b_1^2 \cc a_1^2 \cc e,\\
&\zeta_7= a_1^{-2} \cc b_1^{-2} \cc c_1^{-2} \cc a_2^{-2} \cc b_2^{-2} \cc c_2^{-1} \cc b_2 \cc a_2 \cc b_3 \cc a_3 \cc b_3^{-1} \cc c_3^{-1} \cc c_1 \cc b_1 \cc a_1^{-1} \cc b_1^{-2} \cc c_1^{-2} \cc a_2^{-2} \cc b_2^{-1} \cc c_2 \cc b_2^2 \cc a_2^2 \cc c_1^2 \cc b_1^2 \cc a_1^2 \cc e,\\
&\zeta_8= a_1^{-2} \cc b_1^{-2} \cc c_1^{-2} \cc a_2^{-2} \cc b_2^{-2} \cc c_2^{-1} \cc b_2 \cc a_2\cc c_3 \cc b_3 \cc a_2^{-1} \cc b_2^{-1} \cc c_2 \c b_2^2 \cc a_2^2 \cc c_1^2 \cc b_1^2 \cc a_1^2 \cc e \cc b_1^{-1} \cc c_1^{-1} \cc c_4 \cc a_4^2 \cc c_3^2 \cc b_3^2 \cc a_3 \\
&\bg \cc b_3^{-1} \cc c_3^{-1} \cc c_1 \cc b_1 \cc a_1^{-1} \cc b_1^{-2} \cc c_1^{-2} \cc a_2^{-2} \cc b_2^{-1} \cc c_2 \cc b_2^2 \cc a_2^2 \cc c_1^2 \cc b_1^2 \cc a_1^2 \cc e,\\
&\mu_1 = c_1^{-1} \cc c_3 \cc b_3 \cc a_3^{-1} \cc b_3^{-2} \cc c_3^{-2} \cc a_4^{-2} \cc c_4^{-1} \cc c_1 \cc b_1 \cc e^{-1},\\
&\mu_2 = c_1^{-1} \cc c_3 \cc b_3 \cc a_3^{-1} \cc b_3^{-2} \cc c_3^{-2} \cc a_4^{-2} \cc c_4^{-2}\cc a_5^{-2} \cc e^{-1},\\
&\mu_3 = c_1^{-1} \cc a_2^{-2} \cc b_2^{-1} \cc c_2 \cc b_2^2 \cc a_2^2 \cc c_1^2 \cc b_1^2 \cc a_1^2 \cc e \cc b_1^{-1} \cc c_1^{-1} \cc e^{-1},\\
&\mu_4 = c_1^{-1} \cc a_2^{-2} \cc b_2^{-1} \cc c_2 \cc b_2^2 \cc a_2^2 \cc c_1^2 \cc b_1^2 \cc a_1^2 \cc e \cc b_1^{-1} \cc c_1^{-1}\cc a_5^{-1} \cc e^{-1},\\
&\mu_5 = c_1^{-1} \cc c_3 \cc b_3 \cc a_2^{-1} \cc b_2^{-1} \cc c_2 \cc b_2^2 \cc a_2^2 \cc c_1^2 \cc b_1^2 \cc a_1^2 \cc e \cc b_1^{-1} \cc c_1^{-1}\cc a_5^{-1} \cc e^{-1},\\
&\mu_6 = c_1^{-1} \cc a_2^{-2} \cc b_2^{-1} \cc c_2 \cc b_2^2 \cc a_2^2 \cc c_1^2 \cc b_1^2 \cc a_1^2 \cc e \cc b_1 \cc a_1  \cc b_1^{-1} \cc c_1^{-1}\cc a_4^{-1}\cc c_4^{-1} \cc c_1 \cc b_1 \cc e^{-1},\\
&\mu_7 = c_1^{-1} \cc c_3 \cc b_3 \cc a_2^{-1} \cc b_2^{-1} \cc c_2 \cc b_2^2 \cc a_2^2 \cc c_1^2 \cc b_1^2 \cc a_1^2 \cc e \cc b_1^{-1} \cc c_1^{-1}\cc c_4 \cc a_4 \cc c_1 \cc b_1 \cc a_1^{-1}\cc b_1^{-1} \cc e^{-1},\\
&\mu_8 = c_1^{-1} \cc c_3 \cc b_3 \cc a_3^{-1} \cc b_3^{-2} \cc c_3^{-2} \cc a_4^{-2} \cc c_4^{-1} \cc c_1 \cc b_1 \cc e \cc a_5^2 \cc c_4^2 \cc a_4^2 \cc c_3^2 \cc b_3^2 \cc a_3 \cc b_3^{-1} \cc c_3^{-1} \cc c_1 \cc b_1 \cc a_1^{-1} \cc b_1^{-2} \cc c_1^{-2} \cc a_2^{-2} \cc b_2^{-1}\\
&\bg \cc c_2 \cc b_2^2 \cc a_2^2 \cc c_1^2 \cc b_1^2 \cc a_1^2 \cc e \cc b_1 \cc a_1 \cc b_1^{-1} \cc c_1^{-1} \cc a_4^{-1} \cc c_4^{-1} \cc c_1 \cc b_1 \cc e^{-1},\\
&\mu_9 = c_1^{-1} \cc c_3 \cc b_3 \cc a_3^{-1} \cc b_3^{-1} \cc a_2^{-1} \cc b_2^{-1} \cc c_2 \cc b_2^2 \cc a_2^2  \cc c_1^2 \cc b_1^2 \cc a_1^2  \cc e\cc b_1^{-1} \cc c_1^{-1} \cc  c_4 \cc a_4^2 \cc c_3^2 \cc b_3^2 \cc a_3 \cc b_3^{-1} \cc c_3^{-1} \cc c_1 \cc b_1 \cc a_1^{-1} \\
&\bg \cc b_1^{-2} \cc c_1^{-2} \cc a_2^{-2} \cc b_2^{-1} \cc c_2 \cc b_2^2 \cc a_2^2 \cc c_1^2 \cc b_1^2 \cc a_1^2 \cc e \cc b_1 \cc a_1\cc e \cc a_5 \cc c_1 \cc b_1 \cc e^{-1},\\
&\mu_{10} = c_1^{-1} \cc a_2^{-2} \cc b_2^{-1} \cc  c_2 \cc b_2^2 \cc a_2^2  \cc c_1^2 \cc b_1^2 \cc a_1^2  \cc e\cc b_1^{-1} \cc c_1^{-1} \cc a_5 \cc c_4 \cc c_1 \cc b_1 \cc e \cc a_5^2 \cc c_4^2 \cc a_4^2 \cc c_3^2 \cc b_3^2 \cc a_3 \cc b_3^{-1} \cc c_3^{-1} \\
&\bg \cc c_1 \cc b_1 \cc a_1^{-1} \cc b_1^{-2} \cc c_1^{-2} \cc a_2^{-2} \cc b_2^{-1} \cc c_2 \cc b_2^2 \cc a_2^2 \cc c_1^2 \cc b_1^2 \cc a_1^2 \cc e \cc b_1 \cc a_1 \cc b_1^{-1} \cc c_1^{-1} \cc a_4^{-1} \cc c_4^{-1} \cc c_2 \cc b_1 \cc e^{-1},\\
\end{aligned}
\end{equation*}
\begin{equation*}
\begin{aligned}
&\mu_{11} = c_1^{-1} \cc c_3 \cc b_3 \cc a_2^{-1} \cc b_2^{-1} \cc  c_2 \cc b_2^2 \cc a_2^2  \cc c_1^2 \cc b_1^2 \cc a_1^2  \cc e\cc b_1^{-1} \cc c_1^{-1} \cc c_4\cc a_4  \cc c_1 \cc b_1 \cc e \cc a_5^2 \cc c_4^2 \cc a_4^2 \cc c_3^2 \cc b_3^2 \cc a_3 \cc b_3^{-1} \cc c_3^{-1} \\
&\bg \cc c_1 \cc b_1 \cc a_1^{-1} \cc b_1^{-2} \cc c_1^{-2} \cc a_2^{-2} \cc b_2^{-1} \cc c_2 \cc b_2^2 \cc a_2^2 \cc c_1^2 \cc b_1^2 \cc a_1^2 \cc e \cc b_1 \cc a_1 \cc b_1^{-1} \cc c_1^{-1} \cc a_4^{-1} \cc c_4^{-1} \cc c_2 \cc b_1 \cc e^{-1},\\
&\eta_1 = c_1^{-1} \cc a_2^{-1} \cc b_2 \cc a_2^2 \cc c_1^2 \cc b_1^2 \cc a_1 \cc b_1^{-1},\\
&\eta_2 = c_1^{-1} \cc c_3 \cc b_3 \cc a_2^{-1} \cc b_2^{-1} \cc c_2 \cc b_2^2 \cc a_2^2 \cc c_1^2 \cc b_1^2 \cc a_1^2 \cc e \cc b_1^{-1} \cc c_1^{-1} \cc c_4 \cc a_4 \cc c_1 \cc b_1 \cc a_1^{-1} \cc b_1^{-2} \cc c_1^{-1},\\
&\eta_3 = c_1^{-1} \cc c_3 \cc b_3 \cc a_2^{-1} \cc b_2^{-1} \cc c_2 \cc b_2^2 \cc a_2^2 \cc c_1^2 \cc b_1^2 \cc a_1^2 \cc e \cc b_1^{-1} \cc c_1^{-1} \cc c_4 \cc a_4^2 \cc c_3^2 \cc b_3^2 \cc a_3 \cc b_3^{-1} \cc c_3^{-1} \cc c_1 \cc b_1 \\
&\bg \cc a_1^{-1} \cc b_1^{-2} \cc c_1^{-2} \cc a_2^{-2} \cc b_2^{-1} \cc c_2 \cc b_2^2 \cc a_2^2 \cc c_1^2 \cc b_1^2 \cc a_1^2 \cc e,\\
&\eta_4 =  c_1^{-1} \cc c_3 \cc b_3\cc a_3^{-1} \cc b_3^{-1} \cc a_2^{-1} \cc b_2^{-1} \cc c_2 \cc b_2^2 \cc a_2^2 \cc c_1^2 \cc b_1^2 \cc a_1^2 \cc e \cc b_1^{-1} \cc c_1^{-1} \cc c_4 \cc a_4^2 \cc c_3^2 \cc b_3^2 \cc a_3 \cc b_3^{-1} \cc c_3^{-1} \cc c_1 \cc b_1\\
&\bg \cc a_1^{-1} \cc b_1^{-2} \cc c_1^{-2} \cc a_2^{-2} \cc b_2^{-1} \cc c_2 \cc b_2^2 \cc a_2^2 \cc c_1^2 \cc b_1^2 \cc a_1^2 \cc e.\\
\end{aligned}
\end{equation*}
The admissible quadruples are 
\begin{equation*}
\begin{aligned}
A= \{& (1,6,18,5),(1,6,28,5),(1,6,14,6),(1,6,14,7),(1,6,24,7),(1,6,29,7),(1,6,31,7),\\
&(1,6,14,25),(1,6,28,27),(2,27,20,1),(2,27,32,5),(2,27,19,9),(2,27,32,11),(3,4,11,1),\\
&(4,11,1,6),(5,26,10,1),(5,26,11,1),(5,6,15,1),(5,30,7,2),(5,26,8,4),(5,9,1,6),\\
&(5,26,23,7),(5,6,24,7),(5,6,33,7),(5,26,19,9),(5,26,22,11),(5,6,27,19),(5,6,21,27),\\
&(6,14,7,2),(6,24,7,2),(6,29,7,2),(6,31,7,2),(6,24,7,3),(6,15,1,6),(6,28,5,6),\\
&(6,27,19,9),(6,14,25,11),(6,24,7,16),(6,24,7,17),(6,14,6,18),(6,21,27,20),(6,28,27,20),\\
&(6,18,5,26),(6,14,6,29),(6,28,5,30),(7,16,13,4),(7,16,13,7),(7,3,4,11),(7,16,4,11),\\
&(7,2,27,19),(7,2,27,20),(7,12,27,20),(7,17,27,20),(7,2,27,32),(8,4,11,1),(9,1,6,24),\\
&(10,1,6,24),(11,1,6,14),(11,1,6,18),(11,1,6,24),(11,1,6,28),(11,1,6,31),(12,27,20,1),\\
&(13,4,11,1),(13,7,3,4),(13,7,12,27),(14,25,11,1),(14,6,18,5),(14,6,29,7),(14,7,2,27),\\
&(15,1,6,24),(16,4,11,1),(16,13,7,3),(16,13,4,11),(16,13,7,12),(17,27,20,1),(18,5,26,8),\\
&(18,5,26,23),(18,5,26,33),(19,9,1,6),(20,1,6,18),(21,27,20,1),(22,11,1,6),(23,7,3,4),\\
&(24,7,3,4),(24,7,16,4),(24,7,16,13),(24,7,2,27),(24,7,17,27),(25,11,1,6),(26,19,9,1),\\
&(26,22,11,1),(26,33,7,2),(26,33,7,3),(26,10,1,6),(26,11,1,6),(26,8,4,11),(27,19,9,1),\\
&(27,32,11,1),(27,20,1,6),(27,32,5,9),(27,32,5,26),(28,27,20,1),(28,5,30,7),\\
&(28,5,26,11),(28,5,6,15),(28,5,6,21),(28,5,6,24),(28,5,6,27),(28,5,6,33),(29,7,2,27),\\
&(30,7,2,27),(31,7,2,27),(32,5,9,1),(32,11,1,6),(32,5,26,10),(32,5,26,19),(32,5,26,22),\\
&(32,5,26,23),(33,7,2,27)\}
\end{aligned}
\end{equation*}
Let $V= Span \, \Gamma$ and $W= Span\, \Delta$. The splitting map $S: V\to W$ is given by 
\begin{equation*}
S: \left\{ \begin{aligned} & \gamma_1 \mapsto \eta_1, \quad \gamma_2 \mapsto \mu_4 + \gamma_1 + \zeta_1, \quad \gamma_3 \mapsto \mu_3 + \zeta_1, \quad \gamma_4 \mapsto \mu_5 + \gamma_1 + \zeta_1,\quad \gamma_5 \mapsto \mu_1 + \zeta_3\\
&\gamma_6 \mapsto \mu_7 + \zeta_1, \quad \gamma_7 \mapsto \eta_2 , \quad \gamma_8 \mapsto \mu_1 + \zeta_4, \quad \gamma_9 \mapsto \eta_3, \quad \gamma_{10} \mapsto \eta_4, \quad \gamma_{11} \mapsto \mu_1 + \zeta_7,\\
&\gamma_{12} \mapsto \mu_6 + \gamma_{11} + \gamma_1 + \zeta_1, \quad \gamma_{13} \mapsto \mu_9 + \zeta_2, \quad \gamma_{14} \mapsto \mu_2 +\gamma_{16} + \zeta_2, \quad \gamma_{15} \mapsto \mu_1 + \zeta_8,\\
&\gamma_{16} \mapsto \mu_{10} + \gamma_5 + \zeta_1, \quad \gamma_{17} \mapsto \mu_{10} + \gamma_{11} +\gamma_1 + \zeta_1, \quad \gamma_{18} \mapsto \mu_2 + \gamma_3 + \gamma_4 + \gamma_{11} + \gamma_1 + \zeta_1,\\
&\gamma_{19} \mapsto \mu_9 + \zeta_5, \quad \gamma_{20} \mapsto  \mu_9 + \zeta_6, \quad \gamma_{21} \mapsto \mu_8 + \gamma_{11} + \gamma_1 + \zeta_1, \quad \gamma_{22} \mapsto \mu_9 + \gamma_{11} + \gamma_1 + \zeta_1,\\
&\gamma_{23} \mapsto \mu_1 + \gamma_{27} + \zeta_2, \quad \gamma_{24} \mapsto \mu_2 + \gamma_2 + \gamma_{27} + \zeta_2, \quad \gamma_{25} \mapsto \mu_7 + \gamma_4 + \gamma_{11} + \gamma_1 + \zeta_1,\\
&\gamma_{26} \mapsto \mu_{11} + \gamma_5 + \zeta_1, \quad \gamma_{27} \mapsto \mu_1 + \gamma_{18} + \gamma_5 + \zeta_1, \quad \gamma_{28} \mapsto \mu_2 + \gamma_{16} + \gamma_4 + \gamma_{11} + \gamma_1 + \zeta_1,\\
&\gamma_{29} \mapsto \mu_2 + \gamma_{12} + \gamma_{27} + \zeta_2, \quad \gamma_{30} \mapsto \mu_{11} + \gamma_{27} + \zeta_2, \quad \gamma_{31} \mapsto \mu_2 + \gamma_{17} + \gamma_{27} + \zeta_2,\\
&\gamma_{32} \mapsto \mu_9 + \gamma_8 + \gamma_4 + \gamma_{11} + \gamma_1 + \zeta_1, \quad \gamma_{33} \mapsto \mu_1 + \gamma_{21} + \gamma_{27} + \zeta_2.\\
 \end{aligned}\right.
\end{equation*}
and the merging map $M:V \to W$ is
\begin{equation*}
M: \left\{ \begin{aligned} &\gamma_1 \mapsto \gamma_1, \quad \gamma_2 \mapsto \gamma_2, \quad \gamma_3 \mapsto \gamma_3, \quad \gamma_4 \mapsto \gamma_4, \quad \gamma_5 \mapsto \gamma_5, \quad \gamma_6 \mapsto \zeta_1 + \mu_1,\\
&\gamma_7 \mapsto \zeta_1 + \mu_2, \quad \gamma_8 \mapsto \gamma_8, \quad \gamma_9 \mapsto \zeta_1 + \mu_5, \quad \gamma_{10} \mapsto \zeta_2 + \mu_5, \quad \gamma_{11} \mapsto \gamma_{11}, \quad \gamma_{12} \mapsto \gamma_{12},\\
&\gamma_{13} \mapsto \zeta_2 + \mu_7, \quad \gamma_{14} \mapsto \zeta_3 + \mu_7, \quad \gamma_{15} \mapsto \zeta_4 + \mu_5, \quad \gamma_{16} \mapsto \gamma_{16}, \quad \gamma_{17} \mapsto \gamma_{17}, \\
&\gamma_{18} \mapsto \gamma_{18}, \quad \gamma_{19} \mapsto \zeta_2 + \eta_2 + \mu_3, \quad \gamma_{20} \mapsto \zeta_2 + \eta_2 + \mu_4, \quad \gamma_{21} \mapsto \gamma_{21},\\
&\gamma_{22} \mapsto \zeta_2 + \eta_2 + \mu_6, \quad \gamma_{23} \mapsto \zeta_6 + \eta_1 + \mu_7, \quad \gamma_{24} \mapsto \zeta_7 + \eta_1 + \mu_7, \quad \gamma_{25} \mapsto \zeta_1  +\mu_8,\\
&\gamma_{26} \mapsto \zeta_1 + \mu_9, \quad \gamma_{27} \mapsto \gamma_{27}, \quad \gamma_{28} \mapsto \zeta_3 + \mu_{11}, \quad \gamma_{29} \mapsto \zeta_8 + \eta_1 + \mu_7,\\
&\gamma_{30} \mapsto \zeta_1 + \eta_4 + \eta_1 + \mu_7, \quad \gamma_{31} \mapsto \zeta_3 + \eta_3 + \eta_1 + \mu_7, \quad \gamma_{32} \mapsto \zeta_2 + \eta_2 + \mu_{10},\\
&\gamma_{33} \mapsto \zeta_5 + \eta_3 + \eta_1 + \mu_7.\\
\end{aligned}\right.
\end{equation*}

There is a $v_\lambda \in V$ such that $Sv_\lambda = \lambda Mv_\lambda$ where $\lambda\approx 1.54325 $ is the largest real root of \[\chi(t) =t^{14} - t^{13} -t^{11} + t^{10} - 2 t^9 + t^8 - 2 t^7 + t^6- 2 t^5 + t^4 -t^3 -t+1 \]
and
\begin{equation*}
\begin{aligned}
 v_\lambda \approx ( & 13.509,3.128,2.382,6.224,8.754,12.322,9.605,1.543,2.857,0.176,7.449,\\
 &0.420,1.049,1.693,1,3.124,0.551,3.675,1.543,2.027,0.648,0.272,\\
 &1.313,4.827,0.420,4.713,5.672,3.128,0.648,0.114,0.851,2.382,1)
 \end{aligned}
 \end{equation*}

\subsection{Orbitdata $3,4,6$ with a cyclic permutation}\label{B:346}
Let $f_\mathbf{R}: X(\mathbf{R}) \to X(\mathbf{R})$ be a real diffeomorphism associated with the orbit data $3,4,6$ and a cyclic permutation. 
Let 
\begin{equation*}
\begin{aligned}
&\Gamma = \{ \gamma_1, \dots, \gamma_{42} \},\\
&\Delta = \{ \gamma_1, \gamma_2,\gamma_3, \gamma_4, \gamma_5, \gamma_7,\gamma_8,\gamma_9,\gamma_{10},\gamma_{13},\gamma_{14},\gamma_{16},\gamma_{17},\gamma_{18},\gamma_{22},\gamma_{26},\gamma_{30},\gamma_{39},\gamma_{41},\gamma_{42},\\
&\bg\bg\bg \zeta_1,\dots,\zeta_8,\mu_1,\dots,\mu_{13},\eta_1,\dots,\eta_5\} 
\end{aligned}
\end{equation*}
be two ordered sets of non-cyclic words in $\pi_1( X(\mathbf{R}),P_{fix})$ where
\begin{equation*}
\begin{aligned}
&\gamma_1 = a_1^{-1} \cc b_1^{-1} \cc e^{-1},\ \  \gamma_2 = b_1^{-1} \cc c_1^{-1} \cc a_4^{-1}\cc c_4^{-1} \cc c_1 \cc b_1\cc e^{-1},\ \ \gamma_3 = b_1^{-1} \cc c_1^{-1} \cc c_4^{-1} \cc a_5^{-1} \cc c_1 \cc b_1 \cc e^{-1},\\
&\gamma_4 = b_1^{-1} \cc c_1^{-1} \cc a_4\cc c_3 \cc a_2^{-1} \cc b_2^{-1} \cc c_2 \cc b_2^2 \cc a_2^2 \cc c_1^2 \cc b_1^2 \cc a_1^2 \cc e \cc b_1^{-1} \cc c_1^{-1} \cc c_4^{-1} \cc a_5^{-1} \cc c_1\cc b_1 \cc e^{-1},\\
&\gamma_5 = b_1^{-1} \cc c_1^{-1} \cc a_4\cc c_3 \cc a_2^{-1} \cc b_2^{-1} \cc c_2 \cc b_2^2 \cc a_2^2 \cc c_1^2 \cc b_1^2 \cc a_1^2 \cc e \cc b_1^{-1} \cc c_1^{-1} \cc c_4^{-1} \cc a_5^{-2}\cc a_6^{-2} \cc e^{-1},\\
&\gamma_6 = a_1^{-2} \cc b_1^{-2} \cc c_1^{-2} \cc a_2^{-2} \cc b_2^{-2} \cc c_2^{-1} \cc b_2 \cc a_2^2 \cc c_1^2 \cc b_1^2 \cc a_1 \cc b_1^{-1} \cc c_1^{-1} \cc c_3 \cc b_3 \cc a_2 \cc c_1^2 \cc b_1 \cc e^{-1},\\
&\gamma_7= a_1^{-2} \cc b_1^{-2} \cc c_1^{-2} \cc a_2^{-2} \cc b_2^{-2} \cc c_2^{-1} \cc b_2 \cc a_2^2 \cc c_1^2 \cc b_1^2 \cc a_1\cc e \cc a_6^2 \cc a_5^2 \cc c_4 \cc c_1 \cc b_1 \cc e^{-1},\\
&\gamma_8 = a_1^{-2} \cc b_1^{-2} \cc c_1^{-2} \cc a_2^{-2} \cc b_2^{-2} \cc c_2^{-1} \cc b_2 \cc a_2 \cc b_3^{-1} \cc c_3^{-1} \cc c_1 \cc b_1 \cc a_1^{-1} \cc b_1^{-2} \cc c_1^{-2} \cc a_2^{-2} \cc b_2^{-1}\cc a_2 \cc c_1 \cc e^{-1},\\
&\gamma_9 = a_1^{-2} \cc b_1^{-2} \cc c_1^{-2} \cc a_2^{-2} \cc b_2^{-2} \cc c_2^{-1} \cc b_2 \cc a_2 \cc b_3^{-1} \cc c_3^{-1} \cc c_1 \cc b_1 \cc a_1^{-1} \cc b_1^{-2} \cc c_1^{-2} \cc a_2^{-2} \cc b_2^{-1}\cc a_2 \cc c_1 \cc  a_6^{-1}\cc e^{-1},\\
&\gamma_{10}= a_1^{-2} \cc b_1^{-2} \cc c_1^{-2} \cc a_2^{-2} \cc b_2^{-2} \cc c_2^{-1} \cc a_2^{-1} \cc b_2^{-1}\cc c_2 \cc b_2^2 \cc a_2^2 \cc c_1^2 \cc b_1^2 \cc a_1^2\cc e\cc b_1^{-1} \cc c_1^{-1} \cc a_5^{-1} \cc a_6^{-2} \cc e^{-1},\\
&\gamma_{11} = a_1^{-2} \cc b_1^{-2} \cc c_1^{-2} \cc a_2^{-2} \cc b_2^{-2} \cc c_2^{-1} \cc b_2 \cc a_2^2 \cc c_1^2 \cc b_1^2 \cc a_1 \cc b_1^{-1} \cc c_1^{-1} \cc c_3 \cc b_3 \cc a_3^{-1} \cc b_3^{-2} \cc c_3^{-2}\cc a_4^{-2} \cc c_4^{-1}  \cc c_1 \cc b_1 \cc e^{-1},\\
&\gamma_{12} =a_1^{-2} \cc b_1^{-2} \cc c_1^{-2} \cc a_2^{-2} \cc b_2^{-2} \cc c_2^{-1} \cc b_2 \cc a_2^2 \cc c_1^2 \cc b_1^2 \cc a_1 \cc b_1^{-1} \cc c_1^{-1} \cc c_3 \cc b_3 \cc a_3^{-1} \cc b_3^{-2} \cc c_3^{-2}\cc a_4^{-2} \cc c_4^{-2}\cc a_5^{-2} \cc a_6^{-2} \cc e^{-1},\\
&\gamma_{13} = a_1^{-2} \cc b_1^{-2} \cc c_1^{-2} \cc a_2^{-2} \cc b_2^{-2} \cc c_2^{-1} \cc b_2 \cc a_2 \cc b_3^{-1} \cc c_3^{-1} \cc c_1 \cc b_1 \cc a_1^{-1} \cc b_1^{-2} \cc c_1^{-2} \cc a_2^{-2} \cc b_2^{-1}\cc c_2 \cc b_2^2 \cc a_2^2 \cc c_1^2 \cc b_1^2 \cc a_1^2 \cc e \\
&\bg \cc b_1 \cc a_1 \cc e \cc  a_6 \cc c_1\cc b_1 \cc e^{-1},\\
&\gamma_{14} =  a_1^{-2} \cc b_1^{-2} \cc c_1^{-2} \cc a_2^{-2} \cc b_2^{-2} \cc c_2^{-1} \cc b_2 \cc a_2 \cc b_3^{-1} \cc c_3^{-1} \cc c_1 \cc b_1 \cc a_1^{-1} \cc b_1^{-2} \cc c_1^{-2} \cc a_2^{-2} \cc b_2^{-1}\cc a_2^{-1} \cc b_2^{-1} \cc c_2 \cc b_2^2 \cc a_2^2 \cc c_1^2 \\
&\bg \cc b_1^2 \cc a_1^2 \cc e \cc b_1^{-1} \cc c_1^{-1} \cc a_5^{-1} \cc a_6^{-2} \cc e^{-1},\\
&\gamma_{15} = a_1^{-2} \cc b_1^{-2} \cc c_1^{-2} \cc a_2^{-2} \cc b_2^{-2} \cc c_2^{-1} \cc b_2 \cc a_2 \cc b_2 \cc a_2^2 \cc c_1^2 \cc b_1^2 \cc a_1 \cc b_1^{-1} \cc c_1^{-1} \cc c_3 \cc b_3 \cc a_2^{-1} \cc b_2^{-1} \cc c_2 \cc b2^2 \cc a_2^2 \cc c_1^2 \cc b_1^2 \\
&\bg \cc a_1^2 \cc e \cc b_1^{-1} \cc c_1^{-1} \cc c_4 \cc a_4 \cc c_1\cc b_1 \cc a_1^{-1} \cc b_1^{-1} \cc e^{-1},\\
\end{aligned}
\end{equation*}
\begin{equation*}
\begin{aligned}
&\gamma_{16} = a_1^{-2} \cc b_1^{-2} \cc c_1^{-2} \cc a_2^{-2} \cc b_2^{-2} \cc c_2^{-1} \cc b_2 \cc a_2\cc b_3 \cc a_3 \cc b_3^{-1} \cc c_3^{-1} \cc c_1 \cc b_1 \cc a_1^{-1} \cc b_1^{-2} \cc c_1^{-2} \cc a_2^{-2} \cc b_2^{-1}\cc c_2 \cc b_2^2 \cc a_2^2 \cc c_1^2 \cc b_1^2 \cc a_1^2 \cc e \\
&\bg \cc b_1 \cc a_1 \cc e \cc a_6 \cc c_1 \cc b_1^2 \cc a_1 \cc b_1^{-1} \cc c_1^{-1} \cc a_4^{-1} \cc c_4^{-1} \cc c_1 \cc b_1 \cc e^{-1},\\
&\gamma_{17}= a_1^{-2} \cc b_1^{-2} \cc c_1^{-2} \cc a_2^{-2} \cc b_2^{-2} \cc c_2^{-1} \cc b_2 \cc a_2 \cc c_3^{-1} \cc a_4^{-1} \cc c_1 \cc b_1\cc e \cc a_6^2 \cc a_5^2 \cc c_4^2 \cc a_4^2 \cc c_3 \cc c_1 \cc b_1  \cc a_1^{-1} \cc b_1^{-2} \cc c_1^{-2} \cc a_2^{-2} \cc b_2^{-1}\\
&\bg \cc c_2 \cc b_2^2 \cc a_2^2 \cc c_1^2 \cc b_1^2 \cc a_1^2 \cc e \cc b_1 \cc a_1 \cc b_1^{-1} \cc c_1^{-1} \cc a_4^{-1} \cc c_4^{-1} \cc c_1 \cc b_1 \cc e^{-1},\\
&\gamma_{18} = b_1^{-1} \cc c_1^{-1} \cc a_4 \cc c_3 \cc a_2^{-1} \cc b_2^{-1} \cc c_2 \cc b_2^2 \cc a_2^2 \cc c_1^2 \cc b_1^2 \cc a_1^2 \cc e \cc b_1^{-1} \cc c_1^{-1} \cc c_4 \cc a_4^2 \cc c_3^2 \cc b_3^2 \cc a_3 \cc b_3^{-1} \cc c_3^{-1} \cc c_1 \cc b_1 \cc a_1^{-1} \cc b_1^{-2}\\
&\bg \cc c_1^{-2} \cc a_2^{-2} \cc b_2^{-1} \cc c_2 \cc b_2^2 \cc a_2^2 \cc c_1^2 \cc b_1^2 \cc a_1^2 \cc e\cc b_1 \cc a_1 \cc b_1^{-1} \cc c_1^{-1} \cc a_4^{-1} \cc c_4^{-1} \cc c_1 \cc b_1 \cc e^{-1},\\
&\gamma_{19}=a_1^{-2} \cc b_1^{-2} \cc c_1^{-2} \cc a_2^{-2} \cc b_2^{-2} \cc c_2^{-1} \cc b_2 \cc a_2 \cc b_3 \cc a_3 \cc c_2 \cc b_2^2 \cc a_2^2 \cc c_1^2 \cc b_1^2\cc a_1^2 \cc e \cc c_1^{-1} \cc a_2^{-1} \cc b_2 \cc a_2^2 \cc c_1^2 \cc b_1^2 \cc a_1 \cc b_1^{-1} \cc c_1^{-1} \\
&\bg \cc c_3 \cc b_3 \cc a_2^{-1} \cc b_2^{-1} \cc c_2 \cc b_2^2 \cc a_2^2 \cc c_1^2 \cc b_1^2 \cc a_1^2 \cc e \cc b_1^{-1} \cc c_1^{-1} \cc c_4 \cc a_4 \cc a_1^{-1} \cc b_1^{-1} \cc e^{-1},\\
&\gamma_{20} = a_1^{-2} \cc b_1^{-2} \cc c_1^{-2} \cc a_2^{-2} \cc b_2^{-2} \cc c_2^{-1} \cc b_2 \cc a_2^2 \cc c_1^2 \cc b_1^2\cc a_1 \cc b_1^{-1} \cc c_1^{-1} \cc_3 \cc b_3 \cc a_2^{-1} \cc b_2^{-1} \cc c_2 \cc b_2^2 \cc a_2^2 \cc c_1^2 \cc b_1^2 \cc a_1^2  \cc e\cc b_1^{-1}\cc c_1^{-1}\\
&\bg \cc c_4 \cc a_4 \cc c_1 \cc b_1 \cc a_1^{-1} \cc b_1^{-2} \cc c_1^{-2} \cc a_2^{-2} \cc b_2^{-1} \cc c_2 \cc b_2^2 \cc a_2^2 \cc c_1^2 \cc b_1^2 \cc a_1^2 \cc e \cc b_1^{-1} \cc c_1^{-1} \cc a_6^{-1} \cc e^{-1},\\
&\gamma_{21} = a_1^{-2} \cc b_1^{-2} \cc c_1^{-2} \cc a_2^{-2} \cc b_2^{-2} \cc c_2^{-1} \cc b_2 \cc a_2\cc b_2 \cc a_2^2 \cc c_1^2 \cc b_1^2\cc a_1 \cc b_1^{-1} \cc c_1^{-1} \cc_3 \cc b_3 \cc a_2^{-1} \cc b_2^{-1} \cc c_2 \cc b_2^2 \cc a_2^2 \cc c_1^2 \cc b_1^2 \cc a_1^2  \cc e\\
&\bg \cc b_1^{-1}\cc c_1^{-1} \cc c_4 \cc a_4 \cc c_1 \cc b_1 \cc a_1^{-1} \cc b_1^{-2} \cc c_1^{-2} \cc a_2^{-2} \cc b_2^{-1} \cc c_2 \cc b_2^2 \cc a_2^2 \cc c_1^2 \cc b_1^2 \cc a_1^2 \cc e \cc b_1^{-1} \cc c_1^{-1} \cc a_6^{-1} \cc e^{-1},\\
&\gamma_{22} = a_1^{-2} \cc b_1^{-2} \cc c_1^{-2} \cc a_2^{-2} \cc b_2^{-2} \cc c_2^{-1} \cc a_3^{-1} \cc b_3^{-1} \cc a_2^{-1} \cc b_2^{-1} \cc c_2 \cc b_2^2 \cc a_2^2 \cc c_1^2 \cc b_1^2 \cc a_1^2 \cc e \cc b_1^{-1} \cc c_1^{-1} \cc c_4 \cc a_4^2 \cc c_3^2 \cc b_3^2 \cc a_3\\
&\bg \cc b_3^{-1} \cc c_3^{-1} \cc c_1 \cc b_1\cc \cc a_1^{-1} \cc b_1^{-2} \cc c_1^{-2} \cc a_2^{-2} \cc b_2^{-1} \cc c_2 \cc b_2^2 \cc a_2^2 \cc c_1^2 \cc b_1^2 \cc a_1^2 \cc e\cc b_1 \cc a_1 \cc e \cc a_6^2 \cc a_5 \cc c_1 \cc b_1 \cc e^{-1},\\
&\gamma_{23} = a_1^{-2} \cc b_1^{-2} \cc c_1^{-2} \cc a_2^{-2} \cc b_2^{-2} \cc c_2^{-1} \cc b_2 \cc a_2\cc b_2 \cc a_2^2 \cc c_1^2 \cc b_1^2\cc a_1 \cc b_1^{-1} \cc c_1^{-1} \cc_3 \cc b_3 \cc a_2^{-1} \cc b_2^{-1} \cc c_2 \cc b_2^2 \cc a_2^2 \cc c_1^2 \cc b_1^2 \cc a_1^2  \cc e\\
&\bg \cc b_1^{-1}\cc c_1^{-1} \cc c_4 \cc a_4 \cc c_1 \cc b_1 \cc a_1^{-1} \cc b_1^{-2} \cc c_1^{-2} \cc a_2^{-2} \cc b_2^{-1} \cc c_2 \cc b_2^2 \cc a_2^2 \cc c_1^2 \cc b_1^2 \cc a_1^2 \cc e \cc b_1^{-1} \cc c_1^{-1}\cc a_5^{-1} \cc a_6^{-2} \cc e^{-1},\\
&\gamma_{24} = a_1^{-2} \cc b_1^{-2} \cc c_1^{-2} \cc a_2^{-2} \cc b_2^{-2} \cc c_2^{-1} \cc b_2 \cc a_2\cc b_2 \cc a_2^2 \cc c_1^2 \cc b_1^2\cc a_1 \cc b_1^{-1} \cc c_1^{-1} \cc_3 \cc b_3 \cc a_2^{-1} \cc b_2^{-1} \cc c_2 \cc b_2^2 \cc a_2^2 \cc c_1^2 \cc b_1^2 \cc a_1^2  \cc e\\
&\bg \cc b_1^{-1}\cc c_1^{-1} \cc c_4 \cc a_4 \cc c_1 \cc b_1 \cc a_1^{-1} \cc b_1^{-2} \cc c_1^{-2} \cc a_2^{-2} \cc b_2^{-1} \cc c_2 \cc b_2^2 \cc a_2^2 \cc c_1^2 \cc b_1^2 \cc a_1^2 \cc e \cc b_1^{-1} \cc c_1^{-1}\cc a_5 \cc c_4 \cc c_1 \cc b_1 \cc e^{-1},\\
&\gamma_{25}=a_1^{-2} \cc b_1^{-2} \cc c_1^{-2} \cc a_2^{-2} \cc b_2^{-2} \cc c_2^{-1} \cc b_2 \cc a_2\cc b_3 \cc a_3 \cc c_2 \cc b_2^2 \cc a_2^2 \cc c_1^2 \cc b_1^2 \cc a_1 \cc b_1^{-1} \cc c_1^{-1} \cc c_3^{-1} \cc a_4^{-1} \cc c_1 \cc b_1 \cc e \cc a_6^2 \cc a_5^2 \cc c_4^2\\
&\bg. \cc a_4^2 \cc c_3 \cc c_1 \cc b_1 \cc a_1^{-1} \cc b_1^{-2} \cc c_1^{-2} \cc a_2^{-2} \cc b_2^{-1} \cc c_2 \cc b_2^2 \cc a_2^2 \cc c_1^2 \cc b_1^2 \cc a_1^2 \cc e\cc b_1 \cc a_1  \cc b_1^{-1} \cc c_1^{-1}\cc a_4^{-1} \cc c_4^{-1} \cc c_1 \cc b_1 \cc e^{-1},\\
&\gamma_{26} = b_1^{-1} \cc c_1^{-1} \cc a_4 \cc c_3 \cc c_1 \cc b_1 \cc a_1^{-1} \cc b_1^{-2} \cc c_1^{-2} \cc a_2^{-2} \cc b_2^{-2}\cc c_2^{-1} \cc a_3^{-1} \cc b_3^{-1} \cc a_2^{-1} \cc b_2^{-1} \cc c_2 \cc b_2^2 \cc a_2^2 \cc c_1^2 \cc b_1^2 \cc a_1^2 \cc e \cc b_1^{-1} \cc c_1^{-1} \\
&\bg \cc c_4 \cc a_4^2 \cc c_3^2 \cc b_3^2 \cc a_3 \cc b_3^{-1} \cc c_3^{-1} \cc c_1 \cc b_1 \cc  a_1^{-1} \cc b_1^{-2} \cc c_1^{-2} \cc a_2^{-2} \cc b_2^{-1} \cc c_2 \cc b_2^2 \cc a_2^2 \cc c_1^2 \cc b_1^2 \cc a_1^2 \cc e\cc b_1 \cc a_1\cc e \cc a_6^2 \cc a_5 \cc c_1 \cc b_1 \cc e^{-1},\\
&\gamma_{27} = a_1^{-2} \cc b_1^{-2} \cc c_1^{-2} \cc a_2^{-2} \cc b_2^{-2} \cc c_2^{-1} \cc b_2 \cc a_2 \cc b_3^{-1} \cc c_3^{-1} \cc c_1 \cc b_1 \cc a_1^{-1} \cc b_1^{-2} \cc c_1^{-2} \cc a_2^{-2} \cc b_2^{-1} \cc c_2 \cc b_2^2 \cc a_2^2 \cc c_1^2 \cc b_1^2 \cc a_1^2 \cc e \cc b_1 \cc a_1 \cc e\\
&\bg \cc a_6 \cc c_1^{-1} \cc a_2^{-1} \cc b_2 \cc a_2^2 \cc c_1^2 \cc b_1^2 \cc a_1 \cc b_1^{-1} \cc c_1^{-1} \cc c_3 \cc b_3 \cc a_2^{-1} \cc b_2^{-1} \cc c_2 \cc b_2^2 \cc a_2^2 \cc c_1^2 \cc b_1^2 \cc a_1^2 \cc e \cc b_1^{-1} \cc c_1^{-1} \cc c_4 \cc a_4 \cc c_1 \cc b_1\\
&\bg \cc a_1^{-1} \cc b_1^{-1} \cc e^{-1},\\
&\gamma_{28} =a_1^{-2} \cc b_1^{-2} \cc c_1^{-2} \cc a_2^{-2} \cc b_2^{-2} \cc c_2^{-1} \cc b_2 \cc a_2\cc b_3 \cc a_3 \cc b_3^{-1} \cc c_3^{-1} \cc c_1 \cc b_1 \cc a_1^{-1} \cc b_1^{-2} \cc c_1^{-2} \cc a_2^{-2} \cc b_2^{-1} \cc c_2 \cc b_2^2 \cc a_2^2 \cc c_1^2 \cc b_1^2 \cc a_1^2 \cc e \\
&\bg \cc b_1 \cc a_1 \cc e\cc a_6 \cc c_1^{-1} \cc a_2^{-1} \cc b_2 \cc a_2^2 \cc c_1^2 \cc b_1^2 \cc a_1 \cc b_1^{-1} \cc c_1^{-1} \cc c_3 \cc b_3 \cc a_2^{-1} \cc b_2^{-1} \cc c_2 \cc b_2^2 \cc a_2^2 \cc c_1^2 \cc b_1^2 \cc a_1^2 \cc e \cc b_1^{-1} \cc c_1^{-1}\\
&\bg \cc c_4 \cc a_4 \cc c_1 \cc b_1 \cc a_1^{-1} \cc b_1^{-1} \cc e^{-1},\\
&\gamma_{29} = a_1^{-2} \cc b_1^{-2} \cc c_1^{-2} \cc a_2^{-2} \cc b_2^{-2} \cc c_2^{-1} \cc b_2 \cc a_2^2 \cc c_1^2 \cc b_1^2 \cc a_1 \cc b_1^{-1} \cc c_1^{-1} \cc c_3 \cc b_3 \cc a_3^{-1} \cc b_3^{-1} \cc a_2^{-1} \cc b_2^{-1} \cc c_2 \cc b_2^2 \cc a_2^2 \cc c_1^2 \cc b_1^2 \cc a_1^2 \cc e \\
&\bg \cc b_1^{-1} \cc c_1^{-1} \cc c_4 \cc a_4^2 \cc c_3^2 \cc b_3^2 \cc a_3 \cc b_3^{-1} \cc c_3^{-1} \cc c_1 \cc b_1 \cc a_1^{-1} \cc b_1^{-2} \cc c_1^{-2} \cc a_2^{-2} \cc b_2^{-1} \cc c_2 \cc b_2^2 \cc a_2^2 \cc c_1^2 \cc b_1^2 \cc a_1^2 \cc e\\
&\bg \cc b_1 \cc a_1 \cc e \cc a_6^2 \cc a_5^2 \cc c_1 \cc b_1 \cc e^{-1},\\
&\gamma_{30} = a_1^{-2} \cc b_1^{-2} \cc c_1^{-2} \cc a_2^{-2} \cc b_2^{-2} \cc c_2^{-1} \cc b_2 \cc a_2\cc b_3^{-1} \cc c_3^{-1} \cc c_1 \cc b_1 \cc  a_1^{-1} \cc b_1^{-2} \cc c_1^{-2} \cc a_2^{-2} \cc b_2^{-1} \cc a_2^{-1}\cc b_2^{-1} \cc c_2 \cc b_2^2 \cc a_2^2 \cc c_1^2 \cc b_1^2 \\
&\bg \cc a_1^2 \cc e \cc b_1^{-1} \cc c_1^{-1} \cc a_5 \cc c_4^2 \cc a_4^2 \cc c_3^2 \cc b_3^2 \cc a_3 \cc b_3^{-1} \cc c_3^{-1} \cc c_1\cc b_1 \cc a_1{-1} \cc b_1^{-2} \cc c_1^{-2} \cc a_2^{-2} \cc b_2^{-1} \cc c_2 \cc b_2^2 \cc a_2^2 \cc c_1^2 \cc b_1^2 \cc a_1^2 \cc e\\
&\bg \cc b_1 \cc a_1 \cc b_1^{-1} \cc c_1^{-1} \cc a_4^{-1} \cc c_4^{-1} \cc c_1 \cc b_1 \cc e^{-1},\\
\end{aligned}
\end{equation*}
\begin{equation*}
\begin{aligned}
&\gamma_{31} = a_1^{-2} \cc b_1^{-2} \cc c_1^{-2} \cc a_2^{-2} \cc b_2^{-2} \cc c_2^{-1} \cc b_2 \cc a_2 \cc b_3 \cc a_3 \cc c_2 \cc b_2^2 \cc a_2^2 \cc c_1^2 \cc b_1^2 \cc a_1 \cc b_1^{-1} \cc c_1^{-1} \cc c_3 \cc b_3  \cc a_2^{-1} \cc b_2^{-1} \cc c_2 \cc b_2^2 \cc a_2^2 \cc c_1^2 \cc b_1^2 \cc a_1^2 \cc e \\
&\bg \cc b_1^{-1} \cc c_1^{-1} \cc c_4 \cc a_4^2 \cc c_3^2 \cc b_3^2 \cc a_3 \cc b_3^{-1} \cc c_3^{-1} \cc c_1 \cc b_1 \cc a_1^{-1} \cc b_1^{-2} \cc c_1^{-2} \cc a_2^{-2} \cc b_2^{-1} \cc c_2 \cc b_2^2 \cc a_2^2 \cc c_1^2 \cc b_1^2 \cc a_1^2 \cc e\\
&\bg \cc b_1 \cc a_1 \cc e \cc a_6^2 \cc a_5^2 \cc c_1 \cc b_1 \cc e^{-1},\\
&\gamma_{32} =  a_1^{-2} \cc b_1^{-2} \cc c_1^{-2} \cc a_2^{-2} \cc b_2^{-2} \cc c_2^{-1} \cc b_2 \cc a_2 \cc b_3 \cc a_3 \cc c_2 \cc b_2^2 \cc a_2^2 \cc c_1^2 \cc b_1^2 \cc a_1 \cc b_1^{-1} \cc c_1^{-1} \cc c_3 \cc b_3  \cc a_2^{-1} \cc b_2^{-1} \cc c_2 \cc b_2^2 \cc a_2^2 \cc c_1^2 \cc b_1^2 \cc a_1^2 \cc e \\
&\bg \cc b_1^{-1} \cc c_1^{-1} \cc c_4 \cc a_4^2 \cc c_3^2 \cc b_3^2 \cc a_3 \cc b_2 \cc a_2^2 \cc c_1^2 \cc b_1^2 \cc a_1\cc b_1^{-1} \cc c_1^{-1} \cc c_3 \cc b_3  \cc a_2^{-1} \cc b_2^{-1} \cc c_2 \cc b_2^2 \cc a_2^2 \cc c_1^2 \cc b_1^2 \cc a_1^2 \cc e\\
&\bg \cc b_1^{-1} \cc c_1^{-1} \cc c_4 \cc a_4 \cc c_1 \cc b_1 \cc a_1^{-1} \cc b_1^{-1} \cc e^{-1},\\
&\gamma_{33} =  a_1^{-2} \cc b_1^{-2} \cc c_1^{-2} \cc a_2^{-2} \cc b_2^{-2} \cc c_2^{-1} \cc b_2 \cc a_2 \cc b_3 \cc a_3 \cc c_2 \cc b_2^2 \cc a_2^2 \cc c_1^2 \cc b_1^2 \cc a_1 \cc b_1^{-1} \cc c_1^{-1} \cc c_3 \cc b_3^2 \cc a_3 \cc b_3^{-1} \cc c_3^{-1} \cc c_1 \cc b_1\\
&\bg  \cc a_1^{-1} \cc b_1^{-2} \cc c_1^{-2} \cc a_2^{-2} \cc b_2^{-1} \cc c_2 \cc b_2^2 \cc a_2^2 \cc c_1^2 \cc b_1^2 \cc a_1^2 \cc e \cc a_6 \cc c_1^{-1} \cc a_2^{-1} \cc b_2 \cc a_2^2 \cc c_1^2 \cc b_1^2 \cc a_1 \cc b_1^{-1} \cc c_1^{-1} \cc c_3 \cc b_3 \\
&\bg \cc a_2^{-1} \cc b_2^{-1} \cc c_2 \cc b_2^2 \cc a_2^2 \cc c_1^2 \cc b_1^2 \cc a_1^2 \cc e \cc b_1^{-1} \cc c_1^{-1} \cc c_4 \cc a_4 \cc c_1 \cc b_1 \cc a_1^{-1} \cc b_1^{-1} \cc e^{-1},\\
&\gamma_{34} =  a_1^{-2} \cc b_1^{-2} \cc c_1^{-2} \cc a_2^{-2} \cc b_2^{-2} \cc c_2^{-1} \cc b_2 \cc a_2 \cc b_3 \cc a_3 \cc c_2 \cc b_2^2 \cc a_2^2 \cc c_1^2 \cc b_1^2 \cc a_1^2 \cc e \cc c_1^{-1} \cc a_2^{-1} \cc  b_2 \cc a_2^2 \cc c_1^2 \cc b_1^2 \cc a_1 \cc b_1^{-1} \cc c_1^{-1} \cc c_3 \cc b_3\\
&\bg \cc a_2^{-2} \cc b_2^{-1} \cc c_2 \cc b_2^2 \cc a_2^2 \cc c_1^2 \cc b_1^2 \cc a_1^2 \cc e \cc b_1^{-1} \cc c_1^{-1} \cc c_4 \cc a_4^2 \cc c_3^2 \cc b_3^2 \cc a_3 \cc b_2 \cc a_2^2 \cc c_1^2 \cc b_1^2 \cc a_1 \cc b_1^{-1} \cc c_1^{-1} \cc c_3 \cc b_3 \cc a_2^{-1} \cc b_2^{-1}\\
&\bg \cc c_2 \cc b_2^2 \cc a_2^2 \cc c_1^2 \cc b_1^2 \cc a_1^2 \cc e \cc b_1^{-1} \cc c_1^{-1} \cc c_4 \cc a_4 \cc c_1 \cc b_1 \cc a_1^{-1} \cc b_1^{-1} \cc e^{-1},\\
&\gamma_{35}= a_1^{-2} \cc b_1^{-2} \cc c_1^{-2} \cc a_2^{-2} \cc b_2^{-2} \cc c_2^{-1} \cc b_2 \cc a_2 \cc b_3 \cc a_3 \cc c_2 \cc b_2^2 \cc a_2^2 \cc c_1^2 \cc b_1^2 \cc a_1^2 \cc e \cc c_1^{-1} \cc a_2^{-1} \cc  b_2 \cc a_2^2 \cc c_1^2 \cc b_1^2 \cc a_1 \cc b_1^{-1} \cc c_1^{-1} \cc c_3 \cc b_3\\
&\bg \cc  a_2^{-2} \cc b_2^{-1} \cc c_2 \cc b_2^2 \cc a_2^2 \cc c_1^2 \cc b_1^2 \cc a_1^2 \cc e \cc b_1^{-1} \cc c_1^{-1} \cc c_4 \cc a_4\cc c_1 \cc b_1 \cc e \cc a_6^2 \cc a_5^2 \cc c_4^2 \cc a_4^2 \cc c_3 \cc c_1 \cc b_1 \cc a_1^{-1} \cc b_1^{-2} \cc c_1^{-2} \cc a_2^{-2} \cc b_2^{-1}\\
&\bg \cc c_2 \cc b_2^2 \cc a_2^2 \cc c_1^2 \cc b_1^2 \cc a_1^2 \cc e \cc b_1 \cc a_1 \cc b_1^{-1} \cc c_1^{-1} \cc a_4^{-1} \cc c_4^{-1} \cc c_1 \cc b_1 \cc e^{-1},\\
&\gamma_{36} = a_1^{-2} \cc b_1^{-2} \cc c_1^{-2} \cc a_2^{-2} \cc b_2^{-2} \cc c_2^{-1} \cc b_2 \cc a_2 \cc b_2 \cc a_2^2 \cc c_1^2 \cc b_1^2 \cc a_1 \cc b_1^{-1} \cc c_1^{-1}  \cc c_3 \cc b_3 \cc a_2^{-1} \cc  b_2^{-1} \cc c_2 \cc b_2^2  \cc a_2^2 \cc c_1^2 \cc b_1^2 \cc a_1^2 \cc e\\
&\bg  \cc b_1^{-1} \cc c_1^{-1} \cc c_4 \cc a_4\cc c_1 \cc b_1 \cc a_1^{-1} \cc b_1^{-2} \cc c_1^{-2} \cc a_2^{-2} \cc b_2^{-1}  \cc c_2 \cc b_2^2 \cc a_2^2 \cc c_1^2 \cc b_1^2 \cc a_1^2 \cc e \cc b_1^{-1} \cc c_1^{-1} \cc a_5 \cc c_4 \cc c_1 \cc b_1 \cc e\\
&\bg \cc a_6^2 \cc a_5^2 \cc c_4^2 \cc a_4^2 \cc c_3 \cc c_1 \cc b_1 \cc a_1^{-1} \cc b_1^{-2} \cc c_1^{-2} \cc a_2^{-2} \cc b_2^{-1} \cc c_2 \cc b_2^2 \cc a_2^2 \cc c_1^2 \cc b_1^2 \cc a_1^2 \cc e \cc b_1^{-1} \cc c_1^{-1} \cc a_6^{-1} \cc e^{-1}\\
&\gamma_{37}= a_1^{-2} \cc b_1^{-2} \cc c_1^{-2} \cc a_2^{-2} \cc b_2^{-2} \cc c_2^{-1} \cc b_2 \cc a_2  \cc b_3^{-1} \cc c_3^{-1} \cc c_1 \cc b_1 \cc a_1^{-1} \cc b_1^{-2} \cc c_1^{-2} \cc a_2^{-2} \cc b_2^{-1}  \cc c_2 \cc b_2^2 \cc a_2^2 \cc c_1^2 \cc b_1^2 \cc a_1^2 \cc e\\
&\bg \cc c_1^{-1} \cc a_2^{-1} \cc b_2 \cc a_2^2 \cc c_1^2 \cc b_1^2 \cc a_1 \cc b_1^{-1} \cc c_1^{-1} \cc c_3 \cc b_3 \cc a_2^{-1} \cc b_2^{-1} \cc c_2 \cc b_2^2 \cc a_2^2 \cc c_1^2 \cc b_1^2 \cc a_1^2 \cc e\cc b_1^{-1} \cc c_1^{-1} \cc c_4 \cc a_4^2 \cc c_3^2 \cc b_3^2\\
&\bg a_3 \cc b_2 \cc a_2^2 \cc c_1^2 \cc b_1^2 \cc a_1 \cc b_1^{-1} \cc c_1^{-1} \cc c_3 \cc b_3 \cc a_2^{-1} \cc b_2^{-1} \cc c_2 \cc b_2^2 \cc a_2^2 \cc c_1^2 \cc b_1^2 \cc a_1^2 \cc e \cc b_1^{-1} \cc c_1^{-1} \cc c_4 \cc a_4 \cc c_1 \cc b_1 \cc a_1^{-1} \cc b_1^{-1} \cc e^{-1},\\
&\gamma_{38} = a_1^{-2} \cc b_1^{-2} \cc c_1^{-2} \cc a_2^{-2} \cc b_2^{-2} \cc c_2^{-1} \cc b_2 \cc a_2\cc b_2 \cc a_2^2 \cc c_1^2 \cc b_1^2 \cc a_1 \cc b_1^{-1} \cc c_1^{-1} \cc c_3 \cc b_3 \cc a_2^{-1} \cc b_2^{-1} \cc c_2 \cc b_2^2 \cc a_2^2 \cc c_1^2 \cc b_1^2 \cc a_1^2 \cc e\\
&\bg \cc b_1^{-1} \cc c_1^{-1} \cc c_4 \cc a_4 \cc c_1 \cc b_1 \cc a_1^{-1} \cc b_1^{-2} \cc c_1^{-2} \cc a_2^{-2} \cc b_2^{-1} \cc c_2 \cc b_2^2 \cc a_2^2 \cc c_1^2 \cc b_1^2 \cc a_1^2 \cc e\cc b_1^{-1} \cc c_1^{-1} \cc a_5 \cc c_4 \cc c_1 \cc b_1 \cc e\cc a_6^2 \cc a_5^2\\
&\bg  \cc c_4^2 \cc a_4^2 \cc c_3 \cc c_1 \cc b_1 \cc  a_1^{-1} \cc b_1^{-2} \cc c_1^{-2} \cc a_2^{-2} \cc b_2^{-1} \cc c_2 \cc b_2^2 \cc a_2^2 \cc c_1^2 \cc b_1^2 \cc a_1^2 \cc e \cc b_1 \cc a_1 \cc b_1^{-1} \cc c_1^{-1} \cc a_4^{-1} \cc c_4^{-1} \cc c_1 \cc b_1 \cc e^{-1},\\
&\gamma_{39} = b_1^{-1} \cc c_1^{-1} \cc a_4 \cc c_3 \cc a_2^{-1} \cc b_2^{-1} \cc c_2 \cc b_2^2 \cc a_2^2 \cc c_1^2 \cc b_1^2 \cc a_1^2 \cc e \cc b_1^{-1} \cc c_1^{-1} \cc c_4\cc a_4^2 \cc c_3^2 \cc b_3^2 \cc a_3 \cc b_3^{-1} \cc c_3^{-1} \cc c_1 \cc b_1 \\
&\bg \cc  a_1^{-1} \cc b_1^{-2} \cc c_1^{-2} \cc a_2^{-2} \cc b_2^{-1} \cc c_2 \cc b_2^2 \cc a_2^2 \cc c_1^2 \cc b_1^2 \cc a_1^2 \cc e \cc b_1 \cc a_1 \cc b_1^{-1} \cc c_1^{-1} \cc c_4 \cc a_4^2 \cc c_3^2 \cc b_3^2 \cc a_3 \cc b_3^{-1} \cc c_3^{-1} \cc c_1 \cc b_1\\
&\bg \cc  a_1^{-1} \cc b_1^{-2} \cc c_1^{-2} \cc a_2^{-2} \cc b_2^{-1} \cc c_2 \cc b_2^2 \cc a_2^2 \cc c_1^2 \cc b_1^2 \cc a_1^2 \cc e \cc b_1 \cc a_1 \cc e \cc a_6^2 \cc a_5 \cc c_1 \cc b_1 \cc e^{-1},\\
&\gamma_{40} =  a_1^{-2} \cc b_1^{-2} \cc c_1^{-2} \cc a_2^{-2} \cc b_2^{-2} \cc c_2^{-1} \cc b_2 \cc a_2\cc b_2 \cc a_2^2 \cc c_1^2 \cc b_1^2 \cc a_1 \cc b_1^{-1} \cc c_1^{-1} \cc c_3 \cc b_3 \cc a_2^{-1} \cc b_2^{-1} \cc c_2 \cc b_2^2 \cc a_2^2 \cc c_1^2 \cc b_1^2 \cc a_1^2 \cc e\\
&\bg \cc b_1^{-1} \cc c_1^{-1} \cc c_4 \cc a_4 \cc c_1 \cc b_1 \cc a_1^{-1} \cc b_1^{-2} \cc c_1^{-2} \cc a_2^{-2} \cc b_2^{-1} \cc c_2 \cc b_2^2 \cc a_2^2 \cc c_1^2 \cc b_1^2 \cc a_1^2 \cc e\cc b_1^{-1} \cc c_1^{-1} \cc a_5 \cc c_4 \cc c_1 \cc b_1 \cc e\cc a_6^2 \cc a_5^2\\
&\bg  \cc c_4^2 \cc a_4^2 \cc c_3 \cc c_1 \cc b_1 \cc  a_1^{-1} \cc b_1^{-2} \cc c_1^{-2} \cc a_2^{-2} \cc b_2^{-1} \cc c_2 \cc b_2^2 \cc a_2^2 \cc c_1^2 \cc b_1^2 \cc a_1^2 \cc e \cc b_1^{-1} \cc c_1^{-2} \cc a_2^{-2} \cc b_2^{-1} \cc c_2 \cc b_2^2 \cc a_2^2 \\
&\bg \cc c_1^2 \cc b_1^2 \cc a_1^2 \cc e \cc b_1^{-1} \cc c_1^{-1} \cc a_5^{-1} \cc a_6^{-2} \cc e^{-1},\\
&\gamma_{41} = b_1^{-1} \cc c_1^{-1} \cc a_4 \cc c_3\cc c_1 \cc b_1 \cc a_1^{-1} \cc b_1^{-2} \cc c_1^{-2} \cc a_2^{-2} \cc b_2^{-2} \cc c_2^{-1} \cc a_3^{-1} \cc b_3^{-1} \cc a_2^{-1} \cc b_2^{-1}  \cc c_2 \cc b_2^2 \cc a_2^2 \cc c_1^2 \cc b_1^2 \cc a_1^2 \cc e \cc b_1^{-1} \cc c_1^{-1}\\
&\bg  \cc c_4\cc a_4^2 \cc c_3^2 \cc b_3^2 \cc a_3 \cc b_3^{-1} \cc c_3^{-1} \cc c_1 \cc b_1 \cc  a_1^{-1} \cc b_1^{-2} \cc c_1^{-2} \cc a_2^{-2} \cc b_2^{-1} \cc c_2 \cc b_2^2 \cc a_2^2 \cc c_1^2 \cc b_1^2 \cc a_1^2 \cc e \cc b_1 \cc a_1 \cc b_1^{-1} \cc c_1^{-1}  \cc a_4 \cc c_3^2\\
&\bg \cc b_3 \cc a_2^{-1} \cc b_2^{-1} \cc c_2 \cc b_2^2 \cc a_2^2 \cc c_1^2 \cc b_1^2 \cc a_1^2 \cc e \cc b_1^{-1} \cc c_1^{-1} \cc c_4 \cc a_4^2 \cc c_3^2 \cc b_3^2 \cc a_3 \cc b_3^{-1} \cc c_3^{-1} \cc c_1 \cc b_1 \cc  a_1^{-1} \cc b_1^{-2} \cc c_1^{-2} \cc a_2^{-2} \cc b_2^{-1} \\
&\bg \cc c_2 \cc b_2^2 \cc a_2^2 \cc c_1^2 \cc b_1^2 \cc a_1^2 \cc e \cc b_1 \cc a_1 \cc e \cc a_6^2 \cc a_5 \cc c_1 \cc b_1 \cc e^{-1},\\
\end{aligned}
\end{equation*}
\begin{equation*}
\begin{aligned}
&\gamma_{42} = a_1^{-2} \cc b_1^{-2} \cc c_1^{-2} \cc a_2^{-2} \cc b_2^{-2} \cc c_2^{-1} \cc b_2 \cc a_2\cc b_3^{-1} \cc c_3^{-1} \cc c_1 \cc b_1 \cc a_1^{-1} \cc b_1^{-2} \cc c_1^{-2} \cc a_2^{-2} \cc b_2^{-2} \cc c_2^{-1} \cc a_3^{-1} \cc b_3^{-1} \cc a_2^{-1} \cc b_2^{-1}\\
&\bg  \cc c_2 \cc b_2^2 \cc a_2^2 \cc c_1^2 \cc b_1^2 \cc a_1^2 \cc e \cc b_1^{-1} \cc c_1^{-1} \cc c_4 \cc a_4^2 \cc c_3^2 \cc b_3^2 \cc a_3 \cc b_3^{-1} \cc c_3^{-1} \cc c_1 \cc b_1 \cc a_1^{-1} \cc b_1^{-2} \cc c_1^{-2} \cc a_2^{-2} \cc b_2^{-1} \cc c_2 \cc b_2^2 \\
&\bg \cc a_2^2 \cc c_1^2 \cc b_1^2 \cc a_1^2 \cc e \cc b_1 \cc a_1 \cc b_1^{-1} \cc c_1^{-1} \cc a_4 \cc c_3^2 \cc b_3 \cc a_2^{-1} \cc b_2^{-1} \cc c_2 \cc b_2^2 \cc a_2^2 \cc c_1^2 \cc b_1^2 \cc a_1^2 \cc e \cc b_1^{-1} \cc c_1^{-1} \cc c_4 \cc a_4^2 \cc c_3^2 \\
&\bg \cc b_3^2 \cc a_3 \cc b_3^{-1} \cc c_3^{-1} \cc c_1 \cc b_1 \cc a_1^{-1} \cc b_1^{-2} \cc c_1^{-2} \cc a_2^{-2} \cc b_2^{-1} \cc c_2 \cc b_2^2 \cc. a_2^2 \cc c_1^2 \cc b_1^2 \cc a_1^2 \cc e \cc b_1 \cc a_1 \cc e \cc a_6^2 \cc a_5 \cc c_1 \cc b_1 \cc e^{-1},\\
\end{aligned}
\end{equation*}
and 
\begin{equation*}
\begin{aligned}
& \zeta_1=  a_1^{-2} \cc b_1^{-2} \cc c_1^{-2} \cc a_2^{-2} \cc b_2^{-2} \cc c_2^{-1} \cc b_2 \cc a_2^2 \cc c_1^2 \cc b_1^2 \cc a_1 \cc b_1^{-1},\\
&\zeta_2 =  a_1^{-2} \cc b_1^{-2} \cc c_1^{-2} \cc a_2^{-2} \cc b_2^{-2} \cc c_2^{-1} \cc b_2 \cc a_2 \cc b_2 \cc a_2^2 \cc c_1^2 \cc b_1^2 \cc a_1 \cc b_1^{-1},\\
&\zeta_3 =  a_1^{-2} \cc b_1^{-2} \cc c_1^{-2} \cc a_2^{-2} \cc b_2^{-2} \cc c_2^{-1} \cc b_2 \cc a_2 \cc b_3 \cc a_3 \cc c_2 \cc b_2^2 \cc a_2^2 \cc c_1^2 \cc b_1^2 \cc a_1 \cc b_1^{-1},\\
&\zeta_4 =  a_1^{-2} \cc b_1^{-2} \cc c_1^{-2} \cc a_2^{-2} \cc b_2^{-2} \cc c_2^{-1} \cc b_2 \cc a_2 \cc b_3. \cc a_3 \cc c_2 \cc b_2^2 \cc a_2^2 \cc c_1^2 \cc b_1^2 \cc a_1^2 \cc e,\\
&\zeta_5 =  a_1^{-2} \cc b_1^{-2} \cc c_1^{-2} \cc a_2^{-2} \cc b_2^{-2} \cc c_2^{-1} \cc b_2 \cc a_2 \cc b_3^{-1} \cc c_3^{-1} \cc c_1 \cc b_1 \cc  a_1^{-1} \cc b_1^{-2} \cc c_1^{-2} \cc a_2^{-2} \cc b_2^{-1} \cc c_2 \cc b_2^2 \cc a_2^2 \cc c_1^2 \cc b_1^2 \cc a_1^2 \cc e,\\
&\zeta_6 =  a_1^{-2} \cc b_1^{-2} \cc c_1^{-2} \cc a_2^{-2} \cc b_2^{-2} \cc c_2^{-1} \cc b_2 \cc a_2 \cc b_3^{-1} \cc c_3^{-1} \cc c_1 \cc b_1 \cc  a_1^{-1} \cc b_1^{-2} \cc c_1^{-2} \cc a_2^{-2} \cc b_2^{-1} \cc c_2 \cc b_2^2 \cc a_2^2 \cc c_1^2 \cc b_1^2 \cc a_1^2 \cc e\\
&\bg \cc b_1 \cc a_1 \cc e \cc a_6,\\
&\zeta_7 =  a_1^{-2} \cc b_1^{-2} \cc c_1^{-2} \cc a_2^{-2} \cc b_2^{-2} \cc c_2^{-1} \cc b_2 \cc a_2\cc b_3 \cc a_3 \cc b_3^{-1} \cc c_3^{-1} \cc c_1 \cc b_1 \cc  a_1^{-1} \cc b_1^{-2} \cc c_1^{-2} \cc a_2^{-2} \cc b_2^{-1} \cc c_2 \cc b_2^2 \cc a_2^2 \cc c_1^2\\
&\bg  \cc b_1^2 \cc a_1^2 \cc e \cc b_1 \cc a_1 \cc e \cc a_6,\\
&\zeta_8 =  a_1^{-2} \cc b_1^{-2} \cc c_1^{-2} \cc a_2^{-2} \cc b_2^{-2} \cc c_2^{-1} \cc b_2 \cc a_2 \cc b_3^{-1} \cc c_3^{-1} \cc c_1 \cc b_1 \cc  a_1^{-1} \cc b_1^{-2} \cc c_1^{-2} \cc a_2^{-2} \cc b_2^{-1} \cc b_3^{-1} \cc c_3^{-1} \cc c_1 \cc b_1 \\
&\bg  \cc  a_1^{-1} \cc b_1^{-2} \cc c_1^{-2} \cc a_2^{-2} \cc b_2^{-1} \cc c_2 \cc b_2^2 \cc a_2^2 \cc c_1^2 \cc b_1^2 \cc a_1^2 \cc e \cc b_1 \cc a_1 \cc e \cc a_6,\\
&\mu_1 = c_1^{-1} \cc c_3 \cc b_3 \cc a_2 \cc c_1^2 \cc b_1 \cc e^{-1}, \quad \mu_2 = c_1^{-1} \cc c_3 \cc b_3 \cc a_3^{-1} \cc b_3^{-2} \cc c_3^{-2} \cc a_4^{-2} \cc c_4^{-1} \cc c_1 \cc b_1 \cc e^{-1},\\
&\mu_3 = c_1^{-1} \cc c_3 \cc b_3 \cc a_3^{-1} \cc b_3^{-2} \cc c_3^{-2} \cc a_4^{-2} \cc c_4^{-2} \cc a_5^{-2} \cc a_6^{-2} \cc e^{-1},\\
&\mu_4 = c_1^{-1} \cc a_2^{-2} \cc b_2^{-1} \cc c_2 \cc b_2^2 \cc a_2^2 \cc c_1^2 \cc b_1^2 \cc a_1^2 \cc e \cc b_1^{-1} \cc c_1^{-1} \cc a_6^{-1} \cc e^{-1},\\
&\mu_5 =c_1^{-1} \cc a_2^{-2} \cc b_2^{-1} \cc c_2 \cc b_2^2 \cc a_2^2 \cc c_1^2 \cc b_1^2 \cc a_1^2 \cc e \cc  b_1^{-1} \cc c_1^{-1} \cc a_5^{-1} \cc a_6^{-2} \cc e^{-1},\\ 
&\mu_6 =c_1^{-1} \cc a_2^{-2} \cc b_2^{-1} \cc c_2 \cc b_2^2 \cc a_2^2 \cc c_1^2 \cc b_1^2 \cc a_1^2 \cc e \cc  b_1^{-1} \cc c_1^{-1}\cc a_5 \cc c_4 \cc c_1 \cc b_1 \cc e^{-1},\\
&\mu_7 =c_1^{-1} \cc c_3 \cc b_3 \cc a_2^{-1} \cc b_2^{-1} \cc c_2 \cc b_2^2 \cc a_2^2 \cc c_1^2 \cc b_1^2 \cc a_1^2 \cc e \cc  b_1^{-1} \cc c_1^{-1} \cc c_4 \cc a_4  \cc c_1 \cc b_1 \cc a_1^{-1} \cc b_1^{-1} \cc e^{-1},\\ 
&\mu_8 = c_1^{-1} \cc c_3^{-1} \cc a_4^{-1} \cc c_1 \cc b_1 \cc e \cc a_6^2 \cc a_5^2 \cc c_4^2 \cc a_4^2 \cc c_3 \cc c_1 \cc b_1 \cc a_1^{-1} \cc b_1^{-2} \cc c_1^{-2} \cc a_2^{-2} \cc b_2^{-1} \cc c_2 \cc b_2^2 \cc a_2^2 \cc c_1^2 \cc b_1^2 \cc a_1^2 \cc e\\
&\bg \cc b_1 \cc a_1 \cc b_1^{-1} \cc c_1^{-1} \cc a_4^{-1} \cc c_4^{-1} \cc c_1 \cc b_1\cc e^{-1},\\
&\mu_9=  c_1^{-1} \cc a_2^{-2} \cc b_2^{-1} \cc c_2 \cc b_2^2 \cc a_2^2 \cc c_1^2 \cc b_1^2 \cc a_1^2 \cc e \cc b_1^{-1} \cc c_1^{-1} \cc a_5 \cc c_4 \cc c_1 \cc b_1 \cc e  \cc a_6^2 \cc a_5^2 \cc c_4^2 \cc a_4^2 \cc c_3 \cc c_1 \cc b_1 \cc a_1^{-1} \cc b_1^{-2} \cc c_1^{-2} \\
&\bg \cc a_2^{-2} \cc b_2^{-1} \cc c_2 \cc b_2^2 \cc a_2^2 \cc c_1^2 \cc b_1^2 \cc a_1^2 \cc e \cc b_1^{-1} \cc c_1^{-1} \cc a_6^{-1} \cc e^{-1},\\
&\mu_{10} = c_1^{-1}\cc c_3 \cc b_3  \cc a_2^{-1} \cc b_2^{-1} \cc c_2 \cc b_2^2 \cc a_2^2 \cc c_1^2 \cc b_1^2 \cc a_1^2 \cc e \cc b_1^{-1} \cc c_1^{-1} \cc c_4 \cc a_4^2 \cc c_3^2 \cc b_3^2 \cc a_3 \cc b_3^{-1} \cc c_3^{-1} \cc c_1 \cc b_1 \cc a_1^{-1} \cc b_1^{-2} \cc c_1^{-2} \\
&\bg \cc a_2^{-2} \cc b_2^{-1} \cc c_2 \cc b_2^2 \cc a_2^2 \cc c_1^2 \cc b_1^2 \cc a_1^2 \cc e \cc b_1 \cc a_1 \cc e \cc a_6^2 \cc a_5 \cc c_1 \cc b_1 \cc e^{-1},\\
&\mu_{11} = c_1^{-1}\cc c_3 \cc b_3 \cc a_3^{-1} \cc b_3^{-1}  \cc a_2^{-1} \cc b_2^{-1} \cc c_2 \cc b_2^2 \cc a_2^2 \cc c_1^2 \cc b_1^2 \cc a_1^2 \cc e \cc b_1^{-1} \cc c_1^{-1} \cc c_4 \cc a_4^2 \cc c_3^2 \cc b_3^2 \cc a_3 \cc b_3^{-1} \cc c_3^{-1} \cc c_1 \cc b_1 \cc a_1^{-1}  \\
&\bg \cc b_1^{-2} \cc c_1^{-2} \cc a_2^{-2} \cc b_2^{-1} \cc c_2 \cc b_2^2 \cc a_2^2 \cc c_1^2 \cc b_1^2 \cc a_1^2 \cc e \cc b_1 \cc a_1 \cc e \cc a_6^2 \cc a_5 \cc c_1 \cc b_1 \cc e^{-1},\\
&\mu_{12} =  c_1^{-1}\cc a_2^{-1} \cc b_2^{-1}  \cc c_2 \cc b_2^2 \cc a_2^2 \cc c_1^2 \cc b_1^2 \cc a_1^2 \cc e \cc b_1^{-1} \cc c_1^{-1} \cc a_5 \cc c_4 \cc c_1 \cc b_1 \cc e \cc a_6^2 \cc a_5^2 \cc c_4^2 \cc a_4^2 \cc c_3 \cc c_1 \cc b_1 \cc a_1^{-1}\cc b_1^{-2} \cc c_1^{-2} \cc a_2^{-2}   \\
&\bg \cc b_2^{-1} \cc c_2 \cc b_2^2 \cc a_2^2 \cc c_1^2 \cc b_1^2 \cc a_1^2 \cc e \cc b_1 \cc a_1 \cc b_1^{-1} \cc c_1^{-1} \cc a_4^{-1} \cc c_4^{-1}  \cc c_1 \cc b_1 \cc e^{-1},\\
&\mu_{13} = c_1^{-1}\cc c_3 \cc b_3  \cc a_2^{-1} \cc b_2^{-1} \cc c_2 \cc b_2^2 \cc a_2^2 \cc c_1^2 \cc b_1^2 \cc a_1^2 \cc e \cc b_1^{-1} \cc c_1^{-1} \cc c_4 \cc a_4 \cc c_1 \cc b_1 \cc e\cc a_6^2 \cc a_5^2 \cc c_4^2 \cc a_4^2 \cc c_3 \cc c_1 \cc b_1 \cc a_1^{-1} \cc b_1^{-2} \cc c_1^{-2} \\
&\bg \cc a_2^{-2} \cc b_2^{-1} \cc c_2 \cc b_2^2 \cc a_2^2 \cc c_1^2 \cc b_1^2 \cc a_1^2 \cc e \cc b_1 \cc a_1 \cc b_1^{-1} \cc c_1^{-1} \cc a_4^{-1} \cc c_4^{-1} \cc c_1 \cc b_1 \cc e^{-1},\\
&\eta_1 = c_1^{-1} \cc a_2 ^{-1} \cc b_2 \cc a_2^2 \cc c_1^2 \cc b_1^2 \cc a_1 \cc b_1^{-1},\quad \quad 
\eta_2 = c_1^{-1} \cc c_3 \cc b_3 \cc a_2^{-1} \cc b_2^{-1} \cc c_2 \cc b_2^2 \cc a_2^2 \cc c_1^2 \cc b_1^2 \cc a_1^2 \cc e \cc b_1^{-1} \cc c_1^{-1} \cc c_4 \cc a_4 \cc c_1 \cc b_1 \cc a_1^{-1} \cc b_1^{-2} \cc c_1^{-1}\\
&\eta_3 = c_1^{-1} \cc c_3 \cc b_3^2 \cc a_3 \cc b_3^{-1} \cc c_3^{-1} \cc c_1 \cc b_1 \cc a_1^{-1}\cc b_1^{-2} \cc c_1^{-2} \cc a_2^{-2}  \cc b_2^{-1} \cc c_2 \cc b_2^2 \cc a_2^2 \cc c_1^2 \cc b_1^2 \cc a_1^2 \cc e \cc b_1 \cc a_1 \cc e \cc a_6,\\
\end{aligned}
\end{equation*}
\begin{equation*}
\begin{aligned}&\eta_4 = c_1^{-1} \cc c_3 \cc b_3 \cc a_2^{-1} \cc b_2^{-1} \cc c_2 \cc b_2^2 \cc a_2^2 \cc c_1^2 \cc b_1^2 \cc a_1^2 \cc e \cc b_1^{-1} \cc c_1^{-1} \cc c_4 \cc a_4^2 \cc c_3^2 \cc b_3^2 \cc a_3 \cc b_2 \cc a_2^2. \cc c_1^2 \cc b_1^2 \cc a_1 \cc b_1^{-1},\\
&\eta_5= c_1^{-1} \cc a_2^{-2} \cc b_2^{-1} \cc c_2 \cc b_2^2 \cc a_2^2 \cc c_1^2 \cc b_1^2 \cc a_1^2 \cc e \cc b_1^{-1} \cc c_1^{-1} \cc a_5 \cc c_4 \cc c_1 \cc b_1 \cc e \cc a_6^2 \cc a_5^2 \cc c_4^2 \cc a_4^2 \cc c_3 \cc c_1 \cc b_1 \cc a_1^{-1}\cc b_1^{-2} \cc c_1^{-2}\\
&\bg  \cc a_2^{-2}  \cc b_2^{-1} \cc c_2 \cc b_2^2 \cc a_2^2 \cc c_1^2 \cc b_1^2 \cc a_1^2 \cc e \cc b_1^{-1} \cc c_1^{-1}.\\
 \end{aligned}
\end{equation*}
The admissible $4$-tuples are 
\begin{equation*}
\begin{aligned}
A&= \{ (1,6,11,28),(1,11,8,6),(1,11,8,20),(1,11,8,22),(1,11,16,14),(1,11,17,30),(1,11,19,11),\\
&(1,11,25,9),(1,11,28,12),(1,11,31,37),(1,11,32,11),(1,11,32,12),(1,11,33,12),(1,11,34,12),\\
&(1,11,35,9),(1,11,42,23),(1,12,1,11),(1,12,2,8),(1,12,3,7),(1,12,5,2),(1,12,18,9),(1,12,26,7),\\
&(1,12,26,15),(1,12,39,15),(1,12,39,23),(1,12,41,23),(1,29,13,12),(1,29,21,1),(1,29,27,12),\\
&(1,29,36,1),(1,29,37,12),(1,29,38,14),(1,29,38,30),(2,8,10,1),(2,8,22,23),(2,8,22,24),\\
&(2,9,1,6),(2,9,1,11),(2,9,1,29),(2,14,1,11),(2,30,9,1),(3,7,14,1),(4,7,14,1),(5,2,8,22),\\
&(6,11,28,12),(6,11,42,23),(7,14,1,11),(7,37,12,2),(8,6,11,28),(8,6,11,42),(8,10,1,11),\\
&(8,20,1,6),(8,22,23,1),(8,22,24,17),(8,22,36,1),(8,22,40,1),(9,1,6,11),(9,1,11,8),(9,1,11,42),\\
&(9,1,29,13),(9,1,29,21),(9,1,29,27),(9,1,29,36),(9,1,29,37),(9,1,29,38),(10,1,11,28),\\
&(11,8,6,11),(11,8,20,1),(11,8,22,36),(11,8,22,40),(11,16,14,1),(11,17,30,9),(11,19,11,17),\\
&(11,19,11,25),(11,19,11,31),(11,19,11,32),(11,25,9,1),(11,28,12,2),(11,28,12,26),\\
&(11,28,12,41),(11,31,37,12),(11,32,11,17),(11,32,12,1),(11,32,12,2),(11,33,12,2),\\
&(11,34,12,2),(11,35,9,1),(11,42,23,1),(12,1,11,25),(12,2,8,10),(12,2,8,22),(12,3,7,14),\\
&(12,4,7,14),(12,5,2,8),(12,18,9,1),(12,26,7,14),(12,26,15,12),(12,39,15,12),(12,39,23,1),\\
&(12,41,23,1),(13,7,14,1),(13,7,37,12),(13,12,3,7),(14,1,11,16),(14,1,11,19),(14,1,11,28),\\
&(14,1,11,34),(14,1,11,35),(15,12,3,7),(15,12,4,7),(15,12,18,9),(15,12,39,15),(15,12,39,23),\\
&(16,14,1,11),(17,9,1,6),(17,9,1,29),(17,30,9,1),(18,9,1,11),(19,11,17,30),(19,11,25,9),\\
&(19,11,31,37),(19,11,32,11),(19,11,32,12),(20,1,6,11),(21,1,6,11),(22,23,1,11),(22,24,17,9),\\
&(22,24,17,30),(22,36,1,6),(22,40,1,11),(23,1,11,25),(23,1,11,33),(24,17,9,1),(24,17,30,9),\\
&(25,9,1,29),(26,7,14,1),(26,15,12,3),(26,15,12,4),(26,15,12,18),(26,15,12,39),(27,12,4,7),\\
&(27,12,5,2),(28,12,2,8),(28,12,26,7),(28,12,26,15),(28,12,41,23),(29,13,7,14),(29,13,7,37),\\
&(29,13,12,3),(29,21,1,6),(29,27,12,4),(29,27,12,5),(29,36,1,6),(29,37,12,2),(29,38,14,1),\\
&(29,38,30,9),(30,9,1,29),(31,37,12,2),(32,11,17,30),(32,12,1,11),(32,12,2,8),(33,12,2,8),\\
&(34,12,2,8),(35,9,1,11),(35,9,1,29),(36,1,6,11),(37,12,2,8),(38,14,1,11),(38,30,9,1),\\
&(39,15,12,3),(39,23,1,11),(40,1,11,25),(41,23,1,11),(42,23,1,11)\}
\end{aligned}
\end{equation*}
Let $V= Span \, \Gamma$ and $W= Span\, \Delta$. The splitting map $S: V\to W$ is given by 
\begin{equation*}
S: \left\{ \begin{aligned} & \gamma_1 \mapsto  \eta_1, \quad \gamma_2 \mapsto  \mu_5 + \gamma_1 + \zeta_1, \quad \gamma_3 \mapsto \mu_4 + \gamma_1 + \zeta_1, \quad \gamma_4 \mapsto \mu_9 + \gamma_1 + \zeta_1, \quad \gamma_5 \mapsto \eta_5, \\
& \gamma_6 \mapsto \eta_4, \quad  \gamma_7 \mapsto \mu_1 + \zeta_1, \quad \gamma_8 \mapsto \mu_2 + \zeta_3, \quad \gamma_9 \mapsto \mu_2 + \zeta_4, \quad \gamma_{10} \mapsto \eta_3, \quad \gamma_{11} \mapsto \mu_7 + \zeta_1, \\
& \gamma_{12} \mapsto \eta_2,\quad \gamma_{13} \mapsto \mu_2 + \gamma_8 + \zeta_1, \quad \gamma_{14} \mapsto \mu_2 + \zeta_7, \quad \gamma_{15}  \mapsto \mu_{11} + \zeta_2, \\
&\gamma_{16} \mapsto \mu_3 + \gamma_2 + \gamma_8 + \gamma_{10} +\gamma_1 + \zeta_1, \quad \gamma_{17} \mapsto \mu_3 + \gamma_3 + \gamma_7 + \gamma_{14} + \gamma_1 + \zeta_1,\\
&\gamma_{18} \mapsto \mu_{12} + \gamma_{14} + \gamma_1 + \zeta_1, \quad \gamma_{19}  \mapsto \mu_3  + \gamma_{26} + \zeta_2, \quad \gamma_{20} \mapsto \mu_{10} + \zeta_5,\\
& \gamma_{21} \mapsto \mu_{11} + \zeta_5, \quad \gamma_{22} \mapsto \mu_8 + \gamma_9 + \gamma_1 + \zeta_1, \quad \gamma_{23}  \mapsto \mu_{11} +\zeta_6,\\
 &\gamma_{24} \mapsto \mu_{11} + \gamma_{13} + \zeta_1, \quad \gamma_{25} \mapsto \mu_3 + \gamma_4 + \gamma_7 + \gamma_ {14} + \gamma_1  + \zeta_1, \\
 &\gamma_{26} \mapsto \mu_6 + \gamma_{17} + \gamma_9 + \gamma_1 + \zeta_1, \quad \gamma_{27} \mapsto \mu_2 + \gamma_8 + \gamma_{22} + \zeta_2, \\ 
 \end{aligned}\right.
\end{equation*} 
\begin{equation*}
S: \left\{ \begin{aligned}
 &\gamma_{28} \mapsto \mu_3 + \gamma_2 + \gamma_8 + \gamma_{22} + \zeta_2, \quad \gamma_{29} \mapsto \mu_{13} + \gamma_9 + \gamma_1 + \zeta_1, \\ 
 &\gamma_{30} \mapsto \mu_2 + \gamma_{16} + \gamma_{14} + \gamma_1 + \zeta_1, \quad \gamma_{31} \mapsto \mu_3 + \gamma_{18} + \gamma_9 + \gamma_1 + \zeta_1,\\
 &\gamma_{32} \mapsto \mu_3 + \gamma_{39} + \zeta_2, \quad \gamma_{33} \mapsto \mu_3 + \gamma_5 + \gamma_2 + \gamma_8 + \gamma_{22} + \zeta_2,\\
 &\gamma_{34} \mapsto \mu_3 + \gamma_{41} + \zeta_2, \quad \gamma_{35} \mapsto \mu_3 + \gamma_{26} + \gamma_7 + \gamma_{14} + \gamma_1 + \zeta_1,\\
 &\gamma_{36} \mapsto \mu_{11} + \gamma_{13} + \gamma_7 + \zeta_5, \quad \gamma_{37} \mapsto \mu_2 + \gamma_{42} + \zeta_2,\quad \gamma_{38} \mapsto \mu_{11} + \gamma_{13} + \gamma_7 + \gamma_{14} +\gamma_1 + \zeta_1\\
 & \gamma_{39} \mapsto \mu_{12} + \gamma_{30} + \gamma_9 + \gamma_1 + \zeta_1,\quad \gamma_{40} \mapsto \mu_{11} + \gamma_{13} + \gamma_7 + \zeta_8,\\
 &\gamma_{41} \mapsto \mu_6 + \gamma_{17} + \gamma_{30} + \gamma_9 + \gamma_1 + \zeta_1, \quad \gamma_{42} \mapsto \mu_2 + \gamma_{17} + \gamma_{30} + \gamma_9 + \gamma_1 + \zeta_1.\\
 \end{aligned}\right.
\end{equation*}
and the merging map $M:V \to W$ is
\begin{equation*}
M: \left\{ \begin{aligned} &\gamma_1 \mapsto \gamma_1,\quad \gamma_2 \mapsto \gamma_2,\quad \gamma_3 \mapsto \gamma_3,\quad \gamma_4 \mapsto \gamma_4,\quad \gamma_5 \mapsto \gamma_5 ,\quad \gamma_6 \mapsto \zeta_1 + \mu_1,\\
&\gamma_7 \mapsto \gamma_7,\quad \gamma_8 \mapsto \gamma_8,\quad \gamma_9 \mapsto \gamma_9,\quad \gamma_{10} \mapsto \gamma_{10},\quad \gamma_{11} \mapsto \zeta_1 + \mu_2,\\
&\gamma_{12} \mapsto \zeta_1 + \mu_3,\quad \gamma_{13} \mapsto \gamma_{13},\quad \gamma_{14} \mapsto \gamma_{14},\quad \gamma_{15} \mapsto \zeta_2 + \mu_7,\quad \gamma_{16} \mapsto \gamma_{16},\\
&\gamma_{17} \mapsto \gamma_{17},\quad \gamma_{18} \mapsto \gamma_{18},\quad \gamma_{19} \mapsto \zeta_4 + \eta_1 + \mu_7,\quad \gamma_{20} \mapsto \zeta_1 + \eta_2 + \mu_4,\\
&\gamma_{21} \mapsto \zeta_2 + \eta_2 + \mu_4 ,\quad \gamma_{22} \mapsto \gamma_{22},\quad \gamma_{23} \mapsto \zeta_2 + \eta_2 + \mu_5,\quad \gamma_{24} \mapsto \zeta_2 + \eta_2 + \mu_6,\\
&\gamma_{25} \mapsto \zeta_2 + \mu_8,\quad \gamma_{26} \mapsto \gamma_{26},\quad \gamma_{27} \mapsto \zeta_6 + \eta_1 + \mu_7,\quad \gamma_{28} \mapsto \zeta_7+ \eta_1 + \mu_7,\\
&\gamma_{29} \mapsto \zeta_1 + \mu_{11},\quad \gamma_{30} \mapsto \gamma_{30},\quad \gamma_{31} \mapsto \zeta_3 + \mu_{10} ,\quad \gamma_{32} \mapsto \zeta_3 + \eta_4 + \mu_7,\\
&\gamma_{33} \mapsto \zeta_3 + \eta_3 + \eta_1 + \mu_7,\quad \gamma_{34} \mapsto \zeta_6 + \eta_1 + \mu_7,\quad \gamma_{35} \mapsto \zeta_4 + \eta_1 + \mu_{13},,\quad \gamma_{36} \mapsto \zeta_2+ \eta_2 + \mu_9,\\
&\gamma_{37} \mapsto \zeta_5 + \eta_1 + \eta_4 + \mu_7,\quad \gamma_{38} \mapsto \zeta_2 + \eta_2 + \mu_{12},\quad \gamma_{39} \mapsto \gamma_{39},\\ 
&\gamma_{40} \mapsto \zeta_2 + \eta_2 + \eta_5 + \mu_5,,\quad \gamma_{41} \mapsto \gamma_{41},\quad \gamma_{42} \mapsto \gamma_{42}.
\end{aligned}\right.
\end{equation*}

There is a $v_\lambda \in V$ such that $Sv_\lambda = \lambda Mv_\lambda$ where $\lambda\approx 1.55943 $ is the largest real root of 
\begin{equation*}
\begin{aligned}
\chi(t) = \ & t^{26} + t^{25} + t^{24} - t^{23} -3 t^{22} -5 t^{21} - 5 t^{20} - 4 t^{19} - 2 t^{18} - t^{17} - 3 t^{16} -5 t^{15} \\
& -9 t^{14} - 8 t^{13} - 5t^{12} + 4 t^{10} + 9 t^9+9 t^8+8 t^7 + 4 t^6 + t^5 + t^3 + t^2 + t+1 \\
 \end{aligned}
 \end{equation*}
and
\begin{equation*}
\begin{aligned}
 v_\lambda \approx ( & 22.427,3.726,2.086,1.707,0.271,4.482,6.990,6.405,9.902,0.486,\\
 &20.164,12.245,2.757,7.218,1.513,0.758,3.252,0.442,2.071,1.074,\\
 &0.264,4.150,2.215,2.611,2.661,3.443,1.420,4.628,5.317,1.182,0.689,\\
 &0.334,0.423,0.982,3.298,1.094,1.559,0.420,0.214,0.174,0.629,1)
 \end{aligned}
 \end{equation*}

\subsection{Orbitdata $3,5,5$ with a cyclic permutation}\label{B:355}
Let $f_\mathbf{R}: X(\mathbf{R}) \to X(\mathbf{R})$ be a real diffeomorphism associated with the orbit data $3,5,5$ and a cyclic permutation. Let
\begin{equation*}
\begin{aligned}
&\Gamma = \{ \gamma_1, \dots, \gamma_{34} \},\\
&\Delta = \{ \gamma_1, \gamma_2,\gamma_3, \gamma_8 ,\gamma_9,\gamma_{10},\gamma_{11},\gamma_{13},\gamma_{16},\gamma_{17},\gamma_{18},\gamma_{23},\gamma_{24},\gamma_{26},\gamma_{33},\\
&\bg\bg\bg \zeta_1,\dots,\zeta_5,\mu_1,\dots,\mu_{14},\eta_1\eta_2\} 
\end{aligned}
\end{equation*}
be two ordered sets of non-cyclic words in $\pi_1( X(\mathbf{R}),P_{fix})$ where
\begin{equation*}
\begin{aligned}
&\gamma_1 = c_1^{-1} \cc b_1^{-1} \cc a_4^{-1} \cc c_4^{-1} \cc b_1 \cc c_1 \cc e^{-1}, \quad \gamma_2 = a_1^{-2} \cc c_1^{-2} \cc b_1^{-2} \cc a_2^{-2} \cc c_2^{-2} \cc b_2^{-1} \cc c_2 \cc a_2^2 \cc b_1 \cc e^{-1},\\
&\gamma_3 =  a_1^{-2} \cc c_1^{-2} \cc b_1^{-2} \cc a_2^{-2} \cc c_2^{-2} \cc b_2^{-1} \cc c_2 \cc a_2 \cc b_3^{-1} \cc a_4^{-2} \cc c_4^{-1} \cc b_1 \cc c_1 \cc e^{-1},\\
&\gamma_4 =  a_1^{-2} \cc c_1^{-2} \cc b_1^{-2} \cc a_2^{-2} \cc c_2^{-2} \cc b_2^{-1} \cc c_2 \cc a_2^2 \cc b_1^2 \cc c_1^2 \cc a_1 \cc c_1^{-1} \cc b_1^{-1} \cc a_4^{-1} \cc c_4^{-1} \cc b_1 \cc c_1 \cc e^{-1},\\
&\gamma_5 =  a_1^{-2} \cc c_1^{-2} \cc b_1^{-2} \cc a_2^{-2} \cc c_2^{-2} \cc b_2^{-1} \cc c_2 \cc a_2^2 \cc b_1^2 \cc c_1^2 \cc a_1 \cc c_1^{-1} \cc b_1^{-1} \cc b_3 \cc c_3 \cc a_3^{-1} \cc c_3^{-2} \cc b_3^{-2} \cc a_4^{-2} \cc c_4^{-2} \cc b_1 \cc c_1 \cc e^{-1},\\
\end{aligned}
\end{equation*}
\begin{equation*} 
\begin{aligned}
&\gamma_6 =   a_1^{-2} \cc c_1^{-2} \cc b_1^{-2} \cc a_2^{-2} \cc c_2^{-2} \cc b_2^{-1} \cc c_2 \cc a_2^2 \cc b_1^2 \cc c_1^2 \cc a_1 \cc c_1^{-1} \cc b_1^{-1} \cc b_3 \cc c_3 \cc a_3^{-1} \cc c_3^{-2} \cc b_3^{-2} \cc a_4^{-2} \cc c_4^{-2} \cc a_5^{-2} \cc c_5^{-2} \cc e^{-1},\\
&\gamma_7 = a_1^{-2} \cc c_1^{-2} \cc b_1^{-2} \cc a_2^{-2} \cc c_2^{-2} \cc b_2^{-1} \cc c_2 \cc a_2^2 \cc b_1^2 \cc c_1^2 \cc a_1 \cc c_1^{-1} \cc b_1^{-1} \cc b_3 \cc c_3 \cc a_3^{-1} \cc c_3^{-2} \cc b_3^{-2} \cc a_4^{-2} \cc c_4^{-2} \cc a_5^{-2} \cc c_5^{-1}\cc b_1 \cc c_1 \cc e^{-1},\\
&\gamma_8 = c_1^{-1} \cc b_1^{-1} \cc c_4 \cc a_4^2 \cc b_3^2 \cc c_3^2 \cc a_3 \cc c_3^{-1} \cc b_3^{-1} \cc b_1 \cc c_1 \cc a_1^{-1} \cc c_1^{-2} \cc b_1^{-2} \cc a_2^{-2} \cc c_2^{-1} \cc b_2 \cc c_2^2 \cc a_2^2 \cc b_1^2 \cc c_1^2 \cc a_1^2 \cc e \cc c_1^{-1} \cc b_1^{-1} \cc c_5^{-1} \cc e^{-1},\\ 
&\gamma_9 = a_1^{-2} \cc c_1^{-2} \cc b_1^{-2} \cc a_2^{-2} \cc c_2^{-2} \cc b_2^{-1} \cc c_2 \cc a_2 \cc c_3^{-1} \cc b_3^{-1} \cc b_1 \cc c_1 \cc a_1^{-1} \cc c_1^{-2} \cc b_1^{-2} \cc a_2^{-2} \cc c_2^{-1} \cc b_2 \cc c_2^2 \cc a_2^2 \cc b_1^2 \cc c_1^2 \cc a_1^2 \cc e \cc b_1 \cc c_1 \cc e^{-1},\\
&\gamma_{10} = a_1^{-2} \cc c_1^{-2} \cc b_1^{-2} \cc a_2^{-2} \cc c_2^{-2} \cc b_2^{-1} \cc c_2 \cc a_2 \cc c_3^{-1} \cc b_3^{-1} \cc b_1 \cc c_1 \cc a_1^{-1} \cc c_1^{-2} \cc b_1^{-2} \cc a_2^{-2} \cc c_2^{-1} \cc b_2 \cc c_2^2 \cc a_2^2 \cc b_1^2 \cc c_1^2 \cc a_1^2 \cc e \cc c_1^{-1} \cc b_1^{-1} \cc e^{-1},\\
&\gamma_{11} =  a_1^{-2} \cc c_1^{-2} \cc b_1^{-2} \cc a_2^{-2} \cc c_2^{-2} \cc b_2^{-2} \cc a_3^{-1} \cc c_3^{-1} \cc b_3^{-1} \cc b_1 \cc c_1 \cc a_1^{-1} \cc c_1^{-2} \cc b_1^{-2} \cc a_2^{-1} \cc c_2^{-1} \cc b_2 \cc c_2^2 \cc a_2^2 \cc b_1^2\\
&\bg  \cc c_1^2 \cc a_1^2 \cc e \cc c_1^{-1} \cc b_1^{-1} \cc c_5^{-1} \cc e^{-1},\\
&\gamma_{12} = a_1^{-2} \cc c_1^{-2} \cc b_1^{-2} \cc a_2^{-2} \cc c_2^{-2} \cc b_2^{-1} \cc c_2 \cc a_2^2 \cc b_1^2 \cc c_1^2 \cc a_1 \cc c_1^{-1} \cc b_1^{-1} \cc b_3 \cc c_3 \cc a_2^{-1} \cc c_2^{-1} \cc b_2 \cc c_2^2 \cc a_2^2 \cc b_1^2 \cc c_1^2 \cc a_1^2 \cc e\\
&\bg  \cc c_1^{-1} \cc b_1^{-1} \cc a_5^{-1} \cc c_5^{-1} \cc b_1 \cc c_1\cc e^{-1},\\
&\gamma_{13}= a_1^{-2} \cc c_1^{-2} \cc b_1^{-2} \cc a_2^{-2} \cc c_2^{-2} \cc b_2^{-1} \cc c_2 \cc a_2 \cc b_3 \cc c_3^2 \cc a_3 \cc c_3^{-1} \cc b_3^{-1} \cc b_1 \cc c_1 \cc a_1^{-1} \cc c_1^{-2} \cc b_1^{-2} \cc a_2^{-1} \cc c_2^{-1} \cc b_2 \cc c_2^2 \cc a_2^2 \cc b_1^2\\
&\bg  \cc c_1^2 \cc a_1^2 \cc e \cc c_1^{-1} \cc b_1^{-1} \cc e^{-1},\\
&\gamma_{14} = a_1^{-2} \cc c_1^{-2} \cc b_1^{-2} \cc a_2^{-2} \cc c_2^{-2} \cc b_2^{-1} \cc c_2 \cc a_2^2 \cc b_1^2 \cc c_1^2 \cc a_1 \cc c_1^{-1} \cc b_1^{-1} \cc b_3 \cc c_3 \cc a_2^{-1} \cc c_2^{-1} \cc b_2 \cc c_2^2 \cc a_2^2 \cc b_1^2 \cc c_1^2 \cc a_1^2 \cc e\\
&\bg  \cc c_1^{-1} \cc b_1^{-1}\cc c_4 \cc a_4 \cc b_1 \cc c_1 \cc a_1^{-1} \cc c_1^{-2} \cc b_1^{-1} \cc e^{-1},\\
&\gamma_{15} = a_1^{-2} \cc c_1^{-2} \cc b_1^{-2} \cc a_2^{-2} \cc c_2^{-2} \cc b_2^{-1} \cc c_2 \cc a_2^2 \cc b_1^2 \cc c_1^2 \cc a_1 \cc c_1^{-1} \cc b_1^{-1} \cc b_3 \cc c_3 \cc a_2^{-1} \cc c_2^{-1} \cc b_2 \cc c_2^2 \cc a_2^2 \cc b_1^2 \cc c_1^2 \cc a_1^2 \cc e\\
&\bg  \cc c_1^{-1} \cc b_1^{-1}\cc c_4 \cc a_4 \cc b_1 \cc c_1 \cc e \cc c_5 \cc b_1 \cc c_1\cc e^{-1},\\
&\gamma_{16} = a_1^{-2} \cc c_1^{-2} \cc b_1^{-2} \cc a_2^{-2} \cc c_2^{-2} \cc b_2^{-1} \cc c_2 \cc a_2 \cc c_3^{-1} \cc b_3^{-1} \cc b_1 \cc c_1 \cc a_1^{-1} \cc c_1^{-2} \cc b_1^{-2} \cc a_2^{-2} \cc c_2^{-1} \cc b_2 \cc c_2^2 \cc a_2^2 \cc b_1^2 \cc c_1^2 \cc a_1^2 \cc e\\
&\bg  \cc b_1 \cc c_1^2 \cc a_1 \cc c_1^{-1} \cc b_1^{-1} \cc a_4^{-1} \cc c_4^{-1} \cc e^{-1},\\
&\gamma_{17} = a_1^{-2} \cc c_1^{-2} \cc b_1^{-2} \cc a_2^{-2} \cc c_2^{-2} \cc b_2^{-1} \cc c_2 \cc a_2 \cc c_3^{-1} \cc b_3^{-1} \cc b_1 \cc c_1 \cc a_1^{-1} \cc c_1^{-2} \cc b_1^{-2} \cc a_2^{-2} \cc c_2^{-2}\cc b_2^{-1} \cc a_2^{-1} \cc c_2^{-1} \cc b_2 \cc c_2^2 \cc a_2^2 \\
&\bg \cc b_1^2 \cc c_1^2 \cc a_1^2 \cc e \cc c_1^{-1} \cc b_1^{-1} \cc a_5^{-1} \cc c_5^{-1}\cc b_1 \cc c_1 \cc e^{-1},\\
&\gamma_{18}= a_1^{-2} \cc c_1^{-2} \cc b_1^{-2} \cc a_2^{-2} \cc c_2^{-2} \cc b_2^{-1} \cc c_2 \cc a_2 \cc c_3^{-1} \cc b_3^{-1} \cc b_1 \cc c_1 \cc a_1^{-1} \cc c_1^{-2} \cc b_1^{-2} \cc a_2^{-2} \cc c_2^{-2}\cc b_2^{-1} \cc c_3^{-1} \cc b_3^{-1} \cc b_1 \cc c_1 \cc  a_1^{-1} \\
&\bg \cc c_1^{-2} \cc b_1^{-2} \cc a_2^{-2} \cc c_2^{-1} \cc b_2 \cc c_2^2 \cc a_2^2 \cc b_1^2 \cc c_1^2 \cc a_1^2 \cc e \cc c_1^{-1} \cc b_1^{-1} \cc e^{-1},\\
&\gamma_{19} = a_1^{-2} \cc c_1^{-2} \cc b_1^{-2} \cc a_2^{-2} \cc c_2^{-2} \cc b_2^{-1} \cc c_2 \cc a_2 \cc c_3^{-1} \cc b_3^{-1} \cc b_1 \cc c_1 \cc a_1^{-1} \cc c_1^{-2} \cc b_1^{-2} \cc a_2^{-2} \cc c_2^{-1} \cc b_2 \cc c_2^2 \cc a_2^2 \cc b_1^2 \cc c_1^2 \cc a_1^2 \cc e \cc b_1^{-1} \cc a_2^{-2} \cc c_2^{-1}\\
&\bg \cc b_2  \cc c_2^2 \cc a_2^2 \cc b_1^2 \cc c_1^2 \cc a_1^2 \cc e  \cc c_1^{-1} \cc b_1^{-1} \cc a_5^{-1} \cc c_5^{-1} \cc b_1 \cc c_1 \cc e^{-1},\\
&\gamma_{20} = a_1^{-2} \cc c_1^{-2} \cc b_1^{-2} \cc a_2^{-2} \cc c_2^{-2} \cc b_2^{-1} \cc c_2 \cc a_2 \cc c_3^{-1} \cc b_3^{-1} \cc b_1 \cc c_1 \cc a_1^{-1} \cc c_1^{-2} \cc b_1^{-2} \cc a_2^{-2} \cc c_2^{-2} \cc b_2^{-1} \cc c_2 \cc a_2^2 \cc b_1^2 \cc c_1^2 \cc a_1 \cc c_1^{-1} \cc b_1^{-1} \cc b_3 \cc c_3\\
&\bg \cc a_2^{-1} \cc c_2^{-1} \cc b_2  \cc c_2^2 \cc a_2^2 \cc b_1^2 \cc c_1^2 \cc a_1^2 \cc e  \cc c_1^{-1} \cc b_1^{-1}\cc c_4 \cc a_4 \cc b_1 \cc c_1 \cc e^{-1},\\
&\gamma_{21} = a_1^{-2} \cc c_1^{-2} \cc b_1^{-2} \cc a_2^{-2} \cc c_2^{-2} \cc b_2^{-1} \cc c_2 \cc a_2 \cc c_3^{-1} \cc b_3^{-1} \cc b_1 \cc c_1 \cc a_1^{-1} \cc c_1^{-2} \cc b_1^{-2} \cc a_2^{-2} \cc c_2^{-2} \cc b_2^{-1} \cc c_2 \cc a_2^2 \cc b_1^2 \cc c_1^2 \cc a_1 \cc c_1^{-1} \cc b_1^{-1} \cc b_3 \cc c_3\\
&\bg \cc a_2^{-1} \cc c_2^{-1} \cc b_2  \cc c_2^2 \cc a_2^2 \cc b_1^2 \cc c_1^2 \cc a_1^2 \cc e  \cc c_1^{-1} \cc b_1^{-1}\cc c_4 \cc a_4 \cc b_1 \cc c_1\cc e \cc c_5 \cc b_1 \cc c_1 \cc e^{-1},\\
&\gamma_{22} =  a_1^{-2} \cc c_1^{-2} \cc b_1^{-2} \cc a_2^{-2} \cc c_2^{-2} \cc b_2^{-1} \cc c_2 \cc a_2^2 \cc b_1^2 \cc c_1^2 \cc a_1 \cc c_1^{-1} \cc b_1^{-1} \cc b_3 \cc c_3 \cc a_2^{-1} \cc c_2^{-1} \cc b_2  \cc c_2^2 \cc a_2^2 \cc b_1^2 \cc c_1^2 \cc a_1^2 \cc e  \cc c_1^{-1} \cc b_1^{-1} \\
&\bg \cc c_4 \cc a_4 \cc b_1 \cc c_1 \cc a_1^{-1} \cc c_1^{-2} \cc b_1^{-2} \cc a_2^{-2} \cc c_2^{-1} \cc b_2  \cc c_2^2 \cc a_2^2 \cc b_1^2 \cc c_1^2 \cc a_1^2 \cc e  \cc c_1^{-1} \cc b_1^{-1} \cc c_5^{-1} \cc e^{-1},\\
&\gamma_{23} = a_1^{-2} \cc c_1^{-2} \cc b_1^{-2} \cc a_2^{-2} \cc c_2^{-2} \cc b_2^{-1} \cc c_2 \cc a_2  \cc b_3 \cc c_3 \cc a_2^{-1} \cc c_2^{-1} \cc b_2 \cc c_2^2 \cc a_2^2 \cc b_1^2 \cc c_1^2 \cc a_1^2 \cc e  \cc c_1^{-1} \cc b_1^{-1} \cc c_4 \cc a_4^2 \cc b_3^2 \cc c_3^2 \cc a_3 \cc c_3^{-1} \cc b_3^{-1}\\
&\bg \cc b_1 \cc c_1 \cc a_1^{-1} \cc c_1^{-2} \cc b_1^{-2} \cc a_2^{-2} \cc c_2^{-1} \cc b_2  \cc c_2^2 \cc a_2^2 \cc b_1^2 \cc c_1^2 \cc a_1^2 \cc e  \cc c_1^{-1} \cc b_1^{-1} \cc c_5^{-1} \cc e^{-1},\\
&\gamma_{24} = a_1^{-2} \cc c_1^{-2} \cc b_1^{-2} \cc a_2^{-2} \cc c_2^{-2} \cc b_2^{-1} \cc c_2 \cc a_2  \cc c_3^{-1} \cc b_3^{-1} \cc a_2^{-1} \cc c_2^{-1} \cc b_2 \cc c_2^2 \cc a_2^2 \cc b_1^2 \cc c_1^2 \cc a_1^2 \cc e  \cc c_1^{-1} \cc b_1^{-1} \cc c_4 \cc a_4^2 \cc b_3^2 \cc c_3^2 \cc a_3 \cc c_3^{-1} \cc b_3^{-1}\\
&\bg \cc b_1 \cc c_1 \cc a_1^{-1} \cc c_1^{-2} \cc b_1^{-2} \cc a_2^{-2} \cc c_2^{-1} \cc b_2  \cc c_2^2 \cc a_2^2 \cc b_1^2 \cc c_1^2 \cc a_1^2 \cc e  \cc c_1^{-1} \cc b_1^{-1} \cc c_5^{-1} \cc e^{-1},\\
&\gamma_{25} =a_1^{-2} \cc c_1^{-2} \cc b_1^{-2} \cc a_2^{-2} \cc c_2^{-2} \cc b_2^{-1} \cc c_2 \cc a_2^2 \cc b_1^2 \cc c_1^2 \cc a_1 \cc c_1^{-1} \cc b_1^{-1} \cc b_3 \cc c_3 \cc a_2^{-1} \cc c_2^{-1} \cc b_2  \cc c_2^2 \cc a_2^2 \cc b_1^2 \cc c_1^2 \cc a_1^2 \cc e  \cc c_1^{-1} \cc b_1^{-1} \\
&\bg \cc c_4 \cc a_4 \cc b_1 \cc c_1 \cc a_1^{-1} \cc c_1^{-2} \cc b_1^{-2} \cc a_2^{-2} \cc c_2^{-1} \cc b_2  \cc c_2^2 \cc a_2^2 \cc b_1^2 \cc c_1^2 \cc a_1^2 \cc e  \cc c_1^{-1} \cc b_1^{-1}\cc a_5^{-1} \cc c_5^{-1}\cc b_1 \cc c_1 \cc e^{-1},\\
\end{aligned}
\end{equation*}
\begin{equation*} 
\begin{aligned}
&\gamma_{26} =  a_1^{-2} \cc c_1^{-2} \cc b_1^{-2} \cc a_2^{-2} \cc c_2^{-2} \cc b_2^{-1} \cc c_2 \cc a_2  \cc b_3 \cc c_3 \cc a_2^{-1} \cc c_2^{-1} \cc b_2 \cc c_2^2 \cc a_2^2 \cc b_1^2 \cc c_1^2 \cc a_1^2 \cc e  \cc c_1^{-1} \cc b_1^{-1} \cc c_4 \cc a_4^2 \cc b_3^2 \cc c_3^2 \cc a_3 \cc c_3^{-1} \cc b_3^{-1}\\
&\bg \cc b_1 \cc c_1 \cc a_1^{-1} \cc c_1^{-2} \cc b_1^{-2} \cc a_2^{-2} \cc c_2^{-1} \cc b_2  \cc c_2^2 \cc a_2^2 \cc b_1^2 \cc c_1^2 \cc a_1^2 \cc e  \cc c_1^{-1} \cc b_1^{-1} \cc c_5 \cc a_5 \cc b_1 \cc c_1 \cc e^{-1},\\
&\gamma_{27} = a_1^{-2} \cc c_1^{-2} \cc b_1^{-2} \cc a_2^{-2} \cc c_2^{-2} \cc b_2^{-1} \cc c_2 \cc a_2^2 \cc b_1^2 \cc c_1^2 \cc a_1 \cc c_1^{-1} \cc b_1^{-1} \cc b_3 \cc c_3 \cc a_3^{-1} \cc c_3^{-2} \cc b_3^{-1} \cc a_2^{-1} \cc c_2^{-1} \cc b_2  \cc c_2^2 \cc a_2^2 \cc b_1^2 \cc c_1^2 \cc a_1^2 \cc e \\
&\bg \cc c_1^{-1} \cc b_1^{-1}  \cc c_4 \cc a_4^2 \cc b_3^2 \cc c_3^2 \cc a_3 \cc c_3^{-1} \cc b_3^{-1}  \cc b_1 \cc c_1 \cc a_1^{-1} \cc c_1^{-2} \cc b_1^{-2} \cc a_2^{-2} \cc c_2^{-1} \cc b_2  \cc c_2^2 \cc a_2^2 \cc b_1^2 \cc c_1^2 \cc a_1^2 \cc e  \cc c_1^{-1} \cc b_1^{-1}\cc c_5^{-1} \cc e^{-1},\\
&\gamma_{28} = a_1^{-2} \cc c_1^{-2} \cc b_1^{-2} \cc a_2^{-2} \cc c_2^{-2} \cc b_2^{-1} \cc c_2 \cc a_2^2 \cc b_1^2 \cc c_1^2 \cc a_1 \cc c_1^{-1} \cc b_1^{-1} \cc b_3 \cc c_3 \cc a_2^{-1} \cc c_2^{-1} \cc b_2  \cc c_2^2 \cc a_2^2 \cc b_1^2 \cc c_1^2 \cc a_1^2 \cc e  \cc c_1^{-1} \cc b_1^{-1} \\
&\bg \cc c_4 \cc a_4^2 \cc b_3^2 \cc c_3^2 \cc a_3 \cc c_3^{-1} \cc b_3^{-1} \cc b_1 \cc c_1 \cc a_1^{-1} \cc c_1^{-2} \cc b_1^{-2} \cc a_2^{-2} \cc c_2^{-1} \cc b_2  \cc c_2^2 \cc a_2^2 \cc b_1^2 \cc c_1^2 \cc a_1^2 \cc e  \cc c_1^{-1} \cc b_1^{-1}\cc c_5 \cc a_5 \cc b_1 \cc c_1 \cc e^{-1},\\
&\gamma_{29} = a_1^{-2} \cc c_1^{-2} \cc b_1^{-2} \cc a_2^{-2} \cc c_2^{-2} \cc b_2^{-1} \cc c_2 \cc a_2^2 \cc b_1^2 \cc c_1^2 \cc a_1 \cc c_1^{-1} \cc b_1^{-1} \cc b_3 \cc c_3 \cc b_2 \cc c_2^2 \cc a_2^2 \cc b_1^2 \cc c_1^2 \cc a_1 \cc c_1^{-1} \cc b_1^{-1} \cc b_3 \cc c_3  \cc a_2^{-1} \cc c_2^{-1}\\
&\bg  \cc b_2  \cc c_2^2 \cc a_2^2 \cc b_1^2 \cc c_1^2 \cc a_1^2 \cc e  \cc c_1^{-1} \cc b_1^{-1}  \cc c_4 \cc a_4^2 \cc b_3^2 \cc c_3^2 \cc a_3 \cc c_3^{-1} \cc b_3^{-1} \cc b_1 \cc c_1 \cc a_1^{-1} \cc c_1^{-2} \cc b_1^{-2} \cc a_2^{-2} \cc c_2^{-1} \cc b_2  \cc c_2^2 \cc a_2^2 \cc b_1^2\\
&\bg  \cc c_1^2 \cc a_1^2 \cc e  \cc c_1^{-1} \cc b_1^{-1}\cc e^{-1},\\
&\gamma_{30} =  a_1^{-2} \cc c_1^{-2} \cc b_1^{-2} \cc a_2^{-2} \cc c_2^{-2} \cc b_2^{-1} \cc c_2 \cc a_2 \cc c_3^{-1} \cc b_3^{-1} \cc b_1 \cc c_1 \cc a_1^{-1} \cc c_1^{-2} \cc b_1^{-2} \cc a_2^{-2} \cc c_2^{-2} \cc b_2^{-1} \cc c_2 \cc a_2^2 \cc b_1^2 \cc c_1^2 \cc a_1 \cc c_1^{-1} \cc b_1^{-1}\\
&\bg  \cc b_3 \cc c_3 \cc a_2^{-1} \cc c_2^{-1}  \cc b_2  \cc c_2^2 \cc a_2^2 \cc b_1^2 \cc c_1^2 \cc a_1^2 \cc e  \cc c_1^{-1} \cc b_1^{-1}  \cc c_4 \cc a_4^2 \cc b_3^2 \cc c_3^2 \cc a_3 \cc c_3^{-1} \cc b_3^{-1} \cc b_1 \cc c_1 \cc a_1^{-1} \cc c_1^{-2} \cc b_1^{-2} \cc a_2^{-2} \cc c_2^{-1}\\
&\bg  \cc b_2  \cc c_2^2 \cc a_2^2 \cc b_1^2 \cc c_1^2 \cc a_1^2 \cc e  \cc c_1^{-1} \cc b_1^{-1}\cc e^{-1},\\
&\gamma_{31} = a_1^{-2} \cc c_1^{-2} \cc b_1^{-2} \cc a_2^{-2} \cc c_2^{-2} \cc b_2^{-1} \cc c_2 \cc a_2^2 \cc b_1^2 \cc c_1^2 \cc a_1 \cc c_1^{-1} \cc b_1^{-1} \cc b_3 \cc c_3 \cc b_2 \cc c_2^2 \cc a_2^2 \cc b_1^2 \cc c_1^2 \cc a_1 \cc c_1^{-1} \cc b_1^{-1} \cc b_3 \cc c_3  \cc a_2^{-1} \cc c_2^{-1}\\
&\bg  \cc b_2  \cc c_2^2 \cc a_2^2 \cc b_1^2 \cc c_1^2 \cc a_1^2 \cc e  \cc c_1^{-1} \cc b_1^{-1}  \cc c_4 \cc a_4^2 \cc b_3^2 \cc c_3^2 \cc a_3 \cc c_3^{-1} \cc b_3^{-1} \cc b_1 \cc c_1 \cc a_1^{-1} \cc c_1^{-2} \cc b_1^{-2} \cc a_2^{-2} \cc c_2^{-1} \cc b_2  \cc c_2^2 \cc a_2^2 \cc b_1^2\\
&\bg  \cc c_1^2 \cc a_1^2 \cc e  \cc c_1^{-1} \cc b_1^{-1}\cc c_5 \cc a_5 \cc b_1 \cc c_1 \cc e^{-1},\\
&\gamma_{32} =  a_1^{-2} \cc c_1^{-2} \cc b_1^{-2} \cc a_2^{-2} \cc c_2^{-2} \cc b_2^{-1} \cc c_2 \cc a_2 \cc c_3^{-1} \cc b_3^{-1} \cc b_1 \cc c_1 \cc a_1^{-1} \cc c_1^{-2} \cc b_1^{-2} \cc a_2^{-2} \cc c_2^{-1} \cc b_2 \cc c_2^2 \cc a_2^2 \cc b_1^2 \cc c_1^2 \cc a_1 \cc c_1^{-1} \cc b_1^{-1}\\
&\bg  \cc b_3 \cc c_3 \cc a_2^{-1} \cc c_2^{-1}  \cc b_2  \cc c_2^2 \cc a_2^2 \cc b_1^2 \cc c_1^2 \cc a_1^2 \cc e  \cc c_1^{-1} \cc b_1^{-1}  \cc c_4 \cc a_4^2 \cc b_3^2 \cc c_3^2 \cc a_3 \cc c_3^{-1} \cc b_3^{-1} \cc b_1 \cc c_1 \cc a_1^{-1} \cc c_1^{-2} \cc b_1^{-2} \cc a_2^{-2} \cc c_2^{-1}\\
&\bg  \cc b_2  \cc c_2^2 \cc a_2^2 \cc b_1^2 \cc c_1^2 \cc a_1^2 \cc e  \cc c_1^{-1} \cc b_1^{-1}\cc c_5 \cc a_5 \cc b_1 \cc c_1\cc e^{-1},\\
&\gamma_{33} =   a_1^{-2} \cc c_1^{-2} \cc b_1^{-2} \cc a_2^{-2} \cc c_2^{-2} \cc b_2^{-1} \cc c_2 \cc a_2 \cc b_3 \cc c_3 \cc a_2^{-1} \cc c_2^{-1} \cc b_2  \cc c_2^2 \cc a_2^2 \cc b_1^2 \cc c_1^2 \cc a_1^2 \cc e  \cc c_1^{-1} \cc b_1^{-1}  \cc c_4 \cc a_4^2 \cc b_3^2 \cc c_3^2 \cc a_3 \cc c_3^{-1} \cc b_3^{-1}\\
&\bg  \cc b_1 \cc c_1 \cc a_1^{-1} \cc c_1^{-2} \cc b_1^{-2} \cc a_2^{-2} \cc c_2^{-1}  \cc b_2  \cc c_2^2 \cc a_2^2 \cc b_1^2 \cc c_1^2 \cc a_1^2 \cc e  \cc c_1^{-1} \cc b_1^{-1} \cc c_5 \cc a_5^2 \cc c_4^2 \cc a_4^2 \cc b_3^2 \cc c_3^2 \cc a_3 \cc c_3^{-1} \cc b_3^{-1} \cc b_1 \cc c_1\\
&\bg  \cc a_1^{-1} \cc c_1^{-2} \cc b_1^{-2} \cc a_2^{-2} \cc c_2^{-1}\cc b_2  \cc c_2^2 \cc a_2^2 \cc b_1^2 \cc c_1^2 \cc a_1^2 \cc e  \cc c_1^{-1} \cc b_1^{-1} \cc c_5^{-1} \cc e^{-1},\\
&\gamma_{34}=  a_1^{-2} \cc c_1^{-2} \cc b_1^{-2} \cc a_2^{-2} \cc c_2^{-2} \cc b_2^{-1} \cc c_2 \cc a_2 \cc c_3^{-1} \cc b_3^{-1} \cc b_1 \cc c_1 \cc a_1^{-1} \cc c_1^{-2} \cc b_1^{-2} \cc a_2^{-2} \cc c_2^{-2} \cc b_2^{-1} \cc c_2 \cc a_2^2 \cc b_1^2 \cc c_1^2 \cc a_1 \cc c_1^{-1} \cc b_1^{-1}\\
&\bg  \cc b_3 \cc c_3 \cc a_2^{-1} \cc c_2^{-1}  \cc b_2  \cc c_2^2 \cc a_2^2 \cc b_1^2 \cc c_1^2 \cc a_1^2 \cc e  \cc c_1^{-1} \cc b_1^{-1} \cc c_4 \cc a_4^2 \cc b_3 \cc a_2^{-1} \cc c_2^{-1} \cc b_2  \cc c_2^2 \cc a_2^2 \cc b_1^2 \cc c_1^2 \cc a_1^2 \cc e  \cc c_1^{-1} \cc b_1^{-1}\\
&\bg  \cc c_4 \cc a_4^2 \cc b_3^2 \cc c_3^2 \cc a_3 \cc c_3^{-1} \cc b_3^{-1} \cc b_1 \cc c_1 \cc a_1^{-1} \cc c_1^{-2} \cc b_1^{-2} \cc a_2^{-2} \cc c_2^{-1}  \cc b_2  \cc c_2^2 \cc a_2^2 \cc b_1^2 \cc c_1^2 \cc a_1^2 \cc e  \cc c_1^{-1} \cc b_1^{-1}\cc c_5^{-1} \cc e^{-1},\\
\end{aligned}
\end{equation*}
and 
\begin{equation*}
\begin{aligned}
&\zeta_1 =  a_1^{-2} \cc c_1^{-2} \cc b_1^{-2} \cc a_2^{-2} \cc c_2^{-2} \cc b_2^{-1} \cc c_2 \cc a_2^2 \cc b_1^2 \cc c_1^2 \cc a_1 \cc c_1^{-1},\\
&\zeta_2 = a_1^{-2} \cc c_1^{-2} \cc b_1^{-2} \cc a_2^{-2} \cc c_2^{-2} \cc b_2^{-1} \cc c_2\cc a_2 \cc b_2 \cc c_2 \cc a_2^2 \cc b_1^2 \cc c_1^2 \cc a_1 \cc c_1^{-1},\\
&\zeta_3 = a_1^{-2} \cc c_1^{-2} \cc b_1^{-2} \cc a_2^{-2} \cc c_2^{-2} \cc b_2^{-1} \cc c_2\cc a_2 \cc c_3^{-1} \cc b_3^{-1} \cc b_1 \cc c_1\cc a_1^{-1} \cc c_1^{-2} \cc b_1^{-2} \cc a_2^{-2} \cc c_2^{-1} \cc b_2 \cc c_2^2 \cc a_2^2 \cc b_1^2 \cc c_1^2 \cc a_1 \cc c_1^{-1},\\ 
&\zeta_4 =  a_1^{-2} \cc c_1^{-2} \cc b_1^{-2} \cc a_2^{-2} \cc c_2^{-2} \cc b_2^{-1} \cc c_2\cc a_2 \cc c_3^{-1} \cc b_3^{-1}\cc b_1 \cc c_1  \cc a_1^{-1} \cc c_1^{-2} \cc b_1^{-2} \cc a_2^{-2} \cc c_2^{-2} \cc b_2^{-1}  \cc c_2 \cc a_2^2 \cc b_1^2 \cc c_1^2 \cc a_1 \cc c_1^{-1},\\
&\zeta_5 =  a_1^{-2} \cc c_1^{-2} \cc b_1^{-2} \cc a_2^{-2} \cc c_2^{-2} \cc b_2^{-1} \cc c_2\cc a_2 \cc c_3^{-1} \cc b_3^{-1}\cc b_1 \cc c_1  \cc a_1^{-1} \cc c_1^{-2} \cc b_1^{-2} \cc a_2^{-2} \cc c_2^{-1} \cc b_2  \cc c_2^2 \cc a_2^2 \cc b_1^2 \cc c_1^2 \cc a_1^2 \cc e,\\
&\mu_1 = b_1^{-1} \cc a_4^{-1} \cc c_4^{-1} \cc b_1 \cc c_1 \cc e^{-1}, \quad \mu_2 = b_1^{-1} \cc b_3 \cc c_3 \cc a_3^{-1} \cc c_3^{-2} \cc b_3^{-2} \cc a_4^{-2} \cc c_4^{-1} \cc b_1 \cc c_1 \cc e^{-1},\\
&\mu_3 = b_1^{-1} \cc b_3 \cc c_3 \cc a_3^{-1} \cc c_3^{-2} \cc b_3^{-2} \cc a_4^{-2} \cc c_4^{-2}\cc a_5^{-2} \cc c_5^{-2} \cc e^{-1},\\
&\mu_4=b_1^{-1} \cc b_3 \cc c_3 \cc a_3^{-1} \cc c_3^{-2} \cc b_3^{-2} \cc a_4^{-2} \cc c_4^{-2}\cc a_5^{-2} \cc c_5^{-1}\cc b_1 \cc c_1 \cc e^{-1},\\
&\mu_5 = b_1^{-1} \cc a_2^{-2} \cc c_2^{-1} \cc b_2 \cc c_2^2 \cc a_2^2 \cc b_1^2 \cc c_1^2 \cc a_1^2 \cc e \cc c_1^{-1} \cc b_1^{-1} \cc c_5^{-1} \cc e^{-1},\\
&\mu_6= b_1^{-1} \cc a_2^{-2} \cc c_2^{-1} \cc b_2 \cc c_2^2 \cc a_2^2 \cc b_1^2 \cc c_1^2 \cc a_1^2 \cc e \cc c_1^{-1} \cc b_1^{-1}\cc a_5^{-1} \cc c_5^{-1} \cc b_1 \cc c_1 \cc e^{-1},\\
&\mu_7 = b_1^{-1}\cc b_3 \cc c_3  \cc a_2^{-1} \cc c_2^{-1} \cc b_2 \cc c_2^2 \cc a_2^2 \cc b_1^2 \cc c_1^2 \cc a_1^2 \cc e \cc c_1^{-1} \cc b_1^{-1}\cc a_5^{-1} \cc c_5^{-1} \cc b_1 \cc c_1 \cc e^{-1},\\
\end{aligned}
\end{equation*}
\begin{equation*} 
\begin{aligned}
&\mu_8 = b_1^{-1}\cc b_3 \cc c_3  \cc a_2^{-1} \cc c_2^{-1} \cc b_2 \cc c_2^2 \cc a_2^2 \cc b_1^2 \cc c_1^2 \cc a_1^2 \cc e \cc c_1^{-1} \cc b_1^{-1}\cc c_4 \cc a_4 \cc b_1 \cc c_1 \cc e^{-1},\\
&\mu_9 = b_1^{-1}\cc b_3 \cc c_3  \cc a_2^{-1} \cc c_2^{-1} \cc b_2 \cc c_2^2 \cc a_2^2 \cc b_1^2 \cc c_1^2 \cc a_1^2 \cc e \cc c_1^{-1} \cc b_1^{-1}\cc c_4 \cc a_4 \cc b_1 \cc c_1 \cc a_1^{-1} \cc c_1^{-2} \cc b_1^{-1} \cc e^{-1},\\
&\mu_{10} =  b_1^{-1}\cc b_3 \cc c_3  \cc a_2^{-1} \cc c_2^{-1} \cc b_2 \cc c_2^2 \cc a_2^2 \cc b_1^2 \cc c_1^2 \cc a_1^2 \cc e \cc c_1^{-1} \cc b_1^{-1}\cc c_4 \cc a_4 \cc b_1 \cc c_1 \cc e \cc c_5 \cc b_1 \cc c_1 \cc e^{-1},\\
&\mu_{11} =  b_1^{-1}\cc b_3 \cc c_3  \cc a_2^{-1} \cc c_2^{-1} \cc b_2 \cc c_2^2 \cc a_2^2 \cc b_1^2 \cc c_1^2 \cc a_1^2 \cc e \cc c_1^{-1} \cc b_1^{-1}\cc c_4 \cc a_4^2 \cc b_3^2 \cc c_3^2 \cc a_3 \cc c_3^{-1} \cc b_3^{-1} \cc b_1 \cc c_1 \cc a_1^{-1} \cc c_1^{-2} \cc b_1^{-2}\cc a_2^{-2}  \cc c_2^{-1}\\
&\bg \cc  b_2 \cc c_2^2 \cc a_2^2 \cc b_1^2 \cc c_1^2 \cc a_1^2 \cc e \cc c_1^{-1} \cc b_1^{-1} \cc e^{-1},\\
&\mu_{12} = b_1^{-1}\cc b_3 \cc c_3  \cc a_2^{-1} \cc c_2^{-1} \cc b_2 \cc c_2^2 \cc a_2^2 \cc b_1^2 \cc c_1^2 \cc a_1^2 \cc e \cc c_1^{-1} \cc b_1^{-1}\cc c_4 \cc a_4^2 \cc b_3^2 \cc c_3^2 \cc a_3 \cc c_3^{-1} \cc b_3^{-1} \cc b_1 \cc c_1 \cc a_1^{-1} \cc c_1^{-2} \cc b_1^{-2}\cc a_2^{-2}  \cc c_2^{-1}\\
&\bg \cc  b_2 \cc c_2^2 \cc a_2^2 \cc b_1^2 \cc c_1^2 \cc a_1^2 \cc e \cc c_1^{-1} \cc b_1^{-1} \cc c_5 \cc a_5 \cc b_1 \cc c_1 \cc e^{-1},\\
&\mu_{13} = b_1^{-1}\cc b_3 \cc c_3\cc a_3^{-1} \cc c_3^{-2} \cc b_3^{-1}  \cc a_2^{-1} \cc c_2^{-1} \cc b_2 \cc c_2^2 \cc a_2^2 \cc b_1^2 \cc c_1^2 \cc a_1^2 \cc e \cc c_1^{-1} \cc b_1^{-1}\cc c_4 \cc a_4^2 \cc b_3^2 \cc c_3^2 \cc a_3 \cc c_3^{-1} \cc b_3^{-1} \cc b_1 \cc c_1\\
&\bg  \cc a_1^{-1} \cc c_1^{-2} \cc b_1^{-2}\cc a_2^{-2}  \cc c_2^{-1} \cc  b_2 \cc c_2^2 \cc a_2^2 \cc b_1^2 \cc c_1^2 \cc a_1^2 \cc e \cc c_1^{-1} \cc b_1^{-1} \cc c_5^{-1} \cc e^{-1},\\
&\mu_{14} = b_1^{-1}\cc b_3 \cc c_3  \cc a_2^{-1} \cc c_2^{-1} \cc b_2 \cc c_2^2 \cc a_2^2 \cc b_1^2 \cc c_1^2 \cc a_1^2 \cc e \cc c_1^{-1} \cc b_1^{-1}\cc c_4 \cc a_4^2 \cc b_3 \cc a_2^{-1} \cc c_2^{-1} \cc b_2 \cc c_2^2 \cc a_2^2 \cc b_1^2 \cc c_1^2 \cc a_1^2 \cc e \cc c_1^{-1} \cc b_1^{-1}\\
&\bg \cc c_4 \cc a_4^2 \cc b_3^2 \cc c_3^2 \cc a_3 \cc c_3^{-1} \cc b_3^{-1} \cc b_1 \cc c_1 \cc a_1^{-1} \cc c_1^{-2} \cc b_1^{-2}\cc a_2^{-2}  \cc c_2^{-1}  \cc  b_2 \cc c_2^2 \cc a_2^2 \cc b_1^2 \cc c_1^2 \cc a_1^2 \cc e \cc c_1^{-1} \cc b_1^{-1} \cc c_5^{-1} \cc e^{-1},\\
&\eta_1 = b_1^{-1} \cc b_3 \cc c_3 \cc b_2 \cc c_2^2 \cc a_2^2 \cc b_1^2 \cc c_1^2 \cc a_1 \cc c_1^{-1},\\
&\eta_2 = b_1^{-1} \cc b_3 \cc c_3 \cc a_2^{-1} \cc c_2^{-1} \cc b_2 \cc c_2^2 \cc a_2^2 \cc b_1^2 \cc c_1^2 \cc a_1^2 \cc e \cc c_1^{-1} \cc b_1^{-1} \cc c_4 \cc a_4 \cc b_1 \cc c_1 \cc a_1^{-1} \cc c_1^{-2} \cc b_1.\\   
\end{aligned}
\end{equation*}
The admissible $3$-tuples are 
\begin{equation*}
\begin{aligned}
A&= \{  (1,32,2),(1,9,5),(1,19,5),(1,9,12),(1,16,19),(1,9,27),(1,9,29),(1,9,31),(2,22,1),(2,15,4),\\
&(2,12,5),(2,25,5),(2,15,6),(2,15,7),(2,14,15),(3,32,2),(4,19,5),(5,11,1),(5,23,1),(5,24,1),\\
&(5,33,1),(5,34,1),(5,26,2),(5,32,2),(5,26,4),(5,17,5),(5,21,5),(5,20,11),(5,13,15),(5,18,15),\\
&(5,30,15),(5,3,32),(6,8,1),(6,1,32),(7,5,24),(8,1,9),(8,1,16),(8,1,19),(8,1,32),(9,27,1),(9,31,2),\\
&(9,12,5),(9,31,20),(9,29,15),(9,5,18),(9,31,19),(9,31,28),(9,5,32),(10,15,5),(11,1,19),\\
&(12,5,18),(13,15,5),(14,15,5),(15,6,1),(15,5,3),(15,7,5),(15,6,8),(15,5,11),(15,5,13),(15,4,19),\\
&(15,5,23),(15,5,26),(15,5,33),(16,19,5),(17,5,18),(18,15,5),(19,5,17),(19,5,18),(19,5,20),\\
&(19,5,21),(19,5,24),(19,5,30),(19,5,34),(20,11,1),(21,5,3),(22,1,19),(23,1,32),(24,1,9),\\
&(25,5,24),(26,2,15),(26,4,19),(27,1,9),(28,2,15),(29,15,5),(30,15,5),(31,28,2),(31,19,5),\\
&(31,2,12),(31,2,14),(31,2,15),(31,10,15),(31,2,22),(31,2,25),(32,2,15),(33,1,19),\\
&(34,1,9),(34,1,32)\}
\end{aligned}
\end{equation*}
Let $V= Span \, \Gamma$ and $W= Span\, \Delta$. The splitting map $S: V\to W$ is given by 
\begin{equation*}
S: \left\{ \begin{aligned} & \gamma_1 \mapsto \mu_6 + \zeta_1, \quad \gamma_2 \mapsto \eta_1 , \quad \gamma_3 \mapsto \mu_4 + \zeta_1, \quad \gamma_4 \mapsto \mu_7 + \zeta_1, \quad \gamma_5 \mapsto \mu_{10} + \zeta_1, \quad \gamma_6 \mapsto \eta_2,\\
&\gamma_7 \mapsto \mu_9 + \zeta_1, \quad \gamma_8 \mapsto \mu_5 + \gamma_1 + \zeta_5, \quad \gamma_9 \mapsto \mu_2 + \zeta_4, \quad \gamma_{10} \mapsto \mu_2 + \zeta_3, \quad \gamma_{11} \mapsto \mu_1 + \zeta_5,\\
&\gamma_{12} \mapsto \mu_{11} + \zeta_1, \quad \gamma_{13} \mapsto \mu_3 + \gamma_1 + \zeta_3, \quad \gamma_{14} \mapsto \mu_{12} +\zeta_2, \quad \gamma_{15} \mapsto \mu_{12} + \gamma_2 + \zeta_1, \\ &\gamma_{16} \mapsto \mu_2 + \gamma_{17} + \zeta_1, \quad \gamma_{17} \mapsto \mu_2 + \gamma_{13} + \zeta_1, \quad \gamma_{18} \mapsto \mu_2 +\gamma_3 +\zeta_3, \quad \gamma_{19} \mapsto \mu_2 + \gamma_{18} + \zeta_1,\\
&\gamma_{20} \mapsto \mu_2 + \gamma_{26} + \zeta_1, \quad \gamma_{21} \mapsto \mu_2 + \gamma_{26} + \gamma_2 + \zeta_1, \quad \gamma_{22}\mapsto \mu_{12} + \zeta_5, \\
&\gamma_{23} \mapsto \mu_3 + \gamma_8 + \gamma_1 + \zeta_5, \quad \gamma_{24} \mapsto \mu_2 + \gamma_{11} + \gamma_1+ \zeta_5, \quad \gamma_{25} \mapsto \mu_{12} + \gamma_{10} + \zeta_1,\\
&\gamma_{26} \mapsto \mu_3 + \gamma_8 + \gamma_1 + \gamma_9 + \zeta_1, \quad \gamma_{27} \mapsto \mu_8 + \gamma_{11} + \gamma_1 + \zeta_5, \quad \gamma_{28} \mapsto \mu_{13} + \gamma_1 + \gamma_9+ \zeta_1,\\
&\gamma_{29} \mapsto \mu_{14} +\gamma_1 + \zeta_3, \quad \gamma_{30} \mapsto \mu_2 + \gamma_{23} + \gamma_1 + \zeta_3, \quad \gamma_{31} \mapsto \mu_{14} + \gamma_1 + \gamma_9 + \zeta_1,\\
&\gamma_{32} \mapsto \mu_2 + \gamma_{24} + \gamma_1 + \gamma_9 + \zeta_1, \quad \gamma_{33} \mapsto \mu_3+ \gamma_8 + \gamma_1 + \gamma_{16} + \zeta_5, \quad \gamma_{34} \mapsto \mu_2 + \gamma_{33} + \gamma_1 + \zeta_5.\\
 \end{aligned}\right.
\end{equation*}
and the merging map $M:V \to W$ is
\begin{equation*}
M: \left\{ \begin{aligned} &\gamma_1 \mapsto \gamma_1, \quad \gamma_2 \mapsto \gamma_2, \quad \gamma_3 \mapsto \gamma_3, \quad \gamma_4 \mapsto \zeta_1 + \mu_1, \quad \gamma_5 \mapsto \zeta_1 + \mu_2, \quad \gamma_6 \mapsto \zeta_1 + \mu_3,\\
&\gamma_7 \mapsto \zeta_1 + \mu_4, \quad \gamma_8 \mapsto \gamma_8, \quad \gamma_9 \mapsto \gamma_9, \quad \gamma_{10} \mapsto \gamma_{10}, \quad \gamma_{11} \mapsto \gamma_{11}, \quad \gamma_{12} \mapsto \zeta_1 + \mu_7,\\
&\gamma_{13} \mapsto \gamma_{13}, \quad \gamma_{14} \mapsto \zeta_1 + \mu_9, \quad \gamma_{15}\mapsto \zeta_1 + \mu_{10}, \quad \gamma_{16} \mapsto \gamma_{16}, \quad \gamma_{17} \mapsto \gamma_{17}, \quad \gamma_{18} \mapsto \gamma_{18},\\
&\gamma_{19} \mapsto \zeta_5 + \mu_6, \quad \gamma_{20} \mapsto \zeta_4 + \mu_8, \quad \gamma_{21} \mapsto \zeta_4 + \mu_{10} , \quad \gamma_{22} \mapsto \zeta_1 + \eta_2 + \mu_5, \quad \gamma_{23} \mapsto \gamma_{23},\\
&\gamma_{24} \mapsto \gamma_{24}, \quad \gamma_{25} \mapsto \zeta_1 + \eta_2 + \mu_6, \quad \gamma_{26} \mapsto \gamma_{26}, \quad \gamma_{27} \mapsto \zeta_1 + \mu_{13}, \quad \gamma_{28} \mapsto \zeta_2 + \mu_{12},\\
&\gamma_{29} \mapsto \zeta_1 + \eta_1 + \mu_{11}, \quad \gamma_{30} \mapsto \zeta_4 + \mu_{11}, \quad \gamma_{31} \mapsto \zeta_1+\eta_1 + \mu_{12}, \quad \gamma_{32} \mapsto \zeta_3 + \mu_{12} ,\\
&\gamma_{33} \mapsto \gamma_{33}, \quad \gamma_{34} \mapsto \zeta_4 + \mu_{14}.\\
\end{aligned}\right.
\end{equation*}

There is a $v_\lambda \in V$ such that $Sv_\lambda = \lambda Mv_\lambda$ where $\lambda\approx 1.58235 $ is the largest real root of 
\begin{equation*}
\chi(t) =t^6-t^4 - 2 t^3 -t^2 +1 \end{equation*}
and
\begin{equation*}
\begin{aligned}
 v_\lambda \approx ( & 4.063,2.504,1,0.372,6.269,0.588,0.632,0.487,1.931,0.040,0.588,0.235,\\
 &0.160,0.399,3.868,0.399,0.252,1.582,2.504,0.101,0.094,0.308,0.016,0.771,\\
 &0.064,0.123,0.160,0.252,0.123,0.025,1.459,1.220,0.632,1)\\
 \end{aligned}
 \end{equation*}

\bibliographystyle{plain}
\bibliography{biblio}

 \end{document}